\documentclass[10pt,a4paper]{article}

\usepackage[slovak, english]{babel}

\usepackage{amsthm}
\usepackage{textcomp}

\newtheorem*{thA}{Theorem A}
\newtheorem*{thB}{Theorem B}
\newtheorem*{thC}{Theorem C}

\newtheorem{theorem}{Theorem}[section]
\newtheorem{claim}[theorem]{Claim}
\newtheorem{lemma}[theorem]{Lemma}

\newtheorem{corollary}[theorem]{Corollary}

\newtheorem{definition}[theorem]{Definition}

\newcommand{\aI}{{\alpha_{1}}}
\newcommand{\aII}{{\alpha_{2}}} 
\newcommand{\aIII}{{\alpha_{3}}}

\newcommand{\half}{\tfrac{1}{2}}

\newcommand{\llangle}{{\langle \! \langle}}
\newcommand{\rrangle}{{\rangle \! \rangle}}

\newcommand{\cH}{{\mathcal H}}

\newcommand{\cK}{{\mathcal K}}

\usepackage{amsfonts}
\usepackage{amsmath}
\usepackage{amssymb} 
\usepackage{mathrsfs}
\usepackage{setspace}
\usepackage{paralist}

\usepackage{geometry}
\usepackage{srcltx}
\usepackage{graphicx}
\usepackage{color}
\usepackage{upref}
\usepackage{enumerate}
\usepackage{latexsym}
\usepackage{amsmath}
\usepackage{amsfonts}
\usepackage{amssymb}
\usepackage[ansinew]{inputenc}
\usepackage{varioref}
\usepackage{subfig} 
\usepackage[final, notcite,notref]{showkeys} 
\usepackage{verbatim} 
\usepackage{soul}
\usepackage{complexity}

\usepackage{changebar}
\nochangebars

\usepackage{bm}
\usepackage{tabularx}
\usepackage{placeins}

\newcommand{\nocontentsline}[3]{}
\newcommand{\tocless}[2]{\bgroup\let\addcontentsline=\nocontentsline#1{#2}\egroup}

\makeatletter
\newcommand{\labitem}[2]{%
\def\@itemlabel{({#1})}
\item
\def\@currentlabel{#1}
\label{#2}}
\makeatother

\makeatletter

\newcommand{\Rmnum}[1]{\expandafter\@slowromancap\romannumeral #1@}
\makeatother

\usepackage{amssymb}
\renewcommand{\qed}{\qquad\hspace*{\fill}$\Box$}

\usepackage{hyperref}

\begin{document}

\begin{titlepage}
\begin{center}

\textsc{\LARGE \bf }\\[3cm]

{\huge \bf The Ramsey number of \\[12pt] mixed-parity cycles II}

\vspace{20mm}

{\Large David G. Ferguson} 
\end{center}

\vspace{40mm}

\abstract{
\noindent 
Denote by $R(G_1, G_2, G_3)$ the minimum integer $N$ such that any three-colouring of the edges of the complete graph on $N$ vertices contains
a monochromatic copy of a graph $G_i$ coloured with colour~$i$ for some $i\in{1,2,3}$.
In a series of three papers of which this is the second, we consider the case where $G_1, G_2$ and $G_3$ are cycles of mixed parity. Here and in the previous paper, we consider~$R(C_n,C_m,C_{\ell})$, where $n$ and $m$ are even and $\ell$ is odd.
Figaj and \L uczak determined an asymptotic result for 
this case, which we improve upon to give an exact result. We prove that for~$n,m$ and $\ell$ sufficiently large
$$R(C_n,C_m,C_\ell)=\max\{2n+m-3, n+2m-3, \half n +\half m + \ell - 2\}.$$
The proof of this result is mostly contained within the first paper in this series, however, in the case that the longest cycle is of odd length, we require an additional technical result, the proof of which makes up the majority of this paper.

}

\end{titlepage}

\renewcommand{\baselinestretch}{1.25}\small\normalsize
\setcounter{page}{2}

For graphs $G_1,G_2,G_3$, the Ramsey number $R(G_1,G_2,G_3)$ is the smallest integer~$N$ such that every edge-colouring of the complete graph on~$N$ vertices with up to three colours, results in the graph having, as a subgraph, a copy of~$G_{i}$ coloured with colour~$i$ for some~$i$. 
We consider the case when $G_1,G_2$ and $G_3$ are~cycles, specifically cycles of mixed parity. For a history of this problem see~\cite{DF1}.

Defining $\llangle x \rrangle$ to be the largest even number not greater than~$x$ and $\langle x \rangle$ to be the largest odd number not greater than~$x$, Figaj and \L uczak~\cite{FL2008} proved that, for all $\aI,\aII,\aIII>0$,
\begin{align*}
&\text{(i)}\ R(C_{\llangle \aI n \rrangle },C_{\llangle \aII n \rrangle },C_{\langle \aIII n\rangle  }) = \max \{2\aI+\aII, \aI+2\aII, \half\aI  + \half\aII +\aIII \}n +o(n),\\
&\text{(ii)}\ R(C_{\llangle \aI n \rrangle },C_{\langle \aII n \rangle  },C_{\langle \aIII n\rangle  }) = \max \{4\aI,\aI+2\aII, \aI  +2\aIII \}n +o(n),
\end{align*}
as $n\rightarrow \infty$.

In \cite{DF1}, improving on their result, in the case when exactly one of the cycles is of odd length, we considered the following result:

\phantomsection
\hypertarget{thA}
\phantomsection
\begin{thA}
\label{thA}

For every $\alpha_{1}, \alpha_{2}, \alpha_{3}>0$ such that $\aI \geq \aII$, there exists a positive integer $n_{A}=n_{A}(\aI,\aII,\aIII)$ such that, for $n> n_{A}$,
 \begin{align*}
 R(C_{\llangle \alpha_{1} n \rrangle},C_{\llangle \alpha_{2} n \rrangle}, C_{\langle \alpha_{3} n \rangle }) = \max\{ 2\llangle \alpha_{1} n \rrangle \!+\! \llangle \alpha_{2} n \rrangle  \!-\!\text{\:}3,\text{\:}\half\llangle  \alpha_{1} n \rrangle  \!+\! \half\llangle \alpha_{2} n \rrangle  \!+\! \langle \alpha_{3} n \rangle \!-\! \text{\:}2\}.
\end{align*}
\end{thA}
In order to prove Theorem~A, we required a new Ramsey-stability result for so called connected-matchings.
Owing to the length of the proof of this stability result, part of it was postponed to this paper. 
 The stability result, referred to as Theorem~B, will be restated in the next section, after we have given the necessary definitions. 

The complementary case, that is, the case when
exactly one of the cycles is of even length is considered in
 \cite{DF3}, where we prove the following result, which again improves upon the corresponding result of Figaj and \L uczak:

\phantomsection
\hypertarget{thC}
\phantomsection
\begin{thC}
\label{thC}

For every $\alpha_{1}, \alpha_{2}, \alpha_{3}>0$, there exists a positive integer $n_{C}=n_{C}(\aI,\aII,\aIII)$ such that, for $n> n_{C}$,
 \begin{align*}
 R(C_{\llangle \alpha_{1} n \rrangle},C_{\langle \alpha_{2} n \rangle}, C_{\langle \alpha_{3} n \rangle }) = \max\{
4\llangle \aI n \rrangle,
\llangle \aI n \rrangle+2\langle \aII n \rangle,
\llangle \aI n \rrangle+2\langle \aIII n \rangle\}-3
.
 \end{align*}
\end{thC}

\section {Connected-matching stability result}
\label{s:struct}

Before proceeding to restate state Theorem~B, 
we remind the reader of a few concepts and structures from~\cite{DF1}.

We define a \textit{matching} to be collection of pairwise vertex-disjoint edges. We will sometimes abuse terminology and, where appropriate, refer to a matching by its vertex set rather than its edge set. Recall also that we call a matching with all its vertices in the same component of~$G$ a \textit{connected-matching} and that a connected-matching is called \textit{odd} if the component containing the matching also contains an odd cycle. Note that we call a connected-matching with all its edges contained in a monochromatic component of~$G$ a \textit{monochromatic connected-matching}. 

We say that a graph~$G=(V,E)$ on~$N$ vertices is $a$-\emph{almost-complete} for $0\leq a\leq N-1$ if its minimum degree~$\delta(G)$ is at least $(N-1)-a$. Observe that, if~$G$ is $a$-almost-complete and $X\subseteq V$, then $G[X]$ is also $a$-almost-complete.We say that a graph~$G$ on~$N$ vertices is $(1-c)$-\emph{complete} for $0\leq c\leq 1$ if it is $c(N-1)$-almost-complete, that is, if $\delta(G)\geq (1-c)(N-1)$. Observe that, for $c\leq \half$, any $(1-c)$-complete graph is connected.

We say that a bipartite graph~$G=G[U,W]$ is $a$-\emph{almost-complete} if every $u\in U$ has degree at least $|W|-a$ and every $w\in W$ has degree at least $|U|-a$. Notice that, if~$G[U,W]$ is $a$-almost-complete and $U_1\subseteq U, W_1\subseteq W$, then $G[U_1,W_1]$ is $a$-almost-complete. We say that a bipartite graph~$G=G[U,W]$ is $(1-c)$-\emph{complete} if every $u\in U$ has degree at least $(1-c)|W|$ and every $w\in W$ has degree at least $(1-c)|U|$. Again, notice that, for $c< \half$, any $(1-c)$-complete bipartite graph $G[U,W]$ is connected, provided that~$U,W\neq \emptyset$.

We say that a graph~$G$ on~$N$ vertices is $c$-\emph{sparse} for $0<c<1$ if its maximum degree is at most $c(N-1)$. We say a bipartite graph~$G=G[U,W]$ is $c$-\emph{sparse} if every $u\in U$ has degree at most $c|W|$ and every vertex $w\in W$ has degree at most $c|U|$.

Given a three-coloured graph~$G$, we refer to the first, second and third colours as red, blue and green respectively and use $G_1, G_2, G_3$ to refer to the monochromatic spanning subgraphs of $G$. That is, $G_1$ (resp. $G_2, G_3$) has the same vertex set as~$G$ and includes, as an edge, any edge which (in~$G$) is coloured red (resp. blue, green). If~$G_1$ contains the edge $uv$, we say that $u$ and $v$ are \textit{red neighbours} of each other in $G$. Similarly, if $uv\in E(G_2)$, we say that $u$ and $v$ are \textit{blue neighbours} and, if $uv\in E(G_3)$, we say that that $u$ and $v$ are \textit{green neighbours}.

\begin{definition}
\label{d:H}

For $x_{1}, x_{2}, c_1,c_2$ positive, $\gamma_1,\gamma_2$ colours, let $\cH(x_{1},x_{2}, c_1,c_2, \gamma_1,\gamma_2)$ be the class of edge-multicoloured graphs defined as follows: 

A given two-multicoloured graph $H=(V,E)$ belongs to~$\cH$ if its vertex set can be partitioned into $X_{1}\cup X_{2}$ such that
\begin{itemize}
\item[(i)] $|X_{1}|\geq x_{1}, |X_{2}|\geq x_{2} $;
\item[(ii)] $H$ is $c_1$-almost-complete; and
\item[(iii)] defining $H_1$ to be the spanning subgraph induced by the colour $\gamma_1$ and $H_2$ to be the subgraph induced by the colour $\gamma_2$,
\begin{itemize}
\item[(a)] $H_1[X_{1}]$ is $(1-c_2)$-complete and $H_2[X_{1}]$ is $c_2$-sparse,
\item[(b)] $H_2[X_1,X_2]$ is $(1-c_2)$-complete and $H_1[X_1,X_2]$ is $c_2$-sparse.
\end{itemize}
\end{itemize}
\end{definition}

\begin{figure}[!h]
\vspace{-2mm}
\centering
\includegraphics[height=22mm, page=1]{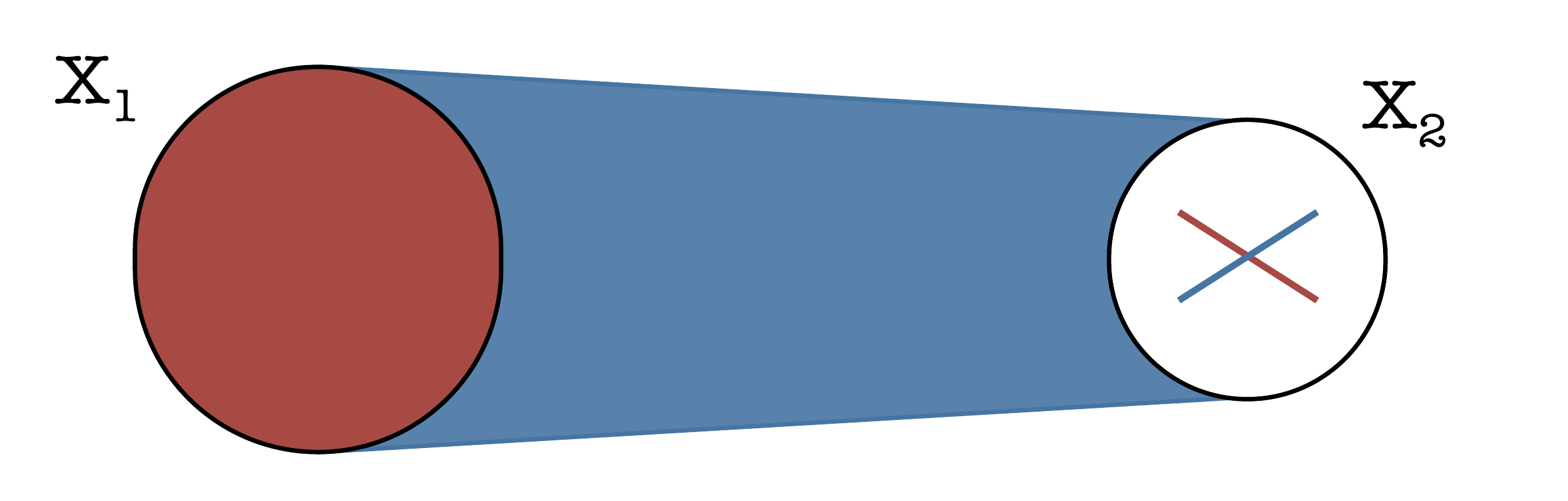}
\vspace{0mm}\caption{$H\in\cH(x_1,x_2,c_1,c_2,red,blue)$.}
\end{figure}

\begin{definition}
\label{d:K}

For $x_{1}, x_{2},x_{3}, c$ positive, let $\cK(x_{1},x_{2}, x_3,c)$ be the class of edge-multicoloured graphs defined as follows: 

A given three-multicoloured graph $H=(V,E)$ belongs to~$\cK$ if its vertex set can be partitioned into $X_{1}\cup X_{2}\cup X_{3}$ such that
\begin{itemize}
\item[(i)] $|X_{1}|\geq x_{1}, |X_{2}|\geq x_{2} , |X_{3}|\geq x_{3} $;
\item[(ii)] $H$ is $c$-almost-complete; and
\item[(iii)] all edges present in $H[X_1,X_3]$ are  coloured exclusively red,
all edges present in $H[X_2,X_3]$ are  coloured exclusively blue, 
and all edges present in $H[X_3]$ are  coloured exclusively green. 
\end{itemize}
\end{definition}

\begin{definition}
\label{d:J}

For $x_{1}, x_{2},y_{1}, y_{2}, z, c$ positive, let $\cK^*(x_{1},x_{2}, y_1, y_2, z, c)$ be the class of edge-multicoloured graphs defined as follows: 

A given three-multicoloured graph $H=(V,E)$ belongs to~$\cK^*$, if its vertex set can be partitioned into $X_{1}\cup X_{2}\cup Y_{1}\cup Y_{2}$ such that
\begin{itemize}
\item[(i)] $|X_{1}|\geq x_{1}, |X_{2}|\geq x_{2} , |Y_{1}|\geq y_{2}, |Y_{2}|\geq y_{2}, |Y_1|+|Y_2|\geq z$;
\item[(ii)] $H$ is $c$-almost-complete; and
\item[(iii)] all edges present in $H[X_1,Y_1]$ and $H[X_2,Y_2]$ are coloured exclusively red,
all edges present in $H[X_1,Y_2]$ and $H[X_2,Y_1]$ are  coloured exclusively blue and
all edges present in $H[X_1,X_2]$ and $H[Y_1,Y_2]$ are  coloured exclusively green.
\end{itemize}
\end{definition}

\begin{figure}[!h]
\centering{
\mbox{
\hspace{-8mm}
{~}\quad
{\includegraphics[width=64mm, page=4]{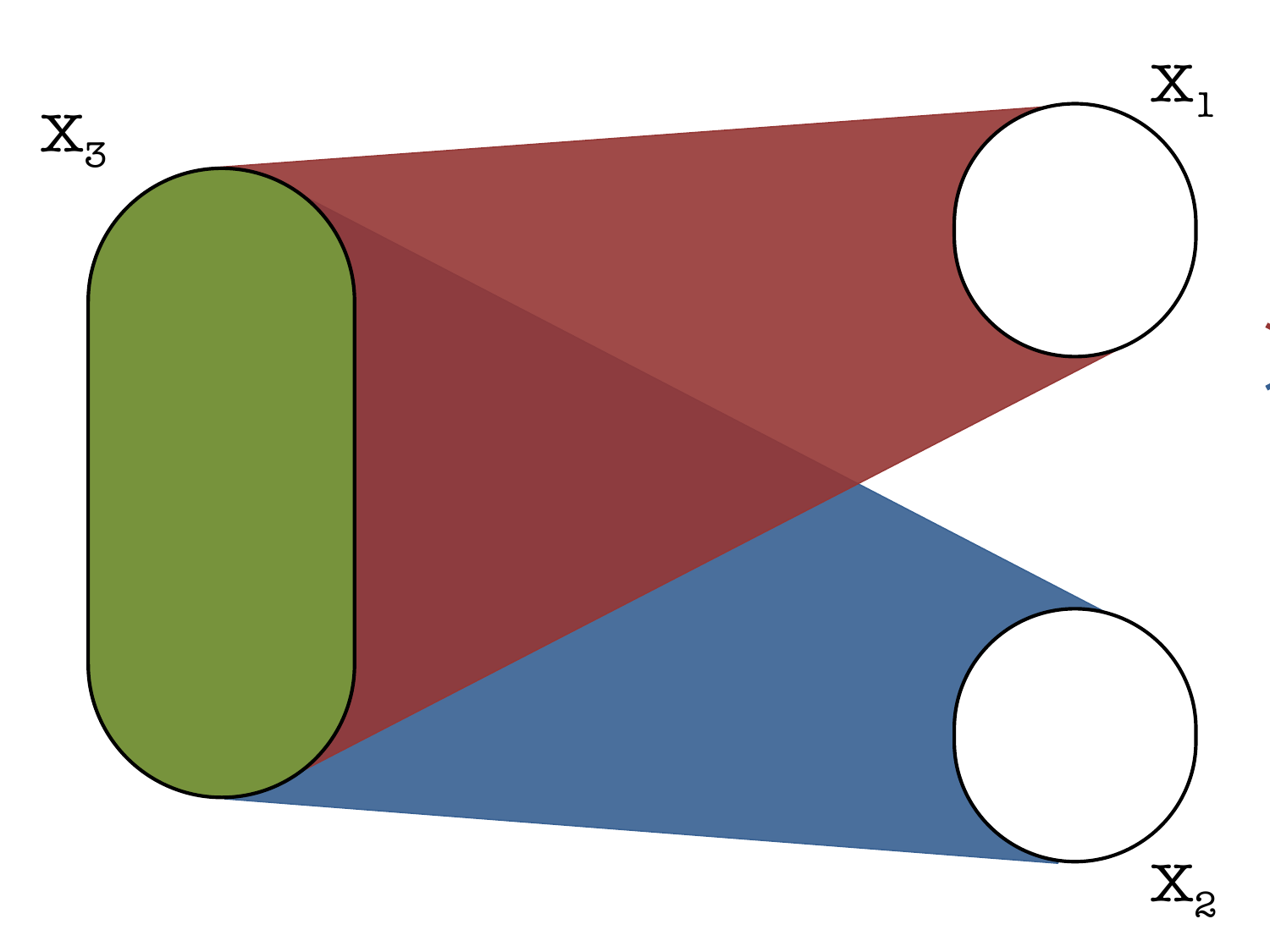}}\quad\quad\quad
{\includegraphics[width=64mm, page=5]{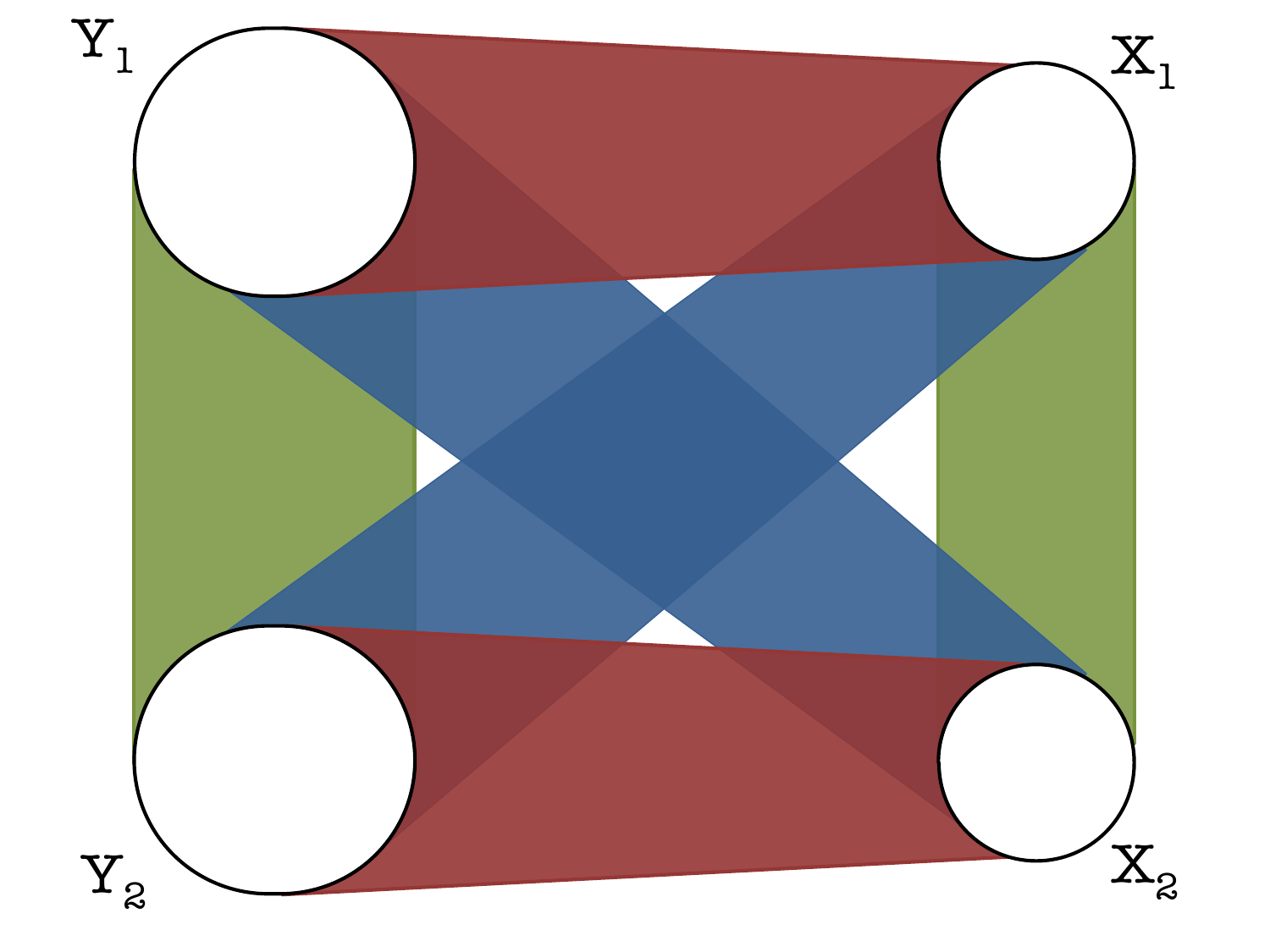}}}}
\caption{$H\in \cK(x_1,x_2,x_3,c)$ and $H\in \cK^{*}(x_1,x_2,y_1,y_2,c)$.}   
\label{th1a}
\end{figure}

Having defined the coloured structures, we now restate Theorem~B:

\FloatBarrier
\phantomsection
\hypertarget{thB}
\phantomsection
\begin{thB}
\label{thB}
\label{th:stab}
For every $\alpha_{1},\alpha_{2},\alpha_{3}>0$ 
such that $\aI\geq \aII$, letting $$c=\max\{2\aI+\aII, \half\aI+\half\aII+\aIII\},$$
there exists $\eta_{B}=\eta_{B}(\aI,\aII,\aIII)$ and $k_{B}=k_{B}(\aI,\aII,\aIII,\eta)$ 
such that, for every $k>k_{B}$ and 
every~$\eta$ such that $0<\eta<\eta_{B}$, every 
three-multicolouring of~$G$, a $(1-\eta^4)$-complete graph on 
$$(c-\eta)k\leq K\leq(c-\half\eta)k$$ vertices,  results in the graph containing at least one of the following:
\vspace{-2mm}
\begin{itemize}
\item [(i)]	a red connected-matching on at least $\aI
 k$ vertices;
\item [(ii)]  a blue connected-matching on at least $\aII k$ vertices;
\item [(iii)]  a green odd connected-matching on at least $\aIII k$ vertices;
\item [(iv)]  
disjoint subsets of vertices $X$ and $Y$ such that $G[X]$ contains a two-coloured spanning subgraph~$H$ from $\cH_{1}\cup\cH_{2}$ and $G[Y]$ contains a two-coloured spanning subgraph $K$ from $\cH_{1}\cup\cH_{2}$, where
\vspace{-2mm}
\begin{align*}
\cH_1=&\cH\left((\aI-2\eta^{1/32})k,(\half\aII-2\eta^{1/32})k,3\eta^4 k,\eta^{1/32},\text{red},\text{blue}\right),\text{ } 
\\ \cH_2=&\cH\left((\aII-2\eta^{1/32})k,(\half\aI-2\eta^{1/32})k,3\eta^4 k,\eta^{1/32},\text{blue},\text{red}\right);
\end{align*}
 \item [(v)]  a subgraph from 
 \vspace{-2mm}
 $$\cK\left((\half\aI-14000\eta^{1/2})k, (\half\aII-14000\eta^{1/2})k, (\aIII-68000\eta^{1/2})k, 4\eta^4 k \right);$$
\item [(vi)]  a subgraph from $\cK_1^{*}\cup \cK_2^*$, where
\vspace{-2mm}
\begin{align*}
\cK_{1}^{*}=\cK^*\big((\half\aI-97\eta^{1/2})k, (\half\aI-97&\eta^{1/2})k, (\half\aI+102\eta^{1/2})k, \\&(\half\aI+102\eta^{1/2})k, (\aIII-10\eta^{1/2})k, 4\eta^4 k\big),\hphantom{l}
\end{align*}
\vspace{-12mm}
\begin{align*}
\cK_{2}^{*}=\cK^*\big((\half\aI-97\eta^{1/2})k, (\half\aII-97\eta^{1/2})k, (\tfrac{3}{4}&\aIII-140\eta^{1/2})k, \\& 100\eta^{1/2}k, (\aIII-10\eta^{1/2})k, 4\eta^4 k\big).
\end{align*}
\end{itemize}
Furthermore, 
\vspace{-2mm}
\begin{itemize}
\item[(iv)] occurs only if $\aIII\leq \tfrac{3}{2}\aI+\half\aII+14\eta^{1/2}$ with $H_1,H_2\in \cH_1$ unless $\aII\geq\aI-\eta^{1/16}$; and
\item[(v)] and (vi) occur only if $\aIII\geq \tfrac{3}{2}\aI+\half\aII-10\eta^{1/2}$.
\end{itemize}
\end{thB}

\setlength{\parskip}{0.1in plus 0.025in minus 0.025in}
This result forms a partially strengthened analogue of the main technical result of the paper of Figaj and \L uczak~\cite{FL2008}. In that paper, Figaj and \L uczak considered a similar graph but on slightly more than $\max\{2\aI+\aII, \half\aI+\half\aII+\aIII\}k$ vertices and proved the existence of a connected-matching, whereas we consider a graph on slightly fewer vertices and prove the existence of either a monochromatic connected-matching or a particular structure.

In \cite[Section 7]{DF1}, we proved Theorem~B, in the case that $\aI\geq\aIII$, the remainder of this paper is dedicated to proving it in the complementary case. 

\section{{Definitions and tools}}
\label{s:pre1}
In this section, we give a few definitions and summarise results that we shall use later in our proof.

Given a graph~$G$, we say $u,v\in V(G)$ are \textit{connected} (in~$G$) if there exists a path in~$G$ between $u$ and~$v$. The graph itself is said to be \textit{connected} if any pair of vertices are connected. By extension, given a subgraph~$H$ of~$G$, we say~$H$ is \textit{connected} if, given any pair $u,v \in V(H)$, there exists a path in~$H$ between $u$ and $v$ and say that~$H$ is \textit{effectively-connected} if, given any pair $u,v \in V(H)$, there exists a path in~$G$ between $u$ and $v$.

\cbstart
A \textit{connected-component} of a graph~$G$ is a maximal connected subgraph. A subgraph of~$H$, a subgraph of~$G$, is an \textit{effectively-connected-component} or \textit{effective-component} of~$H$ if it is a maximal effectively-connected subgraph of~$H$. Thus the effective-components of $H$ are restrictions of the components of~$G$ to $H$.
\cbend

Given a multicoloured graph~$G$, we say that two vertices $u$ and $v$ belong to the same \textit{monochromatic component} of~$G$ if they belong to the same component of $G_i$ for some~$i$. Given a subgraph~$H$ of a multicoloured graph~$G$, we say that two vertices $u$ and $v$ belong to the same \textit{monochromatic effective-component} of~$H$ if they belong to the same effective-component of $G_i$ for some $i$. We can thus talk about, for instance, the \textit{red components} of a graph~$G$ or the \textit{red effective-components} of a subgraph~$H$ of~$G$.

We now summarise the results that we shall use later in our proofs. We beginning with Dirac's Theorem, which gives us a minimum-degree condition for Hamiltonicity:

\begin{theorem}[Dirac's Theorem~\cite{Dirac52}]
\label{dirac}
If~$G$ is a graph on~$n\geq3$ vertices such that every vertex has degree at least $\half n$, then~$G$ is Hamiltonian, that is, $G$ contains a cycle of length exactly $n$.
\end{theorem}

We also make use of the theorem of Erd\H{o}s and Gallai:

\begin{theorem}[\cite{ErdGall59}]
\label{th:eg}
Any graph on~$K$ vertices with at least~$\frac{1}{2}(m-1)(K-1)+1$ edges, where $3\leq m \leq K$, contains a cycle of length at least~$m$.
\end{theorem}

Observing that a cycle on $m$ vertices contains a connected-matching on at least~$m-1$ vertices, the following is an immediate consequence of the above.

\begin{corollary}
\label{l:eg}
For any graph~$G$ on~$K$ vertices and any~$m$ such that $3 \leq m \leq K$, if the average degree $d(G)$ is at least $m$, then~$G$ contains a connected-matching on at least~$m$ vertices.
\end{corollary}

The following decomposition lemma~of Figaj and \L uczak~\cite{FL2008} also follows from the theorem of Erd\H{o}s and Gallai and is crucial in establishing the structure of a graph not containing large connected-matchings of the appropriate parities:

\begin{lemma}[{\cite[Lemma~9]{FL2008}}]
\label{l:decomp}
For any graph~$G$ on~$K$ vertices and any~$m$ such that $3\leq m \leq K$, if no odd component of~$G$ contains a matching on at least~$m$ vertices, then there exists a partition $V=V'\cup V''$ such that
\begin{itemize}
\item [(i)] $G[V']$ is bipartite;
\item [(ii)] every component of $G''=G[V'']$ is odd;
\item [(iii)] $G[V'']$ has at most $\half m |V(G'')|$ edges; and
\item [(iv)] there are no edges in $G[V',V'']$.
\end{itemize}
\end{lemma}

We recall two more results of Figaj and \L uczak. The first, is the main technical result from~\cite{FL2007}. The second from~\cite{FL2008}, allows us to deal with graphs with a \emph{hole}, that is, a subset $W\subseteq V(G)$ such that no  edge of~$G$ lies inside $W$. Note that both of these results can be immediately extended to multicoloured graphs:

\begin{lemma}[{\cite[Lemma~8]{FL2007}}]
\label{l:largeW}
For every $\alpha_{1}, \alpha_{2}, \alpha_{3}>0$ and
 $0<\eta<0.002\min\{\alpha_1^2, \alpha_2^2,\alpha_3^2\}$, there exists $k_{\ref{l:largeW}}=k_{\ref{l:largeW}}(\aI,\aII,\aIII,\eta)$ such that the following holds:

For every $k>k_{\ref{l:largeW}}$ and every $(1-\eta^4)$-complete graph~$G$ on $$K\geq \half\left(\alpha_{1}+\alpha_{2}+\alpha_{3}+\max\{ \alpha_{1}, \alpha_{2}, \alpha_{3} \} +18 \eta^{1/2}\right)k$$ vertices, for every three-colouring of the edges of~$G$, there exists a colour $i\in\{1,2,3\}$ such that~$G_i$ contains a connected-matching on at least $(\alpha_{i}+\eta)k$ vertices.
\end{lemma}

\begin{lemma}[{\cite[Lemma~12]{FL2008}}]
\label{l:hole}
For every $\alpha, \beta>0$, $v\geq0$ and $\eta$ such that
$0<\eta<0.01 \min\{ \alpha,\beta\}$, 
there exists $k_{\ref{l:hole}}=k_{\ref{l:hole}}(\alpha,\beta,v,\eta)$ such that, for every $k>k_{\ref{l:hole}}$, the following holds:

Let $G=(V,E)$ be a graph obtained from a $(1-\eta^4)$-complete graph on $$ K\geq\half\left(\alpha+\beta+\max \{ 2v, \alpha, \beta \} + 6\eta^{1/2} \right)k $$
vertices by removing all edges contained within a subset $W \subseteq V$ of size at most~$vk$. Then, every two-multicolouring of the edges of~$G$ results in {either} a red connected-matching on at least $(\alpha+\eta)k$ vertices {or} a blue connected-matching on at least $(\beta+\eta)k$ vertices.
\end{lemma}

The following pair of lemmas allow us to find large connected-matchings in almost-complete bipartite graphs:

\begin{lemma}[{\cite[Lemma~10]{FL2008}}]
\label{l:ten}
Let $G=G[V_1,V_2]$ be a bipartite graph with bipartition $(V_{1},V_{2})$, where $|V_{1}|\geq|V_{2}|$, which has at least $(1-\epsilon)|V_{1}||V_{2}|$ edges for some $\epsilon$ such that $0<\epsilon<0.01$. Then,~$G$ contains a connected-matching on at least $2(1-3\epsilon)|V_{2}|$ vertices.
\end{lemma}

Notice that, if~$G$ is a $(1-\epsilon)$-complete bipartite graph with bipartition $(V_1,V_2)$, then we may immediately apply the above to find a large connected-matching in~$G$.

\begin{lemma}
\label{l:eleven}
Let $G=G[V_1,V_2]$ be a bipartite graph with bipartition $(V_1,V_2)$. If $\ell$ is a positive integer such that $|V_1|\geq|V_2|\geq \ell$ and~$G$ is $a$-almost-complete for some $a$ such that $0<a/\ell<0.5$, then~$G$ contains a connected-matching on at least $2|V_2|-2a$ vertices.
\end{lemma}

\begin{proof}
Observe that~$G$ is $(1-a/\ell)$-complete. Therefore, since $a/\ell<0.5$,~$G$ is connected. Thus, it suffices to find a matching of the required size. Suppose that we have found a matching with vertex set~$M$ such that $|M|=2k$ for some $k<|V_2|-a$ and consider a vertex $v_2\in V_2\backslash M$. Since $G$ is $a$-almost-complete, $v_2$ has at least $|V_1|-a$ neighbours in~$|V_1|$ and thus at least one neighbour in $v_1\in V_1\backslash M$. Then, the edge $v_1v_2$ can be added to the matching and thus, by induction, we may obtain a matching on $2|V_2|-2a$ vertices.
\end{proof}

We also make use of the following lemma~from~\cite{KoSiSk2}, which is an extension of the two-colour Ramsey result for even cycles and which allows us to find, in any almost-complete two-multicoloured graph on~$K$ vertices, either a large matching or a particular structure.

\begin{lemma}[\cite{KoSiSk2}]
\label{l:SkB}

For every~$\eta$ such that $0<\eta<10^{-20}$, there exists $k_{\ref{l:SkB}}=k_{\ref{l:SkB}}(\eta)$ such that, for every $k>k_{\ref{l:SkB}}$ and every $\alpha,\beta>0$ such that $\alpha \geq \beta \geq 100\eta^{1/2}\alpha$, if $K>(\alpha + \half\beta-\eta^{1/2}\beta)k$ and $G=(V,E)$ is a two-multcoloured $\beta \eta^2 k$-almost-complete graph on $K$ vertices, then at least one of the following occurs:
\begin{itemize}
\item[(i)]~$G$ contains a red connected-matching on at least $(1+\eta^{1/2})\alpha k$ vertices;
\item[(ii)]~$G$ contains a blue connected-matching on at least $(1+\eta^{1/2})\beta k$ vertices;
\item[(iii)] the vertices of~$G$ can be partitioned into three sets $W$, $V'$, $V''$ such that
\begin{itemize}
\item[(a)] $|V'| < (1+\eta^{1/2})\alpha k$, 
$|V''|\leq \half(1+\eta^{1/2})\beta k$,
$|W|\leq \eta^{1/16} k$,
\item[(b)] $G_1[V']$ is $(1-\eta^{1/16})$-complete and $G_2[V']$ is $\eta^{1/16}$-sparse,
\item[(c)] $G_2[V',V'']$ is $(1-\eta^{1/16})$-complete and $G_1[V',V'']$ is $\eta^{1/16}$-sparse;
\end{itemize}
\item[(iv)] we have $\beta > (1-\eta^{1/8})\alpha$ and the vertices of~$G$ can be partitioned into sets $W$, $V'$ and $V''$ such that
\begin{itemize}
\item[(a)] $|V'| < (1+\eta^{1/2})\beta k$,
$|V''|\leq \half(1+\eta^{1/8})\alpha k$,
$|W|\leq \eta^{1/16} k$,
\item[(b)] $G_2[V']$ is $(1-\eta^{1/16})$-complete and $G_1[V']$ is $\eta^{1/16}$-sparse, 
\item[(c)] $G_1[V',V'']$ is $(1-\eta^{1/16})$-complete and $G_2[V',V'']$ is $\eta^{1/16}$-sparse.
\end{itemize}
\end{itemize}
Furthermore, if $\alpha+ \half\beta \geq 2(1+\eta^{1/2})\beta$, then we can replace (i) with
\begin{itemize}
\item[(i')]~$G$ contains a red odd connected-matching on $(1+\eta^{1/2})\alpha k$ vertices.
\end{itemize}
\end{lemma}

We also make use of the following corollary of Lemma~\ref{l:SkB}:

\begin{corollary} 
\label{c:SkBe}
For every $0<\epsilon <10^{-12}$, there exists $k_{\ref{c:SkBe}}=k_{\ref{c:SkBe}}(\epsilon)$ such that, for every $k\geq k_{\ref{c:SkBe}}$, if $K>(1-\epsilon)k$ and $G=(V,E)$ is a two-multicoloured $\tfrac{27}{8}\epsilon^4 k$-almost-complete graph, then~$G$ contains at least one of the following:
\begin{itemize}
\item[(i)] a red connected-matching on $(\tfrac{2}{3}-7\epsilon^{1/8})k$ vertices;
\item[(ii)] a blue connected-matching on $(\tfrac{2}{3}-7\epsilon^{1/8})k$ vertices.
\end{itemize}
\end{corollary}

\begin{proof}
See \cite{DF1}.
\end{proof}

It is a well-known fact that either a graph is connected or its complement is. The following 
three results are
simple extensions of this fact for two-coloured almost-complete graphs, all of which can be immediately extended to two-multicoloured almost-complete graphs. 

\begin{lemma}\label{l:dgf0}
For every $\eta$ such that $0<\eta<1/3$ and every $K\geq 1/\eta$, if $G=(V,E)$ is a two-coloured $(1-\eta)$-complete graph on~$K$ vertices and~$F$ is its largest monochromatic component, then $|F|\geq (1-3\eta)K$.
\end{lemma}

\begin{proof}
See \cite{DF1}.
\end{proof}

The following lemmas form analogues of the above, the first concerns the structure of two-coloured almost-complete graphs with one hole and the second concerns the structure of two-coloured almost-complete graphs with two holes, that is, bipartite graphs. 

\begin{lemma}
\label{l:dgf1} 

For every $\eta$ such that $0<\eta<1/20$ and every $K\geq 1/\eta$, the following holds. For~$W$, any subset of~$V$ such that $|W|,|V\backslash W|\geq 4\eta^{1/2}K$, let $G_{W}=(V,E)$ be a two-coloured graph obtained from~$G$, a $(1-\eta)$-complete graph on~$K$ vertices with vertex set~$V$ by removing all edges contained entirely within~$W$. Let~$F$ be the largest monochromatic component of $G_W$ and define the following two sets:
\begin{align*} 
 W_{r}&= \{\text{$w \in W$ : $w$ has red edges to all but at most  $3\eta^{1/2} K$ vertices in $V \backslash W$}\}; 
\\ 
 W_{b}&=\{ w \in W : w \text{ has blue edges to all but at most } 3\eta^{1/2} K \text{ vertices in } V \backslash W\}.
\end{align*}
Then, at least one of the following holds:
\begin{itemize}
\item [(i)] $|F|\geq (1-2\eta^{1/2})K$; and
\item [(ii)] $|W_{r}|,|W_{b}|>0$.
\end{itemize}
\end{lemma}

\begin{proof}
See \cite{DF1}.
\end{proof}

\begin{lemma}
\label{l:twoholes}
For every $\eta$ such that $0<\eta<0.1$ and $K \geq 2/\eta$, the following holds: Suppose $G=(V,E)$ is a two-multicoloured graph obtained from an $(1-\eta)$-complete graph on $K$ vertices with $V=A\cup B$ and $|A|,|B|\geq 6\eta K$ by removing all edges contained completely within $A$ and all edges contained completely within $B$. Let~$F$ be the largest monochromatic component of~$G$ and define the following sets:\begin{align*}
  A_{r}&=\{ a\in A :  
             a \text{ has red edges to all but at most } 4\eta K \text{ vertices in } B \}; \\
  A_{b}&=\{ a\in A : 
            a \text{ has blue edges to all but at most } 4\eta K \text{ vertices in } B \}; \\ 
  B_{r}&=\{ b\in B : 
           b \text{ has red edges to all but at most } 4\eta K \text{ vertices in } A \};  \\ 
  B_{b}&=\{ b\in B : 
          b \text{ has blue edges to all but at most } 4\eta K \text{ vertices in } A \}. 
\end{align*}

Then, at least one of the following occurs:
\begin{itemize}
\item[(i)] $|F|\geq (1-7\eta)K$;
\item[(ii)] $A,B$ can be partitioned into $A_{1}\cup A_{2}, B_{1}\cup B_{2}$ such that $|A_{1}|, |A_{2}|, |B_{1}|, |B_{2}|\geq 3\eta K$ and all edges present between $A_{i}$ and $B_{j}$ are red for $i=j$, blue for $i\neq j$;
\item[(iii)] $|A_{r}|,|A_{b}|>0$; 
\item[(iv)]  $|B_{r}|,|B_{b}|>0$.
\end{itemize}
\end{lemma}

\begin{proof}
Suppose $|F|< (1-7\eta)K$. Then, without loss of generality, $A$ can be partitioned into $A_{1}\cup A_{2}$, with $|A_1|,|A_2|\geq 3\eta K$, such that $A_1$ and $A_2$ are in different red components. Then, there exists no triple $(a_{1},b,a_{2})$ with $a_{1}\in A_{1}, b\in B, a_{2} \in A_{2}$ and both $a_{1}b$ and $ba_{2}$ coloured red. Thus, we may partition $B$ into $B_{1}\cup B_{2}$ such that there are no red edges present in $G[A_{1},B_{1}]$ or $G[A_{2},B_{2}]$. 

Since~$G$ is $(1-\eta)$-complete, given any subsets $A'\subseteq A, B'\subseteq B$ every vertex in $A'$ has degree at least $|B'|-\eta K$ in $G[A',B']$ and every vertex in $B'$ has degree at least $|A'|-\eta K$ in $G[A',B']$. Thus, if $|B_{1}|,|B_{2}|\geq 3\eta K$, $G[A_1,B_1]$ and $G[A_2,B_2]$ each have a single blue component. Therefore, there can be no blue edges present in $G[A_{1},B_{2}]$ or $G[A_{2},B_{1}]$, giving rise to case (ii).

Thus, without loss of generality, we may assume that  $|B_{1}|<3\eta K$. Then, every vertex in $A_{2}$ is a vertex of $A_{b}$, in which case either every vertex $a\in A_{1}$ has a blue edge to $B$, leading to case (i), or there exists some $a\in A_{1}$ such that $a\in A_{r}$, giving rise to case (iii), thus completing the proof. 

Note that exchanging the roles of A and B above leads to case (iv) in place of case (iii).
\end{proof}

\section{Proof of the stability result  -- Part II}
\label{s:stabp2}

In \cite[Section 7]{DF1}, we proved Theorem~B in the case that $\aI\geq\aIII$, here we we consider the case when $\aIII \geq \aI \geq \aII$. 
We wish to prove that any three-multicoloured $(1-\eta^4)$-complete graph on  slightly fewer than $\max\{2\aI+\aII, \half\aI +\half\aII +\aIII\}k$ vertices will have a red connected-matching on at least~$\alpha_{1}k$ vertices, a blue connected-matching on at least~$\alpha_{2}k$ vertices, a green odd connected-matching on at least~$\alpha_{3}k$ vertices or will have a particular coloured structure.

Thus, given $\aI, \aII, \aIII$ such that $\aIII \geq \aI \geq \aII$, we set $$c=\max\{2\aI+\aII, \half\aI +\half\aII +\aIII\},$$
choose 
$$\eta<\eta_{B2}=\min \left\{
\frac{1}{10^{5}}, \frac{\aII}{10^{24}}, \bigg(\frac{\aII}{100}\bigg)^8, \bigg(\frac{\aII}{1200\aI}\bigg)^2
\right\}$$
and consider $G=(V,E)$, a $(1-\eta^4)$-complete graph on~$K\geq 72/\eta$ vertices, where
$$(c - \eta)k \leq K \leq (c - \tfrac{\eta}{2})k$$ for some integer $k>k_{B2}$
, where $k_{B2}=k_{B2}(\aI,\aII,\aIII,\eta)$ will be defined implicitly during the course of the proof, in that, on a finite number of occasions, we will need to bound~$k$ below in order to apply results from Section~\ref{s:pre1}.

Note that, since $\aIII \geq \aI \geq \aII$, the largest forbidden connected-matching is green and odd. By scaling, we may assume that $2\geq\aIII\geq1\geq \aI\geq\aII$. Thus, $G$ is $3\eta^4k$-almost-complete, as is $G[X]$, for any $X\subset V$. 
We begin by noting that we can use Lemma~\ref{l:largeW} to obtain either a red connected-matching on $\aI k$ vertices, a blue connected-matching on $\aII k$ vertices or a green connected-matching of almost the required size. Note, however, that this green connected-matching need not be odd. Indeed, the graph has $$|V|=vk\geq(c-\eta)k=(\half\aI+\half\aII+(v-\half\aI-\half\aII-9\eta^{1/2})+9\eta^{1/2})k$$ vertices and, since $\aI\geq\aII$ and $\eta\leq(\frac{\aII}{20})^2$, we have 
\begin{align*}
(v-\half\aI-\half\aII-9\eta^{1/2})
&\geq(c-\eta-\half\aI-\half\aII-9\eta^{1/2}) \\
&\geq\max\{2\aI+\aII,\half\aI+\half\aII+\aIII\}-\half\aI-\half\aII-10\eta^{1/2}\\
&\geq\max\{\tfrac{3}{2}\aI+\half\aII,\aIII\}-10\eta^{1/2}\geq \aI\geq\aII.
\end{align*}
Thus, since $\eta\leq 0.002\min\{\alpha_1^2,\alpha_2^2,\alpha_3^2\}$, by Lemma~\ref{l:largeW}, provided $k\geq k_{\ref{l:largeW}}(\aI,\aII,v-\half\aI+\half\aII-9\eta^{1/2}, \eta)$,~$G$ contains either a red connected-matching on at least $\aI k$ vertices, a blue connected-matching on at least $\aII k$ vertices or a green connected-matching on at least 
\begin{equation}
\label{281}
|V|-(\half\aI+\half\aII+9\eta^{1/2})k\geq(\max\{\tfrac{3}{2}\aI+\half\aII, \aIII\}-10\eta^{1/2})k
\end{equation} 
vertices.
Lemma~\ref{l:decomp} gives a decomposition of the green-graph $G_3$ into its bipartite and non-bipartite parts and in doing so gives a decomposition of the vertices of~$G$ into $X\cup Y\cup W$ such that there are no green edges between $X\cup Y$ and $W$ or within $X$ or $Y$. Choosing such a decomposition which maximises $|X\cup Y|$, results in $G_3[X\cup Y]$ being the union of the bipartite green components of~$G$ and $G_3[W]$ being the union of the non-bipartite green components of~$G$. In what follows, we consider the vertices of~$G$ to have been thus partitioned. We will also assume that $|X|\geq |Y|$ and will write $\widehat{V}$ for $X\cup W$ and~$w$ for $|W|/k$. By~(\ref{281}), we may assume that the largest green connected-matching in~$G$ spans at least $(\max\{\tfrac{3}{2}\aI+\half\aII, \aIII\}-10\eta^{1/2})k$ vertices and distinguish three cases: 

\begin{itemize}
\item[(C)] the largest green connected-matching~$F$ is not odd and $w\geq 7\eta^{1/2}$;
\item[(D)] the largest green connected-matching~$F$ is not odd and $w\leq 7\eta^{1/2}$; 
\item[(E)] the largest green connected-matching~$F$ is odd.
\end{itemize}

Within each case, we will, when necessary, distinguish between the two possible forms taken by $c$, that is, between $c=2\aI+\aII$ and $c=\half\aI+\half\aII+\aIII$: 

\phantomsection
\begin{itemize}
\labitem{II\dag}{IId}The first possibility arises only when $\aIII \leq \tfrac{3}{2}\aI+\half\aII$, in which case we may assume that~$F$ spans at least $(\tfrac{3}{2}\aI+\half\aII-10\eta^{1/2})k$ vertices.
\labitem{II\ddag}{IIdd}The second possibility arises only when $\aIII \geq \tfrac{3}{2}\aI+\half\aII$, in which case we may assume that that~$F$ spans at least $(\aIII-10\eta^{1/2})k$ vertices.
\end{itemize}

\subsection*{Case C: {\rm Largest green connected-matching is not odd and $w\geq 7\eta^{1/2}$.}} 

Suppose that we have $c=2\aI+\aII$. Then~$F$ spans at least $(\tfrac{3}{2}\aI+\half\aII-10\eta^{1/2})k$ vertices and is assumed to not be contained in an odd component of~$G$. Thus, by the decomposition, we have
\begin{subequations}
\begin{align}
\label{c1}
|X|\geq|Y| &\geq (\tfrac{3}{4}\aI+\tfrac{1}{4}\aII-5\eta^{1/2})k,\\
\label{c1c}
|X|\geq\frac{K-|W|}{2}&\geq (\aI+\half\aII-\half\eta-\half w)k,\\
\label{c2}
 |W|=K-|X|-|Y| &\leq  (\half\aI+\half\aII+10\eta^{1/2})k.
\end{align}
\end{subequations}
\begin{subequations}
From~(\ref{c1}) and~(\ref{c1c}) we obtain
\begin{align}
\label{c3}
\,\,\,\,\,|\widehat{V}|&=|X|+|W|\geq(\tfrac{3}{4}\aI+\tfrac{1}{4}\aII +w- 5\eta^{1/2})k, \\
\label{c3a}
\!|\widehat{V}|&=|X|+|W|\geq(\aI+\tfrac{1}{2}\aII +\half w- \half\eta)k.
\end{align}
Now, suppose that 
$|\widehat{V}|\geq \half(\aI+\aII+\max\{2w,{\aI},\aII\}+6\eta^{1/2})k.$
In that case, since $\eta\leq 0.01\aII$, provided $k> k_{\ref{l:hole}}(\aI,\aII,w,\eta)$, we may apply Lemma~\ref{l:hole} to obtain either a red connected-matching on $\aI k$ vertices or a blue connected-matching on $\aII k$ vertices. Therefore, we may assume that 
\begin{equation}
\label{c4}
|\widehat{V}|\leq \half\left(\aI+\aII+\max\{2w,{\aI},\aII\}+6\eta^{1/2}\right)k.
\end{equation}
\end{subequations}
If $w\leq \half \aI$, then together~(\ref{c3a}) and~(\ref{c4}) contradict our assumption that $w\geq 7\eta^{1/2}$. Thus, we may assume that $w\geq \half \aI$. In that case,~(\ref{c3}) and~(\ref{c4}) give $$(\tfrac{3}{4}\aI+\tfrac{1}{4}\aII +w -5\eta^{1/2})k\leq |\widehat{V}| \leq (\half \aI +\half \aII+w+3\eta^{1/2})k,$$
which together with~(\ref{c1}) and~(\ref{c2}) gives
\begin{align*}
(\tfrac{3}{4}\aI+\tfrac{1}{4}\aII -5\eta^{1/2})k \leq  |X| & \leq (\half \aI + \half \aII +3\eta^{1/2})k, \\
(\tfrac{3}{4}\aI+\tfrac{1}{4}\aII -5\eta^{1/2})k \leq  |Y| & \leq (\half \aI + \half \aII +3\eta^{1/2})k, \\
(\aI-7\eta^{1/2})k \leq  |W| & \leq (\half \aI + \half \aII +10\eta^{1/2})k.
\end{align*}

The condition for $|X|$ above gives a contradiction unless $\aI\leq\aII+32\eta^{1/2}$. So, recalling that $\aII\leq\aI$, we may obtain
\begin{equation}
\label{c6}
\left.
\begin{aligned}
\quad\quad\quad\quad\quad\quad\quad\quad\quad(\aI -13\eta^{1/2})k \leq  |X| & \leq (\aI + 3\eta^{1/2})k,  \quad\quad\quad\quad\quad\quad\quad \\
(\aI-13\eta^{1/2})k \leq  |Y| & \leq (\aI +3\eta^{1/2})k, \\
(\aI -7\eta^{1/2})k \leq  |W| & \leq (\aI + 10\eta^{1/2})k. 
\end{aligned}
\right\}
\end{equation}
Suppose instead that $\aIII\geq\tfrac{3}{2}\aI+\half\aII$. Then, by~(\ref{281}),~$F$ spans at least $(\aIII-10\eta^{1/2})k$ vertices. Recall that we assume that~$F$  is not contained in an odd component of~$G$, thus, by the decomposition, we have
\begin{subequations}
\begin{align}
\label{c7} |X|\geq|Y| &\geq (\half\aIII-5\eta^{1/2})k,\\
\label{c7c}
|X|\geq\frac{K-|W|}{2}&\geq (\aI+\half\aII-\half\eta-\half w)k,\\
\label{c8} |W| &\leq  (\half\aI+\half\aII+10\eta^{1/2})k.
\end{align}
\end{subequations}
From~(\ref{c7}) and~(\ref{c7c}), we obtain 
\begin{subequations}
\begin{align}
\label{c9} \!\!|\widehat{V}|=|X|+|W|\geq(\half \aIII +w- 5\eta^{1/2})k, \\
\label{c9a}
\quad\,\,\,\,\,|\widehat{V}|=|X|+|W|\geq(\aI+\tfrac{1}{2}\aII +\half w- \half\eta)k.
\end{align}
Again, we may assume that 
\begin{equation}
\label{c10}
|\widehat{V}|\leq \half\left(\aI+\half\aII+\max\{2w,\aI,\aII\}+6\eta^{1/2}\right)k
\end{equation}
\end{subequations}
 since, otherwise, we may apply Lemma~\ref{l:hole} to obtain either a red connected-matching on $\aI k$ vertices or a blue connected-matching on $\aII k$ vertices. Again, we may assume that $w\geq \half \aI$ since, otherwise, together~(\ref{c9a}) and~(\ref{c10}) contradict our assumption that $w\geq 7\eta^{1/2} k$. Then,~(\ref{c9}) and~(\ref{c10}) give $$(\half \aIII +w -5\eta^{1/2})k\leq |\widehat{V}| \leq (\half \aI +\half \aII+w+3\eta^{1/2})k,$$
which, together with~(\ref{c7}) and~(\ref{c8}) gives
\begin{align*}
(\half \aIII -5\eta^{1/2})k \leq  |X| & \leq (\half \aI + \half \aII +3\eta^{1/2})k, \\
(\half \aIII -5\eta^{1/2})k \leq  |Y| & \leq (\half \aI + \half \aII +3\eta^{1/2})k, \\
(\aIII-\half\aI-\half\aII -7\eta^{1/2})k \leq  |W| & \leq (\half \aI + \half \aII +10\eta^{1/2})k.
\end{align*}
 Then, recalling that $\aIII \geq \tfrac{3}{2}\aI+\half\aII$ and that $\aII\leq\aI$, in order to avoid a contradiction, we have $\aI\leq \aII+32\eta^{1/2}$ and obtain
\begin{equation}
\label{c11}
\left.
\begin{aligned}
\quad\quad\quad\quad\quad\quad\quad\quad\quad(\aI -13\eta^{1/2})k \leq  |X| & \leq (\aI + 3\eta^{1/2})k,  \quad\quad\quad\quad\quad\quad\quad \\
(\aI-13\eta^{1/2})k \leq  |Y| & \leq (\aI +3\eta^{1/2})k, \\
(\aI -7\eta^{1/2})k \leq  |W| & \leq (\aI + 10\eta^{1/2})k. 
\end{aligned}
\right\}
\end{equation}
Considering~(\ref{c6}) and~(\ref{c11}), we see that we have obtained the same set of bounds irrespective of the form taken by $c$. Thus, in what follows, we consider both possibilities together.

Recall that, under the decomposition, there are no green edges contained within~$X$ or~$Y$. Then, since~$G$ is $3\eta^4k$-almost-complete, provided $k>k_{\ref{c:SkBe}}(\eta)$, we may apply Corollary~\ref{c:SkBe} 
to each of $G[X]$ and $G[Y]$, thus finding that each contains a monochromatic connected-matching on at least $(\tfrac{2}{3}\aI-8\eta^{1/8})k$ vertices. Thus, provided $\eta<(\aI/120)^8$, we may assume that each of $X$ and $Y$ contain a monochromatic connected-matching on at least $\tfrac{3}{5}\aI k$ vertices. Referring to these matchings as $M_1\subseteq G[X]$ and $M_2\subseteq G[Y]$, we consider three subcases:

\begin{itemize}
\item[(i)] $M_{1}$ and $M_{2}$ are both red;
\item[(ii)] $M_{1}$ and $M_{2}$ are both blue; 
\item[(iii)] $M_{1}$ and $M_{2}$ are different colours.
\end{itemize}

The proof in the first subcase is identical to that of Case B.i in \cite{DF1}, the proof in the second subcase is identical to that of Case B.ii in \cite{DF1} and the proof in the third subcase is identical to that of Case B.iii in \cite{DF1} with the overall result being that~$G$ contains either a red connected-matching on $\aI k$ vertices or a blue connected-matching on $\aII k$ vertices.

\subsection*{Case D: {\rm Largest green connected-matching is not odd and $w\leq 7\eta^{1/2}$.}}

Suppose that $c=2\aII+\aII$. Then, since $\eta\leq\eta_{B2}$, provided $k\geq k_{\ref{l:hole}}(\aI,\aII,w,\eta)$, we obtain bounds on the sizes of $X$ and $Y$ as follows:
\begin{align*}
(\aI +\half\aII-4\eta^{1/2})k \leq  |X| & \leq (\aI + \half\aII+3\eta^{1/2})k, \\
(\aI+\half\aII-12\eta^{1/2})k \leq  |Y| & \leq (\aI +\half\aII+3\eta^{1/2})k. 
\end{align*}
Suppose instead that $c=\half\aI+\half\aII+\aIII$. Then, since $\eta\leq\eta_{B2}$, provided $k\geq k_{\ref{l:hole}}(\aI,\aII,w,\eta)$, we obtain bounds on the sizes of $X$ and $Y$ as follows:
\begin{subequations}
\begin{align}
\label{d1} (\tfrac{1}{4}\aI+\tfrac{1}{4}\aII+\half\aIII-4\eta^{1/2})k \leq  |X| & \leq (\aI +\half\aII+3\eta^{1/2})k, \quad\quad\quad \\
\label{d1a} (\aIII-\half\aI-12\eta^{1/2})k \leq  |Y| & \leq (\aI+\half\aII+3\eta^{1/2})k.
\end{align}
\end{subequations}
Note that the inequalities in~(\ref{d1}) give a contradiction unless $\aIII\leq\tfrac{3}{2}\aI+\half\aII+14\eta^{1/2}$. Since $\aIII\geq\tfrac{3}{2}\aI+\half\aII$, we obtain
\begin{align*}
(\aI +\half\aII-4\eta^{1/2})k \leq  |X| & \leq (\aI + \half\aII+3\eta^{1/2})k, \\
(\aI+\half\aII-12\eta^{1/2})k \leq  |Y| & \leq (\aI +\half\aII+3\eta^{1/2})k. 
\end{align*}
Then, in either case, since $\eta\leq \eta_{B2}$, provided $k>k_{\ref{l:SkB}}(\eta^{1/2})$, we may apply Lemma~\ref{l:SkB} (with $\alpha=\aI, \beta=\aII$) to each of $X$ and $Y$ to find that each contains a red connected-matching on at least~$\aI k$ vertices or a blue connected-matching on at least~$\alpha_{2}k$ vertices {or} has a structure belonging to one of the following classes as a subgraph:
\begin{align*}
\cH_1&=\cH\left((\aI-2\eta^{1/32})k,(\half\aII-2\eta^{1/32})k,3\eta^4 k,\eta^{1/32},\text{red},\text{blue}\right);\text{ } 
\\ \cH_2&=\cH\left((\aII-2\eta^{1/32})k,(\half\aI-2\eta^{1/32}) k,3\eta^4 k,\eta^{1/32},\text{blue},\text{red}\right),
\end{align*}
with the latter case occurring only if $\aII\geq\aI-\eta^{1/16}$.

\subsection*{Case E: {\rm Largest green connected-matching is odd.}}

Recall, from (\ref{281}), that $F$, the largest green connected-matching in~$G$, spans at least $(\max\{\tfrac{3}{2}\aI+\half\aII,\aIII\}-10\eta^{1/2})k$ vertices. We now consider the case when this connected-matching is contained in an odd component of~$G$. 

Thus far, we have made extensive use of the decomposition of Figaj and \L uczak described in Lemma~\ref{l:decomp}. However, in this case, it is necessary to consider an alternative (and somewhat more complicated) decomposition:

We begin by partitioning the vertices of~$G$ into $L\cup P \cup Q$ as follows. We let~$L$ be the vertex set of~$F$. Then, for each $v\in V\backslash L$, if there exists a green edge between $v$ and~$L$, we assign $v$ to~$P$; otherwise, we assign $v$ to~$Q$.

Suppose there exists a green edge $mn$ in~$F$ and distinct vertices $p_1,p_2\in P$ such that $mp_1$ and $np_2$ are both coloured green. Then, we can replace $mn$ with $mp_1$ and $np_2$, contradicting the maximality of~$F$. Thus, after discarding at most one edge from $G[L,P]$ for each edge of~$F$, we may assume that given an edge $uv$ in the matching, at most one of~$u$ or~$v$ has a green edge to~$P$. We may therefore partition~$L$ into $M\cup N$ such that each edge of the matching belongs to $G[M,N]$ and there are no green edges in $G[N,P]$. Observe also that, by maximality of~$F$, there can be no green edges within $G[P]$ or $G[P,Q]$.

In summary, we have a partition $M\cup N\cup P \cup Q$ such that
\phantomsection
\begin{itemize}
\labitem{E1}{E1-0} $M\cup N$ is the vertex set of~$F$ and every edge of~$F$ belongs to $G[M,N]$;
\labitem{E2}{E2-0} every vertex in~$P$ has a green edge to~$M$;
\labitem{E3}{E3-0} there are no green edges in $G[N,P]$, $G[M,Q]$, $G[N,Q]$, $G[P,Q]$ or $G[P]$.
\end{itemize}

\begin{figure}[!h]
\centering
\includegraphics[width=64mm, page=5]{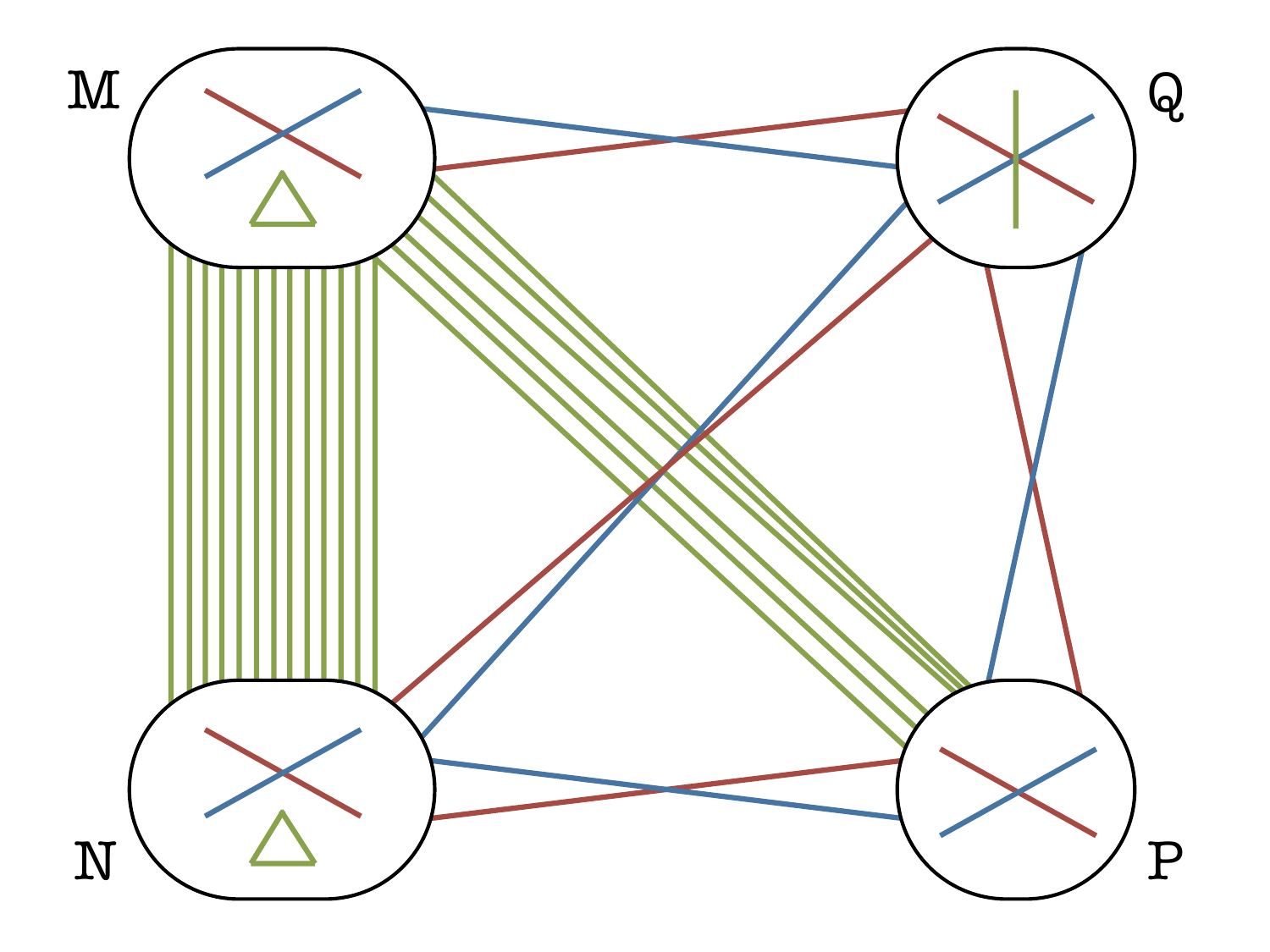}
\vspace{-2.5mm}\caption{Decomposition into $M\cup N\cup P\cup Q$.}
\label{fige}
\end{figure}

Note that, since~$G$ was assumed to be $(1-\eta^4)$-complete and also $3\eta^4 k$-almost-complete, having discarded the green edges described above, provided $k\geq 1/\eta^{4}$, we may now assume that the (new) graph is $(1-\tfrac{3}{2}\eta^4)$-complete and also $4\eta^4 k$-almost-complete. In what follows, on a number of occasions, we will discard vertices from $M\cup N\cup P\cup Q$ but will continue to refer the parts of the partition as $M,N,P$ and~$Q$. The discarded vertices remain in the graph and will be considered later. We need to take care to account of this when considering the sizes of $V(G), M, N, P, Q,$ etc. 

Recalling (\ref{IId}), in the case that $c=2\aI+\aII$, we have $\aIII\leq \tfrac{3}{2} \aI+\half\aII$. Then, since $V(F)=M\cup N$ and $|M|+|N|+|P|+|Q|=K$, we have
\begin{subequations}
\begin{align}
\label{E0a} (\tfrac{3}{4}\aI+\tfrac{1}{4}\aII-5\eta^{1/2})k & \leq  |M|\,,\,|N| \leq \half\aIII k,  \\
\label{E0b} (2\aI+\aII-\aIII-\eta)k & \leq  |P|+|Q| \leq (\half\aI+\half\aII+5\eta^{1/2})k.
\end{align} 
\end{subequations}
The inequalities in~(\ref{E0a}) give a contradiction unless $\aIII\geq\tfrac{3}{2}\aI+\tfrac{1}{2}\aII-10\eta^{1/2}$. Then, since $\aIII \leq \tfrac{3}{2}\aI+\tfrac{1}{2}\aII$, we may re-write~(\ref{E0b}) as
\begin{equation}
\label{E1}
\tag{\ref*{E0b}$^{\prime}$}
(\half\aI+\half\aII-\eta)k  \leq  |P|+|Q| \leq (\half\aI+\half\aII+5\eta^{1/2})k.
\end{equation}
Recalling (\ref{IIdd}), in the case that $c=\half\aI+\half\aII+\aIII$, we have $\aIII\geq \tfrac{3}{2} \aI+\half\aII$. Then, since $V(F)=M\cup N$ and $|M|+|N|+|P|+|Q|=K$, we have
\begin{subequations}
\begin{align}
\label{E2} (\half\aIII-5\eta^{1/2})k & \leq  |M|\,,\,|N| \leq \half\aIII k, \\
\label{E3} (\half\aI+\half\aII-\eta)k & \leq  |P|+|Q| \leq (\half\aI+\half\aII+5\eta^{1/2})k.
\end{align} 
\end{subequations}
We will proceed considering the two possible situations together, assuming that 
\begin{equation}
\label{A3BIG}
\aIII\geq\tfrac{3}{2}\aI+\tfrac{1}{2}\aII-10\eta^{1/2}.
\end{equation}
Comparing~(\ref{E0a}) to~(\ref{E2}) and~(\ref{E1}) to~(\ref{E3}), we will assume that
\begin{subequations}
\begin{align}
\label{E4a-0}\tag{E4a} (\max\{\tfrac{3}{4}\aI+\tfrac{1}{4}\aII,\half\aIII\}-5\eta^{1/2})k & \leq  |M|\,,\,|N| \leq \half\aIII k, \\
\label{E4b-0}\tag{E4b} (\half\aI+\half\aII-\eta)k & \leq  |P|+|Q| \leq (\half\aI+\half\aII+5\eta^{1/2})k.
\end{align} 
\end{subequations}
and distinguish between three possibilities:
\begin{itemize}
\item[(i)] $|P|\leq 95\eta^{1/2}k$;
\item[(ii)] $|Q|\leq 95\eta^{1/2}k$;
\item[(iii)] $|P|, |Q| \geq 95\eta^{1/2}k$.
\end{itemize}

\subsection*{Case E.i: {\rm $|P|\leq 95\eta^{1/2}k$.}}

In this case, we disregard~$P$, and, recalling that $L=M\cup N$, consider $G[L \cup Q]$. From~(\ref{E4a-0}) and~(\ref{E4b-0}), we have
\begin{subequations}
\begin{align}
\label{E6}\tag{E4a$^\prime$} (\max\{\tfrac{3}{2}\aI+\tfrac{1}{2}\aII,\aIII\}-10\eta^{1/2})k & \leq  |L| \leq \aIII k, \\
\label{E7}\tag{E4b$^\prime$} (\half\aI+\half\aII-96\eta^{1/2})k & \leq  |Q| \leq (\half\aI+\half\aII+5\eta^{1/2})k.
\end{align} 
\end{subequations}
By (\ref{E3-0}), we know that all edges in $G[L,Q]$ are coloured red or blue.

\begin{figure}[!h]
\centering
\vspace{-8mm}
\includegraphics[width=64mm, page=6]{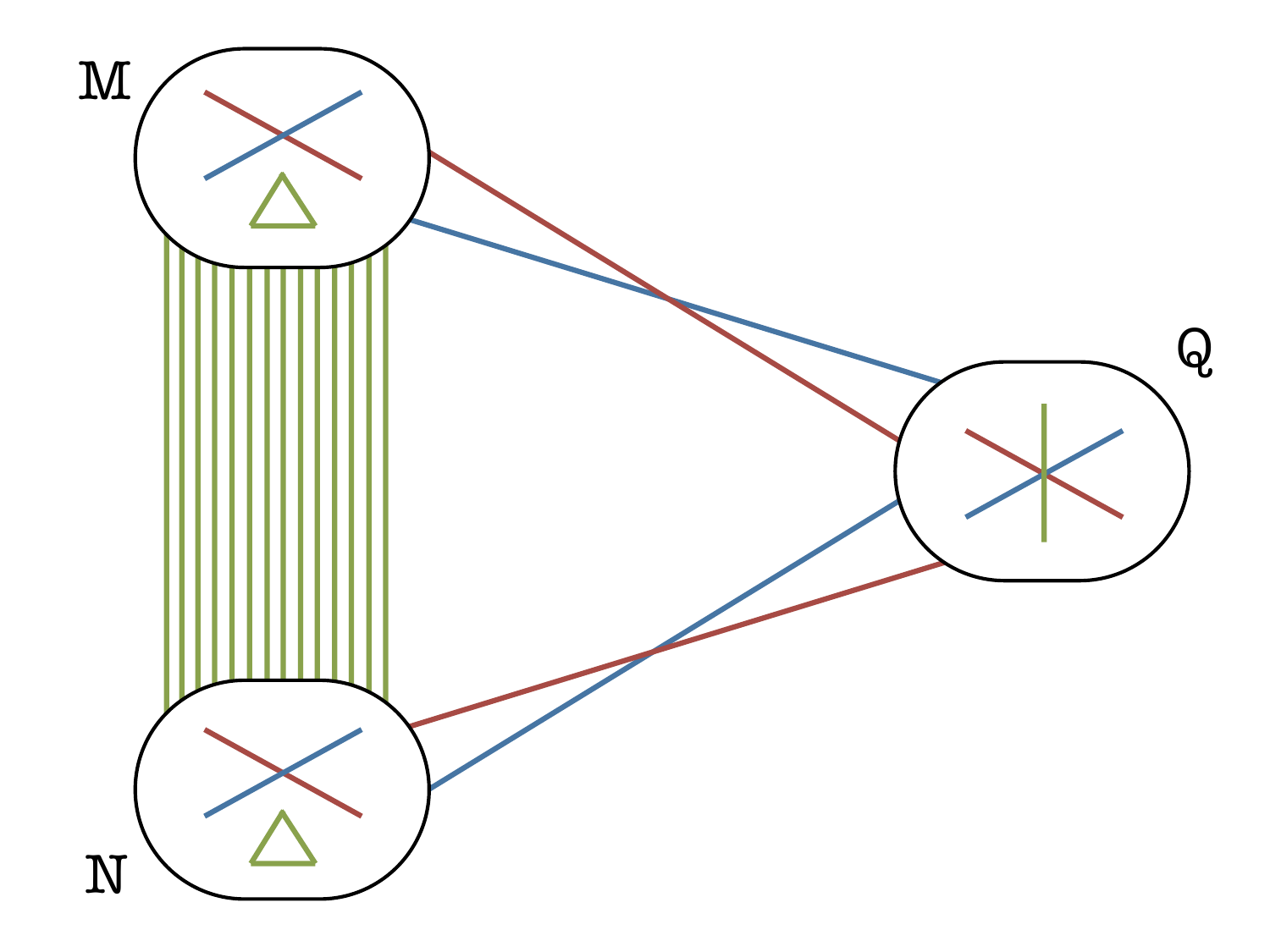}
\vspace{-1mm}\caption{Decomposition in Case E.i.}
  
\end{figure}

Observe that, provided $\eta<(1/200)^2$, we have $$|L|+|Q|\geq\left(\max\{2\aI+\aII,\half\aI+\half\aII+\aIII\}-106\eta^{1/2}\right)k\geq\tfrac{3}{4} K.$$
Thus, since~$G$ is $(1-\tfrac{3}{2}\eta^4)$-complete, $G[L\cup Q]$ is $(1-2\eta^4)$-complete. Also, provided $\eta<10^{-5}$, we have $|L|,|Q|\geq 18\eta^{1/2} \left(|L|+|Q|\right)$. Thus, since $2\eta^4\leq 3\eta^{1/2}$, provided $K\geq 2/\eta^{1/2}$, we may apply Lemma~\ref{l:twoholes}, giving rise to four cases:

\begin{itemize}
\item[(a)] $G[L,Q]$ contains a monochromatic component on at least $(1-21\eta^{1/2})|L\cup Q|\geq|L\cup Q|-63\eta^{1/2}k$ vertices;
\item[(b)] $L,Q$ can be partitioned into $L_{1}\cup L_{2}, Q_{1}\cup Q_{2}$ such that $|L_{1}|, |L_{2}|, |Q_{1}|, |Q_{2}|\geq 9\eta^{1/2}|L\cup Q|\geq9\eta^{1/2}k$ and all edges present between $L_{i}$ and $Q_{j}$ are red for $i=j$, blue for $i\neq j$;
\item[(c)] there exist vertices $v_r, v_b\in L$ such that $v_r$ has red edges to all but $12\eta^{1/2}|L\cup Q|\leq 36\eta^{1/2}k$ vertices in~$Q$ and $v_b$ has blue edges to all but $12\eta^{1/2}|L\cup Q|\leq 36\eta^{1/2}k$ vertices in~$Q$;
\item[(d)] there exist vertices $v_r, v_b\in Q$ such that $v_r$ has red edges to all but $12\eta^{1/2}|L\cup Q|\leq 36\eta^{1/2}k$ vertices in~$L$ and $v_b$ has blue edges to all but $12\eta^{1/2}|L\cup Q|\leq 36\eta^{1/2}k$ vertices in~$L$.
\end{itemize}

\subsection*{Case E.i.a: {\rm $G[L\cup Q]$ has a large monochromatic component.}}

Recall that we assume that $F$, the largest green connected-matching in $G$, spans at least $(\max\{\tfrac{3}{2}\aI+\half\aII,\aIII\}-10\eta^{1/2})k$ vertices and is contained in an odd component of $G$. We have a partition of $V(G)$ into $L\cup P \cup Q$ such that $|P| \leq  95\eta^{1/2}k$,
\begin{subequations}
\begin{align}
\label{Eia1}\tag{E4a$^\prime$}
(\max\{\tfrac{3}{2}\aI+\tfrac{1}{2}\aII,\aIII\}-10\eta^{1/2})k  \leq  |L| & \leq \aIII k, \\
\label{Eia2}\tag{E4b$^\prime$}
(\half\aI+\half\aII-96\eta^{1/2})k \leq  |Q| & \leq (\half\aI+\half\aII+5\eta^{1/2})k. 
\end{align} 
\end{subequations}

Recalling that $L=M\cup N$, by (\ref{E3-0}), all edges present in $G[L,Q]$ are coloured red or blue. Additionally, in this case, we assume that $G[L,Q]$ contains a monochromatic component on at least $|L\cup Q|-63\eta^{1/2}k$ vertices. Suppose this large monochromatic component is red, then
\phantomsection
\begin{itemize}
\labitem{E5}{E5} $G[L,Q]$ has a red component on at least $|L\cup Q|-63\eta^{1/2} k$ vertices.
\end{itemize}
We consider the largest red matching $R$ in $G[L,Q]$ and, thus, partition~$L$ into $L_1\cup L_2$ and~$Q$ into $Q_1\cup Q_2$ where $L_1=L\cap V(R)$, $L_2=L\backslash L_1$, $Q_1=Q\cap V(R)$ and $Q_2=Q\backslash Q_1$. 
By maximality of $R$, all edges present in $G[L_2,Q_2]$ are coloured exclusively blue. Notice that, since $\eta\leq(\aI/100)^2$, we have, by~(\ref{Eia1}) and~(\ref{Eia2}), $|L|\geq|Q|$. Thus, since $|L_1|=|Q_1|$, we have $|L_2|\geq|Q_2|$ and so, in order to avoid having a blue connected-matching on at least $\aII k$ vertices, by Lemma~\ref{l:eleven}, we have $|Q_2|\leq (\half\aII+\eta^{1/2})k$ and, therefore, also $|L_1|=|Q_1|=|Q|-|Q_2|\geq (\half\aI-97\eta^{1/2})k.$

\begin{figure}[!h]
\centering
\includegraphics[width=64mm, page=60]{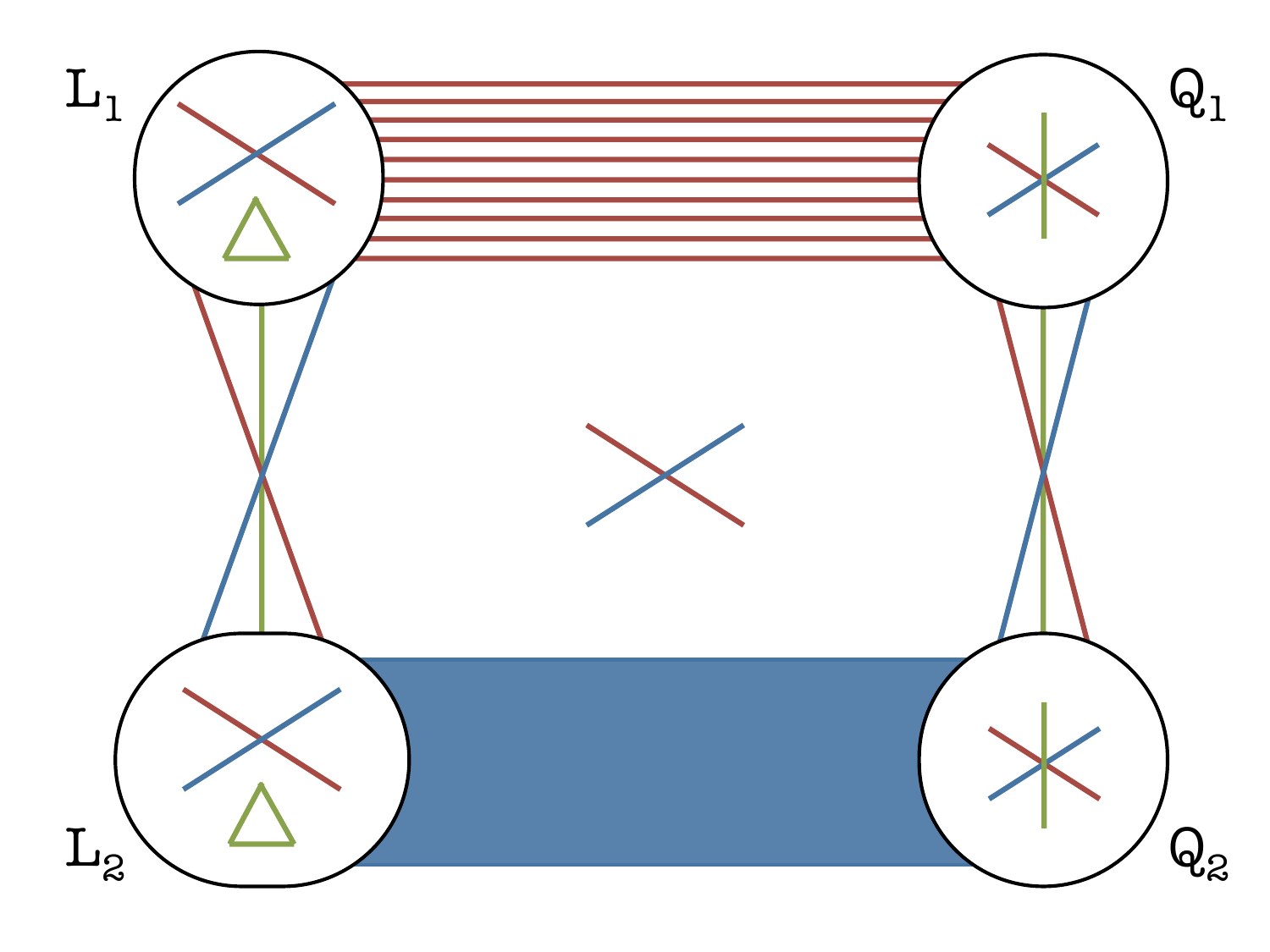}
\vspace{-3mm}\caption{Decomposition into $L_1\cup L_2\cup Q_1 \cup Q_2$ in Case E.i.a.}
  
\end{figure}

Also, by (E5),
in order to avoid having a red connected-matching on at least $\aI k$ vertices, we may assume that $|L_1|=|Q_1|\leq(\half\aI+64\eta^{1/2})k$.
Finally, we have $|Q_2|=|Q|-|Q_1|\geq (\half\aII-160\eta^{1/2})k.$ 

In summary, we have $|L_1|=|Q_1|$,
\begin{subequations}
\begin{align}
\label{Eia3} (\half\aII-97\eta^{1/2})k & \leq |Q_1|\leq (\half\aI+64\eta^{1/2})k,\\
\label{Eia4} (\half\aII-160\eta^{1/2})k  &\leq |Q_2|\leq (\half\aII+\eta^{1/2})k.
\end{align}
Note that, since $\eta\leq (\aI/10000)^2$, by~(\ref{Eia1}),~(\ref{Eia2}), we have $|L|\geq |Q|+3000\eta^{1/2}k$ and thus, since $|Q_1|=|L_1|$, also have
\begin{equation}
\label{Eia5}
|L_2|\geq|Q_2|+3000\eta^{1/2}k.
\end{equation}
Recalling that $|L_2|=|L|-|L_1|=|L|-|Q_1|$, since $\eta\leq (\aI/10000)^2$, considering~(\ref{Eia1}) and~(\ref{Eia3}), we also have 
\begin{equation}
\label{Eia6}
|L_2|\geq|Q_1|+3000\eta^{1/2}k.
\end{equation}
\end{subequations}
Equations~(\ref{Eia5}) and~(\ref{Eia6}) are crucial to the argument that follows since they provide us with the spare vertices we will need in order to establish the coloured structure of~$G$. In what follows, we will show that, after possibly discarding some vertices from each of $L_1, L_2, Q_1$ and $Q_2$, we may assume that all edges present in $G[L,Q_1]$ are coloured exclusively red and that all edges present in $G[L,Q_2]$ are coloured exclusively blue. This is done in three steps, the first dealing with $G[L_2,Q_1]$,  the second dealing with $G[L_1,Q_2]$ and the third dealing with $G[L_1,Q_1]$. Similar arguments will appear many times throughout the remainder of the proof of Theorem B. Note that, in what follows, we mostly omit floors and ceilings for the sake of clarity of presentation, we may do this since we are free to increase $k$ 
where necessary.

\begin{claim}
\label{claimLQ}
We may discard at most $842\eta^{1/2}k$ vertices from each of $L_1$ and $Q_1$, at most $161\eta^{1/2}k$ vertices from $L_2$ and at most $260\eta^{1/2}k$ vertices from $Q_2$ such that all remaining in $G[L,Q_1]$ are coloured exclusively red and all edges present in $G[L,Q_2]$ are coloured exclusively blue.
\end{claim}

\begin{proof}
We begin by considering the blue graph. Observing that $G[L_2,Q_2]$ contains a blue connected-matching of size close to $\aII k$ and that there are `spare' vertices in $L_2$. Thus, we note that there can only be few blue edges in $G[L_2,Q_1]$. Indeed, suppose there exists a blue matching $B_S$ on at least  $322\eta^{1/2}k$ vertices in $G[L_2,Q_1]$. 

\begin{figure}[!h]
\centering
\includegraphics[width=64mm, page=10]{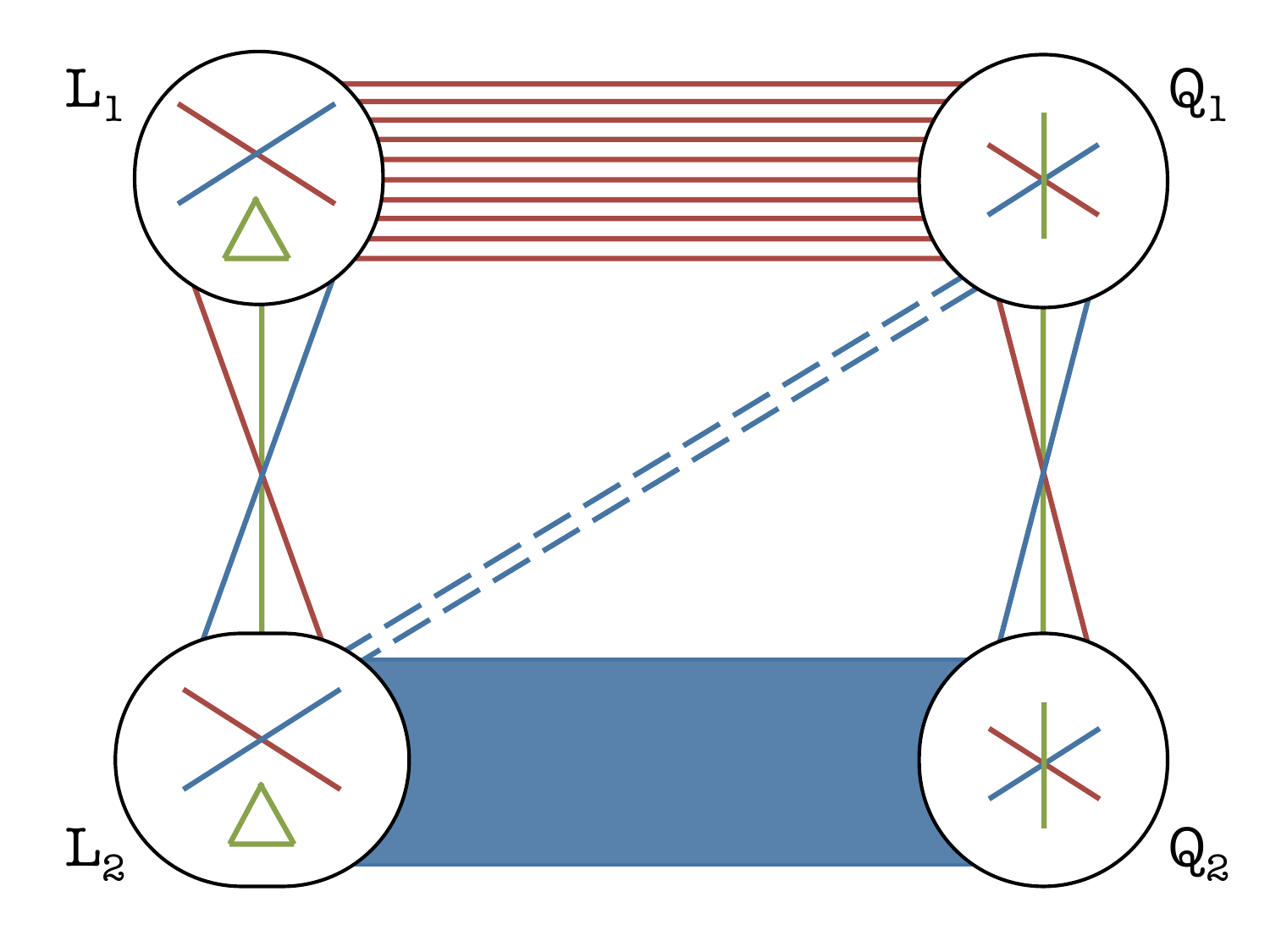}
\vspace{-2mm}\caption{The blue matching $B_S$.}
  
\end{figure}

Then, by (\ref{Eia6}), letting $\widetilde{L}={L_2\backslash V(B_S)}$, we have $|\widetilde{L}|\geq|Q_2|\geq(\half\aII-160\eta^{1/2})k$. Thus, by Lemma~\ref{l:eleven}, 
there exists a blue connected-matching $B_L$ on at least $(\aII -322\eta^{1/2})k$ vertices in $G[\widetilde{L},Q_2]$ which shares no vertices with $B_S$. Notice that, since $G_2[L_2,Q_2]$ is $4\eta^4 k$-almost-complete, $L_2\cup Q_2$ forms a single blue component in~$G$ and, thus, $B_S\cup B_L$ forms a blue connected-matching on $\aII k$ vertices. Therefore, no such matching as $B_S$ can exist. So, after discarding at most $161\eta^{1/2} k$ vertices from each of $Q_1$ and $L_2$, we may assume that all edges present in $G[L_2,Q_1]$ are coloured exclusively red.  

In order to retain the equality $|L_1|=|Q_1|$ and the property that every vertex in $L_1$ belongs to an edge of $R$, we discard from $L_1$ each vertex whose $R$-mate in $Q_1$ has already been discarded. Recalling~(\ref{Eia3})--(\ref{Eia6}), we now have 
\begin{align*}
|L_2|&\geq|Q_1|+2800\eta^{1/2}k, &
  |L_1|=|Q_1|&\geq (\half\aI-258\eta^{1/2})k,\\
|L_2|&\geq|Q_2|+2800\eta^{1/2}k, &  |Q_2|&\geq (\half\aII-160\eta^{1/2})k.
\end{align*}

We now consider the red graph. Since all edges in $G[L_2,Q_1]$ are coloured exclusively red, any two vertices in $Q_1$ have a common red neighbour in $L_2$. Thus, since every vertex in $L_1$ has a red neighbour in $Q_1$, we know that $G[L_1\cup Q_1]$ has a single effective red component. Suppose, then, that there exists a red matching $R_S$ on at least $520\eta^{1/2}k$ vertices in $G[L_1,Q_2]$. Then, recalling that the matching $R$ spans all the vertices of $G[L_1,Q_1]$, we may construct a red connected-matching on at least $\aI k$ vertices as follows.
%

Observe that there exists a set $R^{-}$ of $260\eta^{1/2}k$ edges belonging to $R$ such that $L_1\cap V(R_S)=L_1\cap V(R^{-})$. Define $R^*=R\backslash R^-$ and $\widetilde{Q}=Q_1\cap V(R^-)$ and consider $G[L_2,\widetilde{Q}]$. Since $|L_2|,|\widetilde{Q}|\geq 260\eta^{1/2} k$ and $G[L_2,\widetilde{Q}]$ is $4\eta^4 k$-almost-complete, by Lemma~\ref{l:eleven}, there exists a red connected-matching $R_T$ on at least $518\eta^{1/2}k$ vertices in $G[L_2,\widetilde{Q}]$. Then, $R^{*}\cup R_S \cup R_T$ is a red-connected-matching in 
$G[L_1,Q_1]\cup G[L_1,Q_2]\cup G[L_2,Q_1]$ on at least $2(|Q_1|-261\eta^{1/2}k)+520\eta^{1/2}k+518\eta^{1/2}k\geq \aI k$ vertices. 

\begin{figure}[!h]
\centering
\includegraphics[width=64mm, page=19]{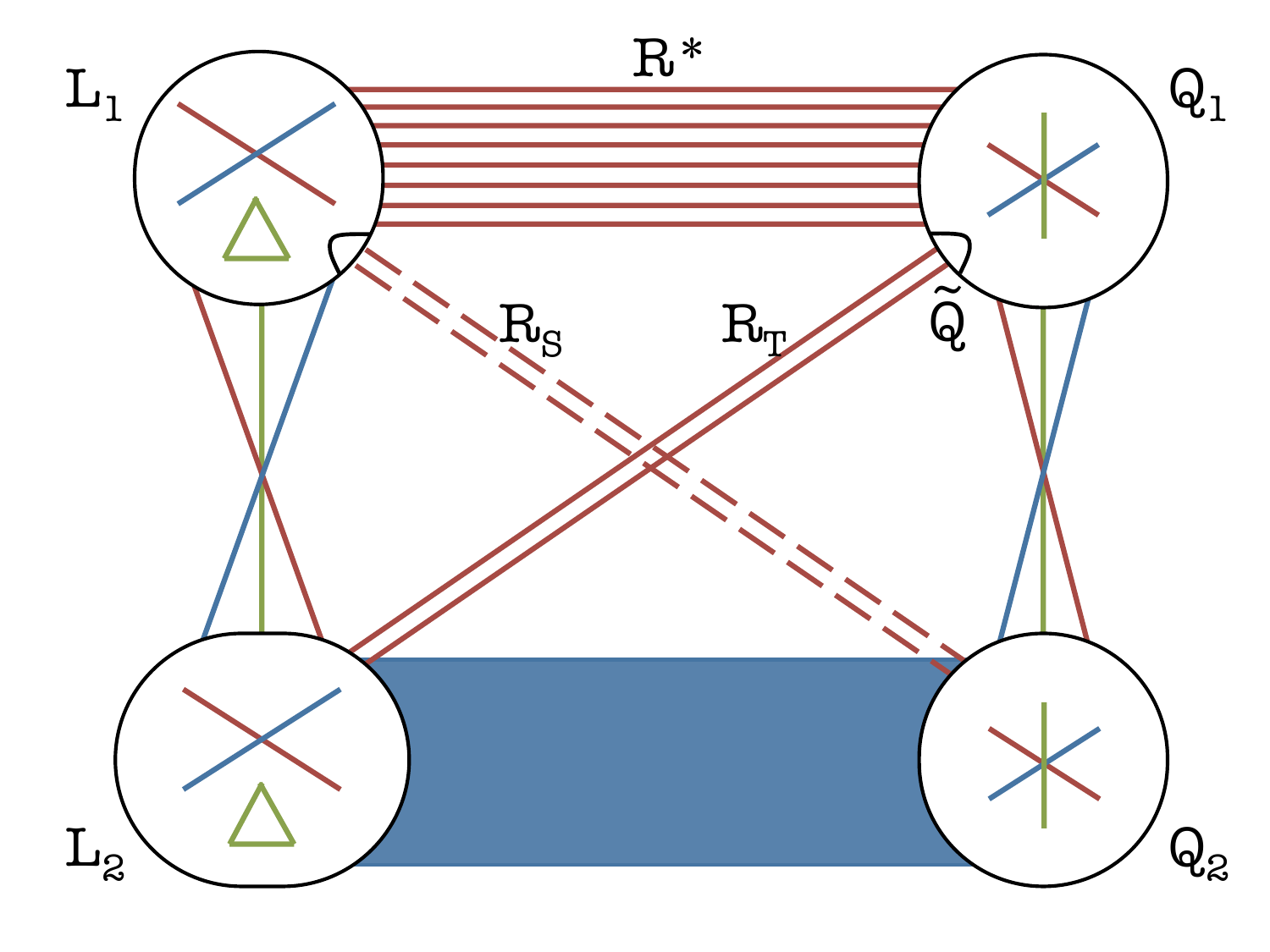}
\vspace{-2mm}\caption{Construction of a red connected-matching on $\aI k$ vertices.}
  
\end{figure}

Therefore, a matching such as $R_S$ cannot exist. Thus, after discarding at most $260\eta^{1/2}k$ vertices from each of $L_1$ and $Q_2$, we may assume that all edges present in $G[L_1,Q_2]$ are coloured exclusively blue. 
%
Note that, in order to retain the equality $|L_1|=|Q_1|$, we also discard from $Q_1$ each vertex whose $R$-mate in $L_1$ has already been discarded. After discarding vertices, we have
$$|L_1|=|Q_1|\geq(\half\aI-518\eta^{1/2}), \text{\hspace{8mm}} |Q_2|\geq(\half\aII-420\eta^{1/2}).$$
  To complete the claim, we return to the blue graph. Since all edges present in $G[L,Q_2]$ are coloured exclusively blue and $G$ is $4\eta^4 k$-almost-complete, $L\cup Q_2$ forms a single blue component. Also, since $|L_2|,|Q_2|\geq (\half\aII-420\eta^{1/2})k$, by Lemma~\ref{l:eleven}, there exists a blue connected-matching $B_2$ on at least $(\aII-842\eta^{1/2})k$ vertices in $G[L_2,Q_2]$. Thus, if there existed a blue matching $B_1$ on at least $842\eta^{1/2}k$ vertices in $G[L_1,Q_1]$, then $B_1\cup B_2$ would form a blue connected-matching on at least $\aII k$ vertices. Thus, after discarding at most $421\eta^{1/2} k$ vertices from each of $L_1$, $Q_1$, we may assume that all edges present in $G[L_1,Q_1]$ are coloured exclusively red. Thus completing the proof of the claim.
    \end{proof}
      \begin{figure}[!h]
\centering
\includegraphics[width=64mm, page=15]{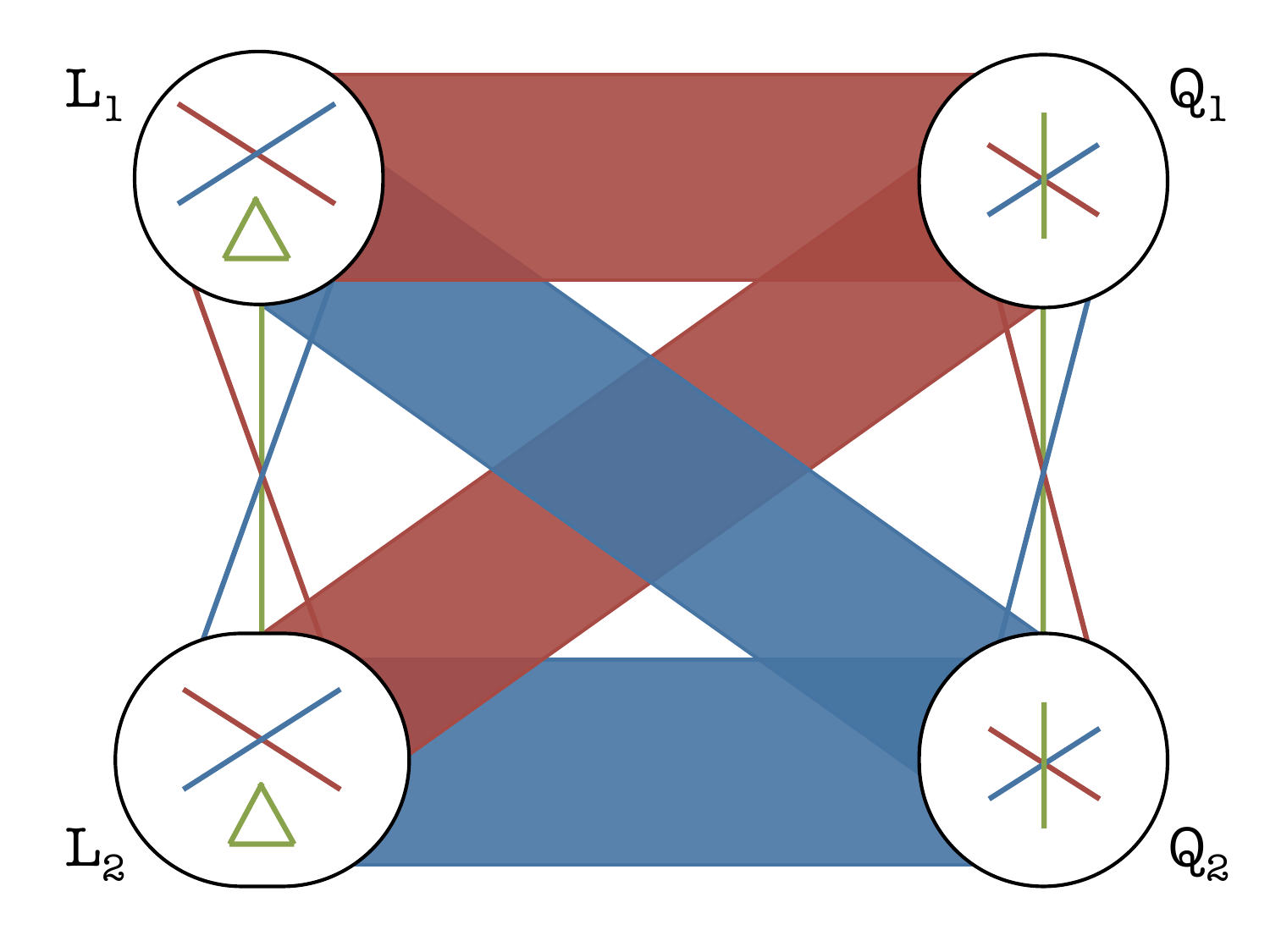}
\vspace{-2mm}\caption{Colouring of $G[L,Q]$ after Claim~\ref{claimLQ}.}
\end{figure}
  
Having proved Claim~\ref{claimLQ}, we know that all edges present in $G[L,Q_1]$ are coloured exclusively red and that all edges present in $G[L,Q_2]$ are coloured exclusively blue. Additionally, we have
\begin{equation}
\left.
\label{name}
\begin{aligned}
\quad\quad\quad\,\,
|L_2|&\geq|Q_1|+2800\eta^{1/2}k, \quad&\quad
|L_1|&=|Q_1|\geq(\half\aI-939\eta^{1/2})k,
\quad\quad\,\,\,\,\\
|L_2|&\geq|Q_2|+2800\eta^{1/2}k, & 
|Q_2|&\geq(\half\aII-420\eta^{1/2})k.
\end{aligned}
\right\}\!
\end{equation}
Now, suppose that there exists a red matching $R^{+}$ on $1880\eta^{1/2}k$ vertices in~$L$. Then, by (\ref{name}), we have $|L\backslash V(R^{+})|\geq |Q_1|\geq(\half\aI-939\eta^{1/2})k$. So, by Lemma~\ref{l:eleven}, there exists a red connected-matching on at least $(\aI-1880\eta^{1/2}) k$ vertices in $G[L\backslash V(R^{+}),Q_1]$, which can be augmented with edges from $R^{+}$ to give a red connected-matching on at least $\aI k$ vertices. Thus, after discarding at most $1880\eta^{1/2} k$ vertices from~$L$, we may assume that there are no red edges present in $G[L]$ and that we then have 
\begin{align*}
|L|&\geq |Q_1|+900\eta^{1/2}k, &|L|&\geq |Q_2|+900\eta^{1/2}k.
\end{align*}
Finally, suppose that there exists a blue connected-matching $B^{+}$ on $842\eta^{1/2}k$ vertices in~$L$. Then, $|L\backslash V(B^{+})|\geq |Q_2|\geq(\half\aII-420\eta^{1/2} )k$ so, by Lemma~\ref{l:eleven}, there exists a blue connected-matching on at least $(\aII-842\eta^{1/2})k$ vertices in $G[L\backslash V(B^{+}),Q_2]$, which can be augmented with edges from $B^{+}$ to give a blue connected-matching on at least~$\aII k$~vertices.
Thus, after discarding at most a further $842\eta^{1/2}$ vertices from~$L$, we may assume that all edges present in $G[L]$ are coloured exclusively green.  

\begin{figure}[!h]
\centering
\includegraphics[width=64mm, page=51]{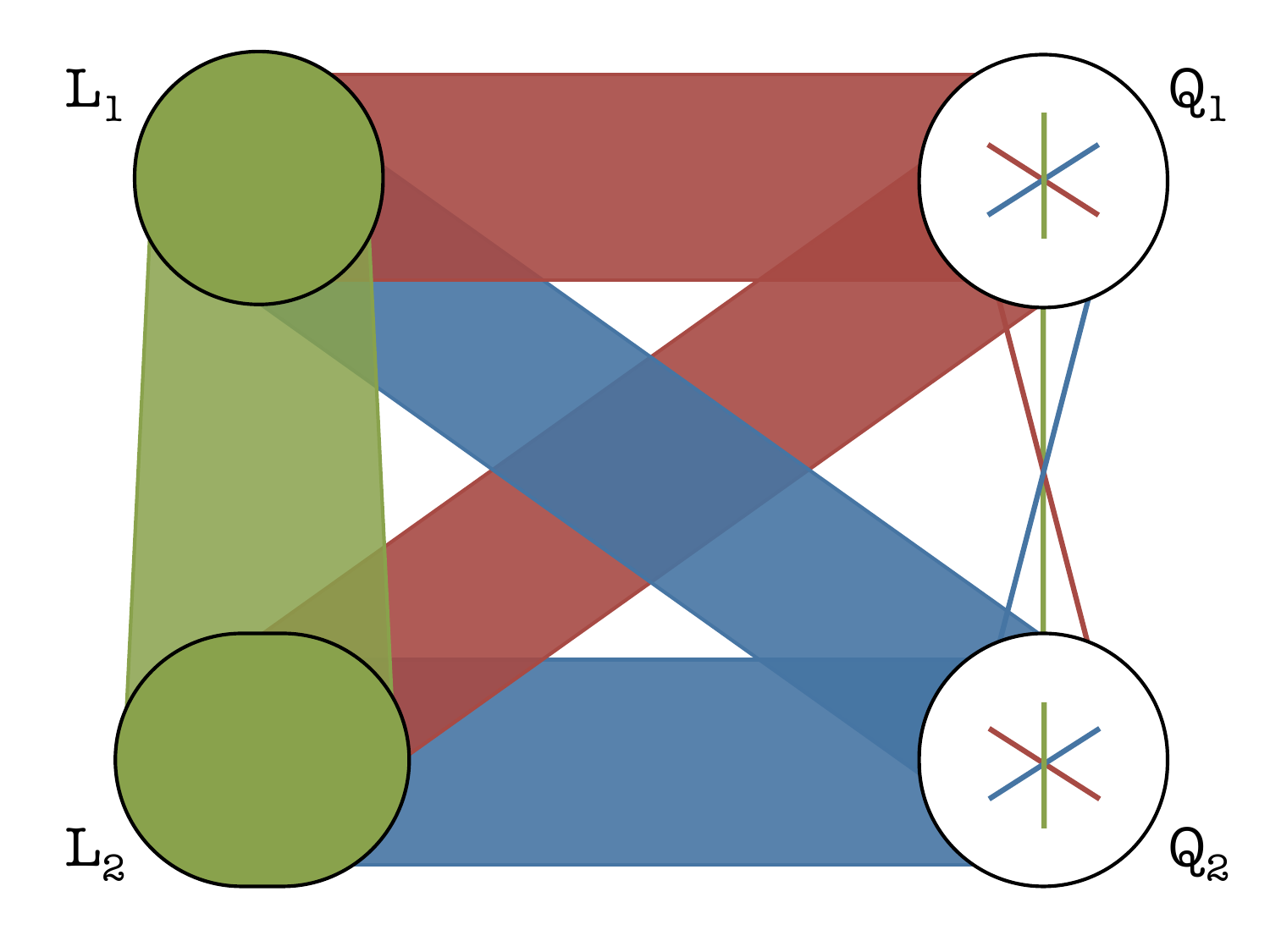}
\vspace{-3mm}\caption{Final colouring in Case E.i.a.}
  
\end{figure}

In summary, having discarded at most $4828\eta^{1/2}k$ vertices, we have 
\begin{align*}
|Q_1|&\geq(\half\aI-939\eta^{1/2})k, &
|Q_2|&\geq(\half\aI-420\eta^{1/2})k, &
|L|&\geq(\half\aIII-3750\eta^{1/2})k,
\end{align*}
and know that all edges present in $G[Q_1,L]$ are coloured exclusively red, all edges present in $G[Q_2,L]$ are coloured exclusively blue and all edges present in $G[L]$ are coloured exclusively green.
Thus, we have found, as a subgraph of~$G$, a graph belonging to $$\cK \left((\half\aI-1000\eta^{1/2})k, (\half\aII-1000\eta^{1/2})k, (\aIII-4000\eta^{1/2})k,4\eta^4k\right).$$

\begin{center}
***
\end{center}
\vspace{-3.5mm}

At the beginning of Case E.i.a, we assumed that the largest monochromatic component in $G[L,Q]$ was red. If, instead, that monochromatic component is blue, then the proof proceeds exactly as above with the roles of red and blue exchanged and with $\aI$ and $\aII$ exchanged. The result is identical.

\subsection*{Case E.i.b: {\rm $L\cup Q$ has a non-trivial partition with `cross' colouring.}}

Recall that we assume that $F$, the largest green connected-matching in $G$, spans at least $(\max\{\tfrac{3}{2}\aI+\half\aII,\aIII\}-10\eta^{1/2})k$ vertices and is contained in an odd component of $G$. We have a partition of $V(G)$ into $L\cup P \cup Q$, such that $|P| \leq  95\eta^{1/2}k$,
\begin{subequations}
\begin{align}
\label{Eib1}\tag{E4a$^\prime$} (\max\{\tfrac{3}{2}\aI+\tfrac{1}{2}\aII,\aIII\}-10\eta^{1/2})k  \leq  |L| & \leq \aIII k, \\
\label{Eib2}\tag{E4b$^\prime$} (\half\aI+\half\aII-96\eta^{1/2})k \leq  |Q| & \leq (\half\aI+\half\aII+5\eta^{1/2})k. \quad\quad\quad
\end{align} 
\end{subequations}
Additionally, in this subcase, we assume that  $L$ and $Q$ can be partitioned into $L_{1}\cup L_{2}$ and $Q_{1}\cup Q_{2}$ such that $|L_{1}|, |L_{2}|, |Q_{1}|, |Q_{2}|\geq 9\eta^{1/2}k$ and all edges present in $G[L_{i},Q_{j}]$ are coloured exclusively red for $i=j$, and exclusively blue for $i\neq j$.
Then, by Lemma~\ref{l:eleven}, there exist red connected-matchings 
\vspace{-1.5mm}
\begin{subequations}
\begin{align}
\label{Eib1a} M_{11}\text{ on at least }2\left(\min\{|L_1|,|Q_1|\}\right)-2\eta^{1/2}k\text{ vertices in }G[L_1,Q_1],\\
\label{Eib1b} M_{22}\text{ on at least }2\left(\min\{|L_2|,|Q_2|\}\right)-2\eta^{1/2}k\text{ vertices in }G[L_2,Q_2],
\end{align}
\vspace{-1.5mm}
and blue connected-matchings 
\vspace{-1.5mm}
\begin{align}
\label{Eib1c} M_{12}\text{ on at least }2\left(\min\{|L_1|,|Q_2|\}\right)-2\eta^{1/2}k\text{ vertices in }G[L_1,Q_2],\\
\label{Eib1d} M_{21}\text{ on at least }2\left(\min\{|L_2|,|Q_1|\}\right)-2\eta^{1/2}k\text{ vertices in }G[L_2,Q_1].
\end{align}
\end{subequations}

\vspace{-5mm}
\begin{figure}[!h]
\centering
\includegraphics[width=64mm, page=16]{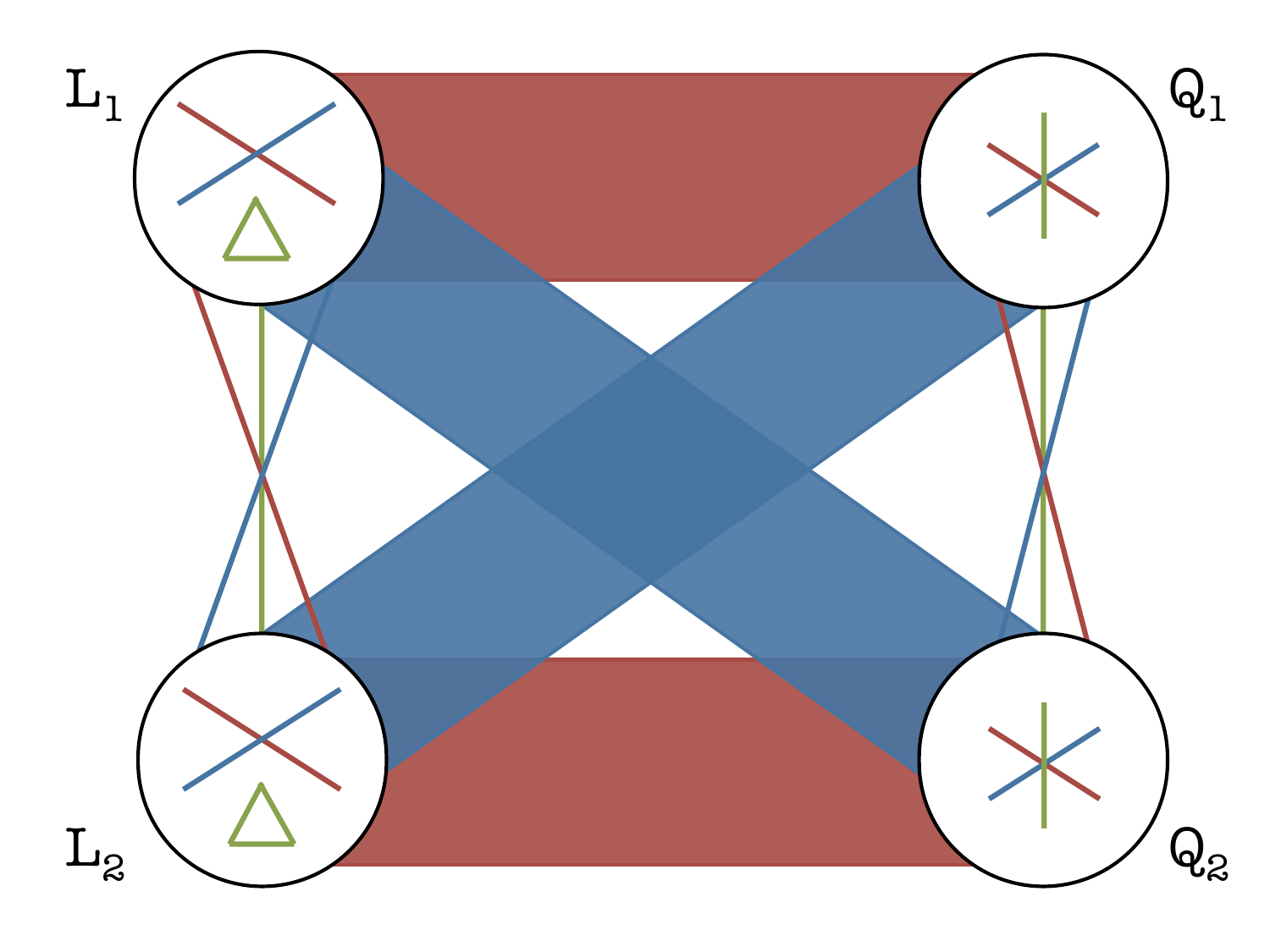}
\vspace{-3mm}\caption{Decomposition into $L_1\cup L_2\cup Q_1 \cup Q_2$ in Case E.i.b.}
\end{figure}

Thus, in order to avoid a red connected-matching on at least $\aI k$ vertices or a blue connected-matching on at least $\aII k$ vertices, we may assume that $|V(M_{11})|,|V(M_{22})|\leq\aI k$ and $|V(M_{12})|,|V(M_{21})|\leq\aII k$. Thus,~(\ref{Eib1a})--(\ref{Eib1d}), above can be used to obtain bounds on the sizes of $L_1,L_2,Q_1$ and $Q_2$ as follows:
\begin{subequations}
\begin{align}
\label{sizea} &\text{Since }|V(M_{11})|, |V(M_{22})|\leq \aI k, \text{we have}  & \min\{|L_1|,|Q_1|\}&\leq (\half\aI+\eta^{1/2})k,\\
\label{sizeb} &\hphantom{\text{Since }|V(M_{11})|, |V(M_{22})|\leq \aI k, \text{we l}}\text{and}  &\min\{|L_2|,|Q_2|\}&\leq (\half\aI+\eta^{1/2})k. \\
\label{sizec} &\text{Since }|V(M_{12})|, |V(M_{21})|\leq \aII k, \text{we have}  & \min\{|L_1|,|Q_2|\}&\leq (\half\aII+\eta^{1/2})k,\\
\label{sized} &\hphantom{\text{Since }|V(M_{12})|, |V(M_{21})|\leq \aII k, \text{we l}}\text{and}  &\min\{|L_2|,|Q_1|\}&\leq (\half\aII+\eta^{1/2})k. 
\end{align}
\end{subequations}
Observe that, since $\eta\leq (\aI/15)^2$, by~(\ref{Eib1}) and~(\ref{Eib2}), we have
\begin{align*}
|L_1|+|L_2|=|L| & \geq (\tfrac{3}{2}\aI+\half\aII-10\eta^{1/2})k\\
& \geq (\half\aI+\half\aII+5\eta^{1/2})k+(\aI-15\eta^{1/2})k\geq |Q|=|Q_1|+|Q_2|.
\end{align*}
Thus, it is not possible to have, for instance, $|Q_1|,|Q_2|\geq|L_1|,|L_2|$. Without loss of generality, we therefore consider four possibilities

\begin{itemize}
\item[(i)] $|L_1|,|L_2|\geq|Q_1|,|Q_2|$;
\item[(ii)] $|L_1|\geq|Q_1|,|Q_2|\geq|L_2|$;
\item[(iii)] $|L_1|\geq|Q_i|\geq|L_2|\geq|Q_j|$ for $\{i,j\}=\{1,2\}$; 
\item[(iv)] $|Q_i|\geq|L_1|\geq|L_2|\geq|Q_j|$ for $\{i,j\}=\{1,2\}$.
\end{itemize}

\subsection*{Case E.i.b.i: $|L_1|,|L_2|\geq|Q_1|,|Q_2|$\rm.}

In this case, by~(\ref{sizec}) and~(\ref{sized}), we may assume that $|Q_1|, |Q_2|\leq (\half\aII+\eta^{1/2})k$ and, consequently, by~(\ref{Eib2}), we have $$|Q_1|, |Q_2|\geq (\half\aI-97\eta^{1/2})k.$$ Thus, by~(\ref{Eib1a})--(\ref{Eib1d}), we have red connected-matchings $M_{11}, M_{22}$, each on at least $(\aI-196\eta^{1/2})k$ vertices and blue connected-matchings $M_{12}, M_{21}$, each on at least $(\aI-196\eta^{1/2})k$ vertices. (Also, in order to avoid having a blue connected-matching on at least~$\aII k$ vertices, we may  assume that $\aI \leq \aII+196\eta^{1/2}$.)
Notice that there can be no red edges present in $G[L_1,L_2]\cup G[Q_1,Q_2]$, since any such edge would mean $M_{11}$ and $M_{22}$ being in the same red-component and so $M_{11}\cup M_{22}$ would form a red connected-matching on at least $\aI k$ vertices. Similarly, there can be no blue edges present in $G[L_1,L_2]\cup G[Q_1,Q_2]$. Therefore, all edges present in $G[L_1,L_2]\cup G[Q_1,Q_2]$ are coloured exclusively~green. 

\vspace{-1mm}
\begin{figure}[!h]
\centering
\includegraphics[width=64mm, page=17]{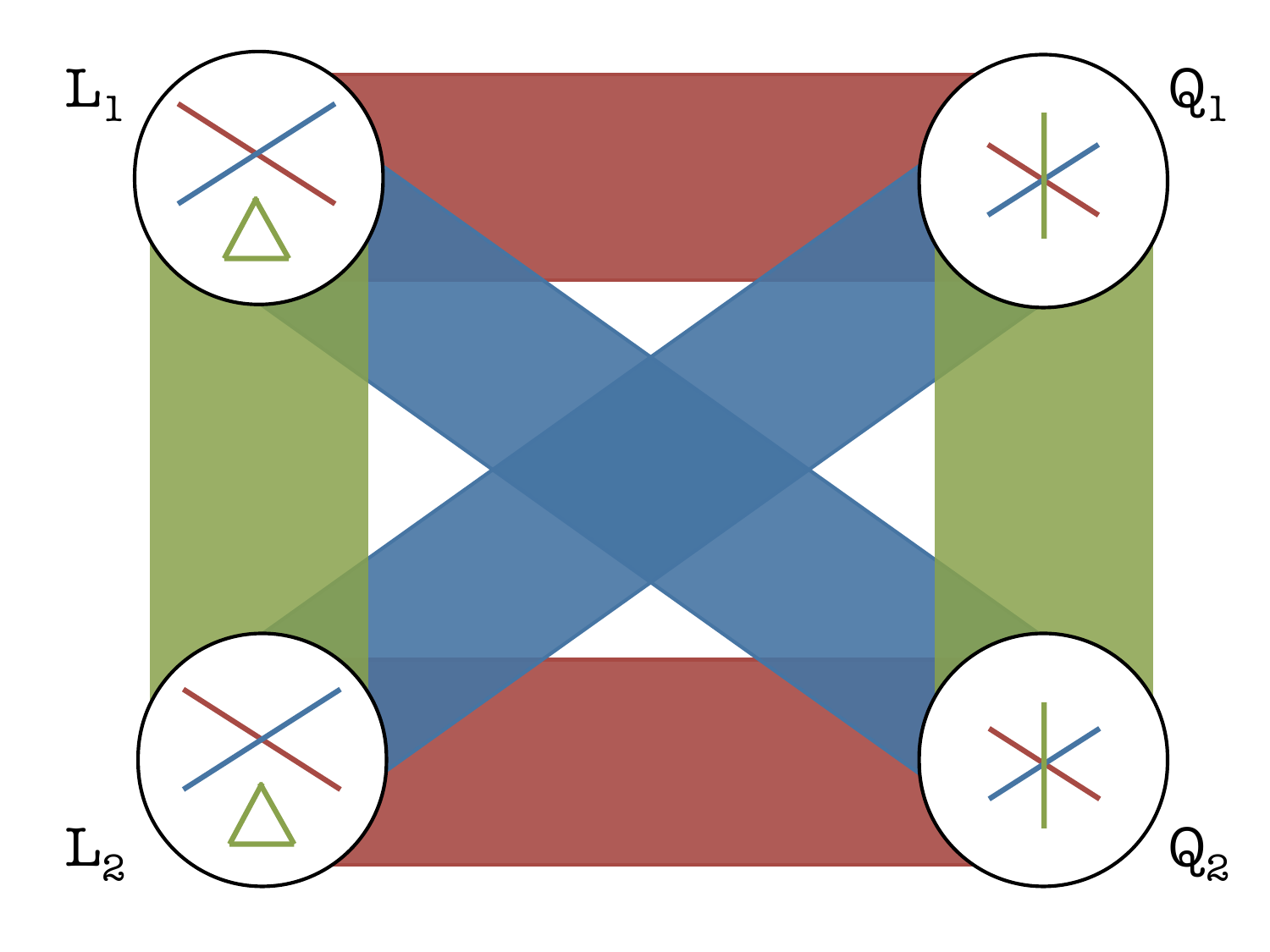}
\vspace{-3mm}\caption{Colouring of $G[L_1\cup L_2\cup Q_1 \cup Q_2]$ in Case E.i.b.i.}
\end{figure}

Thus, we have found, as a subgraph of~$G$, a graph belonging to 
\begin{align*}
\cK^*\big((\half\aI-97\eta^{1/2})k, (\half\aI&-97\eta^{1/2})k, (\half\aI-97\eta^{1/2})k, \\&(\half\aI-97\eta^{1/2})k, (\aIII-10\eta^{1/2})k, 4\eta^4 k\big)\subseteq \cK_1^{*}\cup\cK_2^{*}.
\end{align*}

\subsection*{Case E.i.b.ii: $|L_1|\geq|Q_1|,|Q_2|\geq|L_2|$\rm.}

In this case, by~(\ref{sizea}) and~(\ref{sizec}), we have 
\begin{align}
\label{Eibii0}
|Q_1|&\leq(\half\aI+\eta^{1/2})k, & |Q_2|&\leq (\half\aII+\eta^{1/2})k,
\intertext{and so, by~(\ref{Eib2}),}
\label{Eibii1}|Q_1|&\geq (\half\aI-97\eta^{1/2})k, & |Q_2|&\geq (\half\aII-97\eta^{1/2})k.
\end{align}
Thus, by~(\ref{Eib1a}), we have a red connected-matching $M_{11}$ on at least $(\aI-196\eta^{1/2})k$ vertices in $G[L_1,Q_1]$ and, by~(\ref{Eib1c}), have a blue connected-matching $M_{12}$ on at least $(\aII-196\eta^{1/2})k$ vertices in $G[L_1,Q_2]$. 

Suppose, then, that $|L_2|\geq 100\eta^{1/2}k$. Then, by Lemma~\ref{l:eleven}, there exists a red connected-matching $R_S$ on at least $198\eta^{1/2}k$ vertices in $G[L_2,Q_2]$ and a blue connected-matching~$B_S$ on at least $198\eta^{1/2}k$ vertices in $G[L_2,Q_1]$. Thus, there can be no red edges present in $G[L_1,L_2]\cup G[Q_1,Q_2]$, since otherwise $M_{11}$ and $R_{S}$ would belong to the same red component and together span at least $\aI k$ vertices. Likewise, there can be no blue edges present in $G[L_1,L_2]\cup G[Q_1,Q_2]$, since then $M_{12}$ and $B_{S}$ would then belong to the same blue component and together span at least $\aII k$ vertices. Therefore, we have, as a subgraph of~$G$, a graph belonging to

{~}
\vspace{-6mm}
\begin{align*}
\cK^*\big((\half\aI-97\eta^{1/2})k, (\half\aII-97\eta^{1/2})k, (&\aIII-\half\aII-12\eta^{1/2})k, \\& 100\eta^{1/2}k, (\aIII-10\eta^{1/2})k, 4\eta^4 k\big)\subset \cK_2^{*}.
\end{align*}

Thus, we may assume that $|L_2|\leq 100\eta^{1/2}k$, in which case we have 
\begin{align}
\label{Eibii2}
|L_1|&\geq(\max\{\tfrac{3}{2}\aI+\tfrac{1}{2}\aII,\aIII\}-110\eta^{1/2})k
\end{align}
and know that all edges present in $G[Q_1,L_1]$ are coloured exclusively red and all edges present in $G[Q_2,L_1]$ are coloured exclusively blue. In this case, we disregard $L_2$ and consider $G[L_1]$. 

\begin{figure}[!h]
\centering
\includegraphics[width=64mm, page=53]{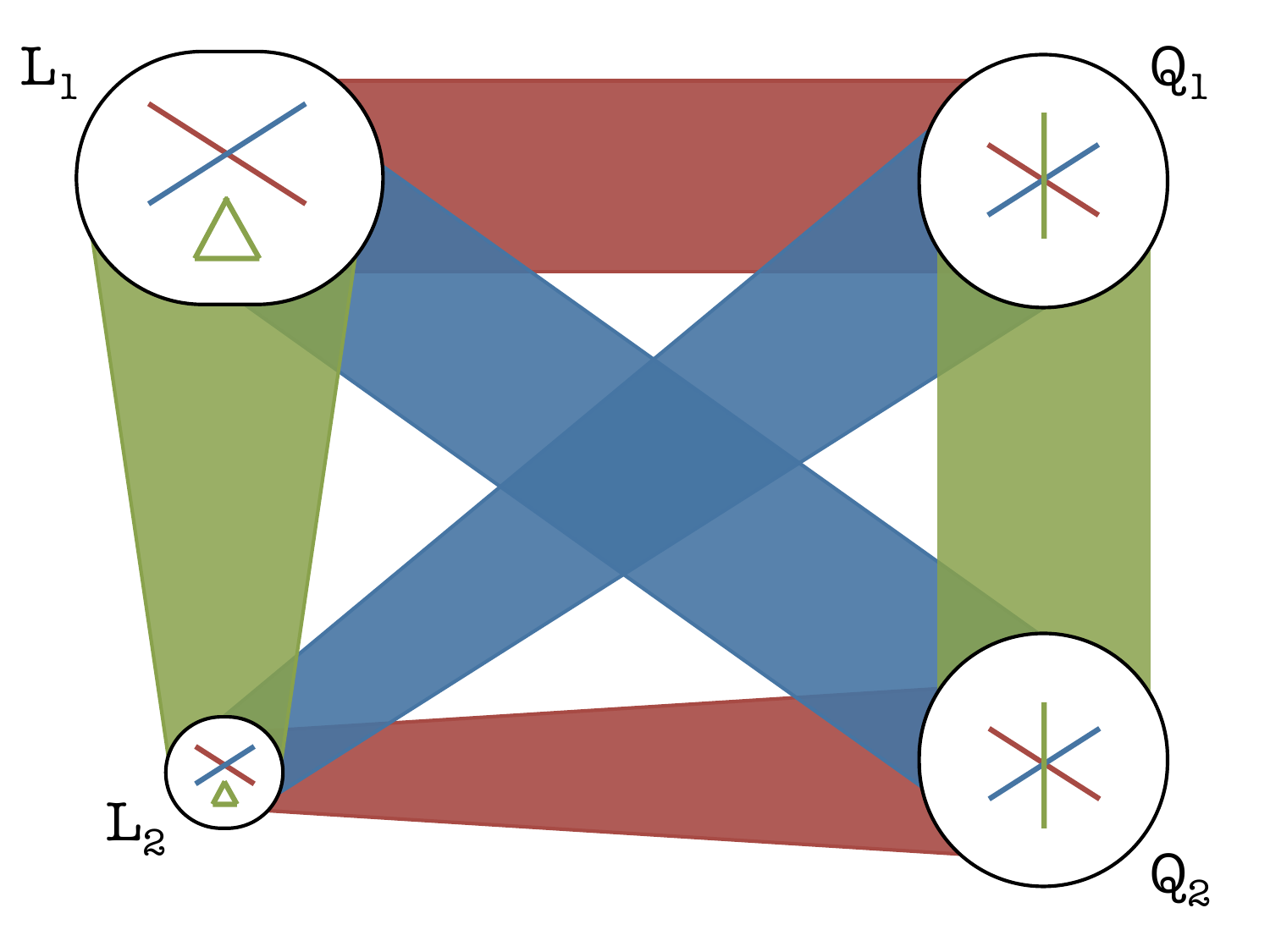}
\vspace{-3mm}\caption{Colouring of $G[L\cup Q]$ in Case E.i.b.ii.}
\end{figure}

Since $\eta\leq (\aI/1000)^2$, by~(\ref {Eibii0}) and~(\ref{Eibii2}),
\begin{align}
\label{Eibii3}
|L_1|&\geq|Q_1|+200\eta^{1/2}k, & |L_1|&\geq|Q_2|+200\eta^{1/2}k.
\end{align}
Suppose there exists a red matching $R_S$ on $196\eta^{1/2}k$ vertices in $G[L_1]$. Then, by~(\ref{Eibii1}) and~(\ref{Eibii3}), $$|L_1\backslash V(R_S)|,|Q_1|\geq(\half\aI-97\eta^{1/2})k.$$ Thus, by Lemma~\ref{l:eleven}, there exists a red connected-matching $R_L$ on at least $(\aI - 196\eta^{1/2})k$ vertices in $G[L_2,Q_1]$ which shares no vertices with $R_S$. Since all edges present in $G[L_1,Q_1]$ are coloured exclusively red and $G$ is $4\eta^4 k$-almost-complete, $R_S$ and $R_L$ belong to the same red component and, thus, together form a red connected-matching on at least $\aI k$ vertices. Therefore, the largest red matching in $G[L_1]$ spans at most~$196\eta^{1/2}k$~vertices.
Similarly, the largest blue connected-matching in $G[L_1]$ spans at most $196\eta^{1/2}k$ vertices. Thus, after discarding at most $392\eta^{1/2}k$ vertices from $L_1$, we may assume that all edges in $G[L]$ are coloured green. 
Thus, we have obtained, as a subgraph of~$G$, a graph in  $$\cK\left((\half\aI-97\eta^{1/2})k, (\half\aII-97\eta^{1/2})k, (\aIII-502\eta^{1/2})k, 4\eta^4 k \right).$$  

\subsection*{Case E.i.b.iii: {\rm $|L_1|\geq|Q_i|\geq|L_2|\geq|Q_j|$ for $\{i,j\}=\{1,2\}$.}}
Suppose that $|L_1|\geq|Q_1|\geq|L_2|\geq|Q_2|$. Then, by~(\ref{sizea})--(\ref{sized}) and~(\ref{Eib2}), we obtain 
\begin{align*}
(\half\aI-97\eta^{1/2})k&\leq|Q_1|\leq(\half\aI+\eta^{1/2})k,\\
(\half\aII-97\eta^{1/2})k&\leq|Q_2|\leq (\half\aII+\eta^{1/2})k.
\end{align*}
Then, given the sizes of $M_{11}$, $M_{12}$, $M_{21}$ and $M_{22}$, there can be no red or blue edges present in $G[L_1,L_2]\cup G[Q_1,Q_2]$ and we may assume that $\aI\leq\aII+196\eta^{1/2}k$.
Thus, we have found, as a subgraph of~$G$, a graph
 in 
\begin{align*}
\cK^*\big((\half\aI-97\eta^{1/2})k, (\half\aII-97&\eta^{1/2})k, (\aIII-\half\aI-12\eta^{1/2})k, \\& (\half\aII-97\eta^{1/2})k, (\aIII-10\eta^{1/2})k, 4\eta^4 k\big)\subset \cK_2^{*}.
\end{align*}
Suppose instead that $|L_1|\geq|Q_2|\geq|L_2|\geq|Q_1|$. Then, by~(\ref{sizea})--(\ref{sized}) and~(\ref{Eib2}), we obtain 
\begin{align*}
(\half\aI-97\eta^{1/2})k&\leq|Q_1|\leq(\half\aII+\eta^{1/2})k,\\
(\half\aI-97\eta^{1/2})k&\leq|Q_2|\leq (\half\aII+\eta^{1/2})k.
\end{align*}
Again, there can be no red or blue edges present in $G[L_1,L_2]\cup G[Q_1,Q_2]$. Thus, we have found, as a subgraph of~$G$, a graph in 
\begin{align*}
\cK^*\big((\half\aI-97\eta^{1/2})k, (\half\aII-97&\eta^{1/2})k, (\aIII-\half\aII-12\eta^{1/2})k, \\& (\half\aI-97\eta^{1/2})k, (\aIII-10\eta^{1/2})k, 4\eta^4 k\big)\subset \cK_2^{*}.
\end{align*}

\subsection*{Case E.i.b.iv: {\rm $|Q_i|\geq|L_1|\geq|L_2|\geq|Q_j|$ for $\{i,j\}=\{1,2\}$.}}

Suppose $|Q_1|\geq|L_1|\geq|L_2|\geq|Q_2|$. Then, $|Q_1|\geq\half|L_1|+\half|L_2|=\half|L|$ and,  since $\eta\leq(\aII/100)^2$, by~(\ref{Eib1}), we have $$|Q_1|\geq(\tfrac{3}{4}\aI+\tfrac{1}{4}\aII-5\eta^{1/2})k\geq(\half\aI+\half\aII-5\eta^{1/2})k\geq(\half\aI+\eta^{1/2})k.$$

Also, by~(\ref{Eib1}), we have $$|L|\geq (\tfrac{3}{2}\aI+\half\aII-10\eta^{1/2})k\geq(\aI+\aII-10\eta^{1/2})k.$$ Thus, either $$|L_1|\geq(\aI-5\eta^{1/2})k, \text{\hspace{3mm} or \hspace{3mm}} |L_2|\geq (\aII-5\eta^{1/2})k.$$  Recall that all edges present in  $G[L_1,Q_1]$ are coloured exclusively red and all edges present in $G[L_2, Q_1]$ are coloured exclusively blue. Thus, by Lemma~\ref{l:eleven}, there exists either a red connected-matching on at least $\aI k$ vertices in $G[L_1,Q_1]$ or a blue connected-matching on at least $\aII k$ vertices in $G[L_1,Q_1]$.

The result is the same in the case that $|Q_2|\geq|L_1|\geq|L_2|\geq|Q_1|$.

\subsection*{Case E.i.c: {\rm $G[L,Q]$ contains red and blue `stars' centred in $L$.}}

We continue to assume that $F$, the largest green connected-matching in $G$, spans at least $(\max\{\tfrac{3}{2}\aI+\half\aII,\aIII\}-10\eta^{1/2})k$ vertices and is contained in an odd component of $G$ and that we have a partition of $V(G)$ into $L\cup P \cup Q$, satisfying~(\ref{E6}) and~(\ref{E7}) such that all edges present in $G[L,Q]$ are coloured red or blue. Additionally, in this case, we have vertices $v_r, v_b\in L$ such that $v_r$ has red edges to all but $36\eta^{1/2} k$ vertices in $Q$ and $v_b$ has blue edges to all but $36\eta^{1/2} k$ vertices in~$Q$. Observe that the existence of $v_r$ means that  
\begin{itemize}
\item[(E5$^{\prime}$)] $G[Q]$ has a red effective-component on at least $|Q|-36\eta^{1/2}k$ vertices.
\end{itemize}
The proof then proceeds as in Case E.i.a but with (\ref{E5}) replaced by (E5$^{\prime}$). The result is the same.

\subsection*{Case E.i.d: {\rm $G[L,Q]$ contains red and blue `stars' centred in $Q$.}}

We continue to assume that $F$, the largest green connected-matching in $G$, spans at least $(\max\{\tfrac{3}{2}\aI+\half\aII,\aIII\}-10\eta^{1/2})k$ vertices and is contained in an odd component of $G$ and that we have a partition of $V(G)$ into $L\cup P \cup Q$, satisfying~(\ref{E6}) and~(\ref{E7}) such that all edges present in $G[L,Q]$ are coloured red or blue. Additionally, in this case,  we have vertices $v_r, v_b\in Q$ such that $v_r$ has red edges to all but $36\eta^{1/2} k$ vertices in $L$ and $v_b$ has blue edges to all but $36\eta^{1/2} k$ vertices in~$L$. Observe that the existence of $v_r$ means that  
\begin{itemize}
\item[(E5$^{\prime\prime}$)] $G[L]$ has a red effective-component on at least $|L|-36\eta^{1/2}k$ vertices.
\end{itemize}
The proof then proceeds as in Case E.i.a but with (\ref{E5}) replaced by (E5$^{\prime\prime}$). The result is the same.

\subsection*{Case E.ii: $|Q|\leq 95\eta^{1/2}k.$}

Recall that we have a decomposition of $V(G)$ into $M\cup N\cup P \cup Q$ such that:
\begin{itemize}
\labitem{E1}{E1-ii}$M\cup N$ is the vertex set of~$F$ and every edge of~$F$ belongs to $G[M,N]$;
\labitem{E2}{E2-ii} every vertex in~$P$ has a green edge to~$M$;
\labitem{E3}{E3-ii} there are no green edges in $G[N,P]$, $G[M,Q]$, $G[N,Q]$, $G[P,Q]$ or $G[P]$.
\end{itemize}
Recall also that
\begin{align}\label{Eii0}\tag{E4a}(\max\{\tfrac{3}{4}\aI+\tfrac{1}{4}\aII,\half\aIII\}-5\eta^{1/2})k & \leq  |M|\,,\,|N| \leq \half\aIII k,\\
\tag{E4b}(\half\aI+\half\aII-\eta)k & \leq  |P|+|Q| \leq (\half\aI+\half\aII+5\eta^{1/2})k.
\end{align}
Here, we consider the case when $Q$ is sufficiently small to be disregarded. In that case, provided $|Q|\leq 95\eta^{1/2} k$, from (E4b), we have 
\begin{equation}
\label{Eii1} 
\tag{E4b$^{\prime\prime}$}(\half\aI+\half\aII-96\eta^{1/2})k\leq|P|\leq (\half\aI+\half\aII+5\eta^{1/2})k.
\end{equation}
By (\ref{E3-ii}), every edge in~$G[P]$ is red or blue. Then, since~$G$ is $4\eta^4 k$-almost-complete, by Lemma~\ref{l:dgf0}, the largest monochromatic component $F_P$ in~$G[P]$ contains at least $|P|-\eta^{1/2}k$ vertices. Suppose that that this component is red. 

We consider the largest red matching $R$ in $G[N,P]$ and partition $N$ into $N_1\cup N_2$ and~$P$ into $P_1\cup P_2$, where $N_1=N\cap V(R)$, $N_2 = N\backslash N_1$, $Q_1=Q\cap V(R)$ and $Q_2=Q\backslash Q_1$.
Then, by maximality of $R$, all edges present in $G[N_2,P_2]$ are coloured exclusively blue. 

\begin{figure}[!h]
\centering{
\mbox{
\hspace{-8mm}
{~}\quad
{\includegraphics[width=64mm, page=7]{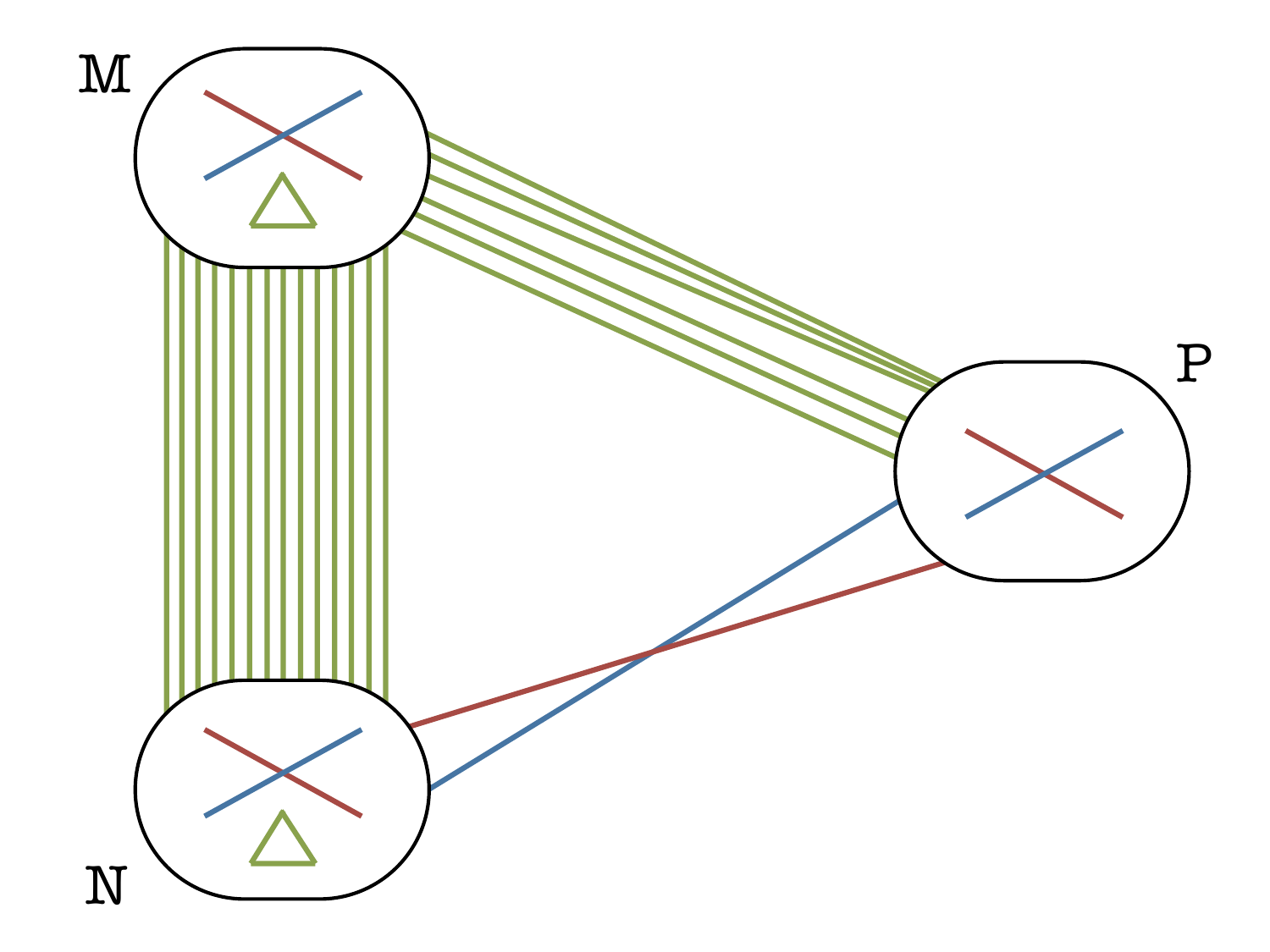}}\quad\quad\quad
{\includegraphics[width=64mm, page=9]{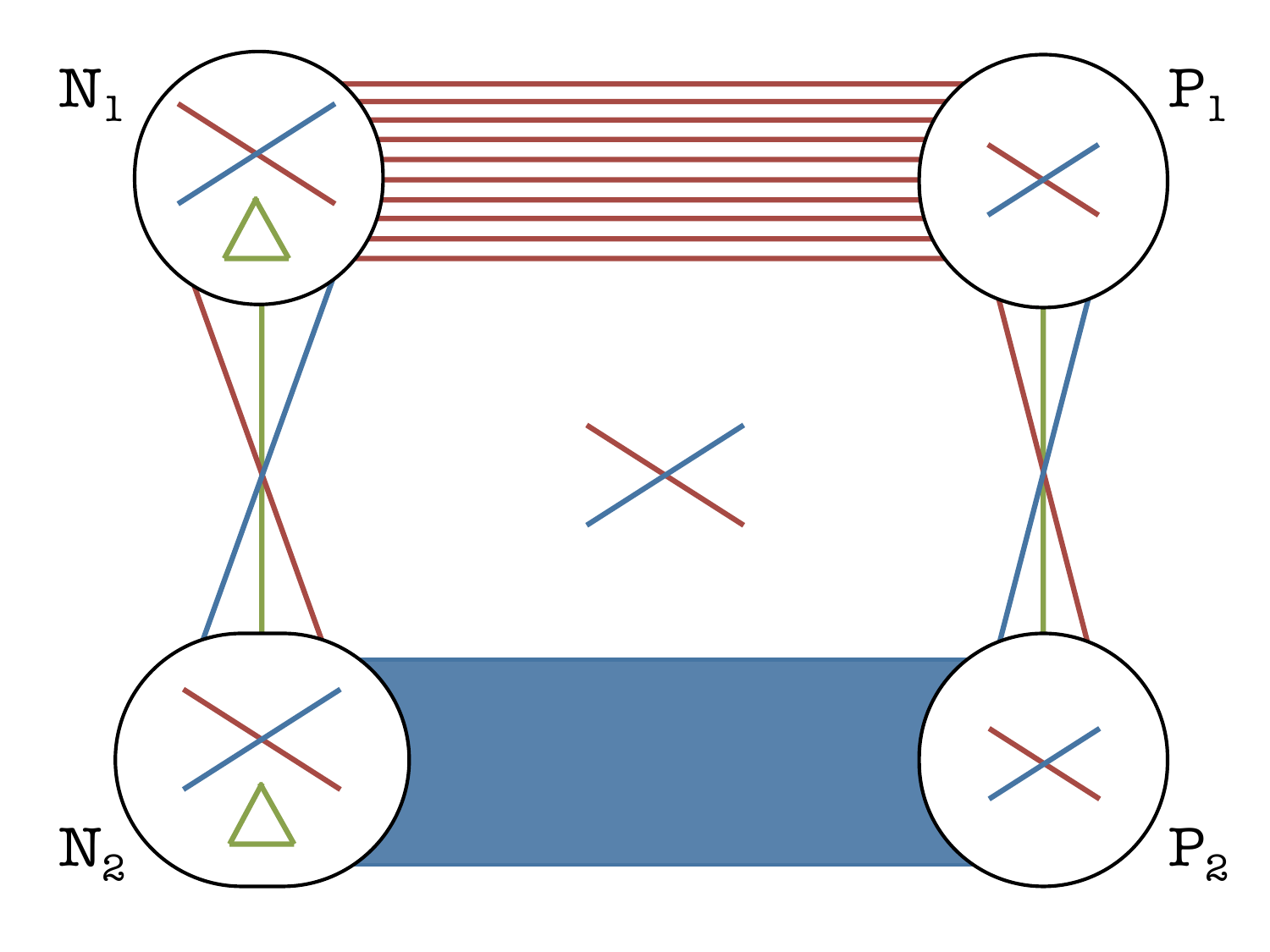}}}}
\caption{Decompositions in Case E.ii.}   
\label{th1a}
\end{figure}

Since~$P$ has a large red connected-component, all but $\eta^{1/2}k$ of the edges of $R$ belong to the same red-component and thus form a red connected-matching. Thus, in order to avoid having a red connected-matching on at least $\aI k$ vertices, we have $|P_1|\leq (\half\aI+\eta^{1/2})k$.

Suppose that $|N_2|\geq|P_2|$. Then, we have $|P_2|\leq(\half\aII+\eta^{1/2})k$, since otherwise, by Lemma~\ref{l:eleven}, $G[N_2,P_2]$ would contain a blue connected-matching on at least $\aII k$ vertices. Thus, by (\ref{Eii0}) and~(\ref{Eii1}), we have
\begin{equation}
\label{Eii2}
\left.
\begin{aligned}
\quad\quad\quad\quad\quad\quad\quad\quad(\half\aI-97\eta^{1/2})k&\leq|P_1|\leq (\half\aI+\eta^{1/2})k,\quad\quad\quad\quad\quad\quad\quad\\
(\half\aII-97\eta^{1/2})k&\leq|P_2|\leq (\half\aII+\eta^{1/2})k,\\
(\half\aI-97\eta^{1/2})k&\leq|N_1|\leq (\half\aI+\eta^{1/2})k,\\
(\half\aII-97\eta^{1/2})k&\leq|N_2|. \\
\end{aligned}
\right\}
\end{equation}
If instead $|N_2|\leq|P_2|$, then, by Lemma~\ref{l:eleven}, we have $|N_2|\leq(\half\aII+\eta^{1/2})k$ since, otherwise, by Lemma~\ref{l:eleven}, $G[N_2,P_2]$ contains a blue connected-matching on at least~$\aII k$ vertices. Thus, by (\ref{Eii0}) and~(\ref{Eii1}), we have
\begin{equation}
\label{Eii3}
\left.
\begin{aligned}
\,\,\,\,\,\quad\quad\quad\quad\quad(\tfrac{3}{4}\aI-\tfrac{1}{4}\aII-6\eta^{1/2})k&\leq|P_1|\leq (\half\aI+\eta^{1/2})k,\quad\quad\quad\quad\quad\quad\quad\\
(\half\aII-97\eta^{1/2})k&\leq|P_2|\leq (\tfrac{3}{4}\aI-\tfrac{1}{4}\aII+11\eta^{1/2})k,\\
(\tfrac{3}{4}\aI-\tfrac{1}{4}\aII-6\eta^{1/2})k&\leq|N_1|\leq (\half\aI+\eta^{1/2})k,\\
(\tfrac{1}{4}\aI+\tfrac{1}{4}\aII-6\eta^{1/2})k&\leq|N_2|\leq(\half\aII+\eta^{1/2})k, 
\end{aligned}
\right\}\,\,
\end{equation}
which yields a contradiction unless $\aI\leq\aII+28\eta^{1/2}$.

Six of the eight bounds obtained in~(\ref{Eii3}) are stronger than the corresponding bounds obtained in~(\ref{Eii2}). The seventh is weaker but can be written in a similar form. Thus, we will combine the two cases and continue under the assumption that
\begin{equation}
\label{Eii4}
\left.
\begin{aligned}
\,\,\,\,\,\,\quad\quad\quad\quad\quad\quad\quad(\half\aI-97\eta^{1/2})k&\leq|P_1|\leq (\half\aI+\eta^{1/2})k,\quad\quad\quad\quad\quad\quad\quad\\
(\half\aII-97\eta^{1/2})k&\leq|P_2|\leq (\half\aII+32\eta^{1/2})k,\\
(\half\aI-97\eta^{1/2})k&\leq|N_1|\leq (\half\aI+\eta^{1/2})k,\\
(\half\aII-97\eta^{1/2})k&\leq|N_2|. 
\end{aligned}
\right\}
\end{equation}

Suppose there exists a blue matching $B_1$ on $198\eta^{1/2}k$ vertices in $G[N_2,P_1]$ and a blue matching $B_2$ on $198\eta^{1/2}k$ vertices in $G[N_1,P_2]$. Then, defining $\widetilde{N}=N_2\backslash V(B_1)$ and $\widetilde{P}=P_2\backslash V(B_2)$, we have $|\widetilde{N}|, |\widetilde{P}|\geq (\half\aII-197\eta^{1/2})k$ and, thus, by Lemma~\ref{l:eleven}, there exists a blue connected-matching $B_3$ on at least $(\aII - 396\eta^{1/2})k$ vertices in $G[\widetilde{N},\widetilde{P}]$. Since $G_2[N_2,P_2]$ is $4\eta^4 k$-almost-complete, all vertices in $N_2\cup P_2$ belong to the same blue component. Thus, $B_1\cup B_2 \cup B_3$ forms a blue connected-matching on at least $\aII k$ vertices.

\begin{figure}[!h]
\centering
\includegraphics[width=64mm, page=22]{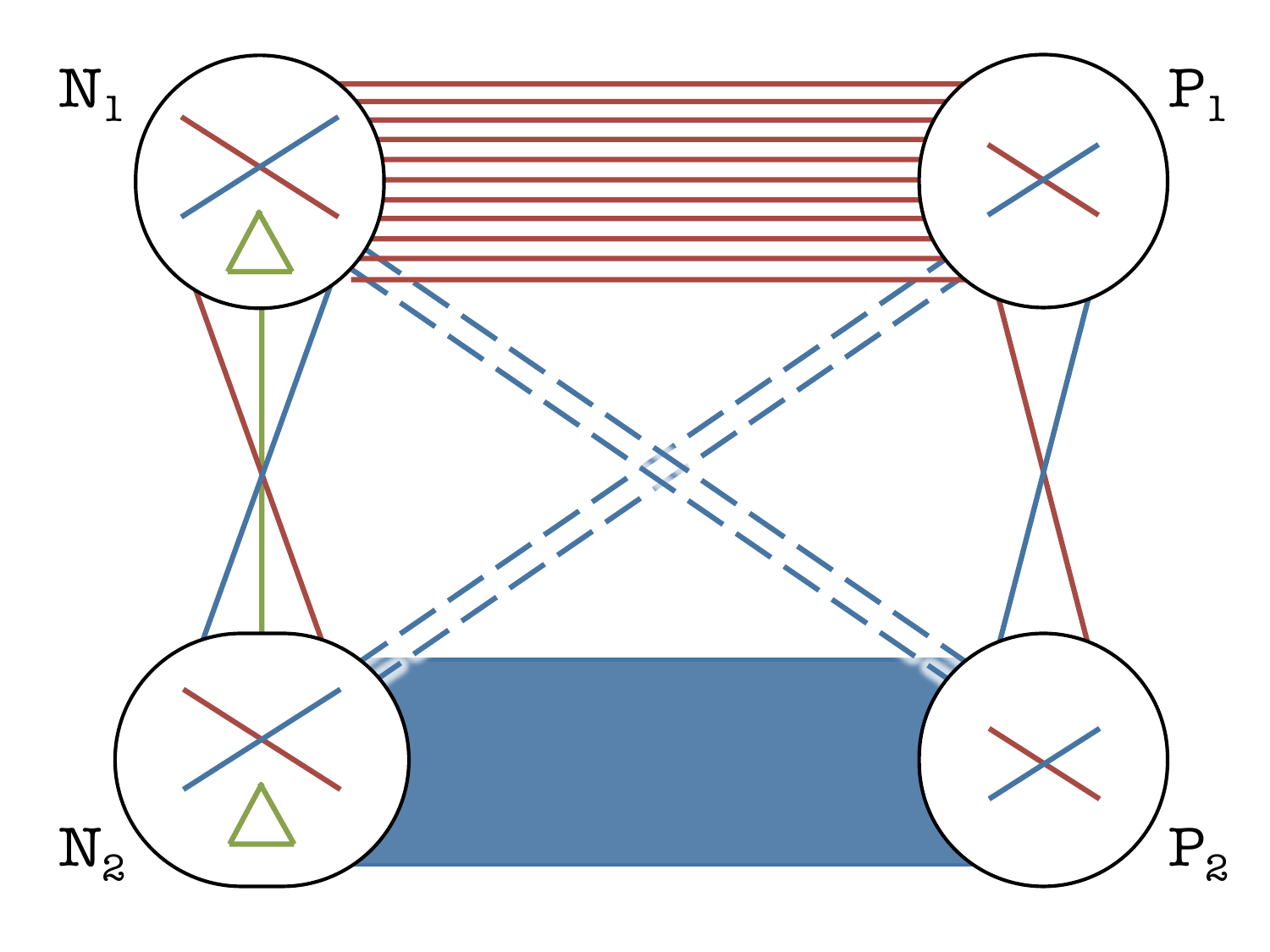}
\vspace{-2mm}\caption{Colouring of the edges of $G[L_1,Q_2], G[L_2,Q_1]$.}
  
\end{figure}

Therefore, we proceed considering the following two subcases:
\begin{itemize}
\item[(a)] the largest blue matching in $G[N_2,P_1]$ spans at most $198\eta^{1/2}k$ vertices;
\item[(b)] the largest blue matching in $G[N_1,P_2]$ spans at most $198\eta^{1/2}k$ vertices.
\end{itemize}

\subsection*{Case E.ii.a: {\rm Most edges in $G_2[N_2,P_1]$ are red.}}

Since the largest blue matching in $G[N_2,P_1]$ spans at most $198\eta^{1/2}k$ vertices, we can discard at most $99\eta^{1/2}k$ vertices from each of $N_2$ and $P_1$ so that all edges present in $G[N_2,P_1]$ are coloured exclusively red. 
%
%
%
In order to retain the equality $|N_1|=|P_1|$ and the property that every vertex in $N_1$ belongs to an edge of $R$, we discard from $N_1$ each vertex whose $R$-mate in $P_1$ has already been discarded. Recalling~(\ref{Eii4}), we then have
\begin{equation*}
\begin{aligned}
\,\,\,\,\,\,\quad\quad\quad\quad\quad\quad\quad(\half\aI-196\eta^{1/2})k&\leq|P_1|\leq (\half\aI+\eta^{1/2})k,\quad\quad\quad\quad\quad\quad\quad\\
(\half\aII-97\eta^{1/2})k&\leq|P_2|\leq (\half\aII+32\eta^{1/2})k,\\
(\half\aI-196\eta^{1/2})k&\leq|N_1|\leq (\half\aI+\eta^{1/2})k,\\
(\half\aII-196\eta^{1/2})k&\leq|N_2|. 
\end{aligned}
\end{equation*}

Now, suppose there exists a red matching $R_U$ on $396\eta^{1/2}k$ vertices in $G[N_1,P_2]$.  Observe that there exists a set $R^{-}$ of $198\eta^{1/2}k$ edges belonging to $R$ such that $N_1\cap V(R_S)=N_1\cap V(R^{-})$. Then, we have $|N_2|,|P_1\cap V(R^-)|\geq 198\eta^{1/2}k$ so, by Lemma~\ref{l:eleven}, there exists a red connected-matching $R_V$ on at least $394\eta^{1/2}k$ vertices in $G[N_2,P_1\cap V(R^-)]$. Then, since all edges present in $G[N_2,P_1]$ are coloured exclusively red and every vertex in $N_1$ has a red neighbour in $P_1$, $(R\backslash R^{-})$, $R_U$ and $R_V$ belong to the same red-component and, thus, together form a red-connected-matching on at least $2(\half\aI-196\eta^{1/2}k-199\eta^{1/2}k)+396\eta^{1/2}k+394\eta^{1/2}k\geq\aI k$ vertices. Thus, we can discard at most $198\eta^{1/2}k$ vertices from each of $N_1$ and $P_2$ so that all edges present in $G[N_1,P_2]$ are coloured exclusively blue. We then have
\begin{align*}
(\half\aI-196\eta^{1/2})k&\leq|P_1|\leq (\half\aI+\eta^{1/2})k,\quad\,\,\,\,~\\
(\half\aII-295\eta^{1/2})k&\leq|P_2|\leq (\half\aII+32\eta^{1/2})k,\\
(\half\aI-394\eta^{1/2})k&\leq|N_1|\leq (\half\aI+\eta^{1/2})k,\\
(\half\aII-196\eta^{1/2})k&\leq|N_2|. 
\end{align*}

\begin{figure}[!h]
\vspace{-2mm}
\centering
\includegraphics[width=64mm, page=23]{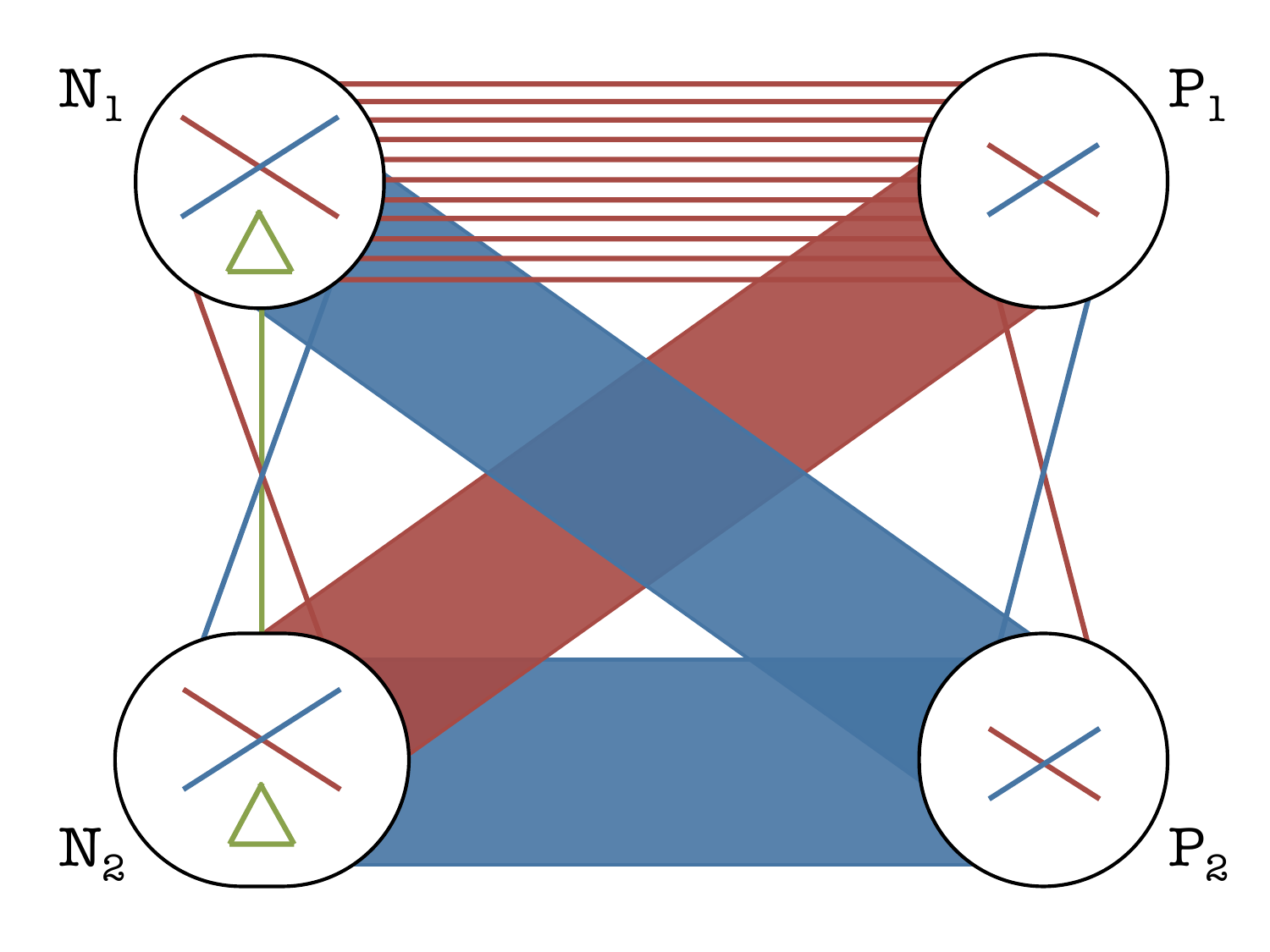}
\vspace{-2mm}\caption{Colouring of the edges of $G[N_1,P_2]$.}
  
\end{figure}

Observe then that, since $|N_2|, |P_2|\geq (\half\aII-295\eta^{1/2})k$, by Lemma~\ref{l:eleven}, there exists a blue connected-matching on at least $(\aII-592\eta^{1/2})k$ vertices in $G[N_2,P_2]$. Thus, since all edges present in $G[N,P_2]$ are coloured exclusively blue, if there existed a blue matching on $592\eta^{1/2}k$ vertices in $G[N_1,P_1]$, we would have a blue connected-matching on at least $\aII k$ vertices. Therefore, after discarding at most $296\eta^{1/2}k$ vertices from each of $P_1$ and~$N_1$, we may assume that all edges in $G[P_1,N_1]$ are coloured exclusively red and that
\begin{equation}
\label{Eiia3}
\left.
\begin{aligned}
\,\,\,\,\,\,\quad\quad\quad\quad\quad\quad\quad(\half\aI-492\eta^{1/2})k&\leq|P_1|\leq (\half\aI+\eta^{1/2})k,\quad\quad\quad\quad\quad\quad\quad\\
(\half\aII-295\eta^{1/2})k&\leq|P_2|\leq (\half\aII+32\eta^{1/2})k,\\
(\half\aI-690\eta^{1/2})k&\leq|N_1|\leq (\half\aI+\eta^{1/2})k,\\
(\half\aII-196\eta^{1/2})k&\leq|N_2|.
\end{aligned}
\right\}
\end{equation}

Then, since $\eta\leq(\aII/20000)^2$, by~(\ref{Eiia3}), we have 
\begin{align}
\label{nlargea}
|N|=|N_1|+|N_2|&\geq|P_1|+5000\eta^{1/2}k,\\
\label{nlargeb}
|N|=|N_1|+|N_2|&\geq|P_2|+5000\eta^{1/2}k.
\end{align}

In particular, $|N|\geq|P_2|$ so, by Lemma~\ref{l:eleven}, we have 
\begin{equation}
\label{p2upb}
|P_2|\leq(\half\aII+\eta^{1/2})k.
\end{equation} 

  \begin{figure}[!h]
\centering{
\mbox{\hspace{-2mm}{\includegraphics[width=64mm, page=25]{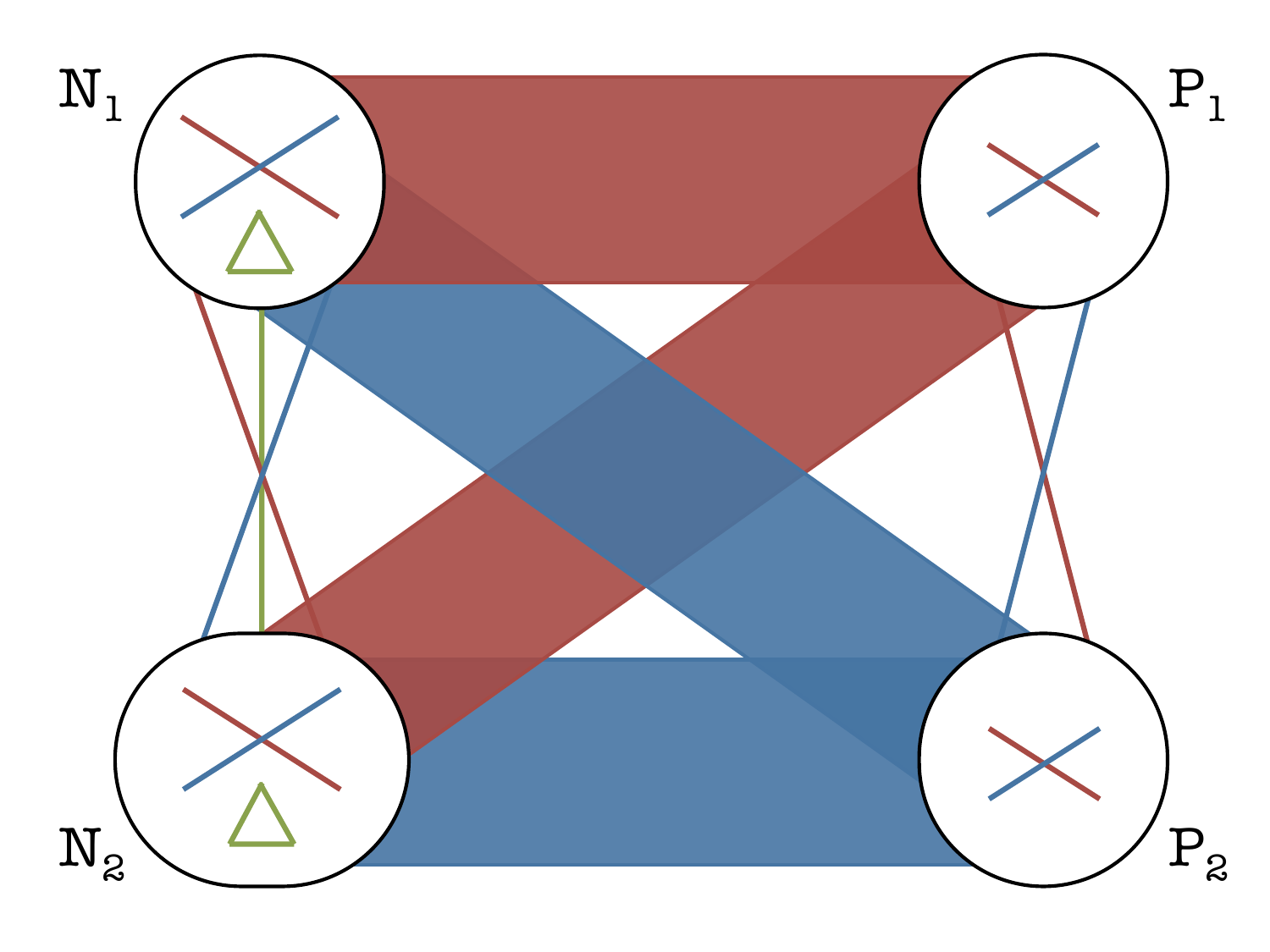}}\quad{\includegraphics[width=64mm, page=26]{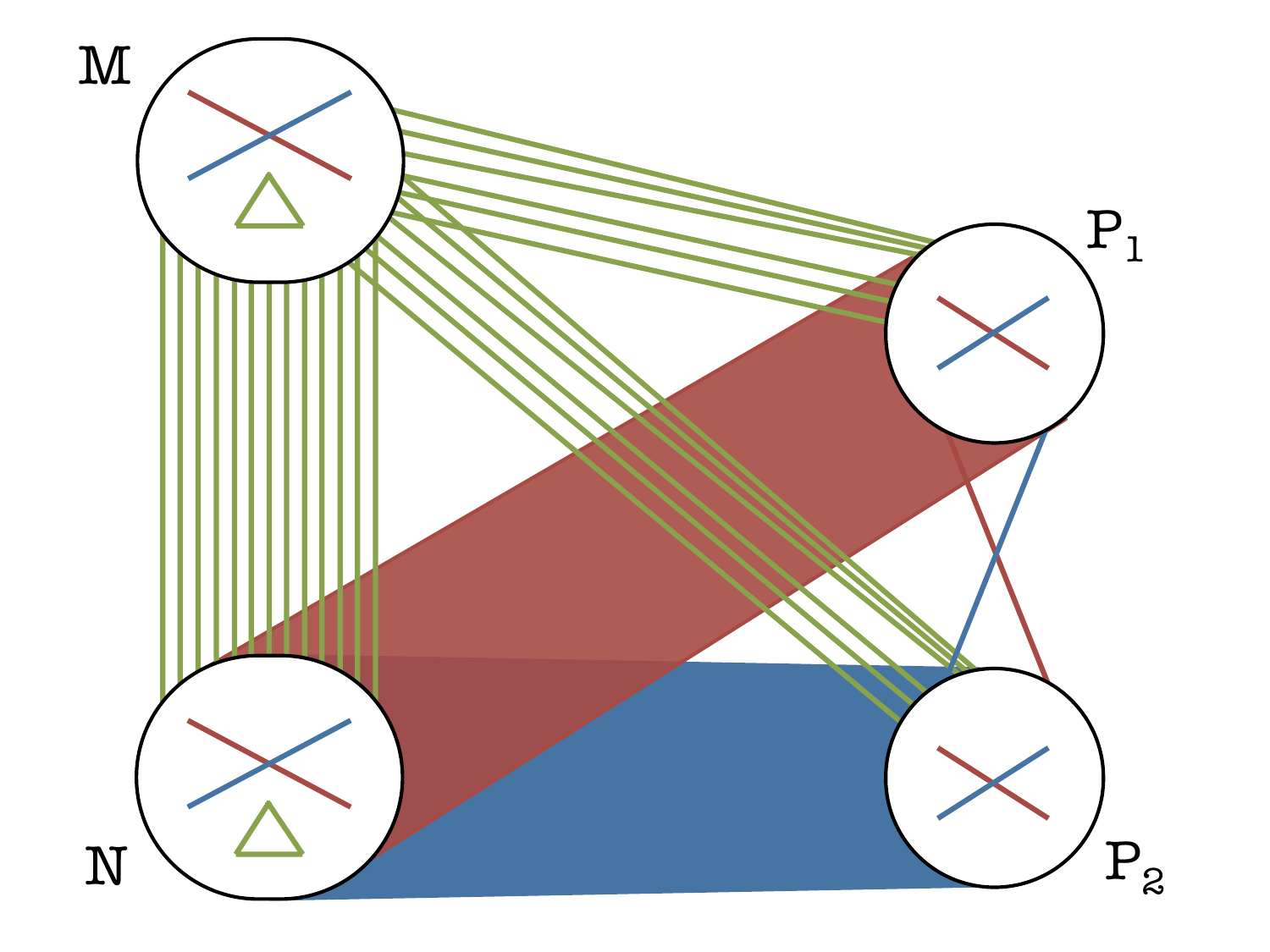}}}}
\caption{Colouring after three rounds of discarding vertices.}   
\end{figure}

Having determined the colouring of the red-blue graph $G[N,P]$, we now expand our sights to $G[M\cup N]$ and $G[M,P]$, each of which can include green edges:

Suppose there exists a red matching $R_A$ on $986\eta^{1/2}k$ vertices in $G[N]\cup G[M,N]$. Then, by~(\ref{Eiia3}) and~(\ref{nlargea}), we have $|N\backslash V(R_A)|\geq|P_1|\geq(\half\aI - 492\eta^{1/2})k$ so, by Lemma~\ref{l:eleven}, there exists a red connected-matching $R_B$ on at least $(\aI -986\eta^{1/2})k$ in $G[N\backslash V(R_A),P_1]$. Since all edges in $G[N,P_1]$ are coloured exclusively red, $R_A$ and $R_B$ belong to the same red component and, thus, together form a red connected-matching on at least $\aI k$ vertices. Similarly, if there exists a blue matching on at least $594\eta^{1/2}k$ vertices in $G[N]\cup G[M,N]$, then this can be used along with $G[N,P_2]$ to give a blue connected-matching on at least $\aII k$ vertices. Thus, after discarding at most $790\eta^{1/2}k$ vertices from~$M$ and at most $2370\eta^{1/2}k$ vertices from~$N$, we may assume that all edges present in $G[N]$ and $G[M\cup N]$ are green.  

\begin{figure}[!h]
\centering
\includegraphics[width=64mm, page=27]{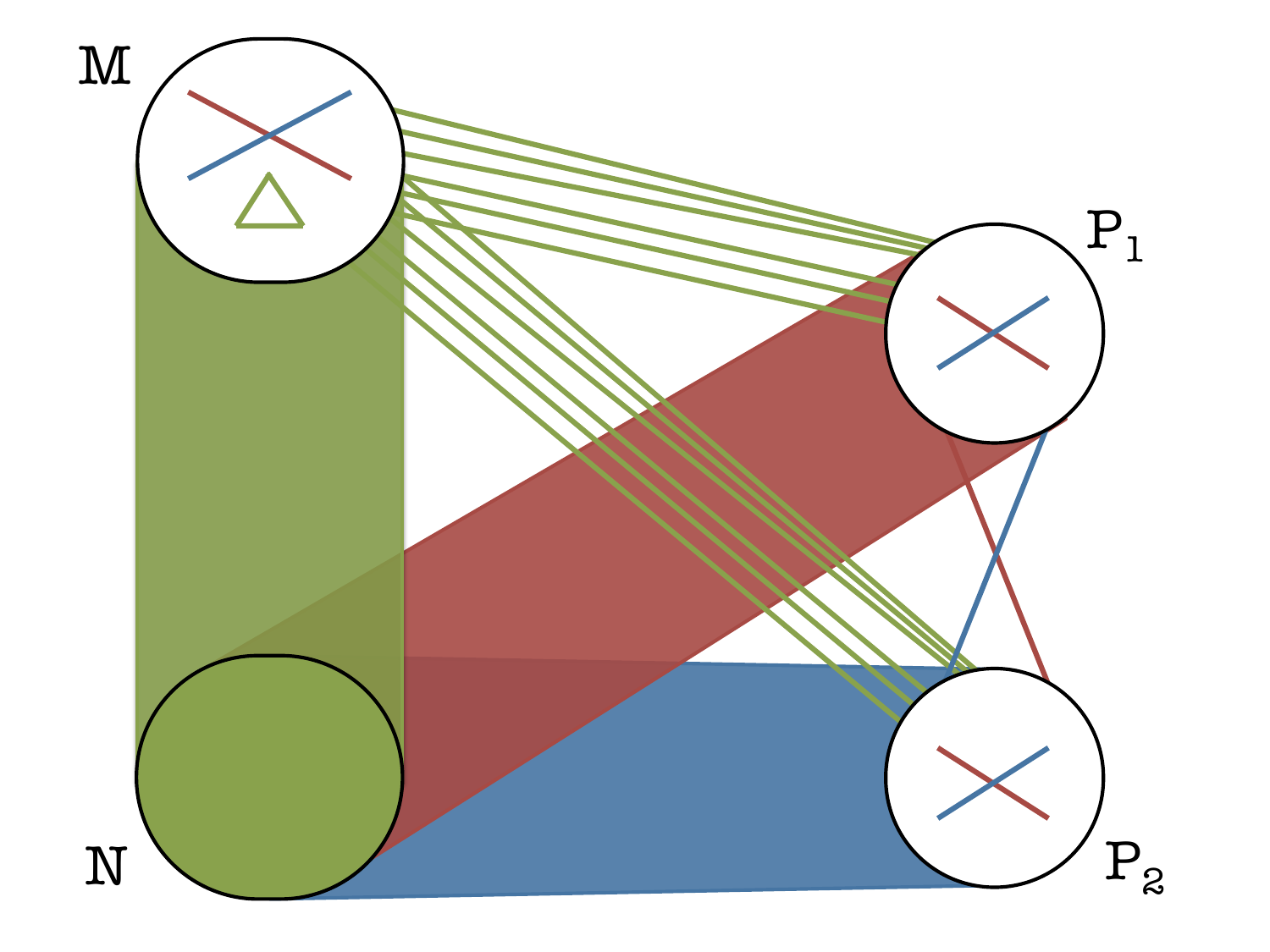}
\vspace{-2mm}\caption{Colouring of the edges of $G[N]\cup G[N,M]$.}
  
\end{figure}

After discarding these vertices, we have
\begin{equation}
\label{Eiia4}
\left.
\begin{aligned}
\,\,\,\,\,\quad\quad\quad\quad\quad\quad\quad\quad\quad(\half\aIII-800\eta^{1/2})k&\leq|M|\leq \half\aIII k,\,\quad\quad\quad\quad\quad\quad\quad\quad\quad  \\
(\half\aIII-3300\eta^{1/2})k&\leq|N|\leq \half\aIII k.
\end{aligned}
\right\}
\end{equation}

Next, suppose there exists a green matching $G_S$ on $8240\eta^{1/2}k$ vertices in $G[M,P]$. By~(\ref{Eiia4}), we have $|M\backslash V(G_S)|\geq (\half\aIII-4920\eta^{1/2})k$. Then, taking $\widetilde{N}$ to be any subset of $(\half\aIII-4920\eta^{1/2})k$ vertices in~$N$, by Lemma~\ref{l:eleven}, there exists a green connected-matching $G_L$ on at least $(\aIII-9842\eta^{1/2})k$ vertices in $G[M\backslash V(G_S), \widetilde{N}]$. Also, since $|N\backslash \widetilde{N}|\geq 1618\eta^{1/2}k$, by Theorem~\ref{dirac}, there exists a green matching $G_T$ on $1610\eta^{1/2}k$ vertices in $G[N\backslash \widetilde{N}]$. Then, together $G_S$, $G_L$ and $G_T$ form a green connected-matching on at least $\aIII k$ vertices which is odd by the definition of the decomposition. Thus, we may discard at most $4120\eta^{1/2}k$ vertices from each of $M$ and~$P$ such that none of the edges present in $G[P,M\cup N]$ are coloured green. 

Then, recalling~(\ref{Eiia3}),~(\ref{p2upb}) and~(\ref{Eiia4}), we have
\begin{equation}
\label{Eiia5a}
\left.
\begin{aligned}
\,\,\,\,\quad\quad\quad\quad\quad\quad\quad(\half\aI-4612\eta^{1/2})k&\leq|P_1|\leq (\half\aI+\eta^{1/2})k,\quad\quad\quad\quad\quad\quad\quad\\
(\half\aII-4415\eta^{1/2})k&\leq|P_2|\leq (\half\aII+\eta^{1/2})k,\\
(\half\aIII-4920\eta^{1/2})k&\leq|M|\leq \half\aIII k,\\
(\half\aIII-3300\eta^{1/2})k&\leq|N|\leq \half\aIII k.
\end{aligned}
\right\}
\end{equation}
Now, suppose there exists a red matching $R_S$ on $9230\eta^{1/2}k$ vertices in $G[M,P_2]$. By~(\ref{nlargea}) and~(\ref{Eiia5a}),  we have $|N|\geq|P_1|\geq(\half\aI -4612\eta^{1/2})k$. So, by Lemma~\ref{l:eleven}, there exists a red connected-matching $R_L$ on at least $2|P_1|-2\eta^{1/2}k\geq (\aI-9226\eta^{1/2})k$ vertices in $G[N,P_1]$. Since $G[P]$ has a red effective-component on at least $|P|-\eta^{1/2}k$, $R_L$ belongs to the same red component as at least $4614\eta^{1/2}k$ of the edges of $R_S$, thus giving a red connected-matching on at least $\aI k$ vertices in $G[M,P_2]\cup G[N,P_1]$. Therefore, after discarding at most $4615\eta k$ vertices from each of $M$ and $P_2$, we may assume that all edges present in $G[M,P_2]$ are coloured exclusively blue. 

We then have
\begin{equation}
\label{Eiia5}
\left.
\begin{aligned}
\,\,\,\,\quad\quad\quad\quad\quad\quad\quad(\half\aI-4612\eta^{1/2})k&\leq|P_1|\leq (\half\aI+\eta^{1/2})k,\quad\quad\quad\quad\quad\quad\quad\\
(\half\aII-9030\eta^{1/2})k&\leq|P_2|\leq (\half\aII+\eta^{1/2})k,\\
(\half\aIII-9335\eta^{1/2})k&\leq|M|\leq \half\aIII k,\\
(\half\aIII-3300\eta^{1/2})k&\leq|N|\leq \half\aIII k.
\end{aligned}
\right\}
\end{equation}
Similarly, suppose there exists a blue matching $B_S$ on  $18062\eta^{1/2}k$ vertices in $G[M,P_1]$. By~(\ref{nlargeb}) and~(\ref{Eiia5}), $|N|\geq|P_2|\geq(\half\aI -9030\eta^{1/2})k$. So, by Lemma~\ref{l:eleven}, there exists a blue connected-matching $B_L$ on at least $2|P_2|-2\eta^{1/2}k\geq (\aI-18062\eta^{1/2})k$ vertices in $G[N,P_2]$. Then, since $G$ is $4\eta^4 k$-almost-complete and all edges present in $G[M,P_2]$ are coloured exclusively blue, $B_L$ and $B_S$ belong to the same blue component of $G$, and thus, together, form a blue connected-matching on at least $\aII k$ vertices. Thus, after discarding at most $9031\eta k$ vertices from each of $M$ and $P_1$, we may assume that all edges present in $G[M,P_1]$ are coloured exclusively red. We then have
\begin{equation}
\label{Eiia6}
\left.
\begin{aligned}
\,\,\,\,\,\,\,\,\quad\quad\quad\quad\quad\quad\quad(\half\aI-13643\eta^{1/2})k&\leq|P_1|\leq (\half\aI+\eta^{1/2})k,\quad\quad\quad\quad\quad\quad\,\\
(\half\aII-9030\eta^{1/2})k&\leq|P_2|\leq (\half\aII+\eta^{1/2})k,\\
(\half\aIII-18366\eta^{1/2})k&\leq|M|\leq \half\aIII k,\\
(\half\aIII-3300\eta^{1/2})k&\leq|N|\leq \half\aIII k.
\end{aligned}
\right\}\!\!
\end{equation}
Finally, since $|N|\geq|P_1|,|P_2|$, if there existed a red matching on at least $27288\eta^{1/2}k$ vertices in $G[M]$ or a blue matching on at least $18062\eta^{1/2}k$ vertices in $G[M]$, then we could obtain a red connected-matching on at least $\aI k$ vertices or a blue connected-matching on at least $\aII k$ vertices. Thus, after discarding at most $45350\eta^{1/2}k$ further vertices from $M$, we may assume that all edges present in $G[M\cup N]$ are green.

\begin{figure}[!h]
\centering
\includegraphics[width=64mm, page=31]{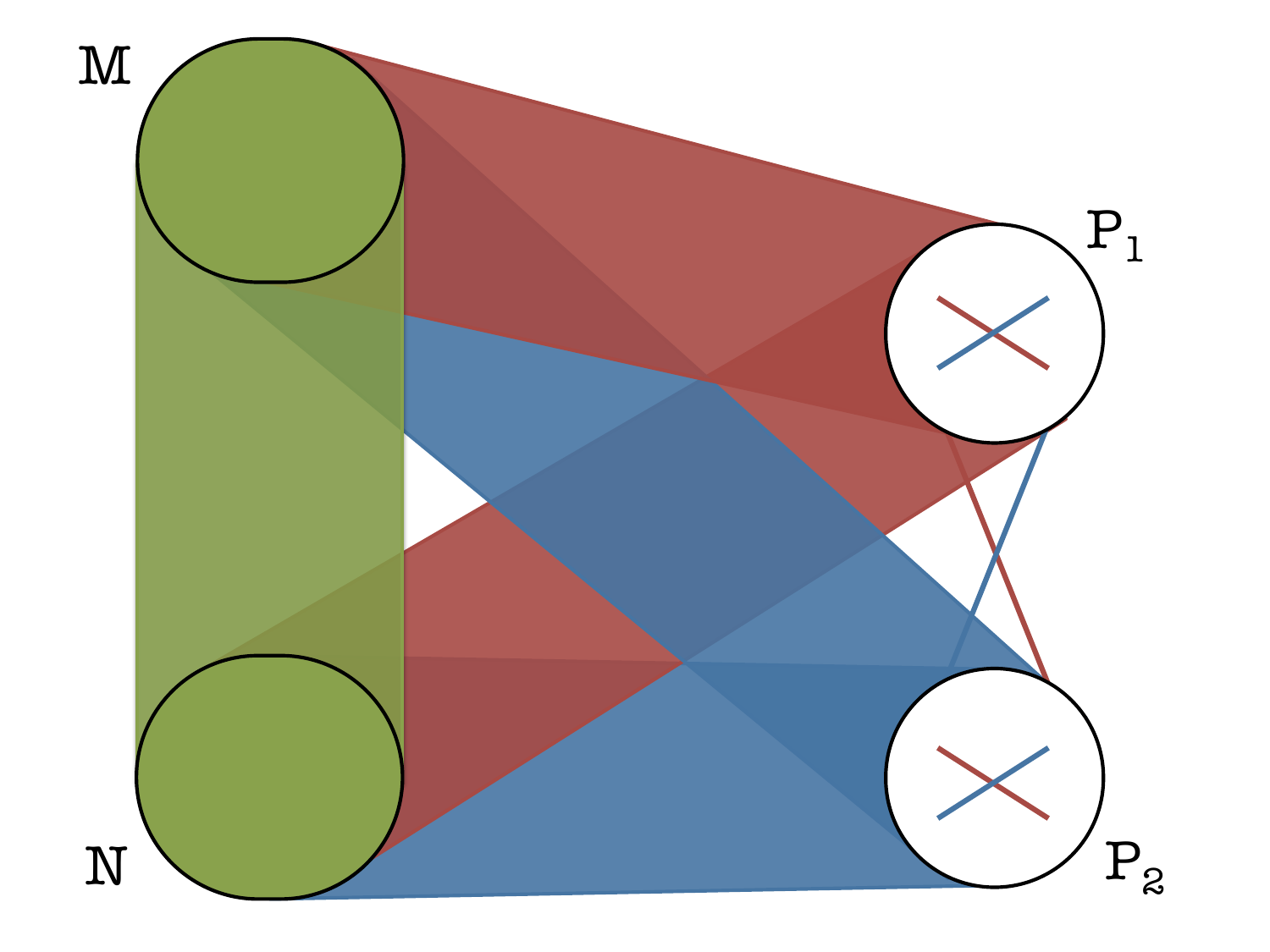}
\vspace{-2mm}\caption{Final colouring in Case E.ii.a.}
  
\end{figure}

In summary, we now have
\begin{align*}
|P_1|&\geq(\half\aI-13643\eta)k,  &
|P_2|&\geq(\half\aII-9030\eta)k,  &
|M\cup N|&\geq(\aIII-67216\eta^{1/2})k,  
\end{align*}

and know that all edges present in $G[M\cup N, P_1]$ are coloured exclusively red, all edges present $G[M\cup N, P_2]$ are coloured exclusively blue and all edges in $G[M\cup N]$ are coloured exclusively green.

We have thus found, as a subgraph of $G$, a graph belonging to
$$\cK\left((\half\aI-14000\eta^{1/2})k, (\half\aII-14000\eta^{1/2})k, (\aIII-68000\eta^{1/2})k, 4\eta^4 k \right).$$ 

\subsection*{Case E.ii.b: {\rm Most edges in $G_2[N_1,P_2]$ are red.}}

Recall that we have a decomposition of $V(G)$ into $M\cup N\cup P \cup Q$ such that:
\begin{itemize}
\labitem{E1}{E1-iib} $M\cup N$ is the vertex set of~$F$ and every edge of~$F$ belongs to $G[M,N]$;
\labitem{E2}{E2-iib} every vertex in~$P$ has a green edge to~$M$;
\labitem{E3}{E3-iib} there are no green edges in $G[N,P]$, $G[M,Q]$, $G[N,Q]$, $G[P,Q]$ or $G[P]$.
\end{itemize}
Recall also that
\begin{align}\label{E4a-iib}\tag{E4a}(\max\{\tfrac{3}{4}\aI+\tfrac{1}{4}\aII,\half\aIII\}-5\eta^{1/2})k & \leq  |M|\,,\,|N| \leq \half\aIII k,\\
\label{E4b-iib}\tag{E4b}(\half\aI+\half\aII-\eta)k & \leq  |P|+|Q| \leq (\half\aI+\half\aII+5\eta^{1/2})k.
\end{align}

Furthermore, recall that the largest red matching $R$ in $G[N,P]$ defines a partition of $N$ into $N_1\cup N_2$ and $P$ into $P_1\cup P_2$ such that the edges of $R$ belong $G[N_1,P_1]$, all edges present in $G[N_2,P_2]$ are coloured exclusively blue and that
\begin{align*}
(\half\aI-97\eta^{1/2})k&\leq|P_1|\leq (\half\aI+\eta^{1/2})k,\\
(\half\aII-97\eta^{1/2})k&\leq|P_2|\leq (\half\aII+32\eta^{1/2})k,\\
(\half\aI-97\eta^{1/2})k&\leq|N_1|\leq (\half\aI+\eta^{1/2})k,\\
(\half\aII-97\eta^{1/2})k&\leq|N_2|. 
\end{align*}

Additionally, in this case, we assume that the largest blue matching in $G[N_1,P_2]$ spans at most $198\eta^{1/2}k$ vertices. Thus, we can discard at most $99\eta^{1/2}k$ vertices from each of~$N_1$ and $P_2$ so that all edges present in $G[N_1,P_2]$ are coloured exclusively red. 

\begin{figure}[!h]
\centering
\includegraphics[width=64mm, page=56]{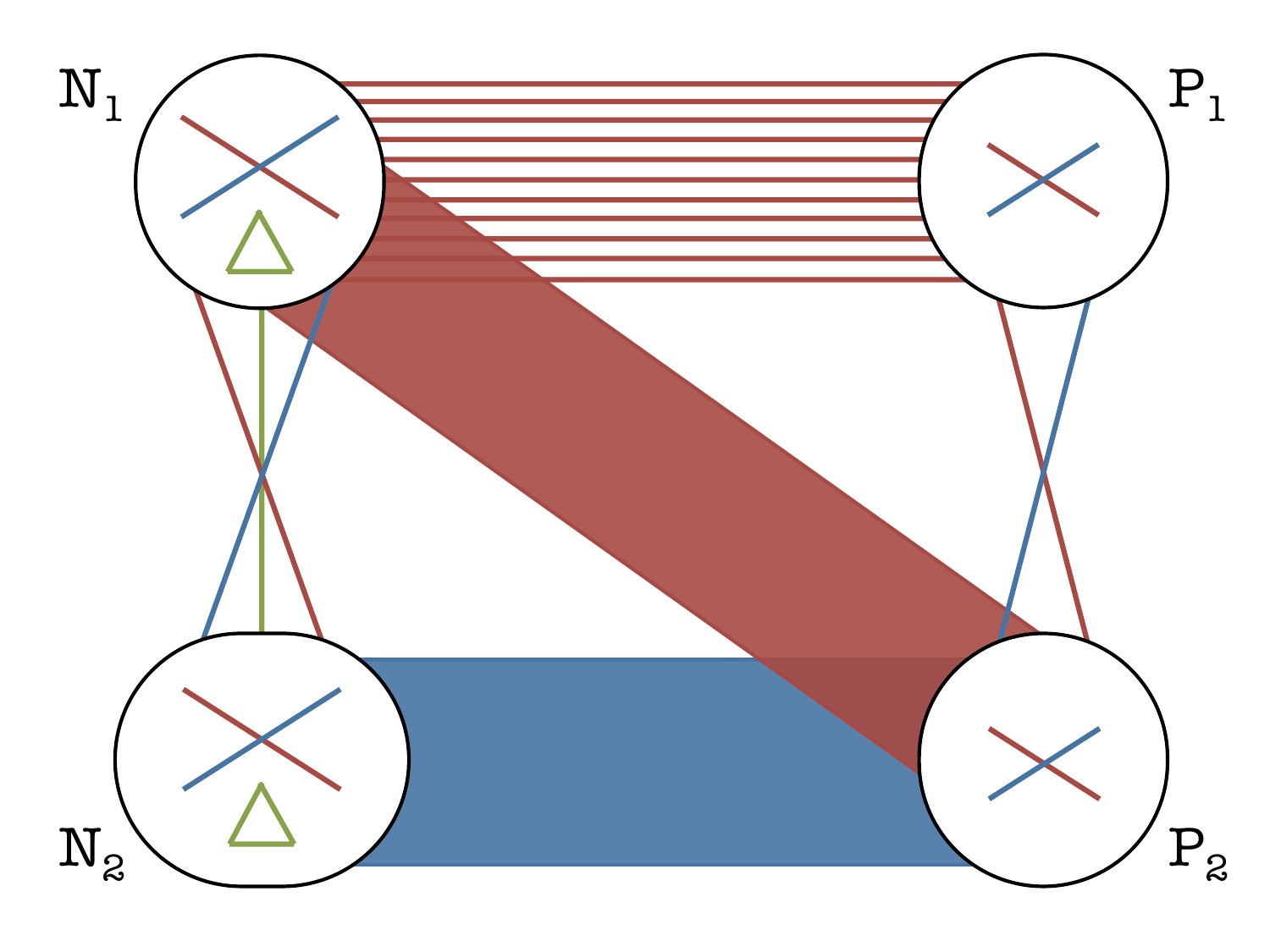}
\vspace{-2mm}\caption{Initial colouring in Case E.ii.b.}
  
\end{figure}

In order to retain the equality $|N_1|=|P_1|$ and the property that every vertex in $P_1$ belongs to an edge of $R$, we discard from $P_1$ each vertex whose $R$-mate in $N_1$ has already been discarded. We then have
\begin{align*}
(\half\aI-196\eta^{1/2})k&\leq|P_1|\leq (\half\aI+\eta^{1/2})k,\\
(\half\aII-196\eta^{1/2})k&\leq|P_2|\leq (\half\aII+32\eta^{1/2})k,\\
(\half\aI-196\eta^{1/2})k&\leq|N_1|\leq (\half\aI+\eta^{1/2})k,\\
(\half\aII-97\eta^{1/2})k&\leq|N_2|. 
\end{align*}

Then, suppose there exists a red matching on $396\eta^{1/2}k$ vertices in $G[N_2,P_1]$. Such a matching could be used together with $G[N_1,P_1]$ and $G[N_1,P_2]$ to give a red connected-matching on at least $\aI k$ vertices. Thus, we can discard at most $198\eta^{1/2}k$ vertices from each of $N_2$ and $P_1$ so that all edges present in $G[N_2,P_1]$ are coloured exclusively blue. 

We then have
\begin{align*}
(\half\aI-394\eta^{1/2})k&\leq|P_1|\leq (\half\aI+\eta^{1/2})k,\\
(\half\aII-196\eta^{1/2})k&\leq|P_2|\leq (\half\aII+32\eta^{1/2})k,\\
(\half\aI-196\eta^{1/2})k&\leq|N_1|\leq (\half\aI+\eta^{1/2})k,\\
(\half\aII-295\eta^{1/2})k&\leq|N_2|. 
\end{align*}

Then, if there existed a blue matching on $592\eta^{1/2}k$ vertices in $G[P_1,N_1]$, this could be used along with $G[P,N_2]$ to give a blue connected-matching on at least $\aII k$ vertices. Thus, after discarding at most $296\eta^{1/2}k$ vertices from each of $N_1$ and $P_1$, we may assume that all edges present in $G[N_1,P_1]$ are coloured exclusively red 
and that 
\begin{equation}
\label{Eiib0}
\left.
\begin{aligned}
\,\,\,\,\,\,\quad\quad\quad\quad\quad\quad\quad (\half\aI-690\eta^{1/2})k&\leq|P_1|\leq (\half\aI+\eta^{1/2})k, \quad\quad\quad\quad\quad\quad\quad \\
(\half\aII-196\eta^{1/2})k&\leq|P_2|\leq (\half\aII+32\eta^{1/2})k,\\
(\half\aI-492\eta^{1/2})k&\leq|N_1|\leq (\half\aI+\eta^{1/2})k,\\
(\half\aII-295\eta^{1/2})k&\leq|N_2|.
\end{aligned}
\right\}
\end{equation}
  Observe, now, that, given any subset $P'$ of $(\half\aI-492\eta^{1/2})k$ vertices from~$P$, by Lemma~\ref{l:eleven}, there exists a red connected-matching on at least $(\aI-986\eta^{1/2})k$ vertices in $G[N_1,P']$. Thus, the largest red matching in $G[P]$ spans at most $986\eta^{1/2}k$ vertices. Similarly, given any subset $P''$ of $(\half\aI-295\eta^{1/2})k$ vertices from~$P$, we can find a blue connected-matching on at least $(\aI-592\eta^{1/2})k$ vertices in $G[N_2,P'']$. Thus, the largest blue matching in $G[P]$ spans at most $592\eta^{1/2}k$ vertices.  Thus, since there are no green edges within $G[P]$, we have $|P|\leq 1600\eta^{1/2}k$, which, since $\eta\leq(\aI/10000)^2$, contradicts~(\ref{Eiib0}), completing Case E.ii.b.

  \begin{figure}[!h]
\centering
\includegraphics[width=64mm, page=33]{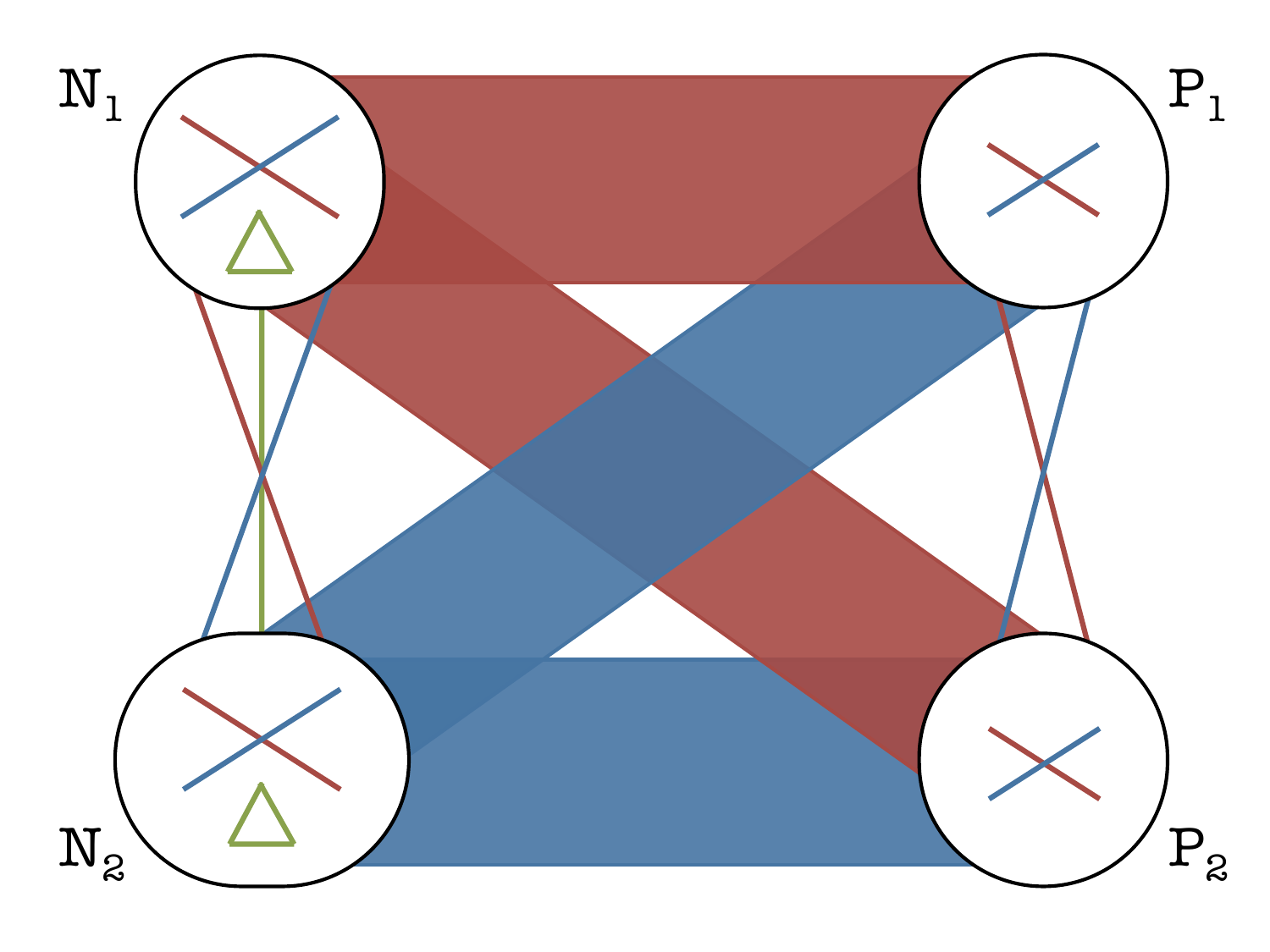}
\vspace{-3mm}\caption{Colouring of the edges of $G[N,P]$.} \label{f236f}
\end{figure}

At the beginning of Case E.ii, we assumed that $F_P$, the largest monochromatic component in $G[P]$, was red. If instead that monochromatic component is blue, then the proof proceeds exactly as above following the same steps with the roles of red and blue exchanged and with $\aI$ and $\aII$ exchanged. The result is identical.

\subsection*{Case E.iii: $|P|,|Q|\geq 95\eta^{1/2}k.$}

We now consider the case when neither $P$ nor $Q$ is trivially small. This case is fairly involved, combining elements of Case E.i and Case E.ii with new arguments. However, because neither $P$ nor $Q$ is trivially small, we can exploit the structure of the two coloured graph $G[P\cup Q]$.

Recall that we have a decomposition of $V(G)$ into $M\cup N\cup P \cup Q$ such that:
\begin{itemize}
\labitem{E1}{E1-iii} $M\cup N$ is the vertex set of~$F$ and every edge of~$F$ belongs to $G[M,N]$;
\labitem{E2}{E2-iii} every vertex in~$P$ has a green edge to~$M$;
\labitem{E3}{E3-iii} there are no green edges in $G[N,P]$, $G[M,Q]$, $G[N,Q]$, $G[P,Q]$ or $G[P]$.
\end{itemize}

\begin{figure}[!h]
\centering
\includegraphics[width=64mm, page=5]{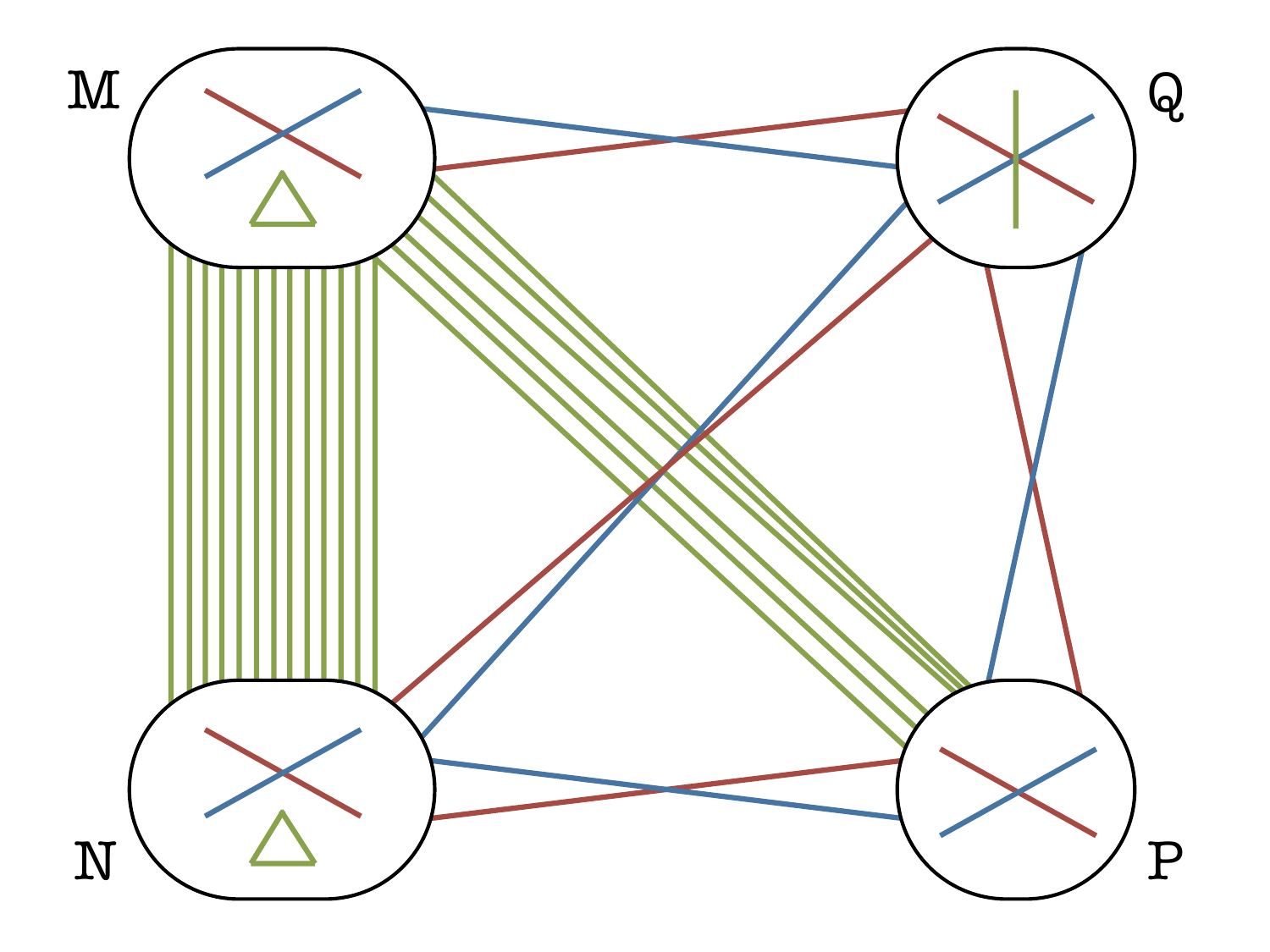}
\vspace{-3mm}\caption{Decomposition in Case E.iii.}
  
\end{figure}

Recall also that
\begin{align}\label{E4a-iii}\tag{E4a}(\max\{\tfrac{3}{4}\aI+\tfrac{1}{4}\aII,\half\aIII\}-5\eta^{1/2})k & \leq  |M|\,,\,|N| \leq \half\aIII k,\\
\label{E4b-iii}\tag{E4b}(\half\aI+\half\aII-\eta)k & \leq  |P|+|Q| \leq (\half\aI+\half\aII+5\eta^{1/2})k.
\end{align}
In this case, we assume that $|P|,|Q|\geq 95\eta^{1/2} k$. 
Recall that $G[P\cup Q]$ is $4\eta^4 k$-almost-complete. Then, since $|P|+|Q|\geq 190\eta^{1/2}k$, we have $4\eta^4 k \leq \eta^2 (|P|+|Q|-1)$ and, thus, $G[P\cup Q]$ is $(1-\eta^2)$-complete. Since $\aII\leq\aI\leq 1$, recalling (E4b), we have $|P|+|Q|\leq(\half\aI+\half\aII+5\eta^{1/2})k\leq(1+5\eta^{1/2})k$ and so $$|P|,|Q|\geq 95\eta^{1/2}k\geq4\eta^{1/2}(1+5\eta^{1/2})\geq4\eta^{1/2}(|P|+|Q|)\geq4\eta(|P|+|Q|).$$ Thus, provided $k\geq 1/(190\eta^{5/2})$, we may apply Lemma~\ref{l:dgf1} to $G[P\cup Q]$, giving rise to two possibilities:
\begin{itemize}
\item[(a)] $P\cup Q$ contains a monochromatic component $F$ on at least $|P\cup Q|-8\eta k$ vertices;
\item[(b)] there exist vertices $w_r, w_b\in Q$ such that $w_r$ has red edges to all but $8\eta k$ vertices in~$P$ and $w_b$ has blue edges to all but $8\eta k$ vertices in~$P$.

\end{itemize}

\subsection*{Case E.iii.a: {\rm $P\cup Q$ has a large monochromatic component.}}

Suppose that $F$, the largest monochromatic component in $P\cup Q$ is red. We consider $G[M,Q]$ and $G[N,P]$, both of which have only red and blue edges, and let $R$ be the largest red matching in $G[M,Q]\cup G[N,P]$. We partition each of $M,N,P,Q$ into two parts such that $M_1=M\cap V(R)$, $M_2=M\backslash M_1$, $N_1=N\cap V(R)$, $N_2=N\backslash N_1$,
$P_1=P\cap V(R)$, $P_2=M\backslash P_1$,
$Q_1=Q\cap V(R)$ and $Q_2=Q\backslash Q_1$. 
Observe that, by maximality of $R$, all edges present in $G[P_2,N_2]$ and $G[Q_2,N_2]$ are blue.

\begin{figure}[!h]
\centering
\includegraphics[width=64mm, page=2]{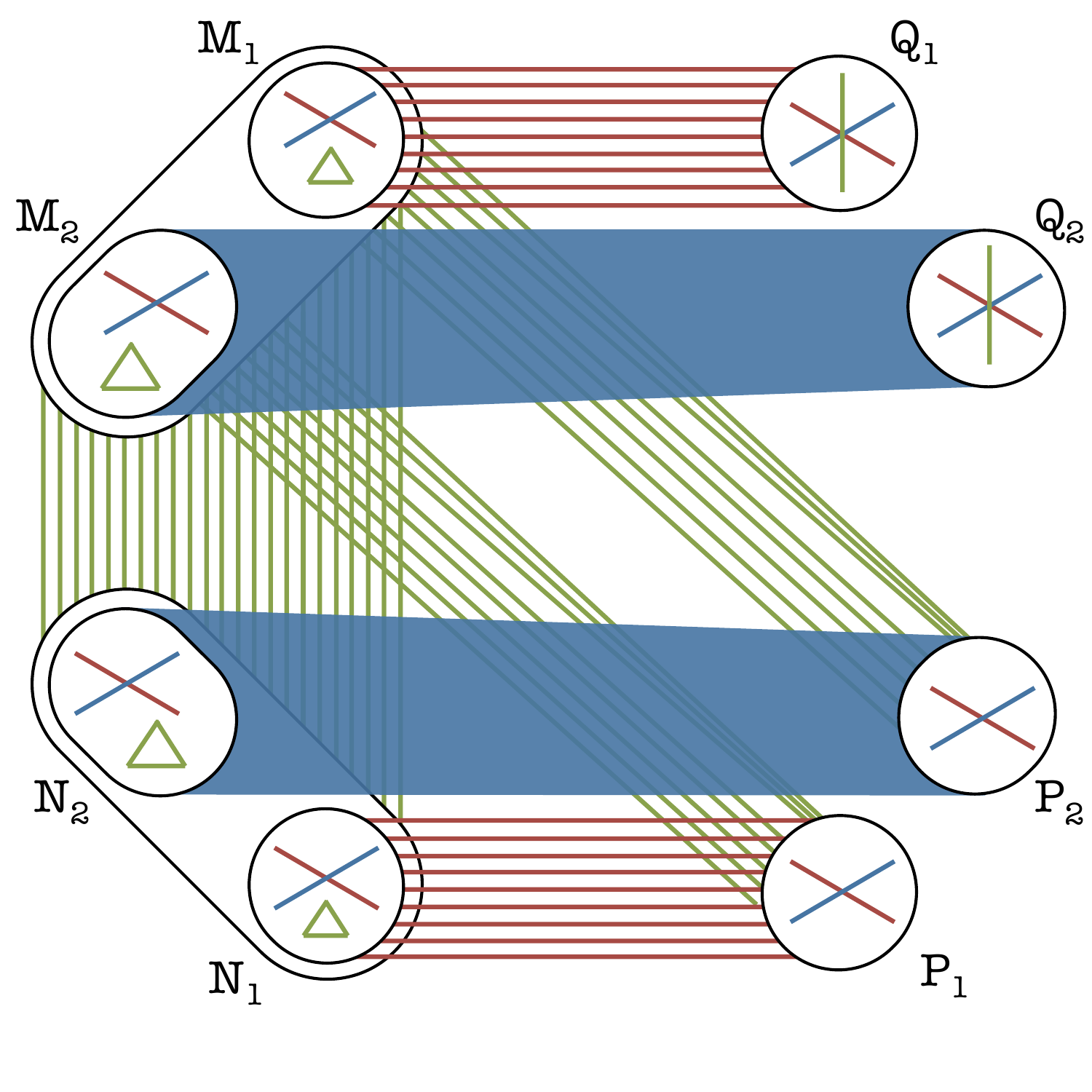}
  
\vspace{-4mm}
\caption{Decomposition into eight parts.}
\end{figure}

Notice that, since $G$ is $4\eta^4 k$-almost-complete, $G[M_2\cup Q_2]$ and $G[N_2\cup P_2]$ each have a single blue component. We then consider two subcases:
\begin{itemize}
\item[(i)] $P_2$ and $Q_2$ belong to the same blue component of $G$;
\item[(ii)] $P_2$ and $Q_2$ belong to different blue components of $G$.
\end{itemize}

\subsection*{Case E.iii.a.i: {\rm $P_2$ and $Q_2$ belong to the same blue component.}}

Since $F$, the largest red component in $P\cup Q$, contains at least $|P\cup Q|-8\eta k$ vertices, all but at most $8\eta k$ of the edges of $R$ belong to $F$. Thus, since $|M_1|=|Q_1|$ and $|N_1|=|P_1|$, we have a red connected-matching on at least $2(|P_1|+|Q_1|-8\eta k)$ vertices and so we may assume that 
\begin{equation}
\label{Eiiai1} |P_1|+|Q_1|\leq (\half\aI+8\eta)k
\end{equation} in order to avoid having a red connected-matching on at least $\aI k$ vertices.

Recalling (\ref{E4a-iii}) and (\ref{E4b-iii}), since $\eta\leq(\aI/100)^2$, we have 
\begin{equation*}
\label{Eiiiai2} |M|+|N|\geq|P|+|Q|+80\eta^{1/2}k
\end{equation*}
and also
$$|M|,|N|\geq\tfrac{3}{4}\aI+\tfrac{1}{4}\aII-5\eta^{1/2}k\geq \half\aI+\half\aII-5\eta^{1/2}k\geq |P|+|Q|-10\eta^{1/2}.$$
Thus, since $|M_1|=|Q_1|$, $|N_1|=|P_1|$ and $|P|,|Q|\geq 95\eta^{1/2}k$, we have
\begin{subequations}
\begin{align}
\label{Eiiai4a}|M_2|&\geq |P|+|Q_2|-10\eta^{1/2}k\geq |Q_2|+85\eta^{1/2}k,\\
\label{Eiiai4b}|N_2|&\geq |P_2|+|Q|-10\eta^{1/2}k\geq |P_2|+85\eta^{1/2}k.
\end{align}
\end{subequations}
In particular, we have $|M_2|\geq|Q_2|$ and $|N_2|\geq|P_2|$. Thus, since $P_2$ and $Q_2$ belong to the same effective-blue component, by Lemma~\ref{l:eleven}, there exists a blue connected-matching on at least $(2|P_2|-2\eta k)+(2|Q_2|-2\eta k)$ vertices in $G[N_2,P_2]\cup G[M_2,Q_2]$. Thus, we may assume that
\begin{equation}
\label{Eiiai5} 
|P_2|+|Q_2|\leq (\half\aII+2\eta)k,
\end{equation}
in order to avoid having a blue connected-matching on at least $\aII k$ vertices.

Then, by (\ref{E4a-iii}),~(\ref{E4b-iii}),~(\ref{Eiiai1}) and~(\ref{Eiiai5}), we have
\begin{align}
\label{Eiiai5a} |P_1|+|Q_1|\geq (\half\aI-3\eta)k, \text{\hspace{8mm}} |P_2|+|Q_2|\geq (\half\aII-9\eta)k.
\end{align}

We now proceed to determine the coloured structure of $G$. We begin by proving the following claim whose proof follows the same three steps as that of Claim~\ref{claimLQ}. Considering in parallel $G[M,Q]$ and $G[N,P]$, the first step in the proof is to show that, after possibly discarding some vertices, all edges contained in $G[M_2,Q_1]\cup G[N_2,P_1]$ are coloured exclusively red, the second is to show that, after possibly discarding further vertices, all edges contained in $G[M_1,Q_2]\cup G[N_1,P_2]$ are coloured exclusively blue and the third is to show that, after possibly discarding still more vertices, all edges contained in $G[M_1,Q_1]\cup G[N_1,P_1]$ are coloured exclusively red.

\begin{claim}
\label{claimMNPQ}
We may discard at most $70\eta k$ vertices from $P_1 \cup Q_1$, at most $48\eta^{1/2}k$ vertices from $P_2\cup Q_2$, at most $94\eta^{1/2}k$ vertices from each of $M_1$ and $N_1$ and at most~$11\eta k$~vertices from $M_2\cup N_2$ such that, in what remains, all edges present in $G[M,Q_1]\cup G[N,P_1]$ are coloured exclusively red and all edges present in $G[M,Q_2]\cup G[N,P_2]$ are coloured exclusively blue.   
\end{claim}
\begin{proof}
Suppose there exists a blue matching $B_S$ on at least $22\eta k$ vertices in $G[M_2, Q_1]\cup G[N_2, P_1]$. By~(\ref{Eiiai4a}) and~(\ref{Eiiai4b}), we have $|M_2\backslash V(B_S)|\geq |Q_2|$ and $|N_2\backslash V(B_S)|\geq |P_2|$. Thus, since $P_2$ and $Q_2$ belong to the same blue component of $G$, applying Lemma~\ref{l:eleven} to each of $G[M_2\backslash V(B_S),Q_2]$ and $G[N_2\backslash V(B_S),P_2]$ gives a blue connected-matching~$B_L$ in $G[M_2\backslash V(B_S),Q_2]\cup G[N_2\backslash V(B_S),P_2]$ on at least $(2|P_2|-2\eta k)+(2|Q_2|-2\eta k)$ vertices which belongs to the same blue component of $G$ as $B_S$ but shares no vertices with it. Thus, $B_L \cup B_S$ forms a blue connected-matching on at least $2|P_2|+2|Q_2|+18\eta k\geq \aII k$ vertices. Therefore, after discarding at most $11\eta k$ vertices from each of $M_2\cup N_2$ and $P_1\cup Q_1$, we may assume that all edges present in $G[M_2,Q_1]\cup G[N_2, P_1]$ are coloured exclusively red. 
\begin{figure}[!h]
\centering
\vspace{-2mm}
\includegraphics[width=64mm, page=1]{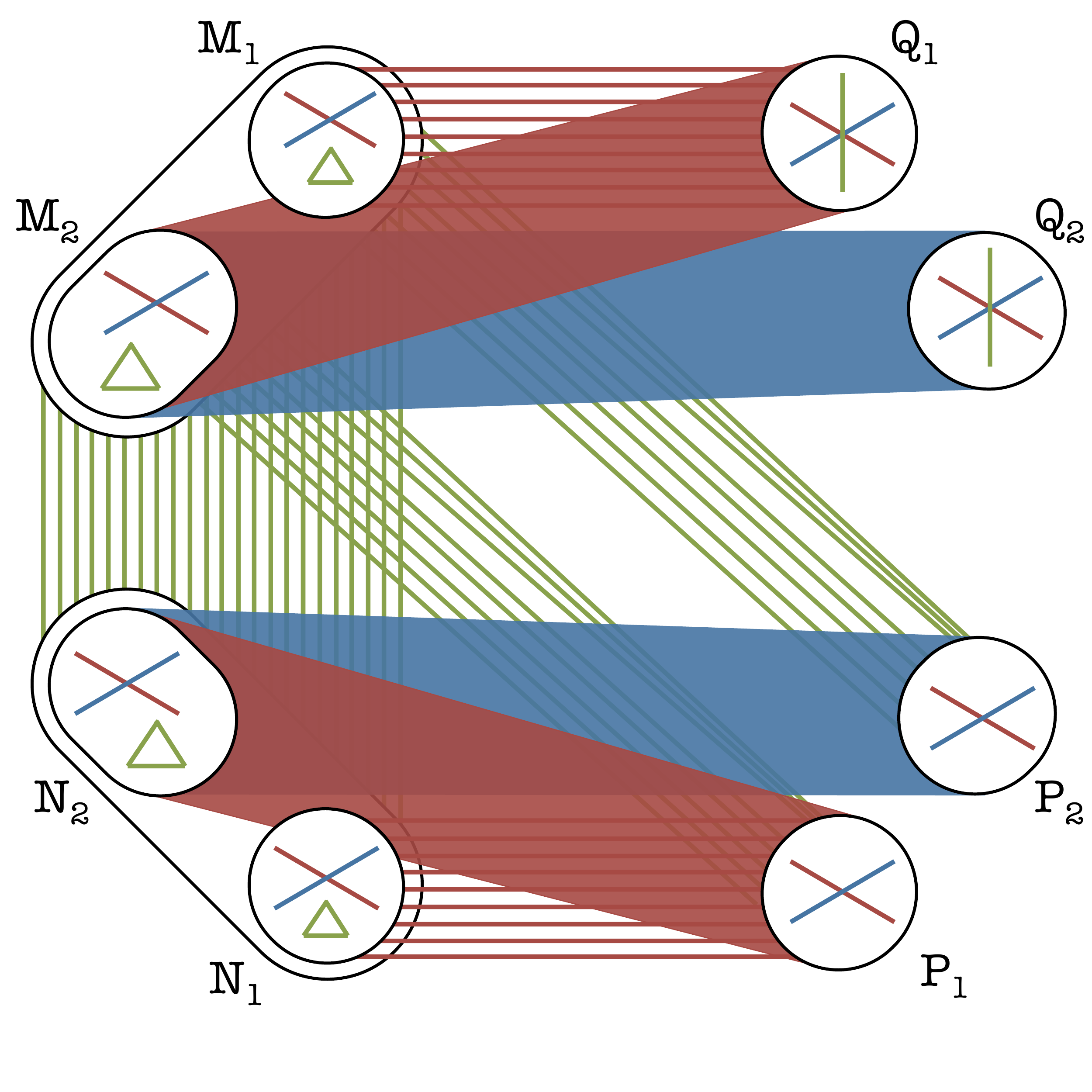}
\vspace{-4mm}\caption{Colouring of $G[M_2,Q_1]\cup G[N_2,P_1]$.}
  
\end{figure}
\vspace{-2mm}

Recalling~(\ref{Eiiai5a}), we then have
\begin{align*}
|P_1|+|Q_1|\geq(\half\aI-14\eta)k, \text{\hspace{8mm}}
|P_2|+|Q_2|\geq(\half\aII-9\eta)k.
\end{align*}

In order to retain the equalities $|M_1|=|Q_1|$ and $|N_1|=|P_1|$ and the property that every vertex in $M_1\cup N_1$ belongs to an edge of $R$, we discard from $M_1\cup N_1$ each vertex whose $R$-mate in $P_1\cup Q_1$ has already been discarded. Thus, we discard at most a further~$11\eta k$ vertices.

Then, if either $G[N_1,P_2]$ or $G[M_1,Q_2]$ contains a red matching on at least $48\eta k$ vertices in $G[N_1,P_2]$, then we may construct a red connected-matching on at least $\aI k$ vertices as follows:

Suppose that there exists a red matching $R_P$ on $48\eta k$ vertices in $G[N_1,P_2]$. Then, observe that there exists a set $R^{-}$ of $24\eta k$ edges belonging to $R$ such that $N_1\cap V(R_P)=N_1\cap V(R^{-})$. Define $\widetilde{P}=P_1\cap V(R^{-})$ and consider $G[N_2,\widetilde{P}]$. Since $|N_2|,|\widetilde{P}|\geq 24\eta k$ and $G[N_2,\widetilde{P}]$ is $4\eta^4 k$-almost-complete, we may apply Lemma~\ref{l:eleven} to find a red connected-matching $R_S$ on at least $46\eta k$ vertices in $G[N_2,\widetilde{P}]$. Then, recalling that $P\cup Q$ has a red effective-component including all but $8\eta k$ of its vertices, we have a red connected-matching $R^{\star}\subseteq (R\backslash R^{-})\cup R_S\cup R_P$ on at least 
\begin{align*}
2\left(|P_1\backslash V(R^{-})|+|Q_1|+|P_1\cap \right. &\left.V(R_S)|+|P_2\cap V(R_P)|-8\eta k \right) \\
&\geq 2\left(\half\aI-39\eta +23\eta +24\eta  - 8\eta \right)k\geq \aI k \text{ vertices}
\end{align*}
in $G[N_1\backslash V(R_P),P_1\backslash V(R_S)]\cup G[M_1,Q_1]\cup G[N_2,P_1\cap V(R_S)]\cup G[N_1\cap V(R_P),P_2].$

Similarly, suppose that there exists a red matching $R_Q$ on $48\eta k$ vertices in $G[M_1,Q_2]$. Then, observe that there exists a set $R^{=}$ of $24\eta k$ edges belonging to $R$ such that $M_1\cap V(R_Q)=M_1\cap V(R^{=})$. Define $\widetilde{Q}=Q_1\cap V(R^{=})$ and consider $G[M_2,\widetilde{Q}]$. Since $|M_2|,|\widetilde{Q}|\geq 24\eta k$ and $G[M_2,\widetilde{Q}]$ is $4\eta^4 k$-almost-complete, we may apply Lemma~\ref{l:eleven} to find a red connected-matching $R_T$ on at least $46\eta k$ vertices in $G[M_2,\widetilde{Q}]$. Then, recalling that $P\cup Q$ has a red effective-component including all but $8\eta k$ of its vertices, we have a red connected-matching $R^{*}\subseteq (R\backslash R^{=})\cup R_T\cup R_Q$ on at least
\begin{align*}
2\left(|P_1|+|Q_1\backslash V(R^{=})|+|Q_1\cap V(R_S)|+|Q_2\cap V(R_Q)|-8\eta k \right)\geq \aI k \text{ vertices}
\end{align*}
in $G[N_1,P_1]\cup G[M_1\backslash V(R_Q),Q_1\backslash V(R_T)] \cup G[M_2,Q_1\cap V(R_T)] \cup G[M_1\cap V(R_Q),Q_2].$

Thus, after discarding at most $24\eta k$ vertices from each of $M_1, N_1, P_2$ and $Q_2$, we may assume that all edges present in $G[M_1,Q_2]\cup G[N_1,P_2]$ are coloured exclusively blue. We then have
\begin{align*}
|P_1|+|Q_1|\geq(\half\aI-14\eta)k, \text{\hspace{8mm}} |P_2|+|Q_2|\geq(\half\aII-57\eta)k.
\end{align*}

\begin{figure}[!h]
\centering
\includegraphics[width=64mm, page=101]{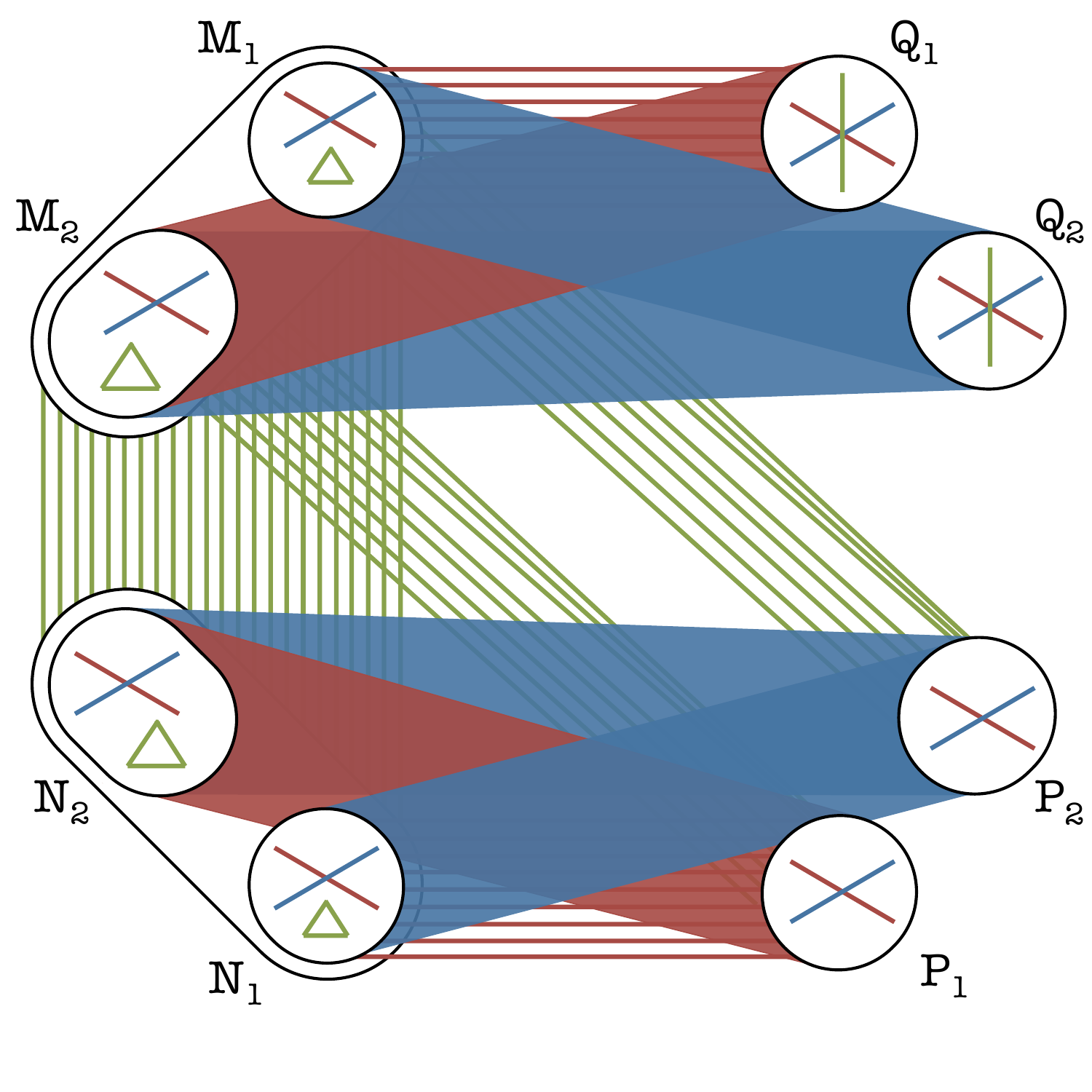}
\vspace{-5mm}\caption{Colouring of $G[M_1,Q_2]\cup G[N_1,P_2]$.}
\end{figure}

Recalling~(\ref{Eiiai4a}) and~(\ref{Eiiai4a}), given that, so far, we have discarded at most $11\eta k$ vertices from $M_2\cup N_2$,  we have
\begin{align*}
|M_2|\geq|Q_2|+14\eta^{1/2}k, \text{\hspace{8mm}} |N_2|\geq|P_2|+14\eta^{1/2}k.
\end{align*}
Thus, by Lemma~\ref{l:eleven}, there exist blue connected-matchings $B_1$ spanning at least $2|Q_2|-2\eta k$ vertices in $G[M_2,Q_2]$, and $B_2$ spanning at least $|P_2|-2\eta k$ vertices in $G[N_2,Q_2]$. Recall that we assume that~$P_2$ and $Q_2$ belong to the same blue effective-component. Then, since all edges present in $G[M,Q_2]$ and~$G[N,P_2]$ are coloured blue, all vertices in $M\cup N\cup P_2\cup Q_2$ belong to the same blue component of~$G$ and $B_1 \cup B_2$ forms a connected-matching on at least $2(|P_2|+|Q_2|)-4\eta k \geq (\aII -118\eta)k$ vertices in that component. 

Thus, the largest blue matching in $G[M_1,Q_1]\cup G[N_1,P_1]$ spans at most than $118\eta k$ vertices. Therefore, after discarding at most $59\eta k$ vertices from each of $M_1\cup N_1$ and $P_1\cup Q_1$, we may assume that all edges present in $G[M_1,Q_1]\cup G[N_1,P_1]$ are coloured exclusively red, completing the proof of Claim~\ref{claimMNPQ}.\end{proof}

In summary, we now know that all edges in $G[M,Q_1]\cup G[N,P_1]$ are coloured exclusively red and that all edges in $G[M,Q_2]\cup G[N,P_2]$ are coloured exclusively blue.

\begin{figure}[!h]
\centering
\includegraphics[width=64mm, page=5]{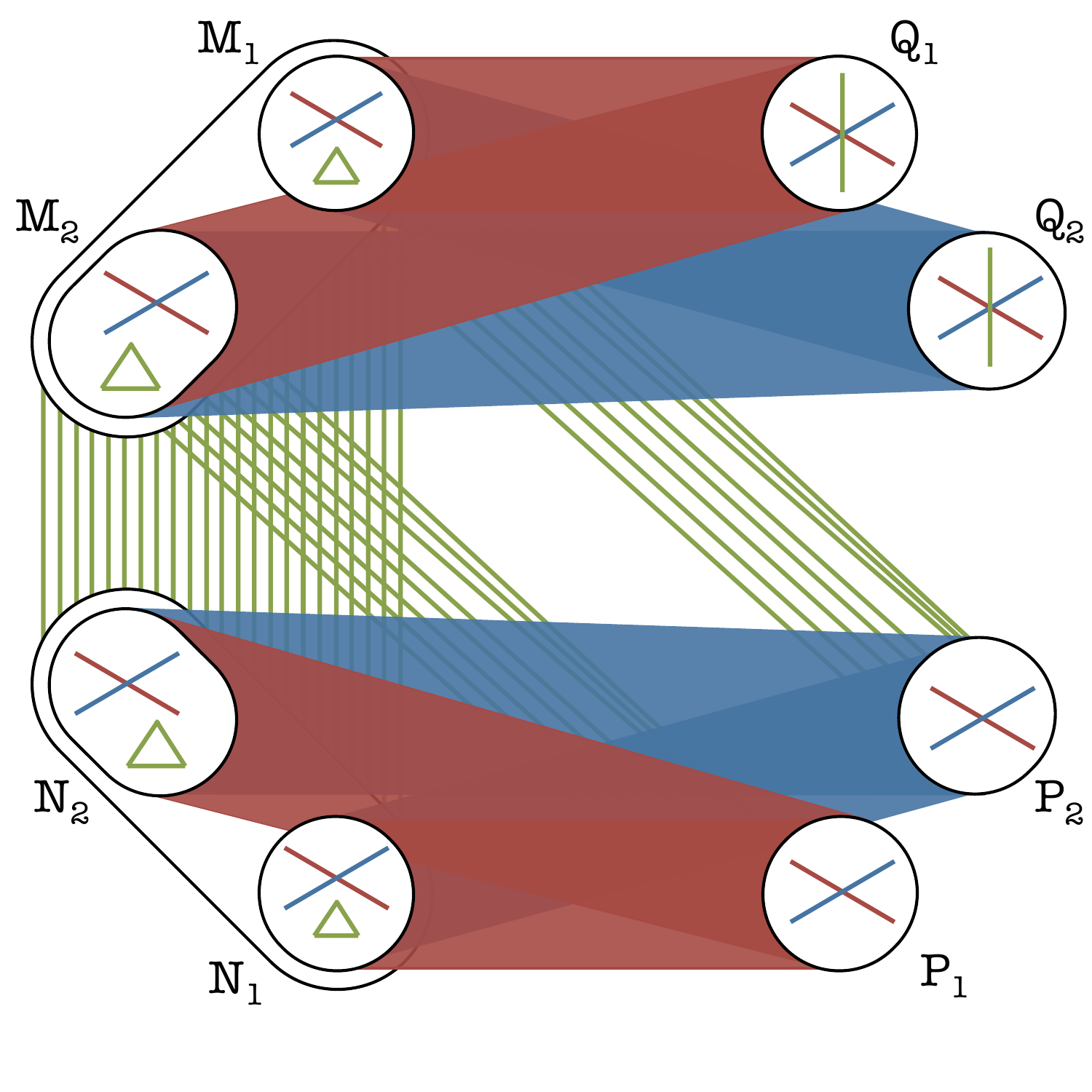}
\vspace{-4mm}\caption{Colouring of $G[M,Q]\cup G[N,P]$.}
  
\end{figure}

Additionally, we have
\begin{equation}
\label{F1}
\left.
\begin{aligned}
\quad\,\,\,\, |M_1|+|M_2|&\geq|Q_1|+|Q_2|+13\eta^{1/2}k,\quad\, & |P_1|+|Q_1|&\geq(\half\aI-73\eta)k,\quad\,\,\,\, \\
|N_1|+|N_2|&\geq|P_1|+|P_2|+13\eta^{1/2}k,   &|P_2|+|Q_2|&\geq(\half\aII-57\eta)k.
\end{aligned}
\right\}\!
\end{equation}
We now move on to consider $G[M,P]\cup G[N,Q]$, taking the same approach as we did for $G[M,Q]\cup G[N,P]$ but recalling the possibility of green edges in $G[M,P]$. We prove the following claim:

\begin{claim}
\label{claimcrossF}
We may discard at most $145\eta k$ vertices from $P_1 \cup Q_1$, at most $84\eta^{1/2}k$ vertices from $P_2\cup Q_2$ and at most $229\eta k$ vertices from $M\cup N$ such that, in what remains, there are no red edges present in $G[M,P_2]\cup G[N, Q_2]$ and no blue edges present in $G[M,P_1]\cup G[N, Q_1]$.
\end{claim}

\begin{proof}
Suppose that there exists a red matching $R^{\times}$ on at least $168\eta k$ vertices in $G[M,P_2]\cup G[N,Q_2]$. Then, by~(\ref{F1}), we have $|M\backslash V(R^{\times})|\geq|Q_1|$ and $|N\backslash V(R^{\times})|\geq|P_1|$. Thus, since all but at most $8\eta k$ vertices of $P\cup Q$ belong to the same red component of $G$, $R^{\times}$ can be used along with edges from $G[N\backslash V(R^{\times}),P_1]$ and $G[M\backslash V(R^{\times}),Q_1]$ to form a red connected-matching on 
$$(2|P_1|-2\eta k)+(2|Q_1|-2\eta k)+(168\eta k-16\eta k) \geq \aI k$$ vertices.  Thus, after discarding at most $84\eta k$ vertices from each of $M\cup N$ and $P_2 \cup Q_2$, we may assume that there are no red edges present in $G[M,P_2]\cup G[N,Q_2]$. In particular, since there are no green edges present in $G[N,Q]$, we know that all edges present in $G[N,Q_2]$ are coloured exclusively blue. We then have
\begin{equation}
\label{F2}
\left.
\begin{aligned}
\quad\,\,\,\,\, |M_1|+|M_2|&\geq|Q_1|+|Q_2|+12\eta^{1/2}k, \hspace{2mm} & |P_1|+|Q_1|&\geq(\half\aI-73\eta)k,\quad\quad \\
|N_1|+|N_2|&\geq|P_1|+|P_2|+12\eta^{1/2}k, \hspace{2mm}   &|P_2|+|Q_2|&\geq(\half\aII-142\eta)k.
\end{aligned}
\right\}\!
\end{equation}
Next, suppose that there exists a blue matching $B^{\times}$ on at least $290\eta k$ vertices in $G[M,P_1]\cup G[N,Q_1]$. Then, by~(\ref{F2}), we have $|M\backslash(V(B^{\times})|\geq |Q_2|$ and $|N\backslash(V(B^{\times})|\geq |P_2|$. Thus, since $M_2\cup N_2\cup P_2\cup Q_2$ belong to the same component of $G$, $B^{\times}$ can be used along with edges from $G[N\backslash V(B^{\times}),P_2]$ and $G[M\backslash V(B^{\times}),Q_2]$ to give a blue connected-matching on at least 
$$(2|P_2|-2\eta k)+(2|Q_2|-2\eta k)+290\eta k\geq\aII k$$ vertices. Thus, after discarding at most $145\eta k$ vertices from each of $P_1\cup Q_1$ and~$N\cup M$, we may assume that there are no blue edges present in $G[M,P_1]\cup G[N,Q_1]$. In particular, since there are no green edges present in $G[N,Q]$, we know that all edges present in $G[N,Q_1]$ are coloured exclusively red. 
 
 \begin{figure}[!h]
\centering
\includegraphics[width=64mm, page=8]{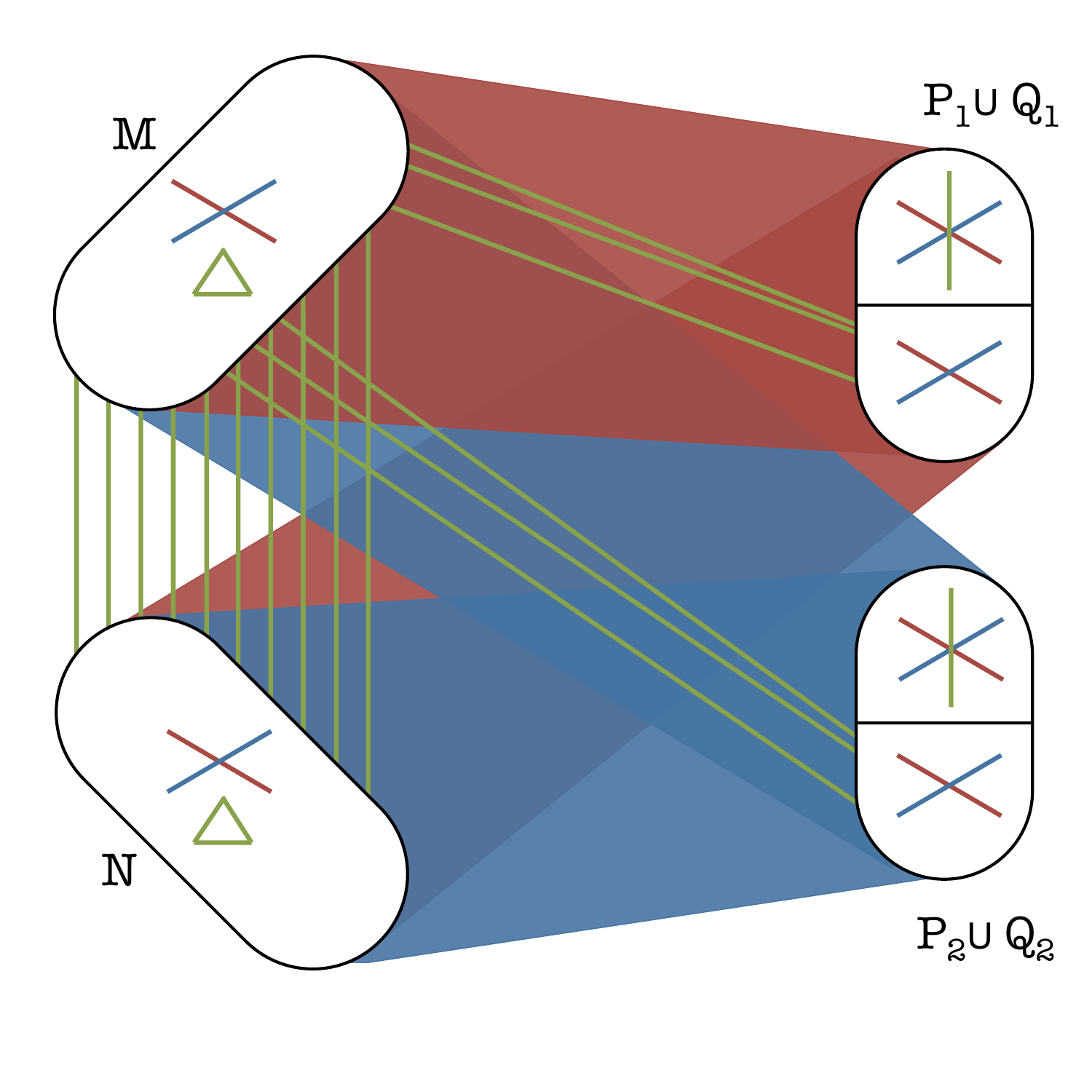}
\vspace{-4mm}\caption{Colouring of $G[M,P]\cup G[N,Q]$ after Claim~\ref{claimcrossF}.}
  
\end{figure}
 
In summary, we have discarded at most $145\eta k$ vertices from $P_1\cup Q_1$, at most $84\eta k$ vertices from $P_2\cup Q_2$ and at most $229\eta k$ vertices from $M\cup N$. Having done so, we now know that there are no red edges present in $G[M,P_2]\cup G[N,Q_2]$ and that there are no blue edges present in $G[M,P_1]\cup G[N,Q_1]$. In particular, since we already knew that are no green edges present in $G[N,Q]$, we know that all edges present in $G[N,Q_1]$ are coloured exclusively red and that all edges present in $G[N,Q_2]$ are coloured exclusively blue, thus completing the proof of Claim~\ref{claimcrossF}.\end{proof}
 
 We now have
\begin{equation}
\label{F3}
\left.
\begin{aligned}
\quad\quad\quad |M|&\geq|Q_1|+|Q_2|+11\eta^{1/2}k,\hspace{2mm} & |P_1|+|Q_1|&\geq(\half\aI-218\eta)k,\quad\quad\quad \\
|N|&\geq|P_1|+|P_2|+12\eta^{1/2}k,\hspace{2mm}    &|P_2|+|Q_2|&\geq(\half\aII-142\eta)k.
\end{aligned}
\right\}
\end{equation}
Finally, we turn our attention to $G[M\cup N]$, proving the following claim.

\begin{claim}
\label{claimMN}
We may discard at most $732\eta k$ vertices from $M\cup N$, such that, in what remains, all edges present in $G[M\cup N]$ are coloured exclusively green.
\end{claim}

\begin{proof} 
Suppose that there exists a red matching $R^{\dagger}$ on $442\eta k$ vertices in $G[M\cup N]$. Then, by~(\ref{F3}), we have $|M\backslash V(R^{\dagger})|\geq|Q_1|$, $|N\backslash V(R^{\dagger})|\geq|P_1|$. Also, since all edges present in $G[M\cup N, Q_1]$ are coloured exclusively red, $M\cup N$ belongs to a single red component of $G$. Thus, by Lemma~\ref{l:eleven}, there exists a red connected-matching $R^{\ddagger}$ in $G[N\backslash V(R^{\dagger}), P_1]\cup G[M\backslash V(R^{\dagger}),Q_1]$ on at least $(2|P_1|-2\eta k)+(2|Q_1|-2\eta k)$ vertices. Then, since $R^{\dagger}$ and $R^{\ddagger}$ belong to the same red component but share no vertices, together they form a red connected-matching on at least $(2|P_1|-2\eta k)+(2|Q_1|-2\eta k)+442\eta k\geq \aI k$ vertices.

Similarly, suppose that there exists a blue matching $B^{\dagger}$ on $290\eta k$ vertices in $G[M\cup N]$. Then, by~(\ref{F3}), we have $|M\backslash V(B^{\dagger})|\geq|Q_2|$, $|N\backslash V(B^{\dagger})|\geq|P_2|$. Since all edges present in $G[M\cup N, Q_2]$ are coloured exclusively blue, $M\cup N$ belongs to a single blue component of $G$. Then, by Lemma~\ref{l:eleven}, there exists a blue connected-matching $B^{\ddagger}$ in $G[N\backslash V(B^{\dagger}), P_2]\cup G[M\backslash V(B^{\dagger}),Q_2]$ on at least $(2|P_2|-2\eta k)+(2|Q_2|-2\eta k)$ vertices. Since $B^{\dagger}$ and $B^{\ddagger}$ belong to the same blue component but share no vertices, together they form a blue connected-matching on at least $(2|P_2|-2\eta k)+(2|Q_2|-2\eta k)+290\eta k\geq \aI k$ vertices.

Thus, after discarding at most $732\eta k$ vertices from $M\cup N$, we can assume that all edges present in $G[M\cup N]$ are coloured exclusively green, completing the proof of Claim~\ref{claimMN}. \end{proof}

\begin{figure}[!h]
\centering
\vspace{-2mm}
\includegraphics[width=64mm, page=37]{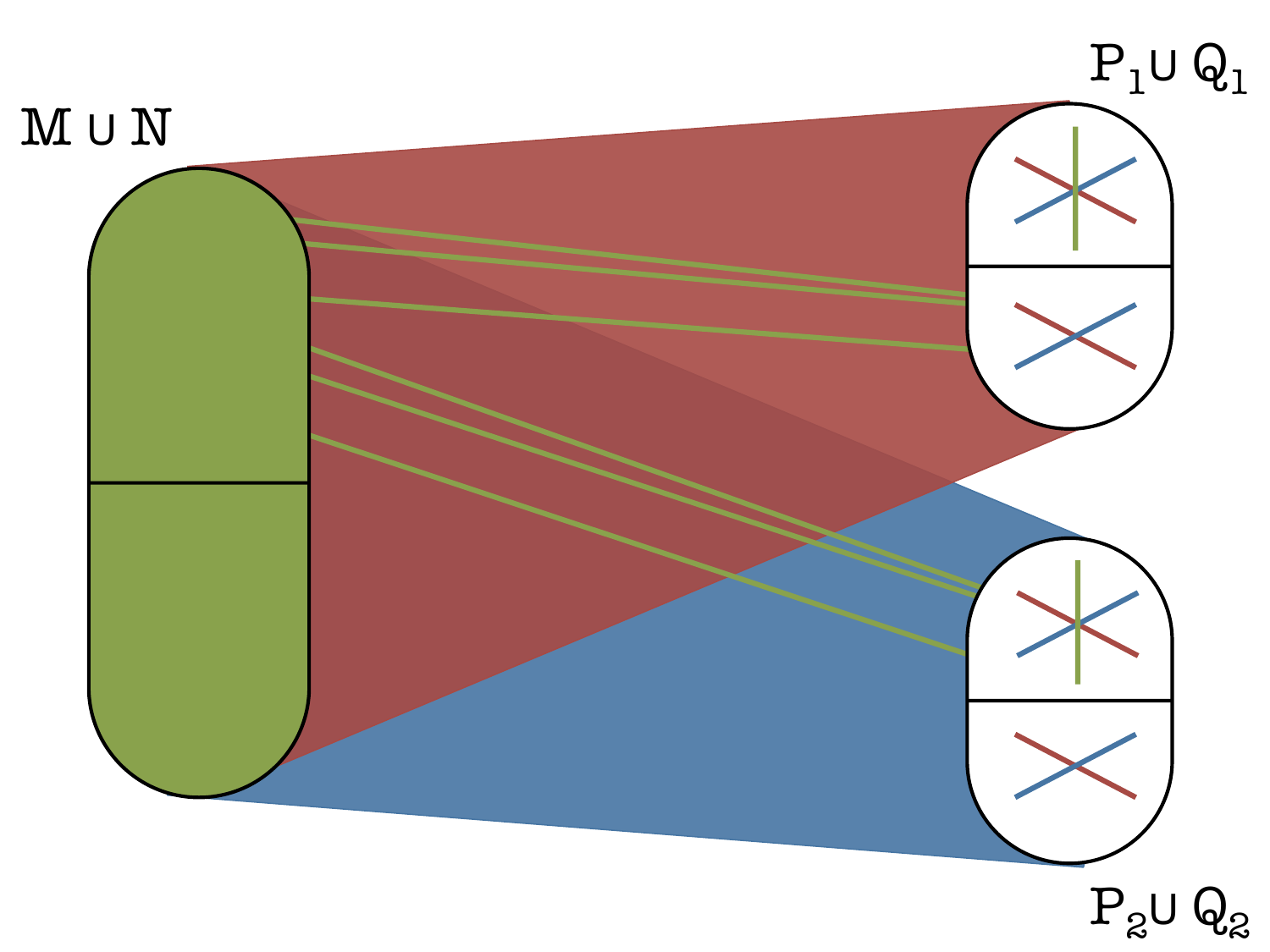}
\vspace{0mm}\caption{Colouring of $G[M\cup N]$.}
  
\end{figure}

Thus far, we have discarded at most $1200\eta k$ vertices from $M\cup N$. Recalling (E4), we now have $|M\cup N|\geq (\aIII - 9\eta^{1/2})k$. Suppose there exists a green matching $G^{\dagger}$ on $20\eta^{1/2}k$ vertices in  $G[M\cup N,P\cup Q]$. Then, we have  $|(M\cup N)\backslash V(G^{\dagger})|\geq (\aIII -19\eta^{1/2})k$. By Theorem~\ref{dirac}, since $G$ is $4\eta^4 k$-almost-complete, $G[(M\cup N)\backslash V(G^{\dagger})]$ contains a green connected-matching on all but at most one of its vertices. Thus, provided $k\geq1/\eta^{1/2}$, there exists a connected-green matching $G^{\ddagger}$ on at least $(\aIII -20\eta^{1/2})k$ vertices in $G[(M\cup N)\backslash V(G^{\dagger})]$. Since $G$ is $4\eta^4 k$-almost-complete and all edges present in $G[M\cup N]$ are coloured exclusivly green, all vertices of $M\cup N$ belong to the same green component of~$G$. Thus, together, $G^{\dagger}$ and $G^{\ddagger}$ form a green connected-matching on at least $\aIII k$ vertices. Thus, we may, after discarding at most $10\eta^{1/2}k$ vertices from each of $M\cup N$ and $P\cup Q$, assume that there are no green edges in $G[M\cup N, P\cup Q]$. 

In summary, we now have
\begin{align*}
 |P_1|+|Q_1|&\geq(\half\aI-218\eta)k,  &
|P_2|+|Q_2|&\geq(\half\aII-142\eta)k,&
|M\cup N|&\geq(\aIII-20\eta^{1/2})k,   
\end{align*}
and know that all edges present in $G[M\cup N, P_1\cup Q_1]$ are coloured exclusively red, all edges present $G[M\cup N, P_2\cup Q_2]$ are coloured exclusively blue and all edges in $G[M\cup N]$ are coloured exclusively green.

Thus, we have found, as a subgraph of~$G$, a graph belonging to
$$\cK\left((\half\aI-10\eta^{1/2})k, (\half\aI-10\eta^{1/2})k, (\aIII-20\eta^{1/2})k, 4\eta^4 k\right),$$ completing Case E.iii.a.i. 

\subsection*{Case E.iii.a.ii: {\rm $P_2$ and $Q_2$ belong to the different components.}}

Recall that we have a decomposition of $V(G)$ into $M\cup N\cup P \cup Q$ satisfying (\ref{E1-iii})--(\ref{E3-iii}) such that
\begin{align}\label{E4a-x}\tag{E4a}(\max\{\tfrac{3}{4}\aI+\tfrac{1}{4}\aII,\half\aIII\}-5\eta^{1/2})k & \leq  |M|\,,\,|N| \leq \half\aIII k\\
\label{E4b-x}\tag{E4b}(\half\aI+\half\aII-\eta)k & \leq  |P|+|Q| \leq (\half\aI+\half\aII+5\eta^{1/2})k.
\end{align}
Recall also that $F$, the largest monochromatic component in $P\cup Q$, is red and spans at least $|P\cup Q|-8\eta k$ vertices and that each of $M,N,P$ and $Q$ have been subdivided into two parts such that $M_1=M\cap V(R)$, $M_2=M\backslash M_1$, $N_1=N\cap V(R)$, $N_2=N\backslash N_1$,
$P_1=P\cap V(R)$, $P_2=M\backslash P_1$,
$Q_1=Q\cap V(R)$, $Q_2=Q\backslash Q_1$, where $R$ is the largest red matching in $G[M,Q]\cup G[N,P]$. By maximality of $R$, all edges present in $G[M_2,Q_2]$ or $G[N_2,P_2]$ are blue. 

Additionally, in this case, we assume that $P_2$ and $Q_2$ belong to different blue components of $G$. Thus, in particular, all edges present in $G[N_2,Q_2]$ and $G[P_2,Q_2]$ are coloured red.

\begin{figure}[!h]
\vspace{3mm}
\centering{
\includegraphics[width=64mm, page=10]{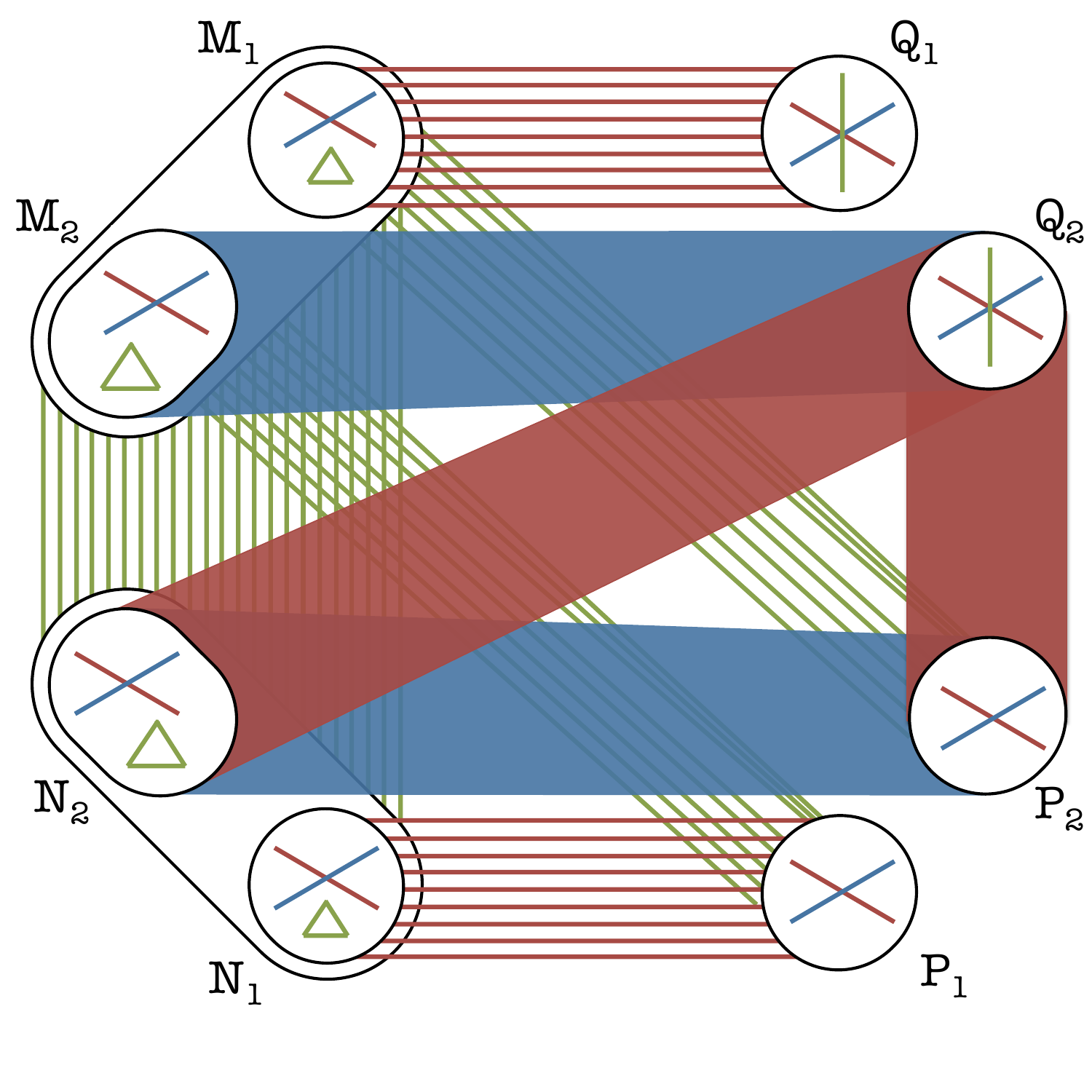}}
\vspace{-4mm}\caption{Initial colouring in Case E.iii.a.ii.}
 \end{figure}

By Lemma~\ref{l:eleven}, there exist red connected-matchings

\begin{tabularx}{\textwidth}{l X l}
& $R_{PQ}$ on at least $2\min\{|P_2|,|Q_2|\}-2\eta k$ vertices in $G[P_2,Q_2]$\\
& $R_{NQ}$ on at least $2\min\{|N_2|,|Q_2|\}-2\eta k$ vertices in $G[N_2,Q_2]$. \\ & \\
\end{tabularx}

Then, since $F$ includes all but at most $8\eta k$ of the vertices of $P\cup Q$, these connected-matchings can be augmented with edges from $R$ to give the red connected-matchings illustrated in Figure~\ref{RpqRnq}:

\begin{tabularx}{\textwidth}{ l X  X  X }
&\multicolumn{2}{l}{$R_1$ on at least $2|P_1|+2|Q_1|+2\min\{|P_2|,|Q_2|\}-18\eta k$ vertices}&  \\
 & & \multicolumn{2}{r}{ in $G[M_1,Q_1]\cup G[N_1,P_1]\cup G[P_2,Q_2]$,}  \\
&\multicolumn{2}{l}{$R_2$ on at least $2|P_1|+2|Q_1|+2\min\{|N_2|,|Q_2|\}-18\eta k$ vertices}&  \\
  && \multicolumn{2}{r}{ in $G[M_1,Q_1]\cup G[N_1,P_1]\cup G[N_2,Q_2]$,}  \\ 
  &\\
\end{tabularx}

Given the existence of these matchings, we can obtain bounds on the sizes of the various sets identified:

Since $|P|, |Q|\geq 95\eta^{1/2}k$, by (\ref{E4b-x}), $|P|, |Q|\leq(\half\aI+\half\aII-90\eta^{1/2})k$.
But, then, $$|M|,|N|\geq (\tfrac{3}{4}\aI+\tfrac{1}{4}\aII-5\eta^{1/2})k\geq \max\{|P|,|Q|\}+85\eta^{1/2}k.$$
So, since $|M_1|=|Q_1|$ and $|N_1|=|P_1|$, we have
\begin{align}
\label{G1}
|M_2|&\geq|Q_2|+85\eta^{1/2}k, &
|N_2|&\geq|P_2|+85\eta^{1/2}k.
\end{align}

  \begin{figure}[!h]
\centering{
\mbox{\hspace{-2mm}{\includegraphics[width=64mm, page=18]{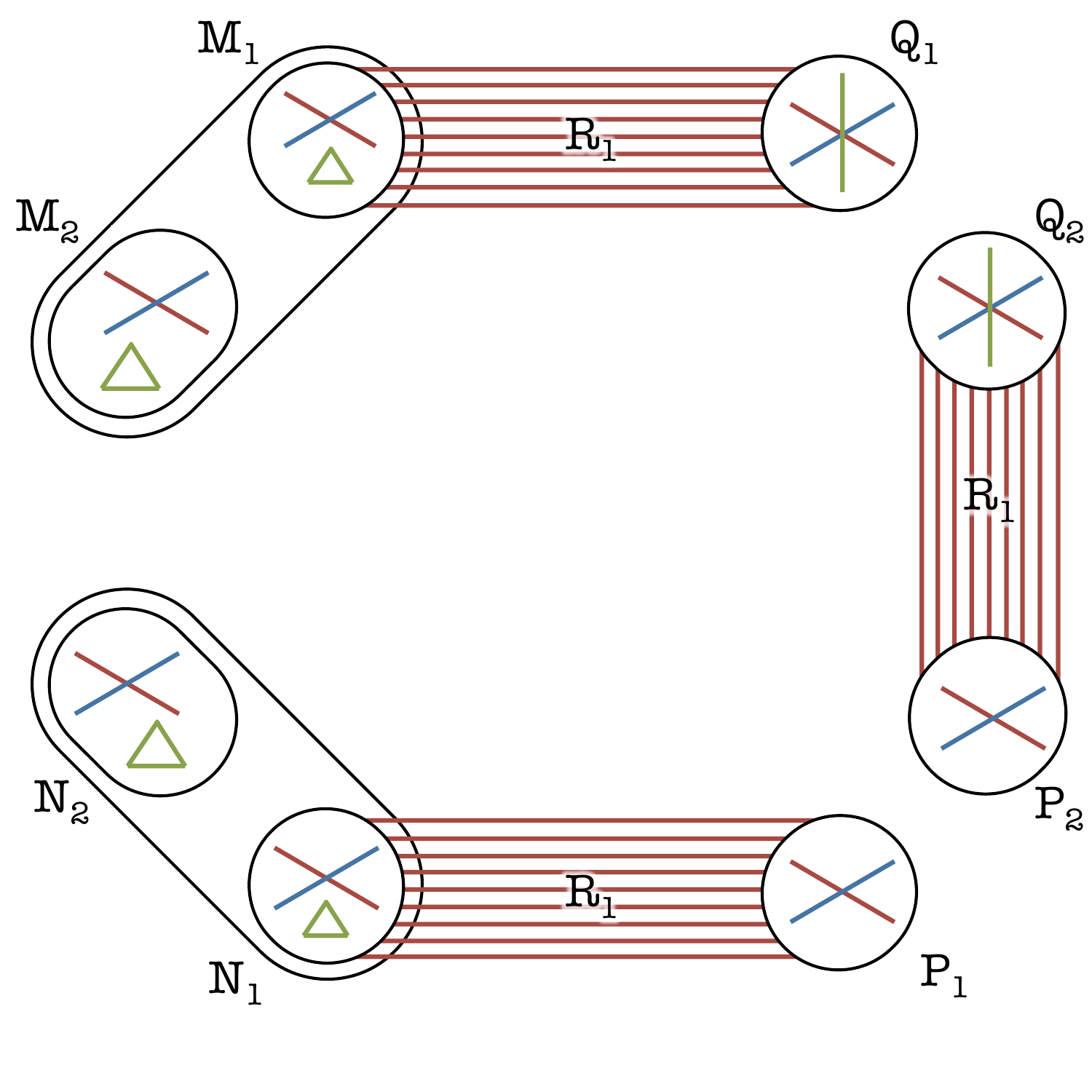}}\quad\quad\quad{\includegraphics[width=64mm, page=19]{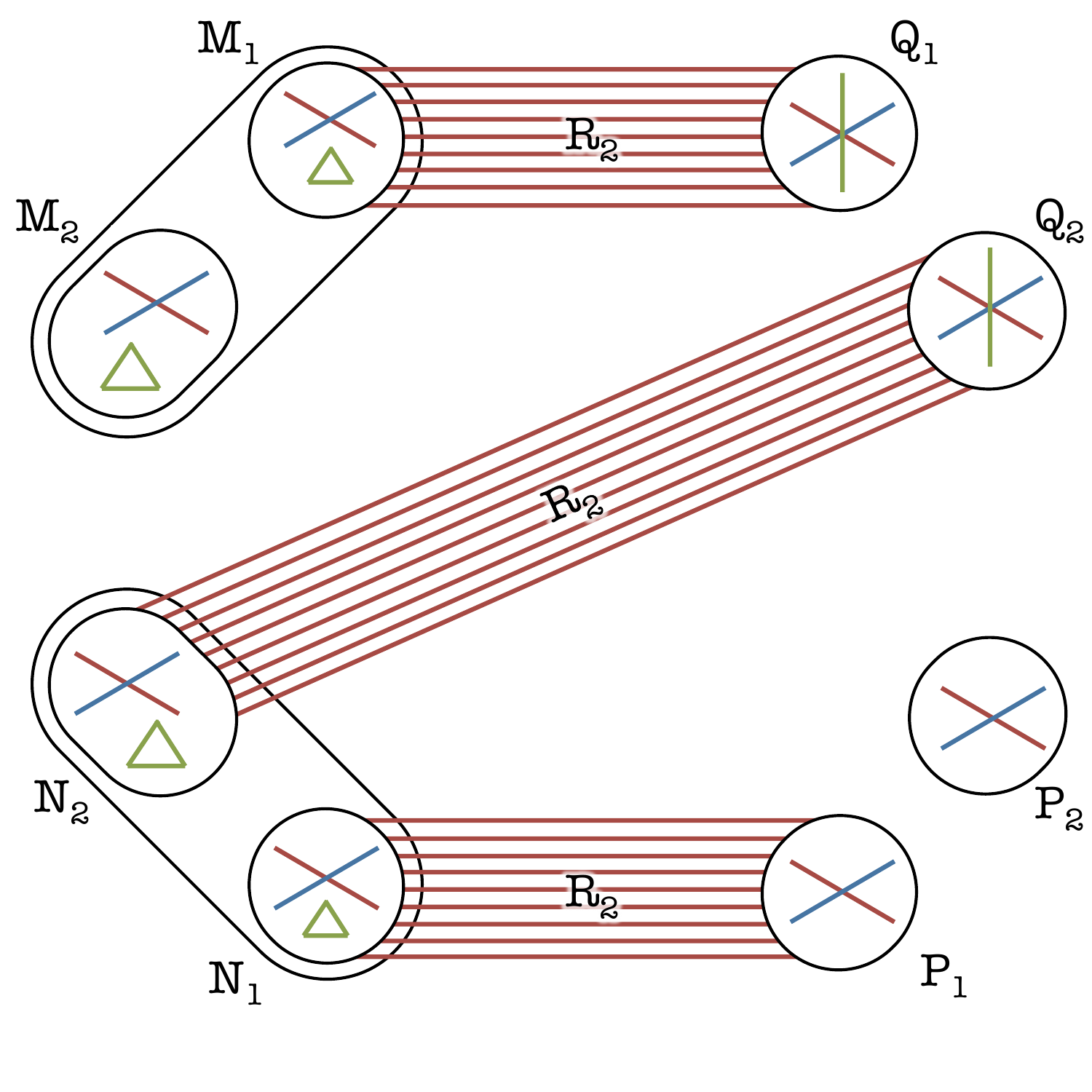}}}}\vspace{-5mm}
\caption{The red connected-matchings $R_1$ and $R_2$.} \label{RpqRnq}
\end{figure}

In particular, $|M_2|\geq |Q_2|$ and $|N_2|\geq|P_2|$. Thus, by Lemma~\ref{l:eleven}, we may obtain a blue connected-matching on $2|P_2|-2\eta k$ vertices in $G[N_2,P_2]$ and one on $2|Q_2|-2\eta k$ vertices in $G[M_2,Q_2]$. Thus, in order to avoid a blue connected-matching on at least $\aII k$ vertices, we may assume that  
\begin{equation}
\label{G2}
|P_2|,|Q_2|\leq (\half\aII+\eta)k.
\end{equation}
Suppose, for now, that $|Q_2|\geq |N_2|$. Then, recalling, that $|N_2|\geq |P_2|$, we have $|Q_2|\geq|P_2|$ and, therefore, $R_1$ spans at least $2|P_1|+2|Q_1|+2|P_2|-18\eta k$ vertices in $G[M_1,Q_1]\cup G[N_1,P_1] \cup G[P_2,Q_2]$. Thus, in order to avoid a red connected-matching on at least $\aI k$ vertices, we may assume that 
$$|P_1|+|Q_1|+|P_2|\leq (\half\aI+10\eta)k.$$
Also, since $|Q_2|\leq (\half\aII+\eta)k$, by (\ref{E4b-x}), we have $$|P_1|+|Q_1|+|P_2|\geq(\half\aI-2\eta)k.$$
Thus, $R_1$ spans at least $(\aI-22\eta)k$ vertices. Now, since $|Q_2|\geq |N_2|$, by~(\ref{G1}), we have $|Q_2|\geq|P_2|+15\eta^{1/2}k$. Thus, there exists $\widetilde{Q}\subseteq{Q_2\backslash V(R_1)}$ such that $|\widetilde{Q}|\geq 15\eta^{1/2}k$. Therefore, by Lemma~\ref{l:eleven}, we may find a red connected-matching $M_3$ on at least $28\eta^{1/2} k$ vertices in $G[N_2,\widetilde{Q}]$ which belongs to the same red component as $R_1$. Thus, together $R_1$ and $R_3$ form a red-connected-matching on $\aI k$ vertices, completing the proof in this case.

Therefore, we may instead assume that $|N_2|\geq |Q_2|$. In that case, $R_2$ spans at least $2|P_1|+2|Q_1|+|Q_2|-20\eta k$ vertices in $G[N_1,P_1]\cup G[M_1,Q_1]\cup G[N_2,Q_2]$. Thus, 
$|P_1|+|Q_1|+|Q_2|\leq(\half\aI+10\eta)k.$ Then, by (\ref{E4b-x}) and~(\ref{G2}), we obtain
\begin{equation}
\label{G3}
\left.
\begin{aligned}
\,\,\quad\quad\quad\quad\quad\quad\quad(\half\aI-2\eta)k\leq |P_1|&+|Q_1|+|Q_2|\leq (\half\aI+10\eta)k,\quad\quad\quad\quad\,\,\,\\
(\half\aI-12\eta)k &\leq |P_2| \leq (\half\aII+\eta)k.
\end{aligned}
\right\}\!
\end{equation}
Observe that, by (\ref{G2}) and (\ref{G3}), we may assume that 
\begin{equation}
\label{G4}
|Q_2|\leq |P_2|+\eta^{1/2}k.
\end{equation}
We are now in a position to examine the coloured structure of $G[N,P]$. We show that, after possibly discarding some vertices, we may assume that all edges contained in $G[N,P_1]$ are coloured exclusively red and all edges contained in $G[N,P_2]$ are coloured exclusively blue. Following the same steps as in the proofs of Claim~\ref{claimLQ} and Claim~\ref{claimMNPQ} we prove:

\begin{claim}
\label{claimNP}
We may discard at most $67\eta k$ vertices from $N_1$, $14\eta k$ vertices from $N_2$, $54\eta k$ vertices from $P_1$ and at most $27\eta k$ vertices from $P_2$ such that, in what remains, all edges present in $G[N,P_1]$ are coloured exclusively red and all edges present in $G[N,P_2]$ are coloured exclusively blue.
\end{claim}
\begin{proof}
Given any $\widetilde{N}_2\subseteq N$ such that $|\widetilde{N}_2|\geq|P_2|\geq(\half\aII-12\eta)k$, by Lemma~\ref{l:eleven}, we can obtain a blue connected-matching on $(\aII-26\eta)k$ vertices in $G[\widetilde{N}_2,P_2]$. Then, since $|N_2|\geq|P_2|+15\eta^{1/2}k$, the existence a blue matching $B_S$ on at least $28\eta k$ vertices $G[N_2,P_1]$ would allow us to obtain a blue connected-matching on at least $\aII k$ vertices. 

  \begin{figure}[!h]
\centering{
\mbox{\hspace{-2mm}{\includegraphics[width=64mm, page=11]{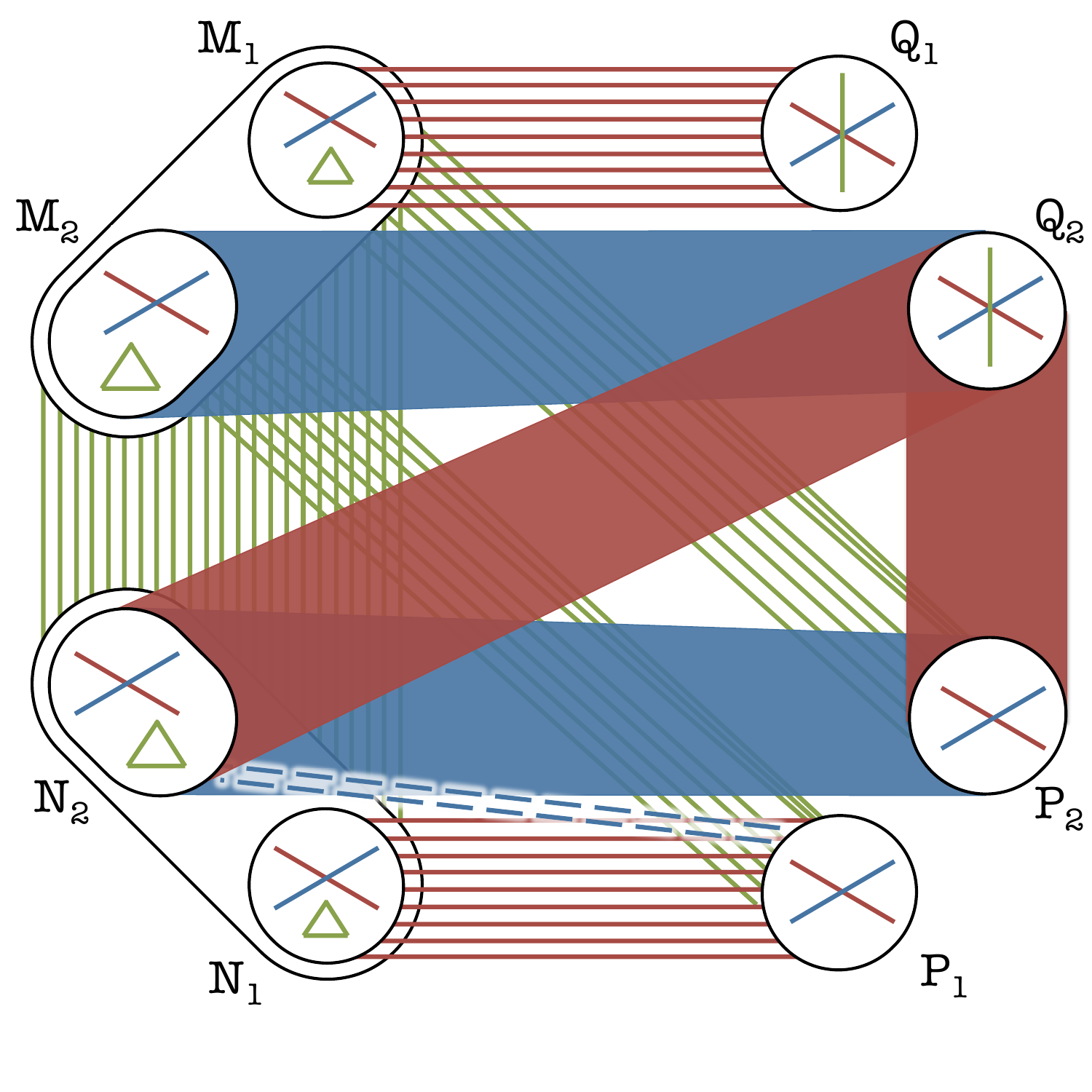}} \quad\quad\quad{\includegraphics[width=64mm, page=12]{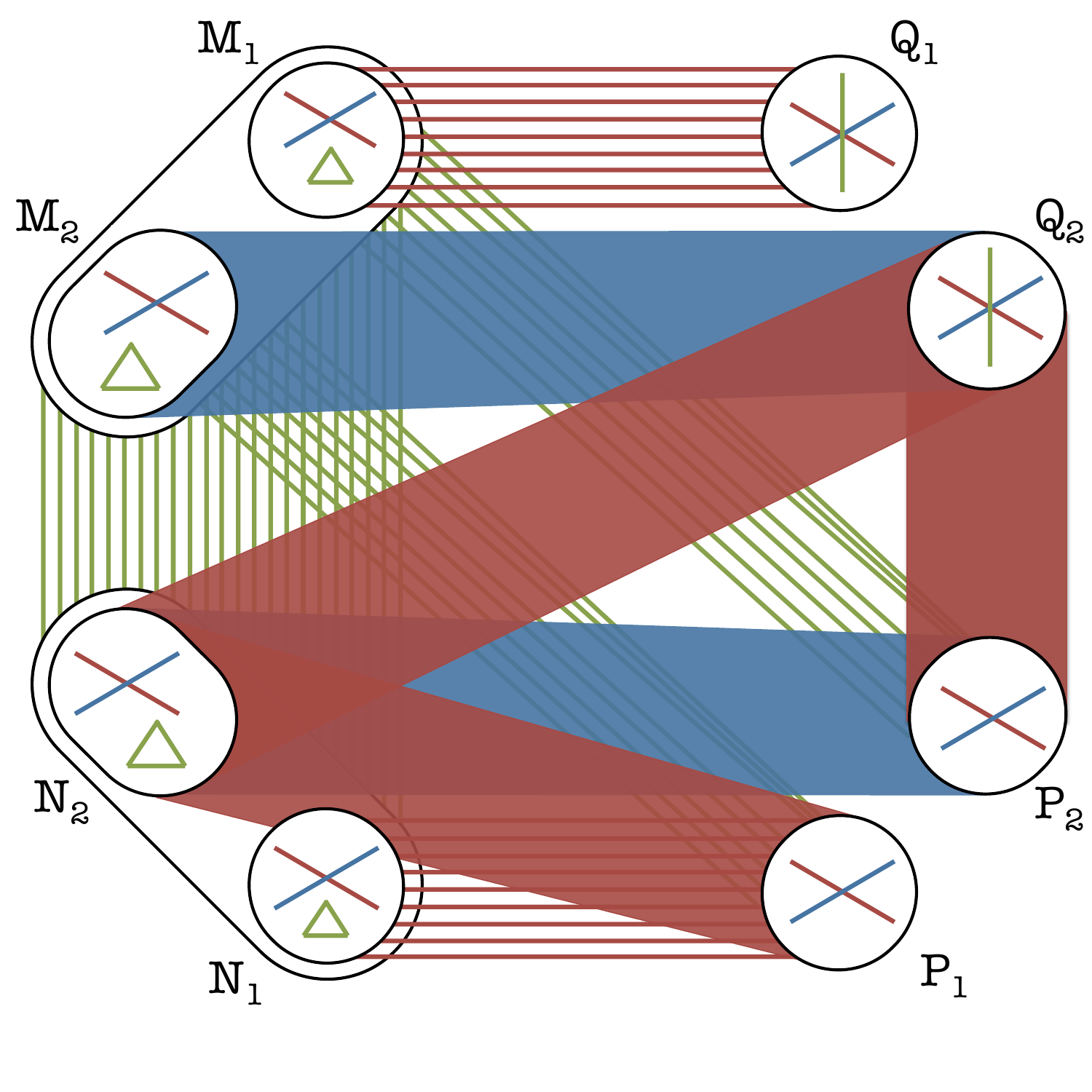}}}}\vspace{-5mm}
\caption{Colouring of the edges of $G[N_2,P_1]$.}   
\end{figure}

Thus, after discarding at most $14\eta k$ vertices from each of $P_1$ and $N_2$, we may assume that all edges present in $G[N_2,P_1]$ are coloured exclusively red.

After discarding these vertices, we have 
\begin{align*}
|P_1|+|Q_1|+|Q_2|\geq(\half\aI-16\eta)k
\end{align*}

and, thus, may assume that $R_2$ spans at least  $(\aI-50\eta)k$ vertices in $G[M_1,Q_1]\cup G[N_1,P_1]\cup G[N_2,Q_2]$. Also, recalling~(\ref{G1}) and~(\ref{G4}), we have $$|N_2|\geq|P_2|+14\eta^{1/2}k\geq|Q_2|+13\eta^{1/2}k.$$
Thus, $$|N_2\backslash V(R_2)|\geq |N_2|-|Q_2|\geq 13\eta^{1/2}k.$$

Suppose there exists a matching $R_S$ on $54\eta k$ vertices in $G[N_1,P_2]$, then we can obtain a red connected-matching on at least $\aI k$ vertices as follows:

Observe that there exists a set $R^{-}$ of $27\eta k$ edges belonging to $R_2$ such that $N_1\cap V(R_S)=N_1\cap V(R^{-})$. Define $R^{*}=R_2\backslash R^{-}$ and $\widetilde{P}=P_1\cap V(R^{-})$, let $\widetilde{N}$ be any set of $27\eta k$ vertices in $N_2 \backslash V(R_2)$ and consider $G[\widetilde{N},\widetilde{P}]$. Since $|\widetilde{N}|,|\widetilde{P}|\geq 27\eta k$, we may apply Lemma~\ref{l:eleven} to find a red connected-matching $R_T$ on at least $52\eta k$ vertices in $G[\widetilde{N},\widetilde{P}]$. Since all edges present in $G[N_2,Q_2]\cup G[P_2,Q_2]$ are coloured exclusively red, $R_S$ and $R_T$ belong to the same red component as $R_2$. Then, $R^{*}\cup R_S \cup R_T$ is a red connected-matching in $$G[M_1,Q_1]\cup G[N_1,P_1]\cup G[N_2,Q_2]\cup G[N_2,P_1] \cup G[N_1,P_2]$$ on at least $(\aI -50\eta  - 55\eta)k +54\eta k+52\eta k\geq \aI k$ vertices.

\begin{figure}[!h]
\centering
\includegraphics[width=64mm, page=14]{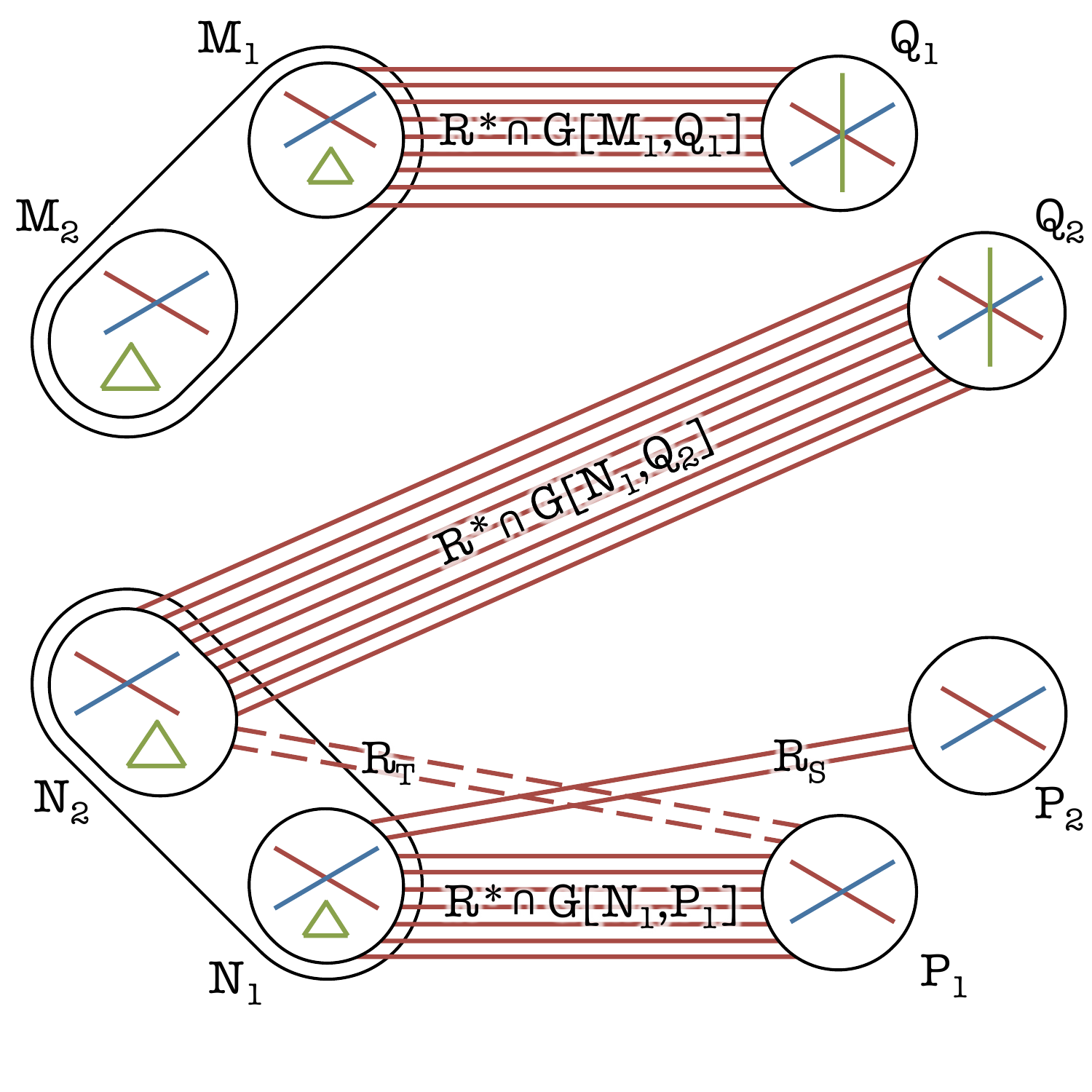}
\vspace{-4mm}\caption{Enlarging the matching $R_2$ with edges from $G[N_1,P_2]$.}
  
\end{figure}

Thus, after discarding at most $27\eta k$ vertices from each of  $N_1$ and $P_2$, we may assume that all edges present in $G[N_1,P_2]$ are coloured exclusively blue and, recalling~(\ref{G3}), that $|P_2|\geq (\half\aII-38\eta)k$. 
\begin{figure}[!h]
\centering
\includegraphics[width=64mm, page=15]{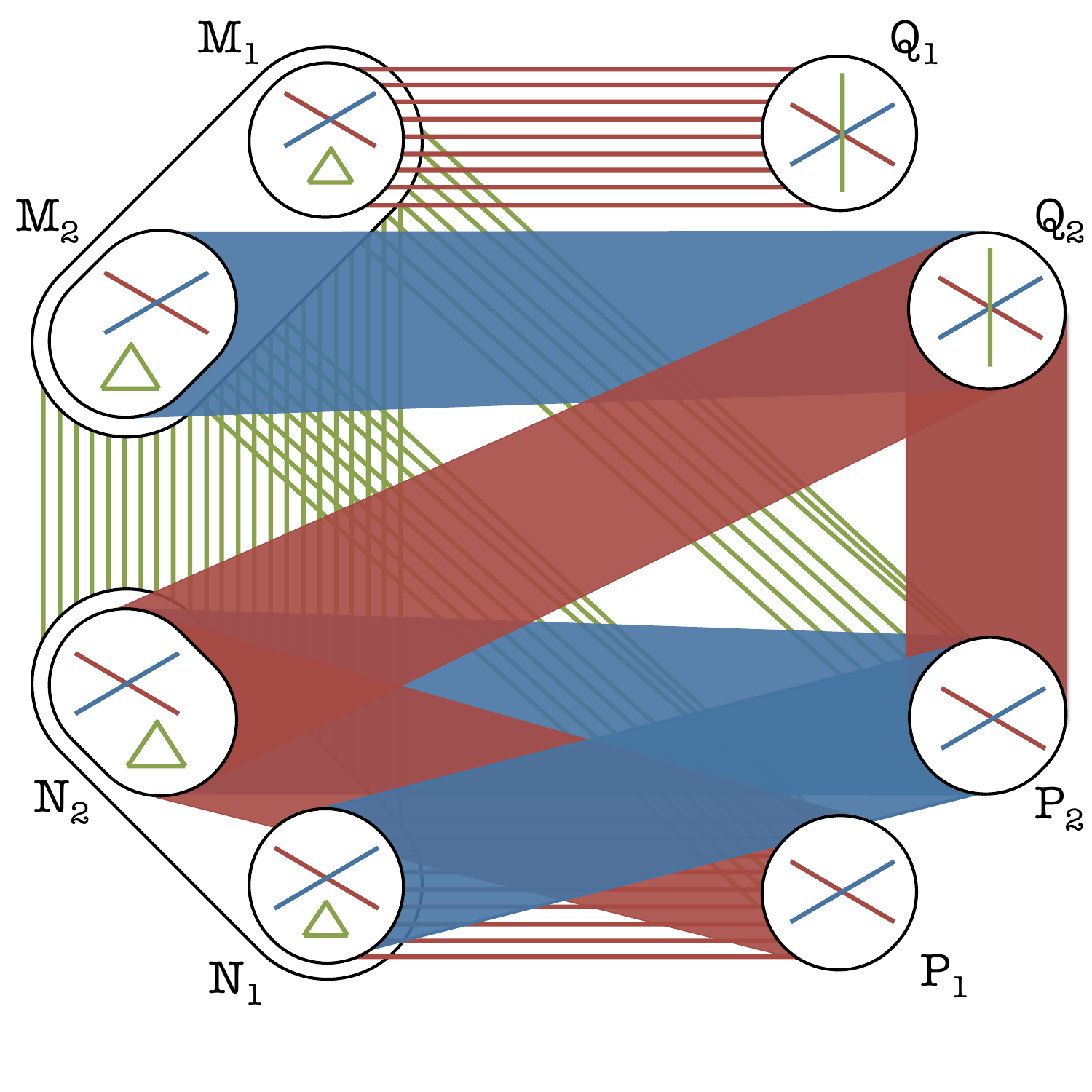}
\vspace{-4mm}\caption{Resultant colouring of the edges of $G[N_1,P_2]$.}
  
\end{figure}

Finally, suppose that there exists a blue matching $B_S$ on at least $80\eta k$ vertices in $G[N_1,P_1]$. Then, we could obtain a blue connected-matching on at least $\aII k$ vertices. Thus, after discarding at most $40\eta k$ vertices from each of $N_1$ and $P_1$, we may assume that all edges present in $G[N_1,P_1]$ are coloured exclusively red, thus completing the proof of the claim.
\end{proof}

In summary, recalling (\ref{G3}), we now have
\begin{align*}
|P_1|+|Q_1|+|Q_2|&\geq(\half\aI-56\eta)k,\\
|P_2|&\geq(\half\aII-40\eta)k,
\end{align*}
and know that all edges present in $G[N,P_1]$ are coloured exclusively red and that all edges present in $G[N,P_2]$ are coloured exclusively blue. Observe, also, that there can be no blue edges present in $G[N_1,Q_2]$ since then $M_2\cup Q_2$ and $N_2\cup P_2$ would belong to the same blue component of $G$.

\begin{figure}[!h]
\centering
\includegraphics[width=64mm, page=17]{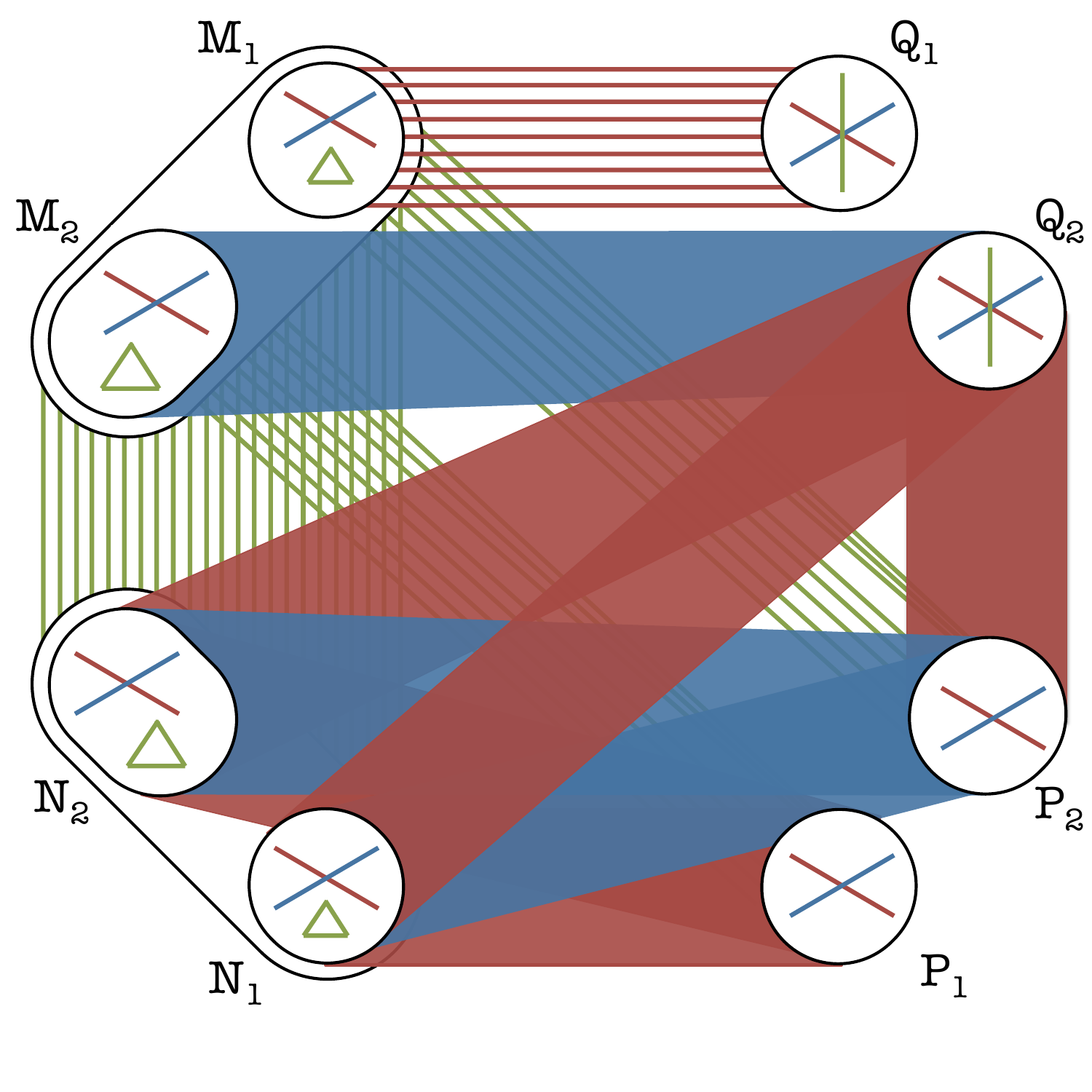}
\vspace{-4mm}
\caption{Colouring after Claim~\ref{claimNP}.}
  
\end{figure}

Next, we consider, in turn $G[M_2,N\cup P_2]$, $G[N]$ and $G[M_1,N]$ showing that after discarding a few vertices, we may assume that all edges remaining in each are coloured exclusively green: 

Suppose there exists a red matching $R_M$ on at least $136\eta k$ vertices in $G[M_2,N\cup P_2]$. Then, since $|N|\geq |P_1|+|Q_2|+12\eta^{1/2}k$, we have $|N\backslash (R_M)|\geq|P_1|+|Q_2|$. Thus, by Lemma~\ref{l:eleven}, there exists a red connected-matching $R_N$ on at least $(2|P_1|-2\eta k)+(2|Q_2|-2\eta k)$ vertices in $G[N,P_1\cup Q_2]$ sharing no vertices with $R_M$. Since all edges present in $G[N\cup P_2,Q_2]\cup G[N,P_1]$ are coloured red, $R_M$ and $R_N$ belong to the same red component of~$G$.

Since $P\cup Q$ has a red effective-component on $F$ on at least $|P\cup Q|-8\eta k$ vertices, all but at most $8\eta k$ of the edges of $R$ contained in $G[M_1,Q_1]$ belong to the same red component as $R_M\cup R_N$. Thus, defining $R_Q$ to be the subset of $R$ belonging to $G[M,Q_1\cap F]$, we have a red connected-matching $R_M\cup R_N\cup R_Q$ on at least $136\eta k+(2|P_1|-2\eta k)+(2|Q_2|-2\eta k)+(2|Q_1|-16\eta k)\geq \aI k$ vertices in $G[M_2,N\cup P_2]\cup G[N,P_1\cup Q_2]\cup G[M_1,Q_1]$.

 \begin{figure}[!h]
\centering
\includegraphics[width=64mm, page=20]{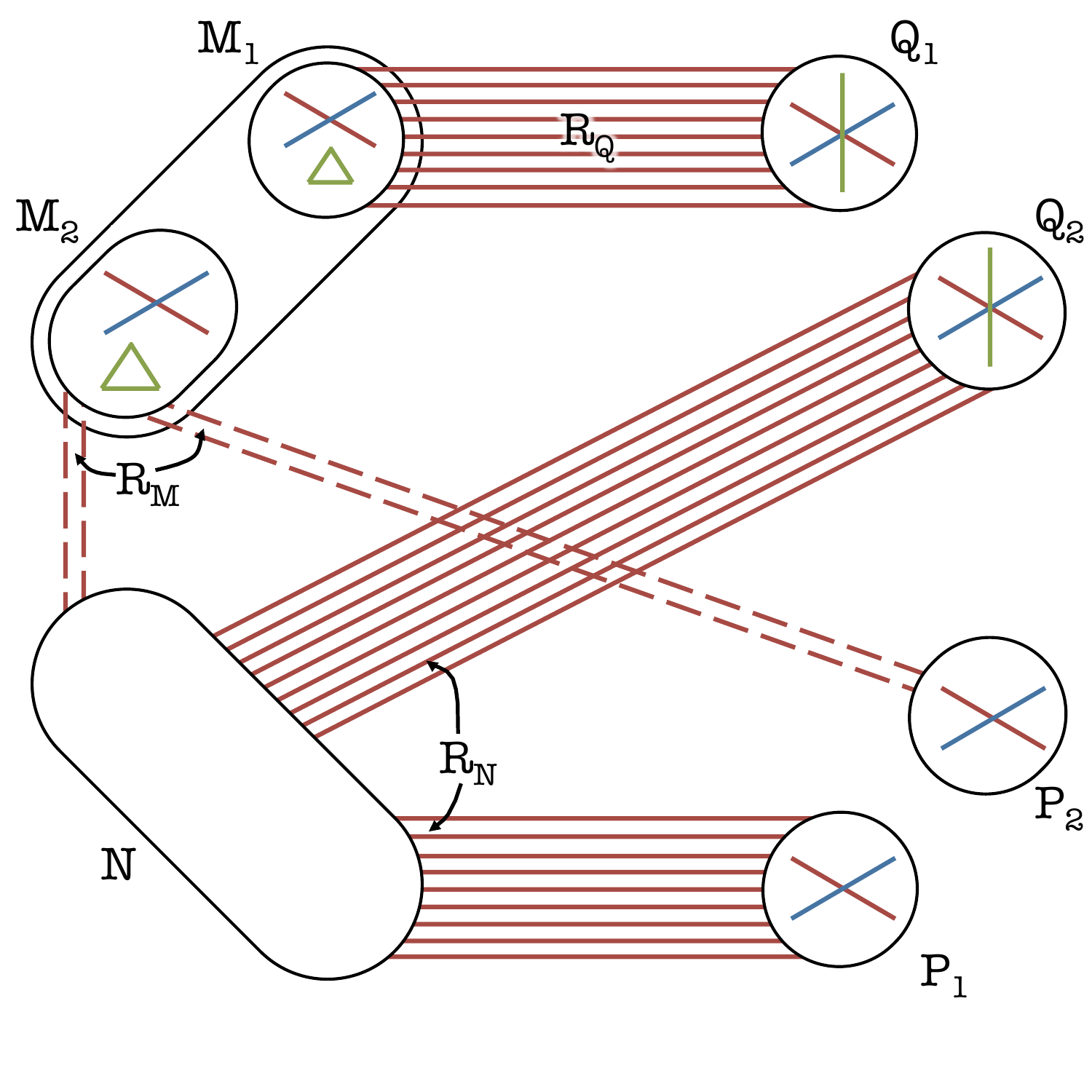}
\vspace{-4mm}\caption{Construction of red connected-matching $R_M\cup R_N\cup R_Q$.}
  
\end{figure}

 Thus, discarding at most $68\eta k$ vertices from each of $M_2$ and $N\cup P_2$, we may assume that there are no red edges $G[M_2,N\cup P_2]$. Thus, recalling that we assume that $P_2$ and~$Q_2$ are in different blue components, all edges present in $G[M_2,N\cup P_2]$ are coloured exclusively green and have
\begin{equation}
\label{G8}
\left.
\begin{aligned}
\quad|N|&\geq|P_1|+|Q_2|+11\eta^{1/2}k,\quad\quad & |P_1|+|Q_1|+|Q_2|&\geq(\half\aI-57\eta)k,\quad\quad\\
|N|&\geq|P_2|+11\eta^{1/2}k, & |P_2|&\geq(\half\aII-108\eta)k.
\end{aligned}
\right\}
\end{equation}
 Next, suppose there exists a red matching $R_A$ on $136\eta k$ vertices in $G[N]$. Then, by~(\ref{G8}), we have $|N\backslash V(R_A)|\geq|P_1|+|Q_2|$. So, by Lemma~\ref{l:eleven}, there exists a red connected-matching $R_B$ on at least $(2|P_1|-2\eta k)+(2|Q_2|-2\eta k)$ vertices in $G[N\backslash V(R_A),P_1\cup Q_2]$. Since all edges in $G[N,P_1]$ are coloured red, all edges of $R_A$ and $R_B$ belong to the same red component. Also, since the red component $F$ spans all but at most $8\eta k$ vertices of $P\cup Q$, there exists a red-matching $R_C\subseteq R$ in $G[M_1,Q_1]$ on at least $2|Q_1|-16\eta k$ vertices belonging to the same red component as $R_A$ and $R_B$. Then, together $R_A$, $R_B$ and $R_C$ form a red connected-matching on at least 
 $\aI k$ vertices in \mbox{$G[N]\cup G[N,P_1\cup Q_2]\cup G[M_1\cup Q_1]$}.
 
Similarly, if there exists a blue matching $B_A$ on $218\eta k$ vertices in $G[N]$, then we can construct a blue connected-matching on at least $\aII k$ vertices as follows. By~(\ref{G8}), we have $|N\backslash V(R_B)|\geq|P_2|$. So, by Lemma~\ref{l:eleven}, there exists a blue connected-matching~$B_B$ on at least $2|P_2|-2\eta k$ vertices in $G[N\backslash V(B_A),P_2]$. Since all edges in $G[N,P_2]$ are coloured blue, all edges in $B_A$ and $B_B$ belong to the same blue component. Thus, together, $B_A$ and $B_B$ form a blue connected-matching on at least $2|P_2|+216\eta k\geq \aII k$ vertices in $G[N]\cup G[N,P_1\cup Q_2]\cup G[M_1\cup Q_1]$. Thus, discarding at most $354\eta k$ vertices from~$N$, we have 
\begin{align}
\label{G9}
|N|&\geq|P_1|+|Q_2|+10\eta^{1/2}k, &|N|&\geq|P_2|+10\eta^{1/2}k
\end{align} and may assume that all edges in $G[N]$ are coloured exclusively green.
 
\begin{figure}[!h]
\centering
\includegraphics[width=64mm, page=24]{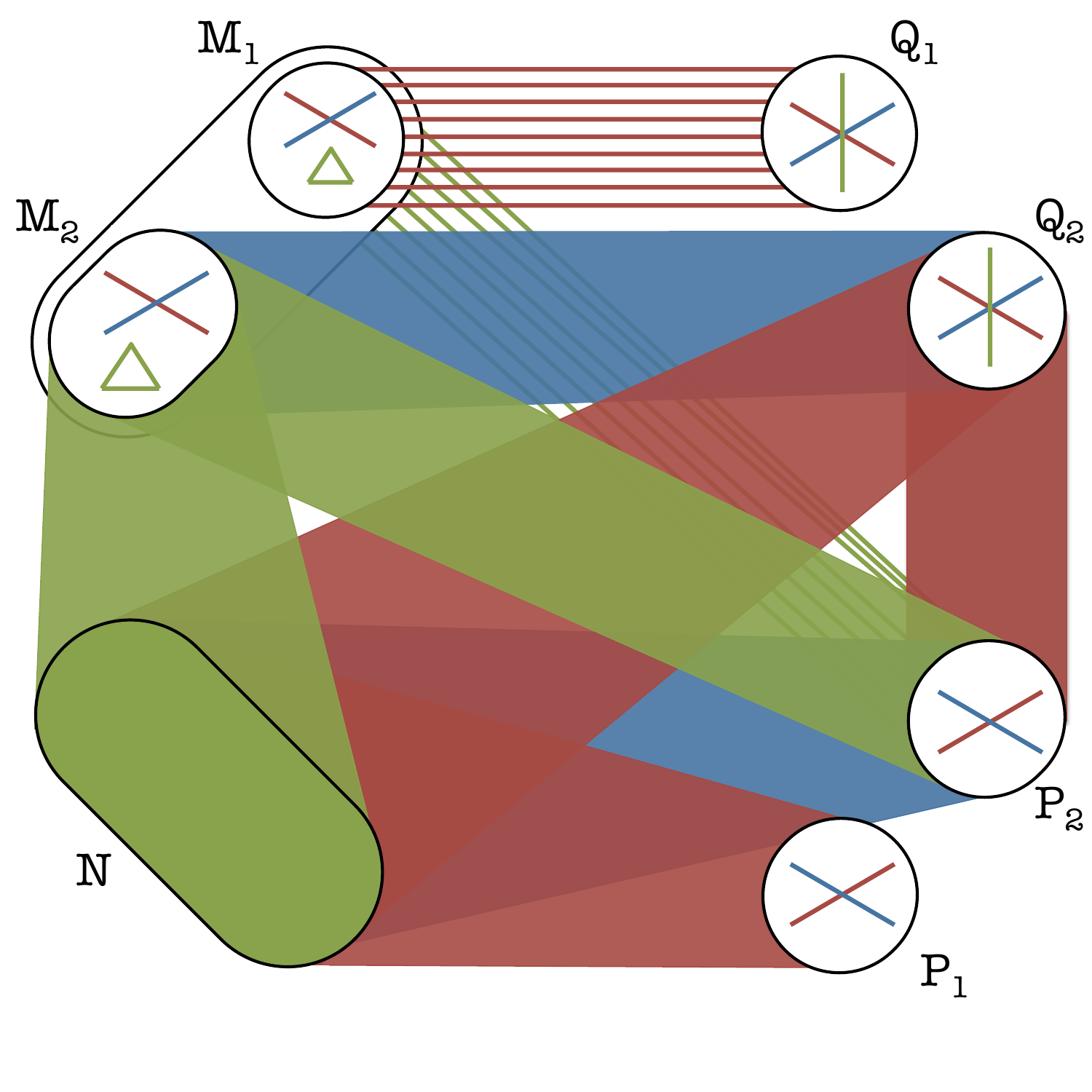}
\vspace{-4mm}\caption{Colouring of $G[M_2,P_2]\cup G[N]$.}
  
\end{figure}

Finally, we consider $G[M_1,N]$. Suppose there exists a red matching $R_D$ on $138\eta k$ vertices in $G[M_1,N]$. Then, 
by~(\ref{G9}), we have $|N\backslash V(R_D)|\geq|P_1|+|Q_2|$. Therefore, by Lemma~\ref{l:eleven}, there exists a red connected-matching $R_E$ on at least $(2|P_1|-2\eta k)+(2|Q_1|-2\eta k)$ vertices in $G[N,P_1\cup Q_2]$ which shares no vertices with $R_D$. Then, since all edges present in $G[N,P_1]$ are coloured red, $R_D$ and $R_E$ belong to the same red component. Since $F$, the largest red component in $G[P\cup Q]$ includes all but at most $8\eta k$ of the vertices of $P\cup Q$, there exists a matching $R_F\subseteq R$ in $G[M_1,Q_1]$ on at least $2(|Q_1|-|M_1\cap V(R_D)|-8\eta k)$ vertices which shares no vertices with $R_D$ but belongs to the same red component as it. Thus, together, $R_D$, $R_E$ and $R_F$ form a red connected-matching on at least $2(|P_1|+|Q_1|+|Q_2|+|N\cap V(R_D)|)-20\eta k \geq \aI k$ vertices.

Similarly, if there exists a blue matching $B_D$ on $220\eta k$ vertices in $G[M_1,N]$. Then, by~(\ref{G9}), we have $|N\backslash V(B_D)|\geq|P_2|$. Therefore, by Lemma~\ref{l:eleven}, there exists a blue connected-matching $B_E$ on at least $(2|P_2|-2\eta k)$ vertices in $G[N,P_2]$ which shares no vertices with $B_D$. Since all edges present in $G[N,P_2]$ are coloured blue, $B_D$ and $B_E$ belong to the same blue component. Thus together $B_D$ and $B_E$ form a blue connected-matching on at least $2|P_2|-2\eta k+220\eta k\geq \aII k$ vertices. Therefore, after discarding at most $174\eta k$ vertices from each of $M_1$ and~$N$, we may assume that all edges present in $G[M_1,N]$ are coloured exclusively green.

\begin{figure}[!h]
\centering
\includegraphics[width=64mm, page=25]{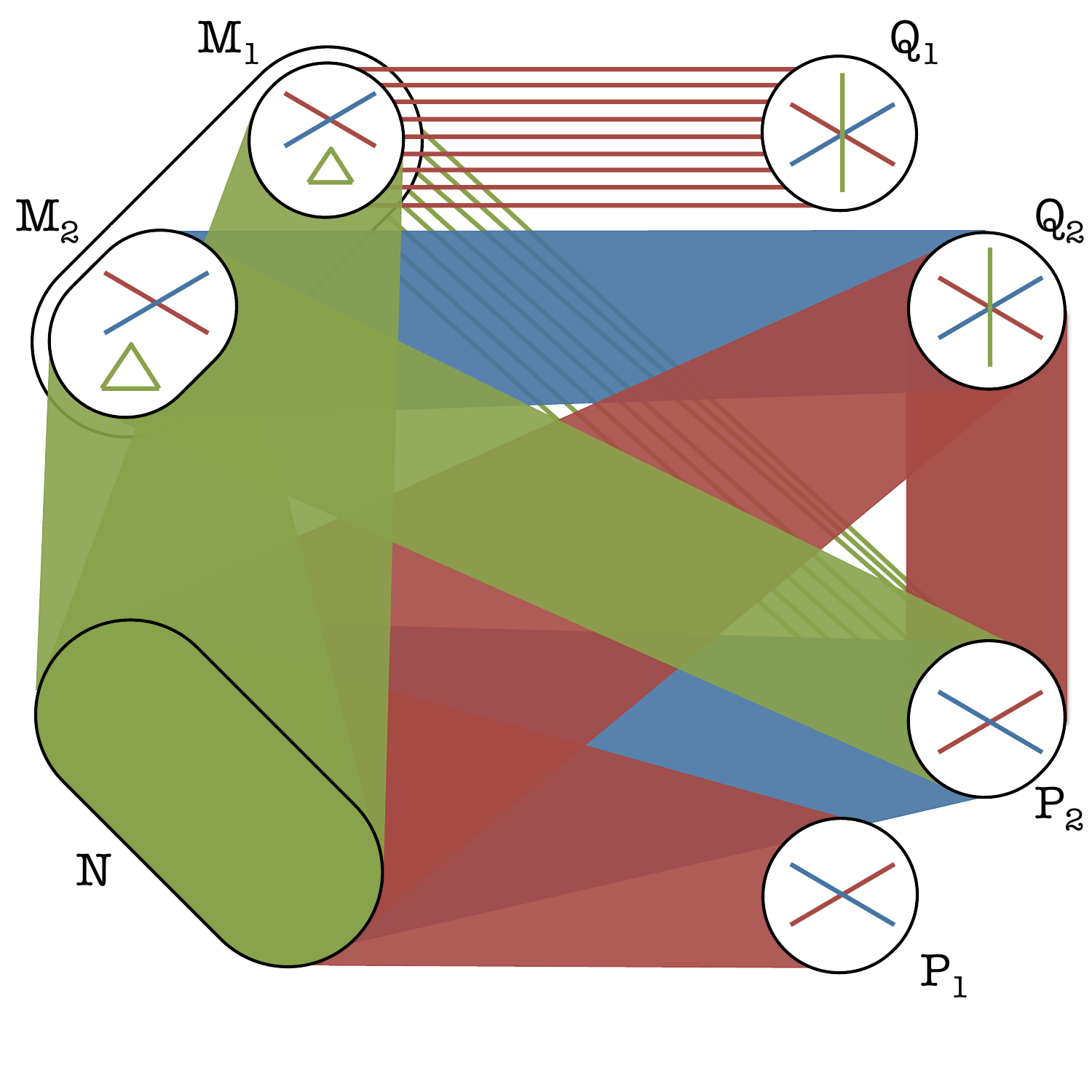}
\vspace{-4mm}\caption{Final Colouring in Case E.iii.a.ii.}
\label{f250}
\end{figure}

Given the colouring found so far, we now show that we may obtain a green connected-matching on at least $\aIII k$ vertices: Having discarded at most $1000\eta k$ vertices from $M\cup N$, recalling (\ref{E4a-x}), since $\eta<10^{-7}$, we have $|M|,|N|\geq (\half\aIII-6\eta^{1/2})k$. Recalling $(\ref{G1})$, we have $|M_2|\geq|Q_2|+14\eta^{1/2}k\geq 14\eta^{1/2}$ and, by~(\ref{G8}), have $|P_2|\geq(\half\aII-108\eta)k\geq14\eta^{1/2}k$. Letting $M^{\prime}$ be a subset of $M_2$ and $N^{\prime}$ a subset of $N$ such that $14\eta^{1/2}k\leq|M^{\prime}|=|N^{\prime}|\leq 15\eta^{1/2}k$, by Lemma~\ref{l:eleven}, there exist green matchings $G_{MP}$ on at least $2|M^{\prime}|-2\eta k$ vertices in $G[M^{\prime},P_2]$ and $G_{MN}$ on at least $2\min\{|M\backslash M^{\prime}|,|N\backslash N^{\prime}|\}-2\eta k$ vertices in $G[M\backslash M^{\prime}, N \backslash N^{\prime}]$. Finally, by  Theorem~\ref{dirac}, provided $k\geq 1/\eta^2$, there exists a connected-matching $G_N$ on at least $|N^{\prime}|-1$ vertices in $G[N^{\prime}]$. Then, since all edges present in $G[M,N]$ are coloured green, $G_{NP}$, $G_{MN}$ and $G_N$ belong to the same green component and, since they share no vertices, form a green connected-matching on at least 
\begin{align*}
(2|M^{\prime}|-2\eta k)+(2\min\{|M\backslash M^{\prime}|,|N\backslash N^{\prime}|\}&-2\eta k) + |N^{\prime}|-1\\
&\geq \aIII k+\eta^{1/2} k - 4\eta k -1\geq \aIII k
\end{align*}
vertices. By the definition of the decomposition $M\cup N\cup P\cup Q$, this connected-matching is odd, thus completing Case E.iii.a.ii.

At the begining of Case E.iii.a, we made the assumption that $F$, the largest monochromatic component in $G[P\cup Q]$ was red. If, instead, $F$ is blue, then the proof is essentially identical to the above with the roles of red and blue reversed. The result is the same, that is,~$G$ will either contain a red connected-matching on at least $\aI k$ vertices, a blue connected-matching on at least $\aII k$ vertices, a green odd connected-matching on at least $\aIII k$ vertices or a subgraph in $$\cK\left((\half\aI-10\eta^{1/2})k, (\half\aI-10\eta^{1/2})k, (\aIII-20\eta^{1/2})k, 4\eta^4 k \right),$$ thus completing Case E.iii.a.
 
\subsection*{Case E.iii.b: {\rm $G[P,Q]$ contains red and blue stars centred in $Q$.}}

Recall that we have a decomposition of $V(G)$ into $M\cup N\cup P \cup Q$ satisfying (\ref{E1-iii})--(\ref{E3-iii}) with $|P|,|Q|\geq95\eta^{1/2}k$ and
\begin{align}\label{E4a-y}\tag{E4a}(\max\{\tfrac{3}{4}\aI+\tfrac{1}{4}\aII,\half\aIII\}-5\eta^{1/2})k & \leq  |M|\,,\,|N| \leq \half\aIII k\\
\label{E4b-y}\tag{E4b}(\half\aI+\half\aII-\eta)k & \leq  |P|+|Q| \leq (\half\aI+\half\aII+5\eta^{1/2})k.
\end{align}

Additionally, in this case, we assume that the sets 
\begin{itemize}
\item[] $W_r=\{ q\in Q$ : $q$ has red edges to all but at most $8\eta k$ vertices in $P\}$,
\item[] $W_b=\{ q\in Q$ : $q$ has blue edges to all but at most $8\eta k$ vertices in $P\}$,
\end{itemize}
are both non-empty.

\begin{figure}[!h]
\centering
\includegraphics[width=64mm, page=57]{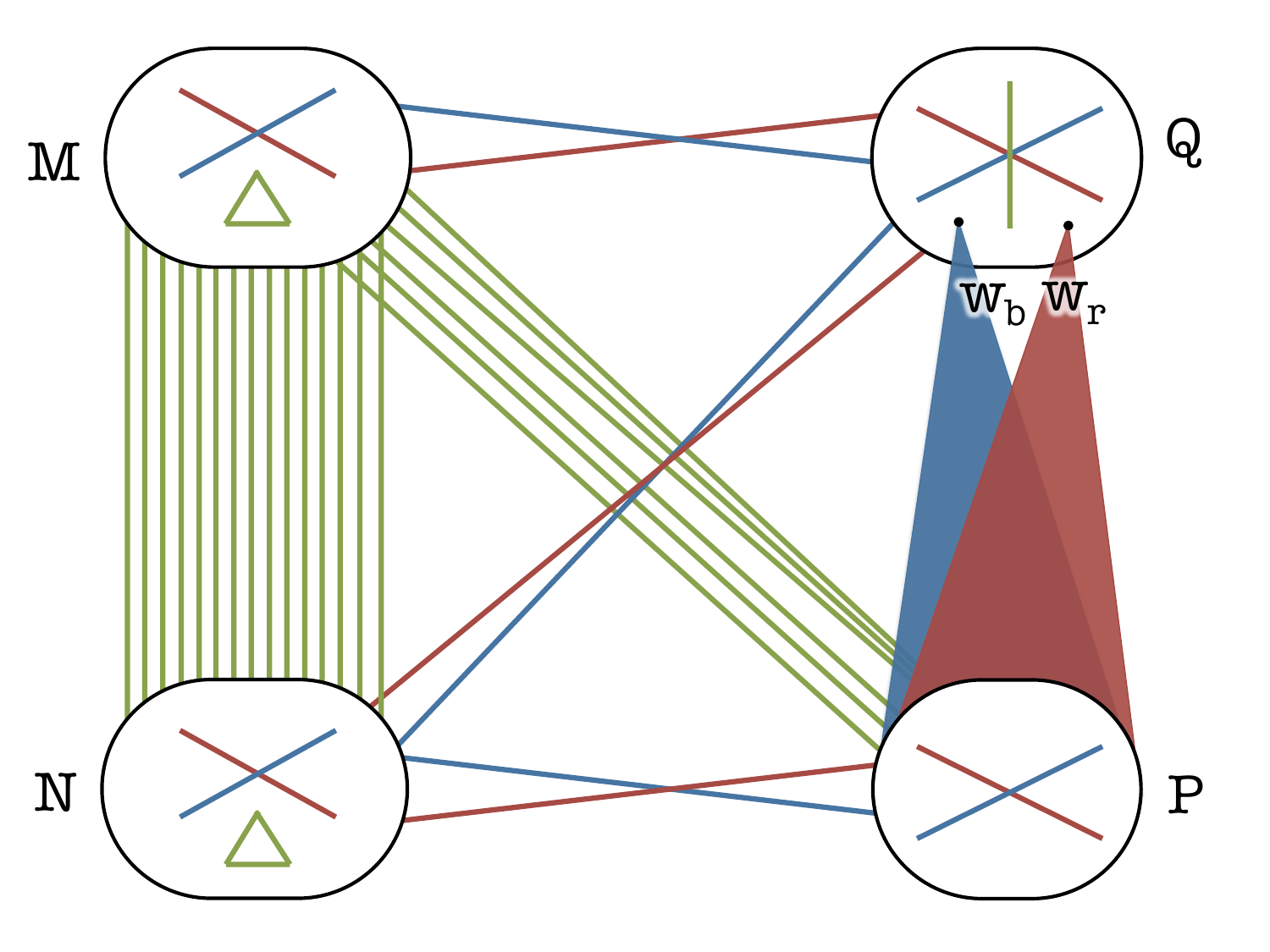}
\vspace{-3mm}\caption{Red and blue stars centred in $Q$.}
  
\end{figure}

We define two further sets which will be useful in what follows:
\begin{itemize}
\item[] $P_r=\{p\in P$ : $p$ has a red edge to some vertex $q \in W_r\}$,
\item[] $P_b=\{p\in P$ : $p$ has a blue edge to some vertex $q \in W_b\}$. 
\end{itemize}

Observe that, since $|P|,|Q|\geq 95\eta^{1/2}k$, by (E4b), we have
$$95\eta^{1/2}k\leq |Q|\leq (\half\aI +\half\aII-90\eta^{1/2})k.$$
Thus, considering (E4a) and (E4b), we have
$$|N|\geq|Q|\geq95\eta^{1/2}\geq6(2\eta^{1/2})|N\cup Q|\geq6(2\eta)|N\cup Q|.$$

Recall that, at the start of Case E, after discarding some edges, $G$ was assumed to be $(1-\tfrac{3}{2}\eta^4)$-complete. Recall also that, from (\ref{A3BIG}), we have $\aIII\geq\tfrac{3}{2}\aI+\half\aII-10\eta^{1/2}$. Thus, considering (E4a), since $|Q|\geq95\eta^{1/2}k$, we have
$$|N\cup Q|\geq\half\aIII k\geq\tfrac{1}{4}(\aIII+\tfrac{3}{2}\aI+\half\aII-10\eta^{1/2})k\geq\tfrac{1}{4}(\aIII+\half\aI+\half\aII+10\eta^{1/2})k\geq\tfrac{1}{4}K.$$
Then, since $\eta<0.1$, we have $6\eta^4\leq 2\eta$ so $G[N\cup Q]$ is $(1-2\eta)$-complete and, provided $|N\cup Q|\geq 1/\eta$, we may apply Lemma~\ref{l:twoholes} to $G[N,Q]$ and distinguish four cases: 

\begin{itemize}
\item[(i)] $G[N\cup Q]$ has a monochromatic component $E$ on at least $|N\cup Q|-28\eta k$ vertices;
\item[(ii)] $N,Q$ can be partitioned into $N_1\cup N_2$, $Q_1\cup Q_2$ such that $|N_1|, |N_2|, |Q_1|, |Q_2|\geq 3\eta k$ and all edges present between $N_i$ and $Q_j$ are red for $i=j$ and blue for $i\neq j$;
\item[(iii)] there exist vertices $n_r, n_b\in N$ such that $n_r$ has red edges to all but $16\eta k$ vertices in~$Q$ and $n_b$ has blue edges to all but $16\eta k$ vertices in~$Q$;
\item[(iv)] there exist vertices $q_r, q_b\in Q$ such that $q_r$ has red edges to all but $16\eta k$ vertices in~$N$ and $q_b$ has blue edges to all but $16\eta k$ vertices in~$N$.
\end{itemize}

\subsection*{Case E.iii.b.i: {\rm $G[N\cup Q]$ has a large monochromatic component.}}

Suppose that $E$, the largest monochromatic component in $G[N\cup Q]$, is red. In that case, if $G[Q\cap E, P_r]$ contains a red edge, then $G[P\cup Q]$ has a red effective-component on at least $|P\cup Q|-36\eta k$ vertices. Alternatively, every edge in $G[Q\cap E,P_r]$ is blue, in which case, $G[P\cup Q]$ has a blue connected-component on at least $|P\cup Q|-36\eta k$ vertices.

\begin{figure}[!h]
\centering
\includegraphics[width=64mm, page=59]{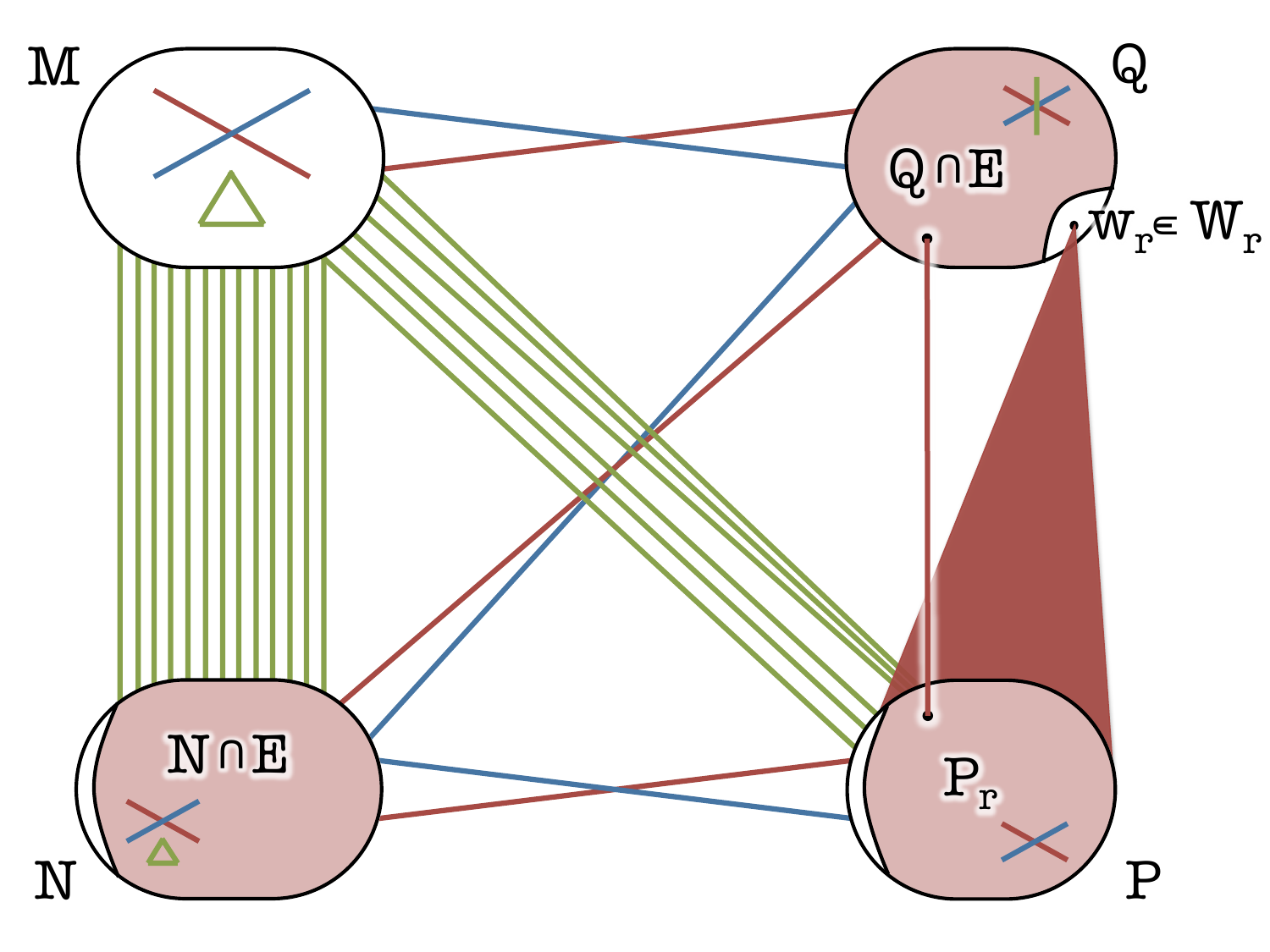}
\vspace{-3mm}\caption{Large red effective-component.}
  
\end{figure}
Suppose instead that $E$ is blue. In that case, if $G[Q\cap E,P_b]$ contains a blue edge, then $G[P\cup Q]$ has a blue effective-component on at least $|P\cup Q|-36\eta k$ vertices. Alternatively, every edge present in $G[Q\cap E,P_b]$ is red, in which case, $G[P\cup Q]$ has a red connected-component on at least $|P\cup Q|-36\eta k$ vertices. 

In either case the proof proceeds via exactly the same steps as Case E.iii.a with the result being that~$G$ will either contain a red connected-matching on at least $\aI k$ vertices, a blue connected-matching on at least $\aII k$ vertices, a green odd connected-matching on at least $\aIII k$ vertices or a subgraph in $$\cK\left((\half\aI-100\eta^{1/2})k, (\half\aI-100\eta^{1/2})k, (\aIII-200\eta^{1/2})k, 4\eta^4 k \right),$$ 
thus completing Case E.iii.b.i.
 
\subsection*{Case E.iii.b.ii: ${N\cup Q}$ {\rm has a non-trivial partition with `cross' colouring.}}

In this case, we assume that $N$ and $Q$ can be partitioned into $N_1\cup N_2$, $Q_1\cup Q_2$ such that $|N_1|, |N_2|, |Q_1|, |Q_2|\geq 3\eta k$ and all edges present between $N_i$ and $Q_j$ are red for $i=j$ and blue for $i\neq j$.

\begin{figure}[!h]
\centering
\includegraphics[width=64mm, page=27]{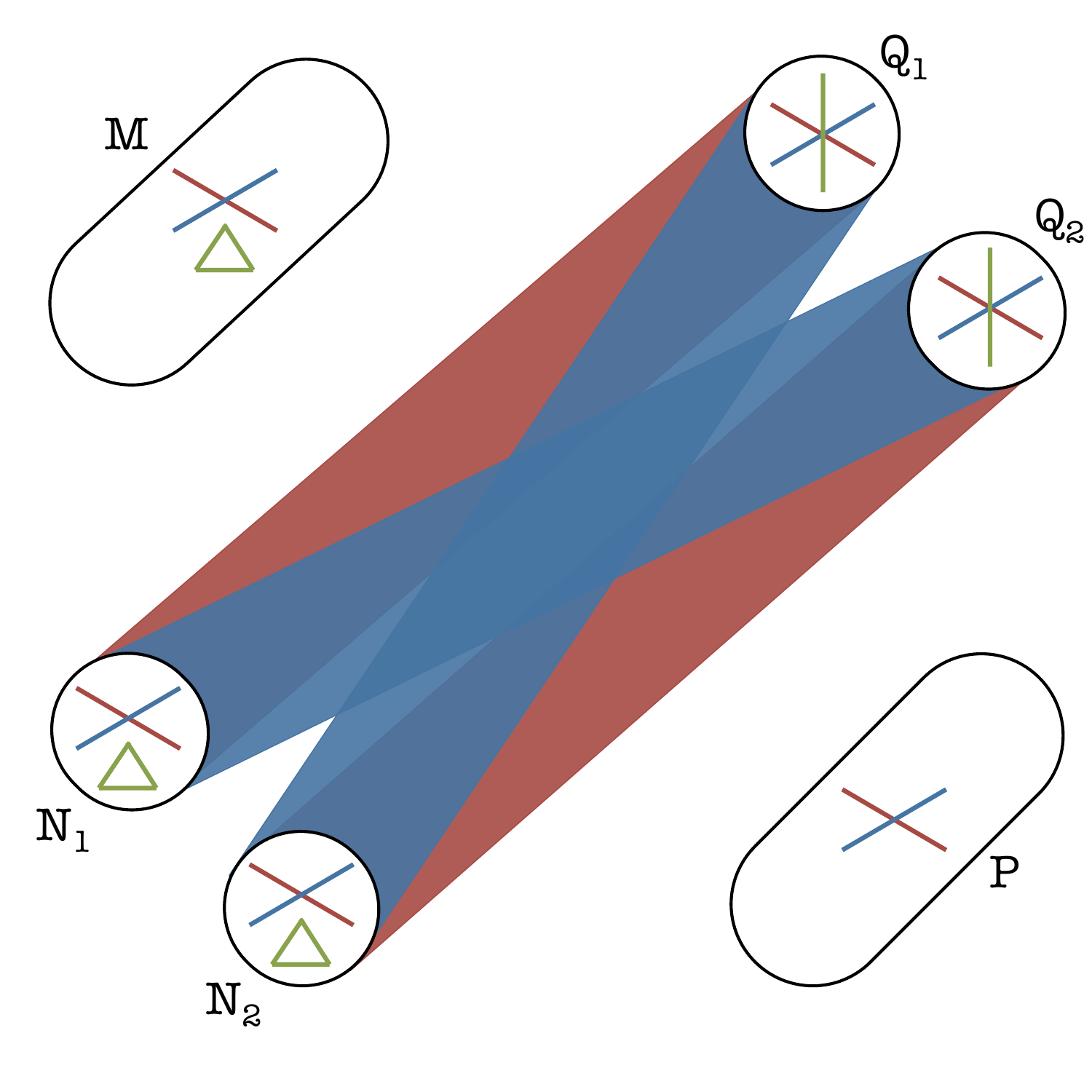}
\vspace{-4mm}\caption{`Cross' colouring of $G[N,Q]$.}
  
\end{figure}

Recall that $W_r$ and $W_b$ are both non-empty, that is, there exist vertices $w_r, w_b\in Q$ such that $w_r$ has red edges to all but $8\eta k$ vertices in~$P$ and $w_b$ has blue edges to all but $8\eta k$ vertices in~$P$.

Observe that, we may assume that $P\cup Q$ does not have a monochromatic effective-component on at least $|P\cup Q|-16\eta k$ vertices. Indeed, otherwise, the proof proceeds via the same steps as Case E.iii.a, with the result being that~$G$ will either contain a red connected-matching on at least $\aI k$ vertices, a blue connected-matching on at least $\aII k$ vertices, a green odd connected-matching on at least $\aIII k$ vertices or a subgraph in $$\cK\left((\half\aI-100\eta^{1/2})k, (\half\aI-100\eta^{1/2})k, (\aIII-200\eta^{1/2})k, 4\eta^4 k \right).$$ 

Without loss of generality, we assume $w_r\in Q_1$. Observe then that the existence of a red edge in $G[Q_2, P_r]$ would result in $P \cup Q$ having a red effective-component on at least $|P\cup Q|-8\eta k$ vertices. Thus, we assume that every edge present in $G[Q_2, P_r]$ is blue. Similarly, we may assume every edge present in $G[Q_1, P_b]$ is red. 
 Since $W_r$ and~$W_b$ are non-empty, $|P_r|,|P_b|\geq|P|-8\eta k$. Then, as $|P|\geq 95\eta^{1/2} k$, we have $|P_r \cap P_b|\geq|P|-16\eta k>0$ and know that all edges present in $G[Q_2, P_r\cap P_b]$ are coloured exclusively red and all edges present in $G[Q_1, P_r\cap P_b]$ are coloured exclusively blue. 

\begin{figure}[!h]
\centering
\includegraphics[width=64mm, page=26]{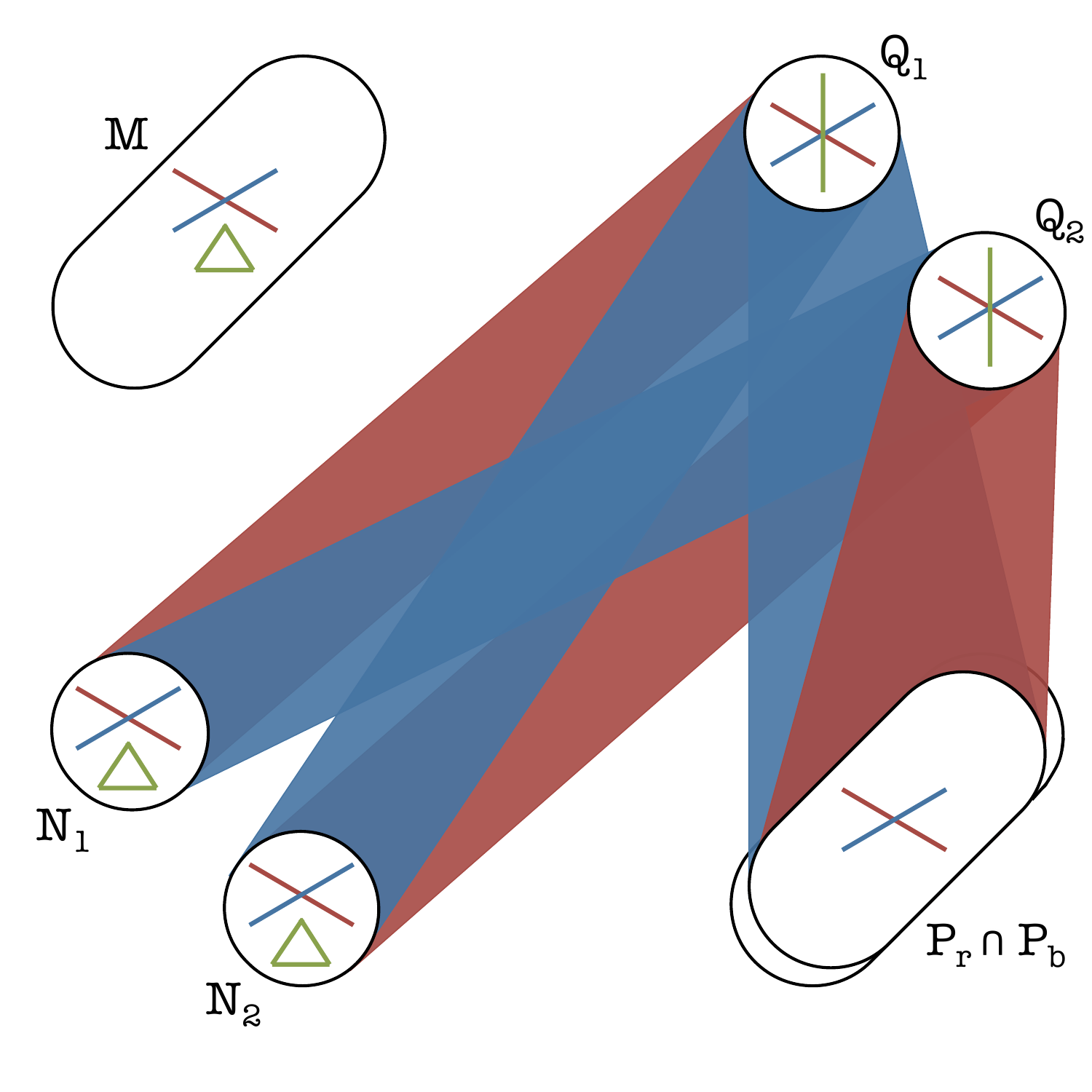}
\vspace{-4mm}\caption{Red and blue edges in $G[N,Q]\cup G[P,Q]$.}
  
\end{figure}

Recall that, by (\ref{E3-iii}), there are no green edges in $G[P,N]$ and notice that, if there existed a vertex $p\in P$ with red edges to both $N_1$ and $N_2$, then $P\cup Q$ would have an effective red-component on at least $|P\cup Q|-16\eta k$ vertices. Thus, we assume that there is no such vertex. Similarly, we assume there does not exist a vertex in $P$ with blue edges to both $N_1$ and $N_2$. Thus,~$P$ can be partitioned into $P_1\cup P_2$ such that all edges present between $P_i$ and $N_j$ are red for $i=j$ and blue for $i\neq j$. 

  \begin{figure}[!h]
\centering{
\mbox{\hspace{-2mm}{\includegraphics[width=64mm, page=28]{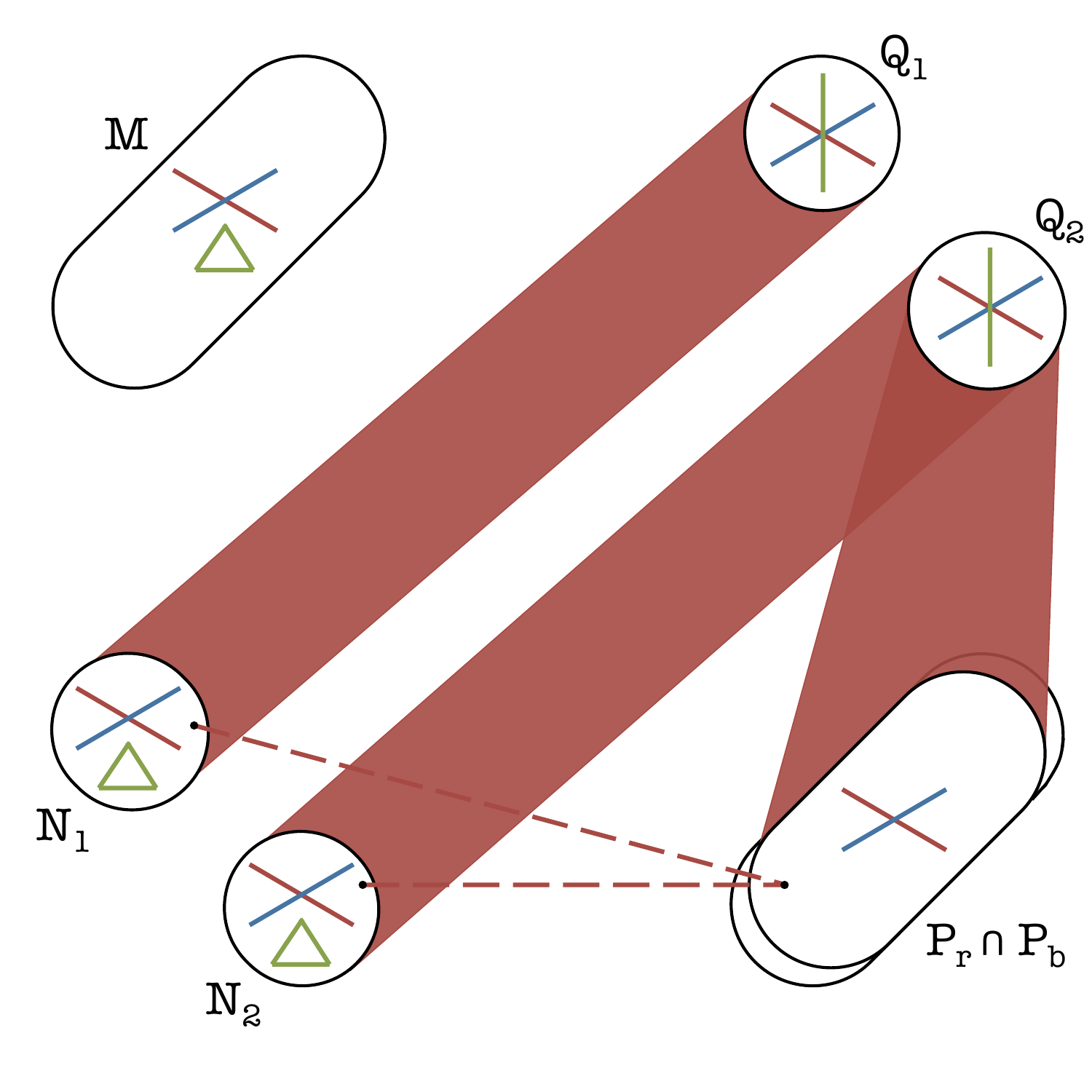}}\quad\quad\quad{\includegraphics[width=64mm, page=29]{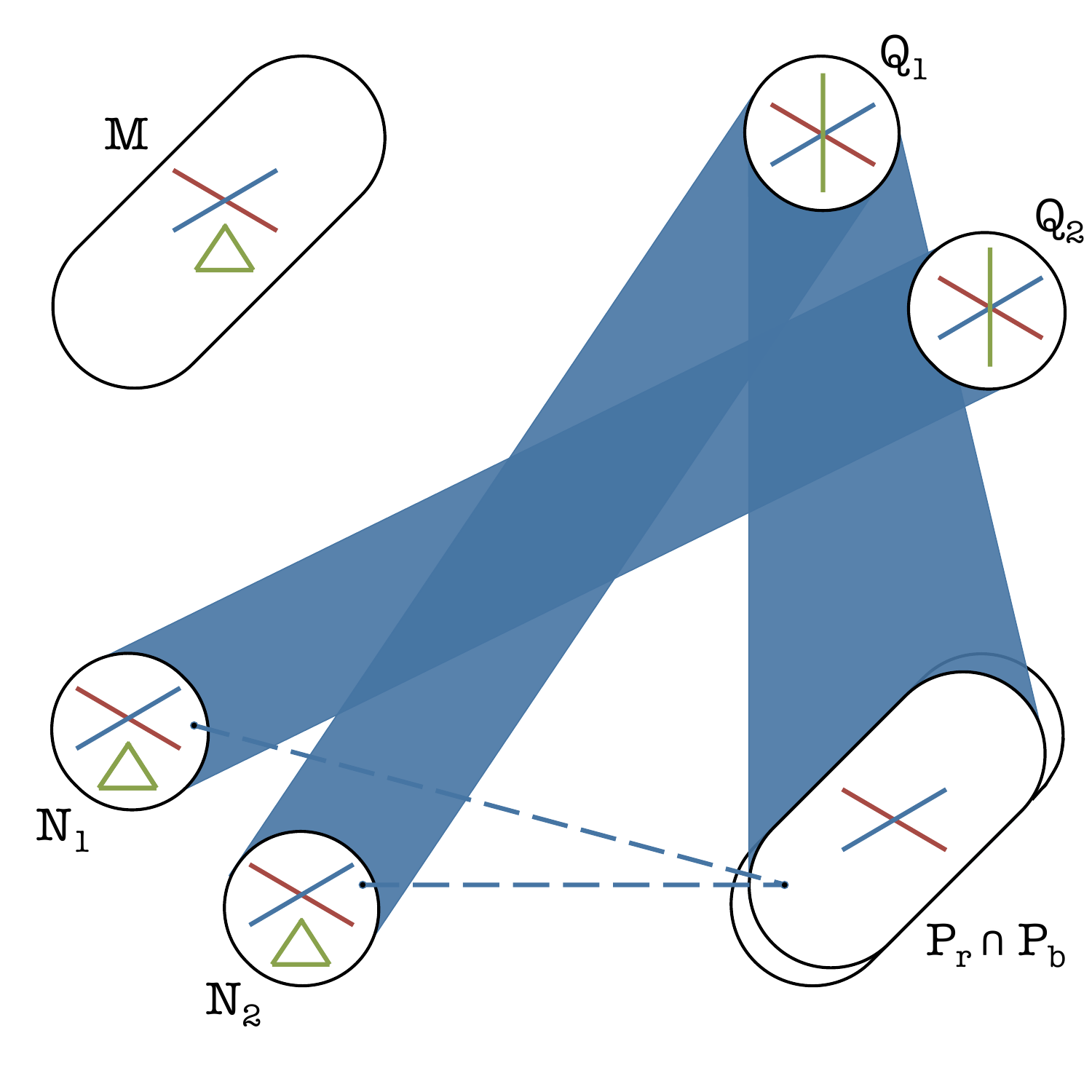}}}}\vspace{-5mm}
\caption{Partitioning of $P$.}   
\end{figure}

  \begin{figure}[!h]
\centering{
\mbox{\hspace{-2mm}{\includegraphics[width=64mm, page=31]{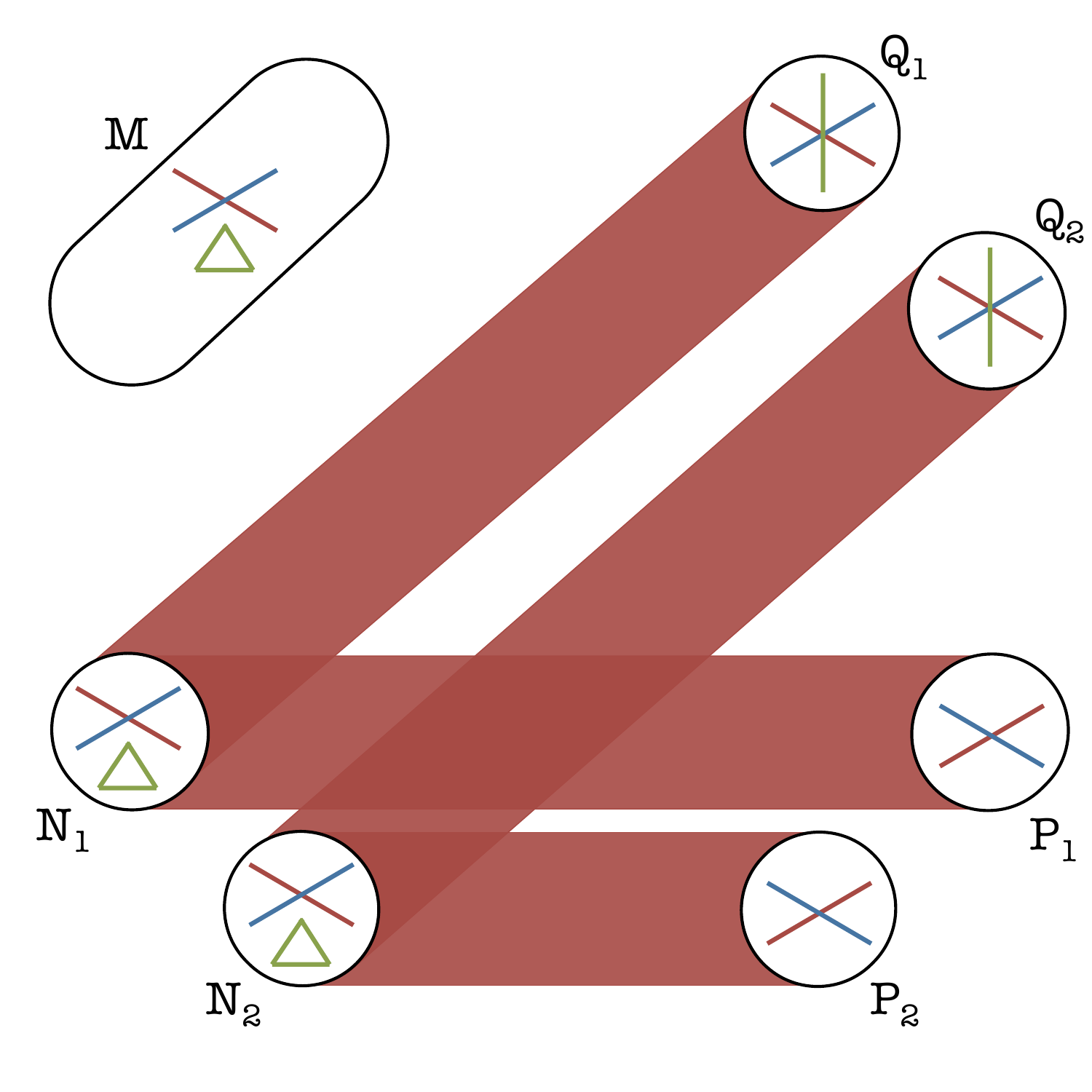}}\quad\quad\quad{\includegraphics[width=64mm, page=32]{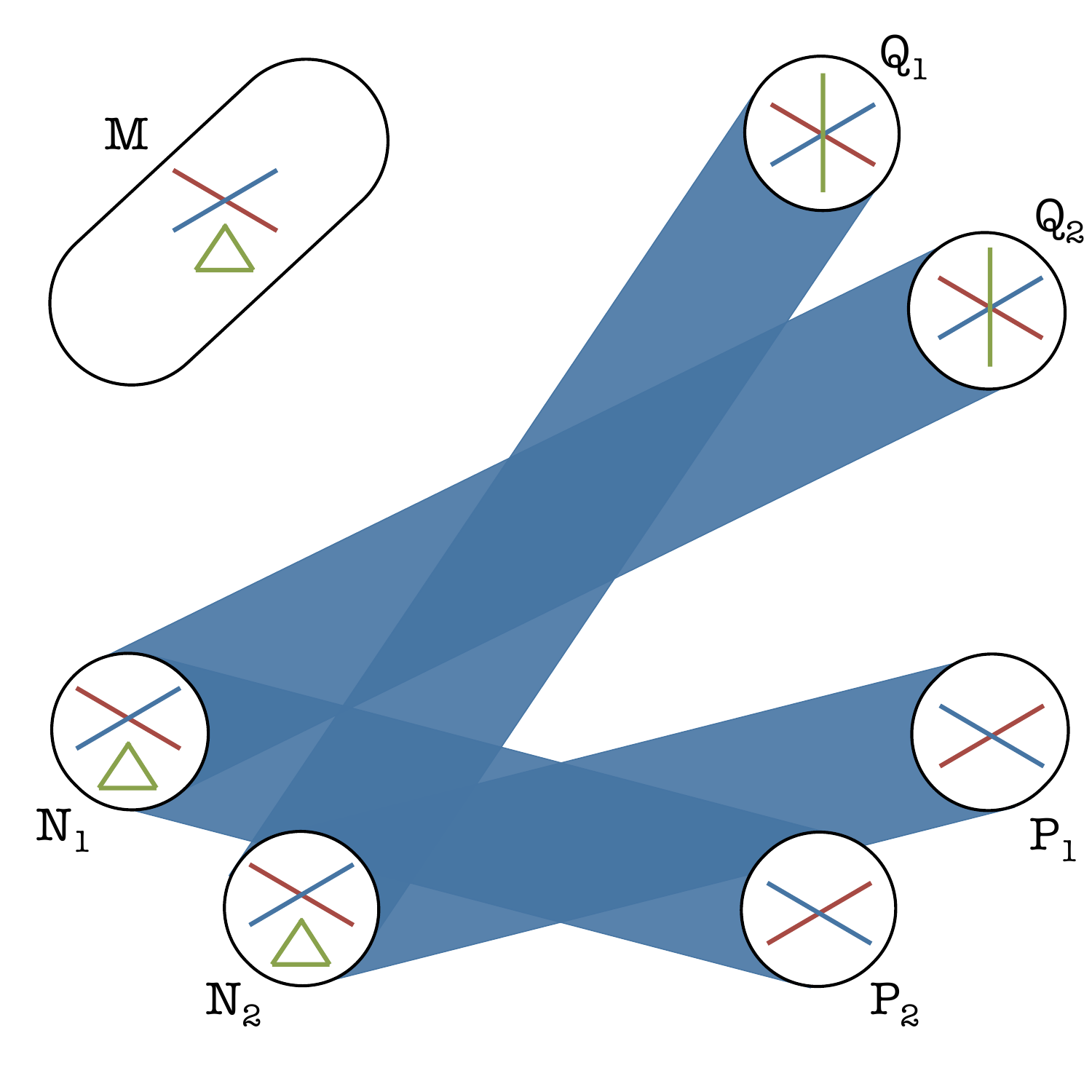}}}\vspace{-3mm}
\caption{Resultant colouring of $G[N,P]$.}   }
\end{figure}

Recall, however, that all edges present in $G[Q_1,P_r]$ are red and that all edges present in $G[Q_2,P_b]$ are blue. Thus, in order to avoid $G[P,Q]$ having a red effective-component on at least $|P\cup Q|-16\eta k$ vertices, we must have $P_r\subseteq P_1$ but then, since $|P_r\cap P_b|>0$, there exists a blue edge in $G[Q_2,P_1]$, giving rise to a blue effective-component on at least $|P\cup Q|-16\eta k$ vertices, completing Case E.iii.a.ii.

\subsection*{Case E.iii.b.iii: {\rm $G[N,Q]$ contains red and blue stars centred in $N$.}}

Recall that we have a decomposition of $V(G)$ into $M\cup N\cup P \cup Q$ satisfying (\ref{E1-iii})--(\ref{E4b-y}) and that there exists vertices $w_r\in W_r$ and $w_b\in W_b$, where 
\begin{itemize}
\item[] $W_r=\{ q\in Q$ : $q$ has red edges to all but at most $8\eta k$ vertices in $P\}$,
\item[] $W_b=\{ q\in Q$ : $q$ has blue edges to all but at most $8\eta k$ vertices in $P\}$.
\end{itemize}

Additionally, in this case, we assume that there exist vertices $n_r\in N_r$ and $n_b\in N_b$, where
\begin{itemize}
\item[] $N_r=\{ n\in N$ : $n$ has red edges to all but at most $16\eta k$ vertices in $Q\}$,
\item[] $N_b=\{ n\in N$ : $n$ has blue edges to all but at most $16\eta k$ vertices in $Q\}$.
\end{itemize}

Recall that 
\begin{itemize}
\item[] $P_r=\{p\in P$ : $p$ has a red edge to some vertex $q \in W_r\}$,
\item[] $P_b=\{p\in P$ : $p$ has a blue edge to some vertex $q \in W_b\}$, 
\end{itemize}

and define $Q_N=\{q\in Q$ : $q$ has a red edge to some vertex $n\in N_r\}$.

\begin{figure}[!h]
\centering
\includegraphics[width=64mm, page=44]{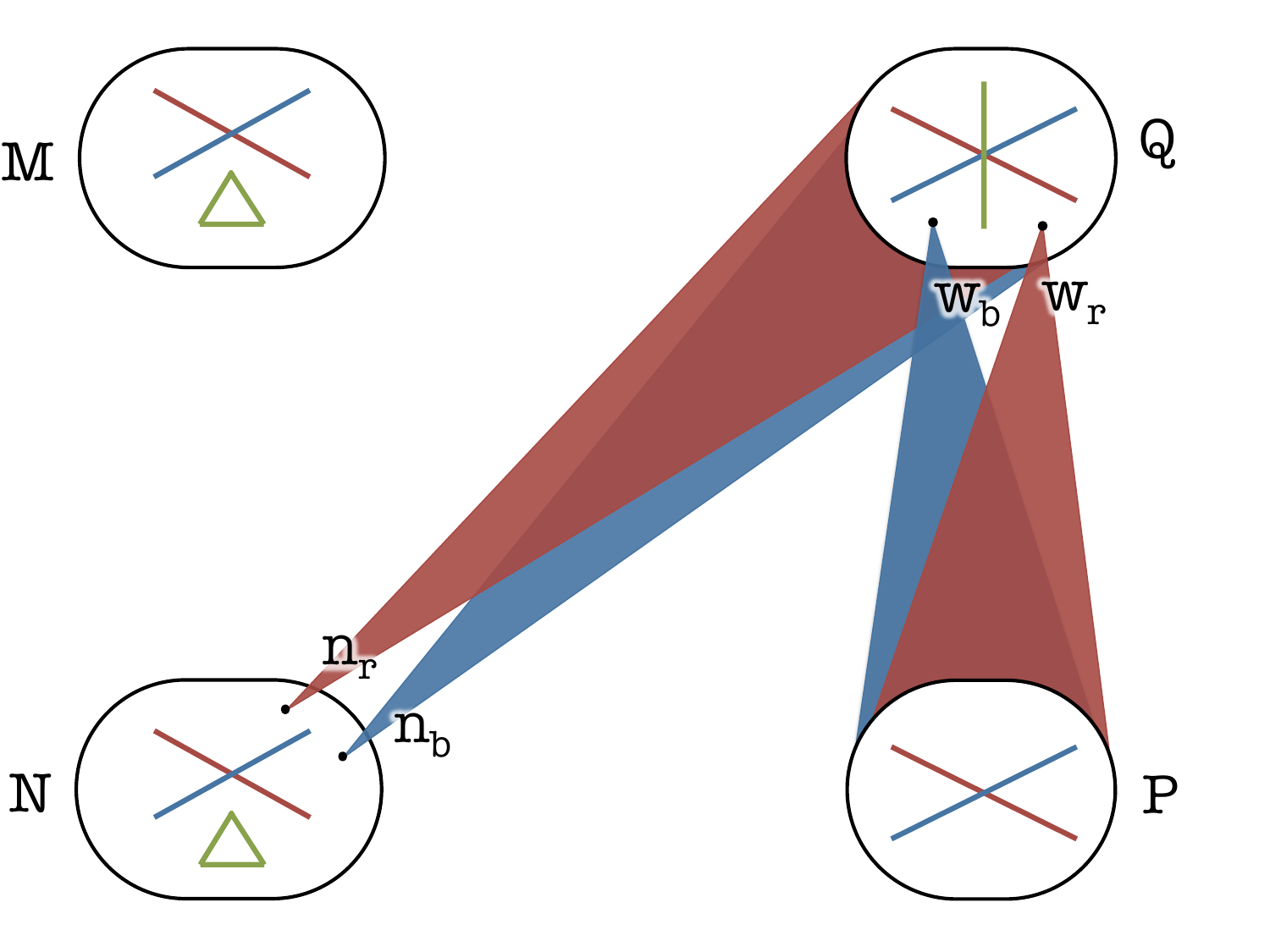}
\vspace{-3mm}\caption{`Stars' centred at $w_r, w_b, n_r$ and $n_b$.}
  
\end{figure}

Notice that, since $|W_r|,|W_b|>0$, we have $|P_r|,|P_b|\geq|P|-8\eta k$ and, since $|N_r|>0$, we have $|Q_N|\geq|Q|-16\eta k$. So, if there exists a red edge in $G[P_r,Q_N]$, then $G[P\cup Q]$ has a red effective-component on at least $|P\cup Q|-24\eta k$ vertices. Alternatively, every edge present in $G[P_r,Q_N]$ is blue. Then $G[P\cup Q]$ is has a blue component on at least $|P\cup Q|-24\eta k$ vertices. In either case, the proof then follows the same steps has in Case E.iii.a with the result being that~$G$ will either contain a red connected-matching on at least $\aI k$ vertices, a blue connected-matching on at least $\aII k$ vertices, a green odd connected-matching on at least $\aIII k$ vertices or a subgraph in $$\cK\left((\half\aI-100\eta^{1/2})k, (\half\aI-100\eta^{1/2})k, (\aIII-200\eta^{1/2})k, 4\eta^4 k \right),$$ thus completing case E.iii.b.iii.

\subsection*{Case E.iii.b.iv: {\rm $G[N,Q]$ contains red and blue stars centred in $Q$.}}

Recall that we have a decomposition of $V(G)$ into four parts $M\cup N\cup P \cup Q$ satisfying

\begin{itemize}
\labitem{E1}{E1-z} $M\cup N$ is the vertex set of~$F$ and every edge of~$F$ belongs to $G[M,N]$;
\labitem{E2}{E2-z} every vertex in~$P$ has a green edge to~$M$;
\labitem{E3}{E3-z} there are no green edges in $G[N,P]$, $G[M,Q]$, $G[N,Q]$, $G[P,Q]$ or $G[P]$;
\end{itemize}
such that the sizes of the four parts satisfy
\begin{align}\label{E4a-z}\tag{E4a}(\max\{\tfrac{3}{4}\aI+\tfrac{1}{4}\aII,\half\aIII\}-5\eta^{1/2})k & \leq  |M|\,,\,|N| \leq \half\aIII k,\\
\label{E4b-z}\tag{E4b}(\half\aI+\half\aII-\eta)k & \leq  |P|+|Q| \leq (\half\aI+\half\aII+5\eta^{1/2})k.
\end{align}
Recall, also, that there exists vertices $w_r\in W_r$ and $w_b\in W_b$, where 
\begin{itemize}
\item[] $W_r=\{ q\in Q$ : $q$ has red edges to all but at most $8\eta k$ vertices in $P\}$,
\item[] $W_b=\{ q\in Q$ : $q$ has blue edges to all but at most $8\eta k$ vertices in $P\}$.
\end{itemize}

Additionally, in this case, we assume that there exist vertices $q_r\in Q_r$ and $q_b\in Q_b$, where
\begin{itemize}
\item[] $Q_r=\{ n\in Q$ : $n$ has red edges to all but at most $16\eta k$ vertices in $N\}$,
\item[] $Q_b=\{ n\in Q$ : $n$ has blue edges to all but at most $16\eta k$ vertices in $N\}$.
\end{itemize}

We consider $G[Q\cup M]$. Since all edges present in $G[Q,M]$ are coloured red or blue,  {Lemma~\ref{l:twoholes} can be applied to $G[M\cup Q]$ as it was previously applied to $G[N\cup Q]$ with the same four possible outcomes:
\begin{itemize}
\item[(i)] $G[M\cup Q]$ has a monochromatic component on at least $|M\cup Q|-28\eta k$ vertices;
\item[(ii)] $M,Q$ can be partitioned into $M_1\cup M_2$, $Q_1\cup Q_2$ such that $|M_1|, |M_2|, |Q_1|, |Q_2|\geq 3\eta k$ and all edges present between $M_i$ and $Q_j$ are red for $i=j$ and blue for $i\neq j$;
\item[(iii)] there exist vertices $m_r, m_b\in N$ such that $m_r$ has red edges to all but $16\eta k$ vertices in~$Q$ and $m_b$ has blue edges to all but $16\eta k$ vertices in~$Q$;
\item[(iv)] there exist vertices $q_r, q_b\in Q$ such that $q_r$ has red edges to all but $16\eta k$ vertices in~$M$ and $q_b$ has blue edges to all but $16\eta k$ vertices in~$M$.
\end{itemize}

 For possibilities (i)--(iii), the proof proceeds via exactly the same steps as in the corresponding cases above with the same possible outcomes.
Thus, we consider possibility (iv). That is, we  assume that (in addition to $w_r,w_b, q_r, q_b$) there exist vertices $v_r, v_b\in Q$ such that $v_r$ has red edges to all but $16\eta k$ vertices in~$M$ and $v_b$ has blue edges to all but $16\eta k$ vertices in~$M$.

\begin{figure}[!h]
\centering
\includegraphics[width=64mm, page=46]{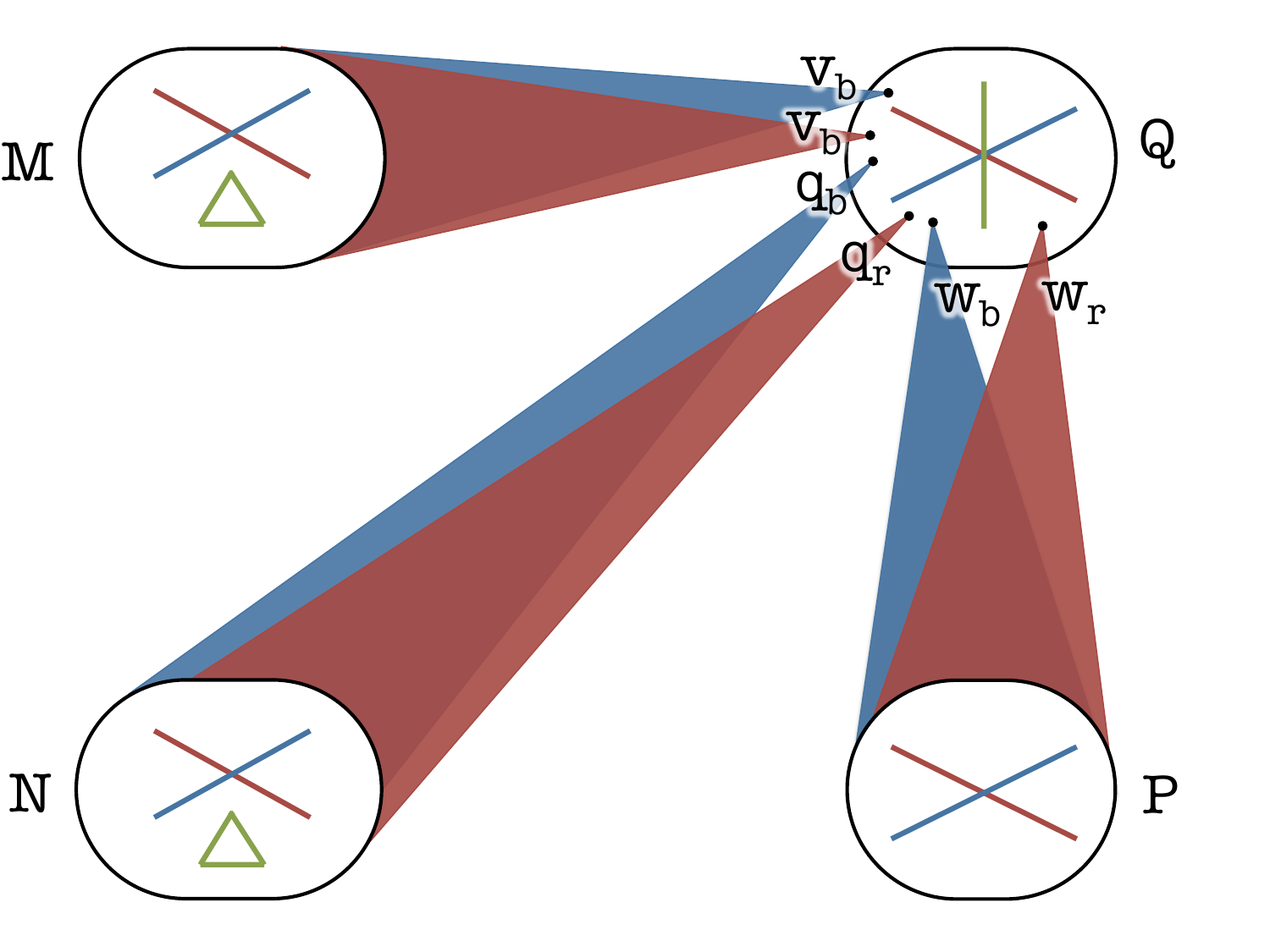}
\vspace{-3mm}\caption{`Stars' centred in $Q$.}
  
\end{figure}

Consider the largest red matching $R$ in $G[N,P\cup Q]$ and partition each of~$N$,~$P$ and~$Q$ into two parts such that $N_1=N\cap V(R)$, $N_2=N\backslash N_1$, $P_1=P\cap V(R)$, $P_2=P\backslash P_1$, $Q_1=Q\cap V(R)$ and $Q_2=Q\backslash Q_1$. By maximality of $R$, all edges present in $G[N_2,P_2\cup Q_2]$ are coloured exclusively blue. 

\begin{figure}[!h]
\centering
\includegraphics[width=64mm, page=37]{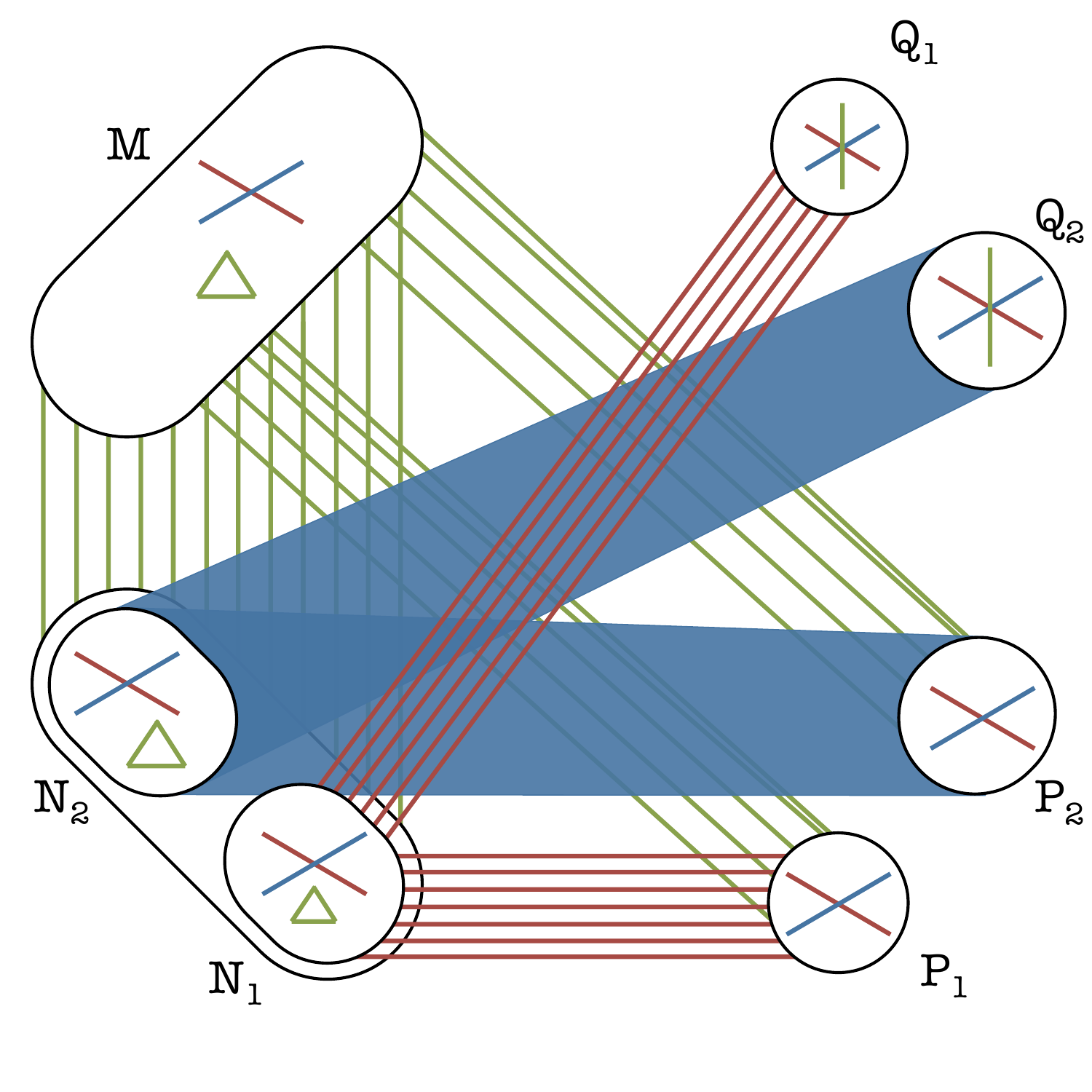}
\vspace{-5mm}\caption{Partition of $N$, $P$ and $Q$.}
  
\end{figure}

Notice that, in order to avoid having a blue connected-matching on at least $\aII k$ vertices,  by Lemma~\ref{l:eleven}, we may assume that $\min\{|P_2 \cup Q_2|,|N_2|\}\leq (\half\aII+\eta^{1/2})k.$
Thus, by (\ref{E4a-z}) and (\ref{E4b-z}), we may thus assume that 
\begin{subequations}
\begin{equation}
\label{M0a}
|N_1|=|P_1\cup Q_1|\geq (\half\aI-6\eta^{1/2})k.
\end{equation}
Since $q_r$ has red edges to all but at most $16\eta k$ vertices in~$N$, all but at most $16\eta k$ of the edges of $R$ belong to the same red component and we have a red connected-matching on $2(|N_1|-16\eta k)$ vertices in $G[N_1,P_1\cup Q_1]$. Thus, we may assume that 
 \begin{equation*}
 \label{M0ab}
 |N_1|=|P_1\cup Q_1|\leq (\half\aI+\eta^{1/2})k
 \end{equation*} and, therefore, by (E4a) and (E4b), that 
 \begin{equation}
 \label{M0b}
|N_2|, |P_2\cup Q_2| \geq (\half\aII-6\eta^{1/2}).
 \end{equation}
 \end{subequations}
Recall that  $v_r$ has red edges to all but $16\eta k$ vertices in~$M$ and $v_b$ has blue edges to all but $16\eta k$ vertices in~$M$. Thus, we may, after discarding at most $32\eta k$ vertices from~$M$, assume that~$G[M]$ is effectively red-connected and effectively blue-connected.

Similarly, we may, after discarding at most $16\eta k$ vertices from~$G[P]$, assume that~$P$ is effectively red-connected and effectively blue-connected and, after discarding at most $32\eta k$ vertices from~$N$, assume that~$N$ is effectively red-connected and effectively blue-connected. In order to maintain the equality $|N_1|=|P_1|$, we also discard from $N_1\cup P_1$ any vertex whose $R$-mate has already been discarded.

Then, in summary, having discarded some vertices, we have a decomposition of $V(G)$ into $M\cup N\cup P\cup Q$ and a refinement into $M\cup N_1\cup N_2 \cup P_1\cup P_2 \cup Q_1\cup Q_2$ such that
\begin{itemize}
\labitem{E3}{E3-a} there are no green edges in $G[N,P]$, $G[M,Q]$, $G[N,Q]$, $G[P,Q]$ or $G[P]$;
\labitem{E6a}{E6a-a} $G[M]$, $G[N]$ and $G[P]$ each have a single red and a single blue effective-component;
\labitem{E6b}{E6b-a} $G[N_1,P_1\cup Q_1]$ contains a red matching utilising every vertex in $G[N_1\cup P_1\cup Q_1]$;
\labitem{E6c}{E6c-a} all edges present in $G[N_2,P_2\cup Q_2]$ are coloured exclusively blue.

\end{itemize}

Having discarded some vertices, recalling (E4a),~(\ref{M0a}) and~(\ref{M0b}), we have
\begin{equation}
\label{M1}
\tag{E7}
\left.
\begin{aligned}
\!\!\!\!\! |M|&\geq\big(\!\max\{\tfrac{3}{4}\aI+\tfrac{1}{4}\aII,\half\aIII\}-6\eta^{1/2}\big)k, 
& \,|N_1|=|P_1\cup Q_1|&\geq (\half\aI-8\eta^{1/2})k,\\
\!\!\!\!\! |N|&\geq\big(\!\max\{\tfrac{3}{4}\aI+\tfrac{1}{4}\aII,\half\aIII\}-7\eta^{1/2}\big)k, 
& |N_2|, \, |P_2\cup Q_2| &\geq (\half\aII-8\eta^{1/2})k.
\end{aligned}
\,\,\, \right\}\!
\end{equation}
and also
\vspace{-2mm}
\begin{align}
|P|&\geq 90\eta^{1/2} k, & \max\{|Q_1|,|Q_2|\}&\geq 45\eta^{1/2} k. 
\tag{E8}\label{M2}
\end{align}

\vspace{-2mm}
For the final time, we distinguish between three cases:
\vspace{-2mm}
\begin{itemize}
\item[(a)] $G[M,N\cup P\cup Q_1]$ contains a red edge;
\item[(b)] $G[M, N\cup P\cup Q_2]$ contains a blue edge;
\item[(c)] $G_1[M,N\cup P\cup Q_1]$ and $G_2[M, N\cup P\cup Q_2]$ each contain no edges.
\end{itemize}

\vspace{-2mm}
\subsection*{Case E.iii.b.iv.a: {\rm $G[M,N\cup P\cup Q_1]$ contains a red edge.}}

\vspace{-2mm}
Given the existence of a red edge in $G[M,N\cup P\cup Q_1,M]$, then the existence of a red matching on at least $14\eta^{1/2} k$ vertices in $G[P_2\cup Q_2,M]$ would give a red connected-matching on at least $\aI k$ vertices. Thus, we may, after discarding at most $7\eta^{1/2} k$ vertices from each of $M$ and $P_2\cup Q_2$, assume that all present  edges in $G[M,Q_2]$ are coloured exclusively blue and that there are no red edges in $G[M,P_2]$. 

\vspace{-1mm}
\begin{figure}[!h]
\centering
\includegraphics[width=64mm, page=38]{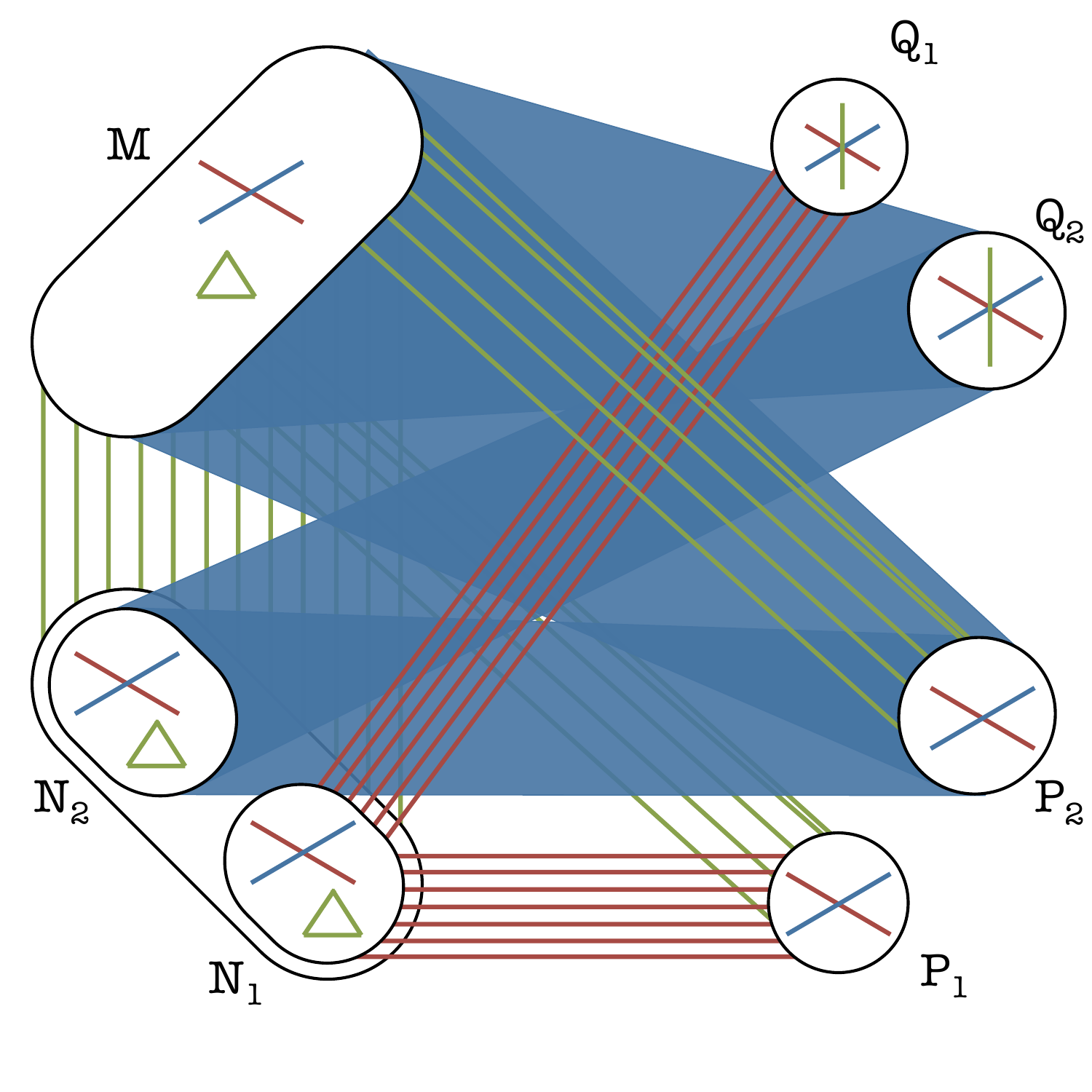}
\vspace{-5mm}\caption{Colouring immediately before Claim~\ref{claimLAST}.}
\label{beforeclaim}
\end{figure}

We then have 
\begin{equation}
\label{K1}
\left.
\begin{aligned}
\,\quad\quad\quad\quad\quad\quad\quad\quad\quad\quad|N_1|=|P_1\cup Q_1|&\geq(\half\aI-8\eta^{1/2})k, \quad\quad\quad\quad\quad\quad\quad\quad\\
|N_2|, \,|P_2\cup Q_2|&\geq(\half\aII-15\eta^{1/2})k.
\end{aligned}
\right\}\!
\end{equation}
The following pair of claims establish the coloured structure of $G[M\cup N,P\cup Q]$:

\begin{claim}
\label{claimLAST}
\hspace{-3mm} {\rm \bf a.} If $|Q_1|\geq 45\eta^{1/2}k$,  we may discard at most $43\eta^{1/2} k$ vertices from $N_1$, at most $43\eta^{1/2} k$ vertices from $N_2$, at most $16\eta^{1/2} k$ vertices from $M$, at most $59\eta^{1/2} k$ vertices from $P_1\cup Q_1$ and at most $27\eta^{1/2} k$ vertices from $P_2\cup Q_2$ such that, in what remains, there are no blue edges present in $G[M\cup N, P_1\cup Q_1]$ and no red edges present in $G[M\cup N,P_2\cup Q_2]$. 

{\rm \bf Claim~\ref{claimLAST}.b.} If $|Q_2|\geq 45\eta^{1/2}k$,  we may discard at most $45\eta^{1/2} k$ vertices from $N_1$, at most $18\eta^{1/2} k$ vertices from $N_2$, at most $18\eta^{1/2} k$ vertices from $M$, at most $18\eta^{1/2} k$ vertices from $P_1\cup Q_1$ and at most $27\eta^{1/2} k$ vertices from $P_2\cup Q_2$ such that, in what remains, there are no blue edges present in $G[M\cup N, P_1\cup Q_1]$ and no red edges present in $G[M\cup N,P_2\cup Q_2]$. 
\end{claim}

\begin{proof} (a)
We begin by considering $G[M\cup N_1,P_1\cup Q_1]$. Suppose there exists a blue matching $B_S$ on $32\eta^{1/2}k$ vertices in $G[M\cup N_1,P_1\cup Q_1]$. Observe
that, since $|N_2|,|P_2 \cup Q_2| \geq (\half\aII-15\eta^{1/2})k$, by Lemma~\ref{l:eleven}, there exists a blue 
connected-matching $B_L$ on at least $(\aII-32\eta^{1/2})k$ vertices in $G[N_2,P_2 \cup Q_2]$. Then, since all edges present in $G[M\cup N_2,
Q_2]$ are coloured blue and $G[N]$ is blue effectively-connected, $B_S$ and $B_L$ belong to the same blue component and thus form a blue connected-matching on at least~$\aII k$ vertices. Thus, after discarding at most $16\eta^{1/2}k$ vertices from each of $M\cup N_1$ and $P_1\cup Q_1$, we may assume that there are no blue edges present in $G[M\cup N_1,P_1\cup Q_1]$. Notice then, in particular, 
that all edges present in $G[M, Q_1]\cup G[N_1,P_1\cup Q_1]$ are coloured exclusively red. 

  \begin{figure}[!h]
\centering{
\mbox{\hspace{-2mm}{\includegraphics[width=64mm, page=102]{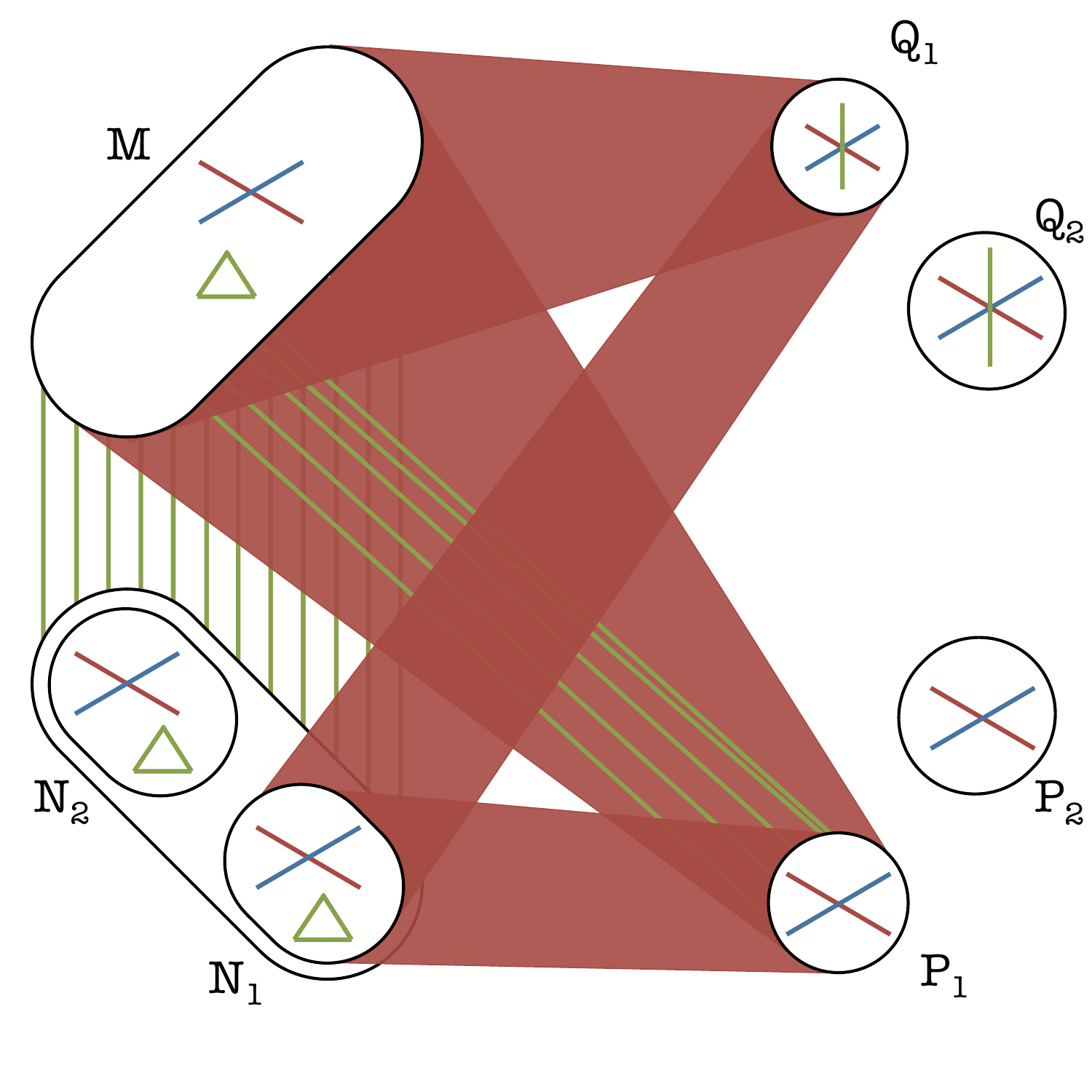}}\quad\quad\quad{\includegraphics[width=64mm, page=103]{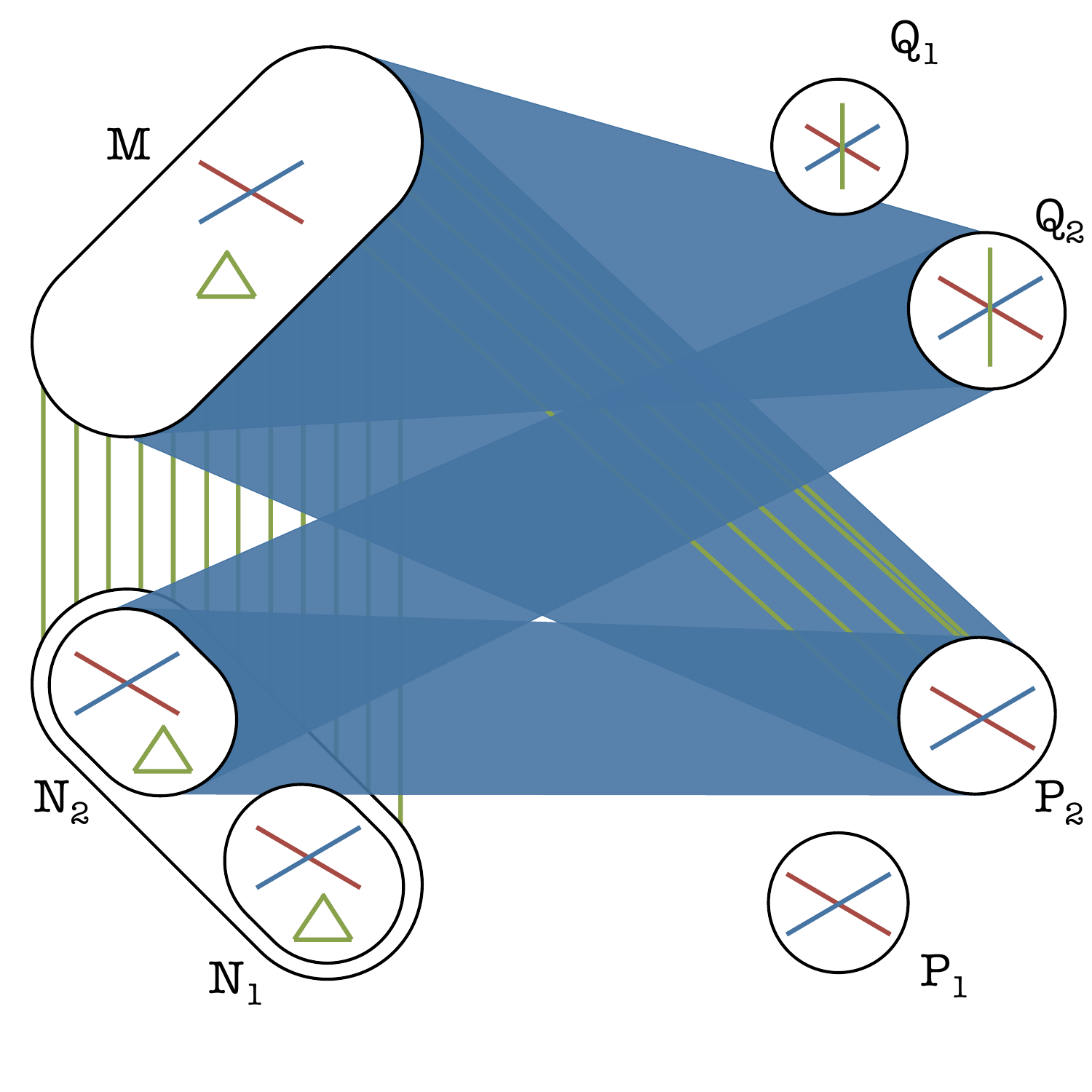}}}}\vspace{-4mm}
\caption{Colouring of the edges of $G[M\cup N_1,P_1\cup Q_1]$ in Claim~\ref{claimLAST}.a.}   
\end{figure}

Having discarding these vertices, we have 
\begin{align*}
|Q_1|&\geq 29\eta^{1/2}k, & |N_1|&\geq|P_1\cup Q_1|\geq(\half\aI-24\eta^{1/2})k.
\end{align*}

Now, suppose there exists a red matching $R_S$ on at least $54\eta^{1/2}k$ vertices in $G[N_1,P_2\cup Q_2]$. Since $|P_1\cup Q_1|\geq(\half\aI -24\eta^{1/2})k$ and $|Q_2|\geq 29\eta^{1/2} k$, there exists $\widetilde{Q}\subseteq Q_1$ such that $|P_1\cup \widetilde{Q}|\geq(\half\aI - 52\eta^{1/2})k$ and $|Q_1\backslash \widetilde{Q}|\geq 27\eta^{1/2} k$. Then, since $|N_1\backslash V(R_S)|,|P_1\cup \widetilde{Q}|\geq(\half\aI - 52\eta^{1/2})k$ and $|M|, |Q_1\backslash \widetilde{Q}|\geq 27\eta^{1/2} k$, by Lemma~\ref{l:eleven}, there exist red connected-matchings $R_L$ on at least $(\aI -106\eta^{1/2})k$ vertices in $G[N_1,P_1\cup Q_1]$ and $R_T$ on at least $52\eta^{1/2}k$ vertices in $G[Q_1,M]$ sharing no vertices with each other or $R_S$. Then, since all edges in $G[N_1,Q_1]$ are coloured red, $R_L, R_S$ and $R_T$ belong to the same red component and, thus, together, form a red connected-matching on at least $\aI k$ vertices. Therefore, after discarding at most $27\eta^{1/2}k$ vertices from each of $N_1$ and $P_2\cup Q_2$, we may assume that all edges in $G[N_1,P_2\cup Q_2]$ are coloured exclusively blue.

  \begin{figure}[!h]
\centering{
\mbox{\hspace{-2mm}{\includegraphics[width=64mm, page=104]{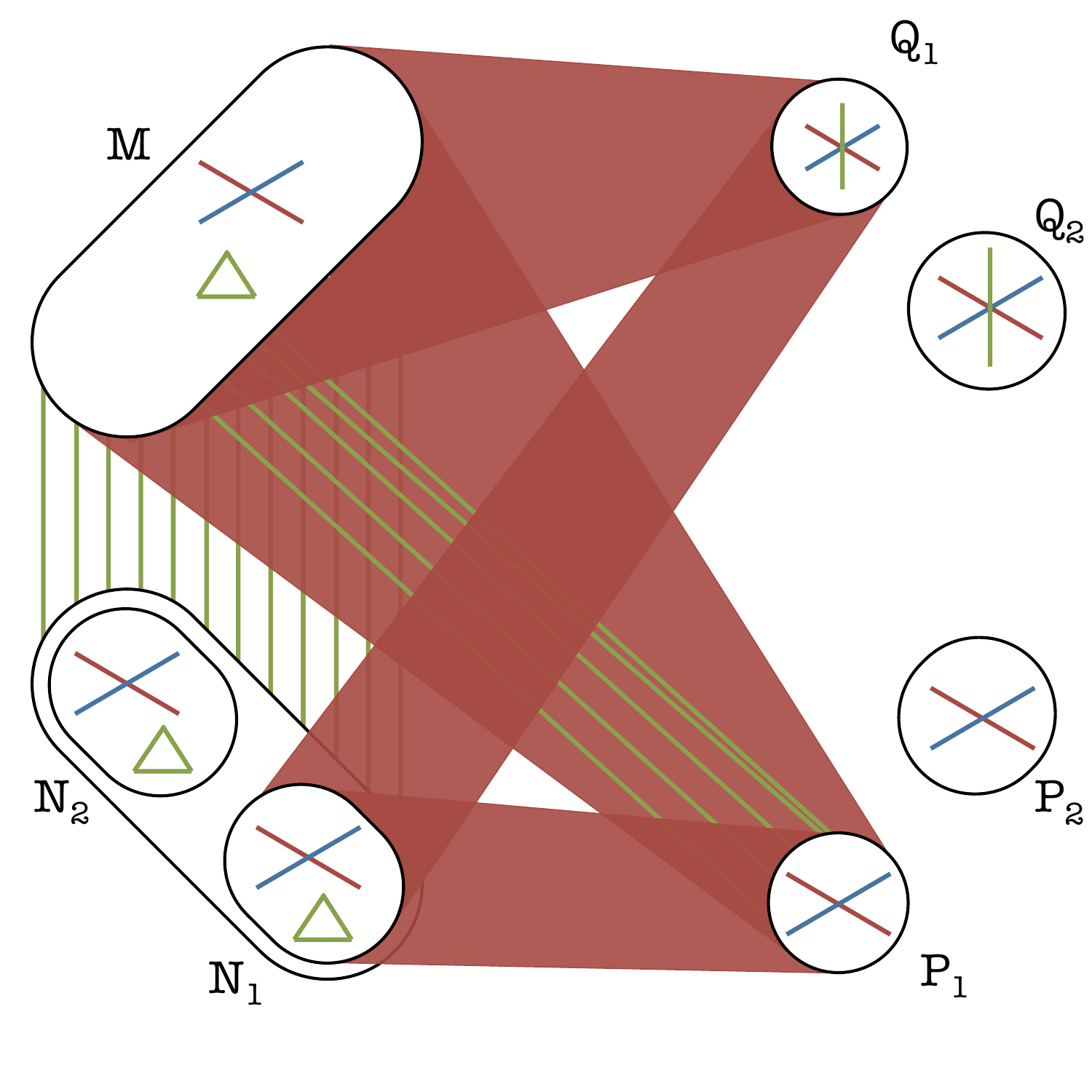}}\quad\quad\quad{\includegraphics[width=64mm, page=105]{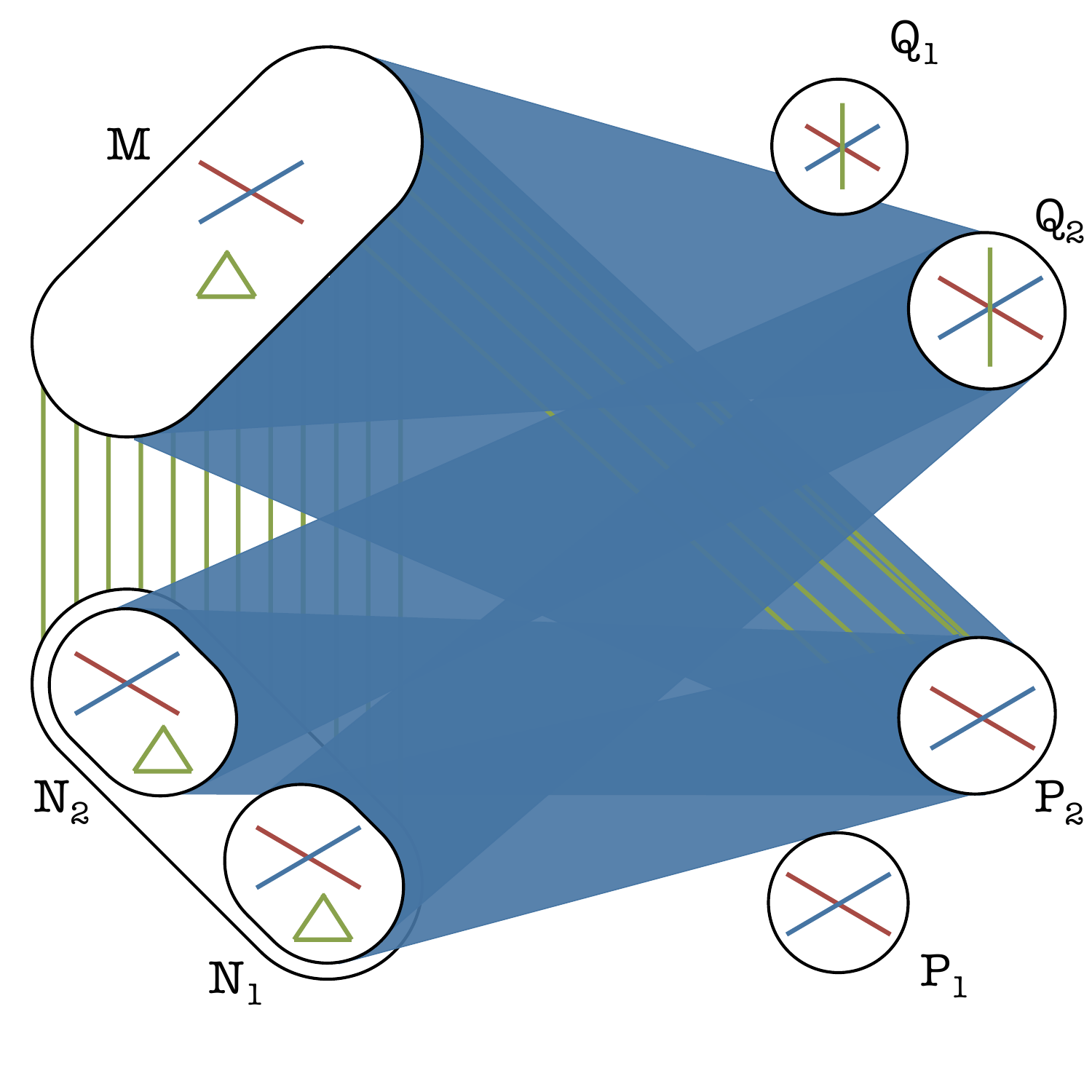}}}}\vspace{-5mm}
\caption{Colouring of the edges of $G[N_1,P_2\cup Q_2]$ in Claim~\ref{claimLAST}.a.}   
\end{figure}

We then have
\begin{align*}
|N_1|&\geq(\half\aI-51\eta^{1/2})k, &
|P_1\cup Q_1|&\geq(\half\aI-24\eta^{1/2})k,\\
|N_2|&\geq(\half\aII-15\eta^{1/2})k, &
|P_2\cup Q_2|&\geq(\half\aII-42\eta^{1/2})k.
\end{align*}
Suppose now that there exists a blue matching $B_U$ on $86\eta^{1/2}k$ vertices in $G[N_2,P_1\cup Q_1]$. Then, since $|N\backslash V(B_U)|,|P_2\cup Q_2|\geq (\half\aII-42\eta^{1/2})k$, by Lemma~\ref{l:eleven}, there exists a blue connected-matching $B_V$ on at least $(\aII -86\eta^{1/2})k$ vertices in $G[N\backslash V(B_U),P_2\cup Q_2]$. Therefore, $B_U\cup B_V$ forms a blue connected-matching on at least $\aII k$ vertices in $G[N,P\cup Q]$. Thus, after discarding at most $43\eta^{1/2}k$ vertices from each of $N_2$ and $P_1\cup Q_1$, we may assume that all edges present in $G[N_1,P_1\cup Q_1]$ are coloured exclusively red, thus completing the proof of Claim~\ref{claimLAST}.a.

\FloatBarrier
(b) Suppose that $G$ is coloured as in Figure~\ref{beforeclaim} and that $|Q_2|\geq 45\eta^{1/2}k$. We begin by considering $G[M\cup N,P_1\cup Q_1]$. Suppose there exists a blue matching $B_S$ on $36\eta^{1/2}k$ vertices in $G[M\cup N,P_1\cup Q_1]$, then we can obtain a blue connected-matching on at least $\aI k$ vertices as follows: Recalling~(\ref{K1}), we have $|P_2\cup Q_2|\geq(\half\aII  -15\eta^{1/2})k$. Then, since $|Q_2|\geq45\eta^{1/2}k$, there exists $\widetilde{Q}\subseteq Q_2$ such that $|P_2 \cup\widetilde{Q}|\geq (\half \aII - 34\eta^{1/2})k$ and $|Q_2\backslash \widetilde{Q}|\geq 18\eta^{1/2} k$. Then, we have $|N_2\backslash V(B_S)|,|P_2\cup \widetilde{Q}|\geq(\half\aII k - 34\eta^{1/2})k$ and $|M\backslash V(B_S)|,|Q_2\backslash \widetilde{Q}|\geq 18\eta^{1/2} k$. Thus, by Lemma~\ref{l:eleven}, there exist blue connected-matchings $B_L$ on at least $(\aII k-70\eta^{1/2})k$ vertices in $G[N\backslash V(B_S),P_2\cup\widetilde{Q}]$ and $B_T$ on at least $34\eta^{1/2}k$ vertices in $G[M\backslash V(B_S),Q_2\backslash \widetilde{Q}]$. Since $M$ and $N$ are each blue effectively-connected and all edges present in $G[M\cup N_2,P_2]$ are coloured blue, $B_L, B_S$ and $B_T$ belong to the same blue component in $G$ and, thus, together, form a blue connected-matching on at least $\aII k$ vertices. Therefore, after discarding at most $18\eta^{1/2} k$ vertices from each of $M\cup N$ and $P_1\cup Q_1$, we may assume that all edges present in $G[M\cup N,P_1\cup Q_1]$ are coloured exclusively red.

\vspace{-2mm}
\begin{figure}[!h]
\centering
\mbox{\hspace{-2mm}{\includegraphics[width=64mm, page=106]{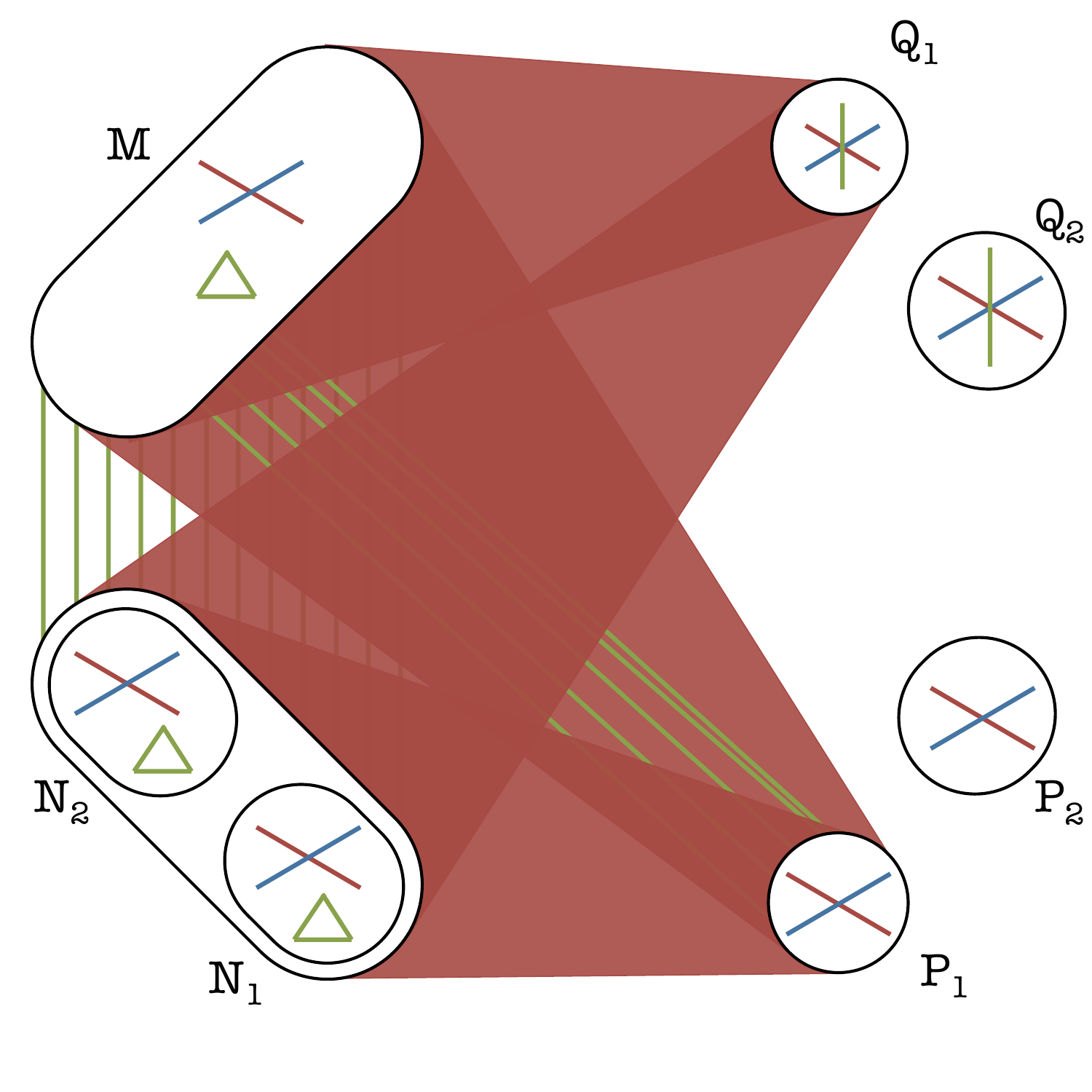}}\quad\quad\quad{\includegraphics[width=64mm, page=107]{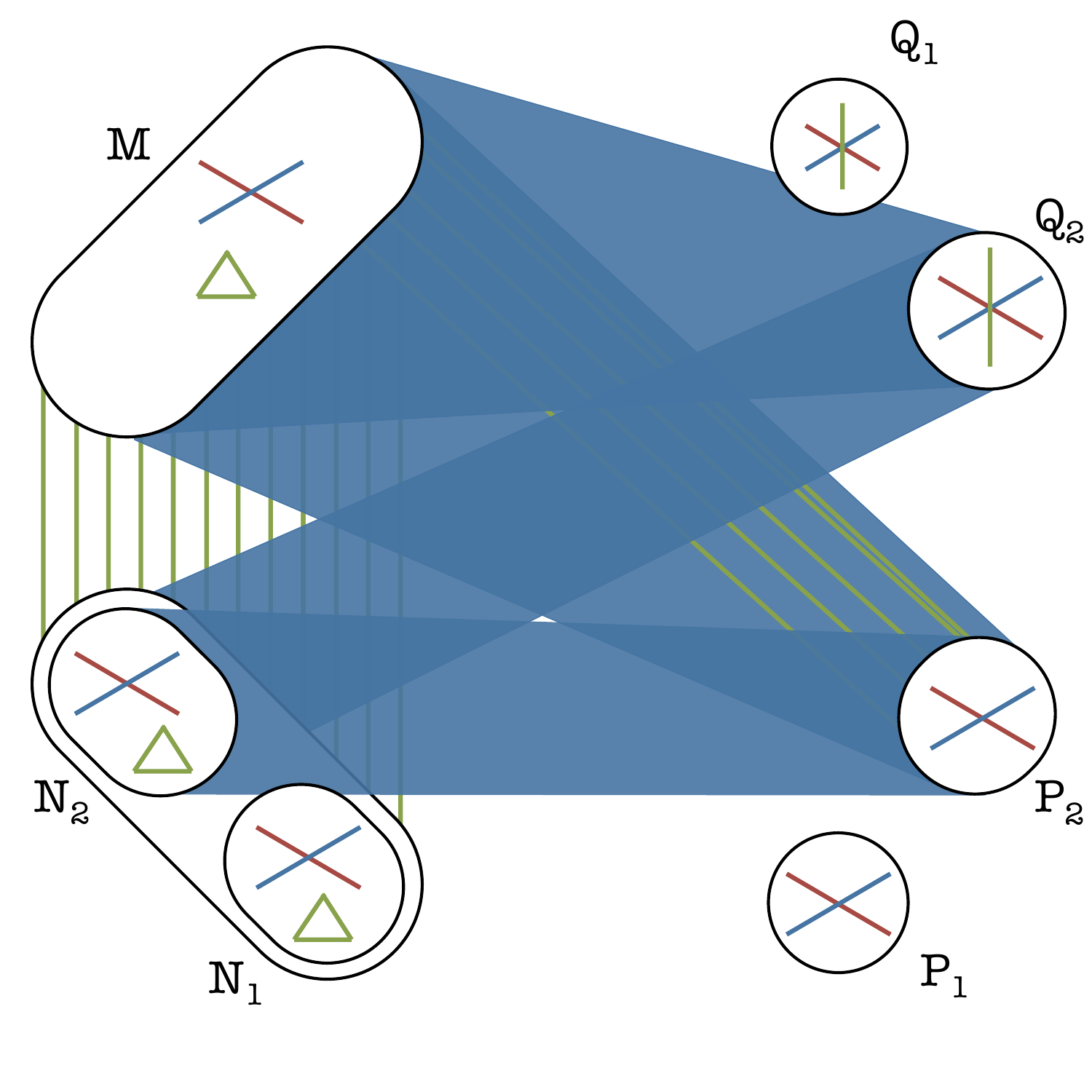}}}\vspace{-5mm}\caption{Colouring of $G[M\cup N,P_1\cup Q_1]$ in Claim~\ref{claimLAST}.b.}
  
\end{figure}

Recalling (\ref{K1}), we then have
\vspace{-1mm}
\begin{align*}
|N_1|&\geq(\half\aI-26\eta^{1/2})k, &
|P_1\cup Q_1|&\geq(\half\aI-26\eta^{1/2})k,\\
|N_2|&\geq(\half\aII-33\eta^{1/2})k, &
|P_2\cup Q_2|&\geq(\half\aII-15\eta^{1/2})k.
\end{align*}
Finally, suppose there exists a red matching $R_S$ on at least $54\eta^{1/2}k$ vertices in $G[N_1,P_2\cup Q_2]$. Then, $|N\backslash V(R_S)|,|P_1\cup Q_1|\geq (\half\aI-26\eta^{1/2})k$ so, by Lemma~\ref{l:eleven}, there exists a red connected-matching $R_L$ on at least $(\aI - 54\eta^{1/2})k$ vertices in $G[N\backslash V(R_S),P_1\cup Q_1]$. The red matchings $R_S$ and $R_L$ share no vertices and, since $G[N]$ is red effectively-connected, belong to the same red component of $G$, thus, together, they form a red connected-matching on at least $\aI k$ vertices. Therefore, after discarding at most $27\eta^{1/2}k$ vertices from each of $N_1$ and $P_2\cup Q_2$, we may assume that all edges in $G[N,P_2\cup Q_2]$ are coloured blue, thus completing the proof of Claim~\ref{claimLAST}.b.
\end{proof}

In summary, combining the two cases above, we may now assume that there are no blue edges present in $G[M\cup N,P_1\cup Q_1]$ and no red edges present in $G[M\cup N,P_2\cup Q_2]$. In particular, we may assume that all 
 all edges present in $G[N,P_1\cup Q_1]$ are coloured exclusively red  and that all edges in $G[N,P_2\cup Q_2]$ are coloured exclusively blue.

 \begin{figure}[!h]
\centering
\mbox{\hspace{-2mm}{\includegraphics[width=64mm, page=108]{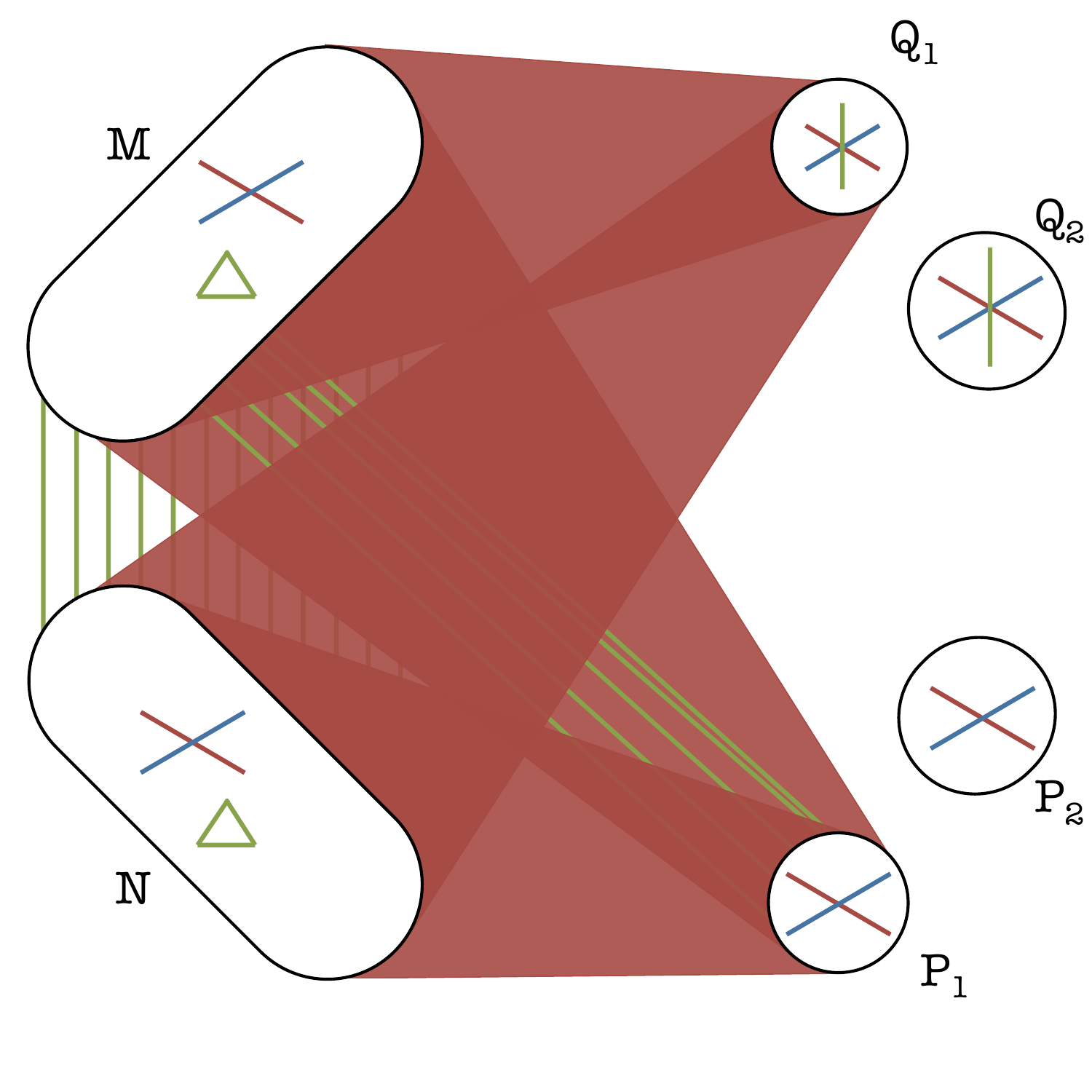}}\quad\quad\quad{\includegraphics[width=64mm, page=109]{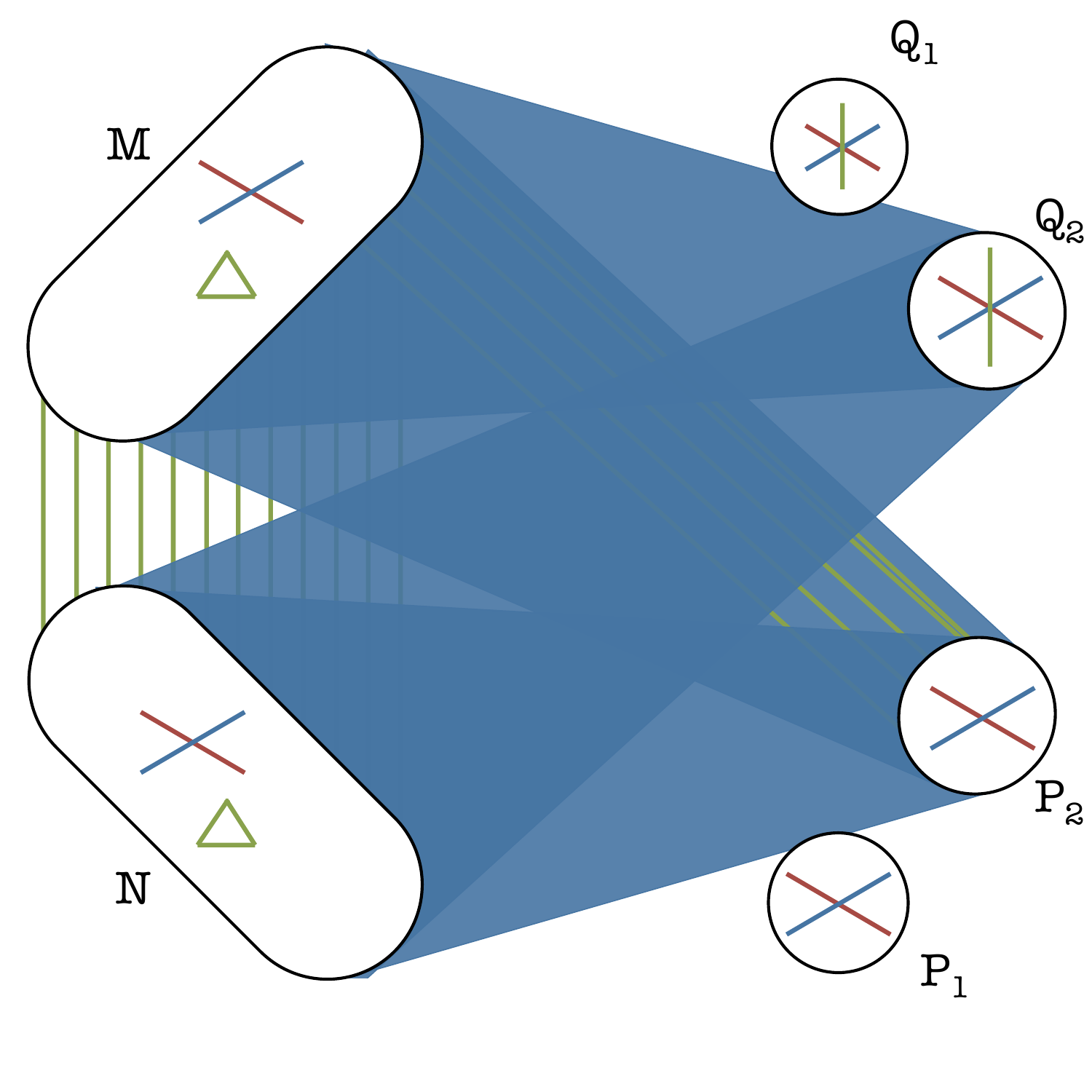}}}\vspace{-5mm}\caption{Colouring after Claim~\ref{claimLAST}.}
  
\end{figure}

 We then have 
\begin{equation}
\label{K2}
\left.
\begin{aligned}
\quad\quad\quad\,
|N_1|&\geq(\half\aI-54\eta^{1/2})k, \quad\quad\quad&
|P_1\cup Q_1|&\geq(\half\aI-67\eta^{1/2})k,\quad\quad\quad\\
|N_2|&\geq(\half\aII-58\eta^{1/2})k, &
|P_2\cup Q_2|&\geq(\half\aII-42\eta^{1/2})k.
\end{aligned}
\right\}\!
\end{equation}
Recalling (\ref{M1}), having discarded at most $86\eta^{1/2}k$ vertices from $N$ and  at most $25\eta^{1/2}k$ vertices from $M$, we have 
\begin{align*}
\,|M|&\geq\left(\max\{\tfrac{3}{4}\aI+\tfrac{1}{4}\aII,\half\aIII\}-31\eta^{1/2}\right)k, & |N|&\geq\left(\max\{\tfrac{3}{4}\aI+\tfrac{1}{4}\aII,\half\aIII\}-93\eta^{1/2}\right)k.
\end{align*}

Thus, provided that $\eta\leq(\aII/600)^2$, we have 
\begin{align*}
|N|& \geq|P_1\cup Q_1|+400\eta^{1/2} k, & |N|&\geq|P_2\cup Q_2|+400\eta^{1/2} k.
\end{align*}
 Therefore, if there existed either a red matching on $136\eta^{1/2}k$ vertices or a blue matching on $86\eta^{1/2}k$ vertices in $G[N]\cup G[N,M]$, then these could be used together with edges from $G[N,P\cup Q]$ to obtain a red connected-matching on at least $\aI k$ vertices or a blue connected-matching on at least $\aII k$ vertices. Thus, after discarding at most $111\eta^{1/2}k$ vertices from~$M$ and $333\eta^{1/2}k$ vertices from~$N$, we may assume that all edges present in $G[N]\cup G[M,N]$ are coloured exclusively green.
 
 \begin{figure}[!h]
\centering
\includegraphics[width=64mm, page=112]{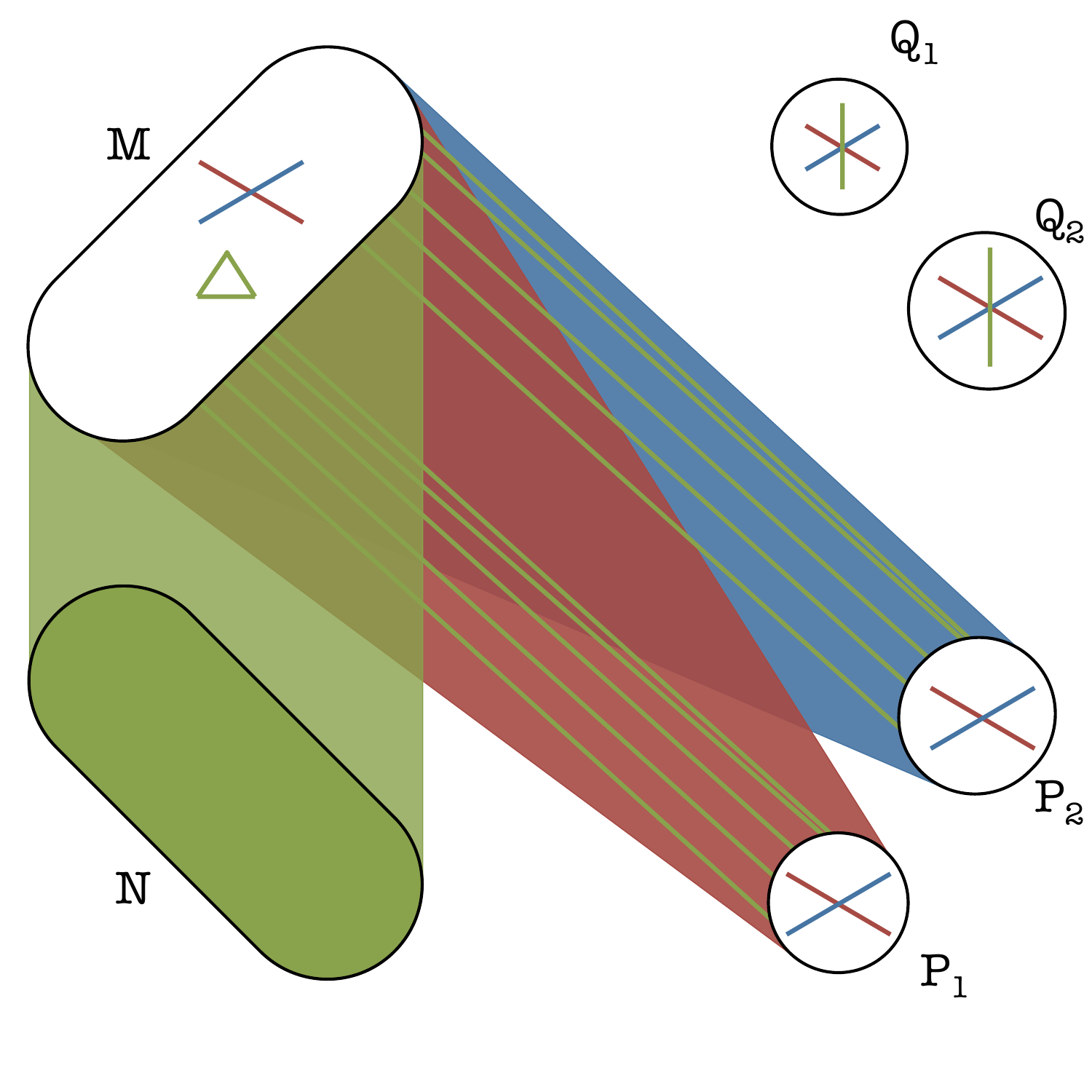}
\vspace{-5mm}\caption{The green graph after considering $G[N]\cup G[M,N]$.}
  
\end{figure}
 
 We then have
 \begin{align*}
|M|&\geq\left(\max\{\tfrac{3}{4}\aI+\tfrac{1}{4}\aII,\half\aIII\}-142\eta^{1/2}\right)k, \\ |N|&\geq\left(\max\{\tfrac{3}{4}\aI+\tfrac{1}{4}\aII,\half\aIII\}-426\eta^{1/2}\right)k.
\end{align*} 
Now, suppose there exists a green matching $G_{MP}$ on $1146\eta^{1/2}k$ vertices in $G[M,P]$. Letting $\widetilde{N}$ be any subset of $290\eta^{1/2}k$, by Lemma~\ref{dirac}, provided $k\geq1/\eta^2$, there exists a green connected-matching $G_N$ on at least $289\eta^{1/2}k$ vertices in $G[\widetilde{N}]$. We then have $|M\backslash V(G_{MP})|,|N\backslash V(G_N)|\geq (\half\aIII-716\eta^{1/2})k$ and, thus, by Lemma~\ref{l:eleven}, we have a green connected-matching $G_{MN}$ on at least $(\aIII-1434\eta^{1/2})k$ vertices in $G[M,N]$, which shares no vertices with $G_{MP}$ or $G_{N}$. Then, since all edges present in $G[M,N]$ are coloured green, together, $G_{MP}$, $G_{MN}$ and $G_{N}$ form a green connected-matching on at least $\aIII k$ vertices which, since all edges in $G[N]$ are green, is odd. Thus, after discarding at most $574 \eta^{1/2}k$ vertices from each of~$M$ and~$P$, we may assume that there are no green edges in $G[M, P\cup Q]$. In particular, since earlier we found that there could be no blue edges in $G[M,P_1]$ and no red edges in $G[M,P_2]$, we now know that all edges present in $G[M,P_1]$ are coloured exclusively red and all edges present in $G[M,P_2]$ are coloured exclusively blue.

 \begin{figure}[!h]
\centering
\includegraphics[width=64mm, page=113]{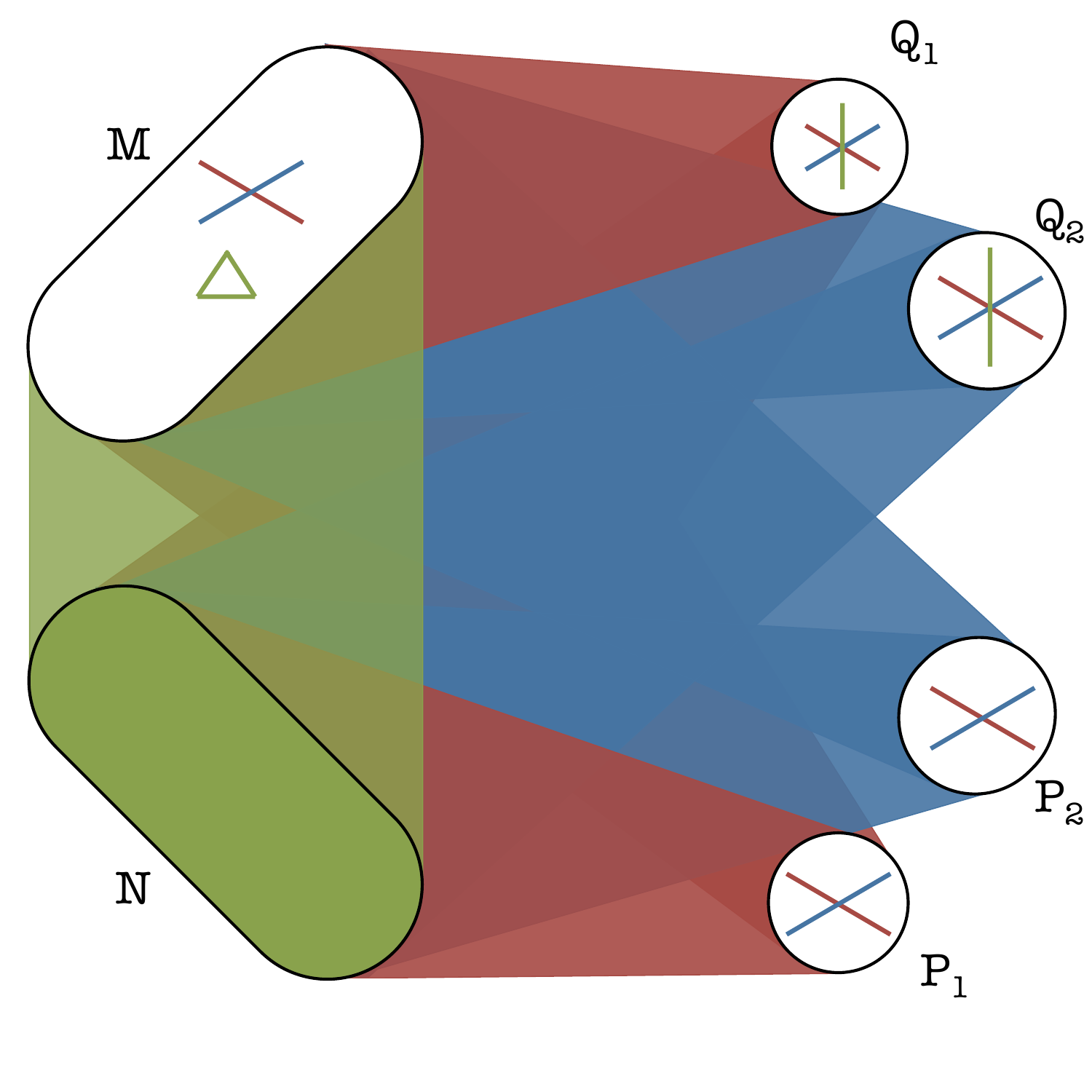}
\vspace{-5mm}\caption{Colouring of $G[M,P]$.}
  
\end{figure}

In summary, having discarded these vertices, we may assume that all edges present in $G[M\cup N,P_1\cup Q_1]$ are coloured exclusively red, that all edges present in $G[M\cup N,P_2\cup Q_2]$ are coloured exclusively blue, that all edges in $G[N]\cup G[M,N]$ are coloured exclusively green and that we have
\begin{align*}
|M|&\geq\left(\max\{\tfrac{3}{4}\aI+\tfrac{1}{4}\aII,\half\aIII\}-716\eta^{1/2}\right)k,
&|P_1\cup Q_1|\geq (\half\aI-642\eta^{1/2})k,\\
|N|&\geq\left(\max\{\tfrac{3}{4}\aI+\tfrac{1}{4}\aII,\half\aIII\}-426\eta^{1/2}\right)k,
&|P_2\cup Q_2|\geq (\half\aI-616\eta^{1/2})k.
\end{align*}
Notice that, provided that $\eta\leq(\aII/5000)^2$, we have $|M|\geq|P_1\cup Q_1|+2200\eta^{1/2}k$. Thus, there cannot exist in $G[M]$ a red matching on $1286\eta^{1/2}k$ vertices or a blue matching on $1234\eta^{1/2}k$ vertices. Therefore, after discarding at most $2520\eta^{1/2}k$ vertices from $M$, we may assume that all edges present in $G[M]$ are coloured exclusively green and that $$|M\cup N|\geq(\aIII-3662\eta^{1/2})k.$$
In summary, we know that all edges present in $G[M\cup N, P_1\cup Q_1]$ are coloured exclusively red, all edges present $G[M\cup N, P_2\cup Q_2]$ are coloured exclusively blue and all edges in $G[M\cup N]$ are coloured exclusively green.
Thus, we have found, as a subgraph of~$G$, a graph in $$\cK\left((\half\aI-700\eta^{1/2})k, (\half\aI-700\eta^{1/2})k, (\aIII-4000\eta^{1/2})k, 4\eta^4 k \right),$$ 
thus completing Case E.iii.b.iv.a.

\subsection*{Case E.iii.b.iv.b: {\rm $G[M,N\cup P\cup Q_2]$ contains a blue edge.}}

In this case, following similar steps as in Case E.iii.b.iv.a will result in either a red connected-matching on at least $\aI k$ vertices, a blue connected-matching on at least $\aII k$ vertices, a green odd connected-matching on at least $\aIII k$ vertices or a subgraph in $$\cK\left((\half\aI-700\eta^{1/2})k, (\half\aI-700\eta^{1/2})k, (\aIII-4000\eta^{1/2})k, 2\eta^4 \right),$$ 
thus completing Case E.iii.b.iv.a.

\subsection*{Case E.iii.b.iv.c: {\rm $G[M,N\cup P\cup Q_1]$ contains no red edges.}}

\vspace{-5mm}
\hspace{45mm} {\large and $G[M,N\cup P\cup Q_2]$ contains no blue edges.}

Recall that we that we have a decomposition $V(G)$ into $M\cup N\cup P\cup Q$ and a refinement into $M\cup N_1\cup N_2 \cup P_1\cup P_2 \cup Q_1\cup Q_2$, such that
\begin{itemize}
\labitem{E3}{E3-c} there are no green edges in $G[N,P]$, $G[M,Q]$, $G[N,Q]$, $G[P,Q]$ or $G[P]$;
\labitem{E6a}{E6a-c} $G[M]$, $G[N]$ and $G[P]$ each have a single red and a single blue effective-component;
\labitem{E6b}{E6b-c}  $G[N_1,P_1\cup Q_1]$ contains a red matching utilising every vertex in $G[N_1\cup P_1\cup Q_1]$;
\labitem{E6c}{E6c-c}  all edges present in $G[N_2,P_2\cup Q_2]$ are coloured exclusively blue.
\end{itemize}
This decomposition also satisfies 
\begin{equation}
\label{E7-c}
\tag{E7}
\left.
\begin{aligned}
\!\!\!\!\! |M|&\geq\big(\!\max\{\tfrac{3}{4}\aI+\tfrac{1}{4}\aII,\half\aIII\}-6\eta^{1/2}\big)k, 
& \,|N_1|=|P_1\cup Q_1|&\geq (\half\aI-8\eta^{1/2})k,\\
\!\!\!\!\! |N|&\geq\big(\!\max\{\tfrac{3}{4}\aI+\tfrac{1}{4}\aII,\half\aIII\}-7\eta^{1/2}\big)k, 
& |N_2|, \, |P_2\cup Q_2| &\geq (\half\aII-8\eta^{1/2})k.
\end{aligned}
\,\,\, \right\}\!
\end{equation}

\vspace{-10mm}
\begin{align}
|P|&\geq 90\eta^{1/2} k, & \max\{|Q_1|,|Q_2|\}&\geq 45\eta^{1/2} k. 
\tag{E8}\label{E8-c}
\end{align}
Additionally, in this case, we may assume that 
\begin{itemize}
\labitem{E9a}{E9a-c}  all edges present in $G[Q_2,M]$ are coloured exclusively red;
\labitem{E9b}{E9b-c} all edges present in $G[Q_1,M]$ are coloured exclusively blue; and
\labitem{E9c}{E9c-c} all edge present in $G[M,P_1\cup P_2]\cup G[M,N]$ are green.
\end{itemize}

\vspace{2mm}
   \begin{figure}[!h]
\centering{
\mbox{\hspace{-2mm}{\includegraphics[width=64mm, page=115]{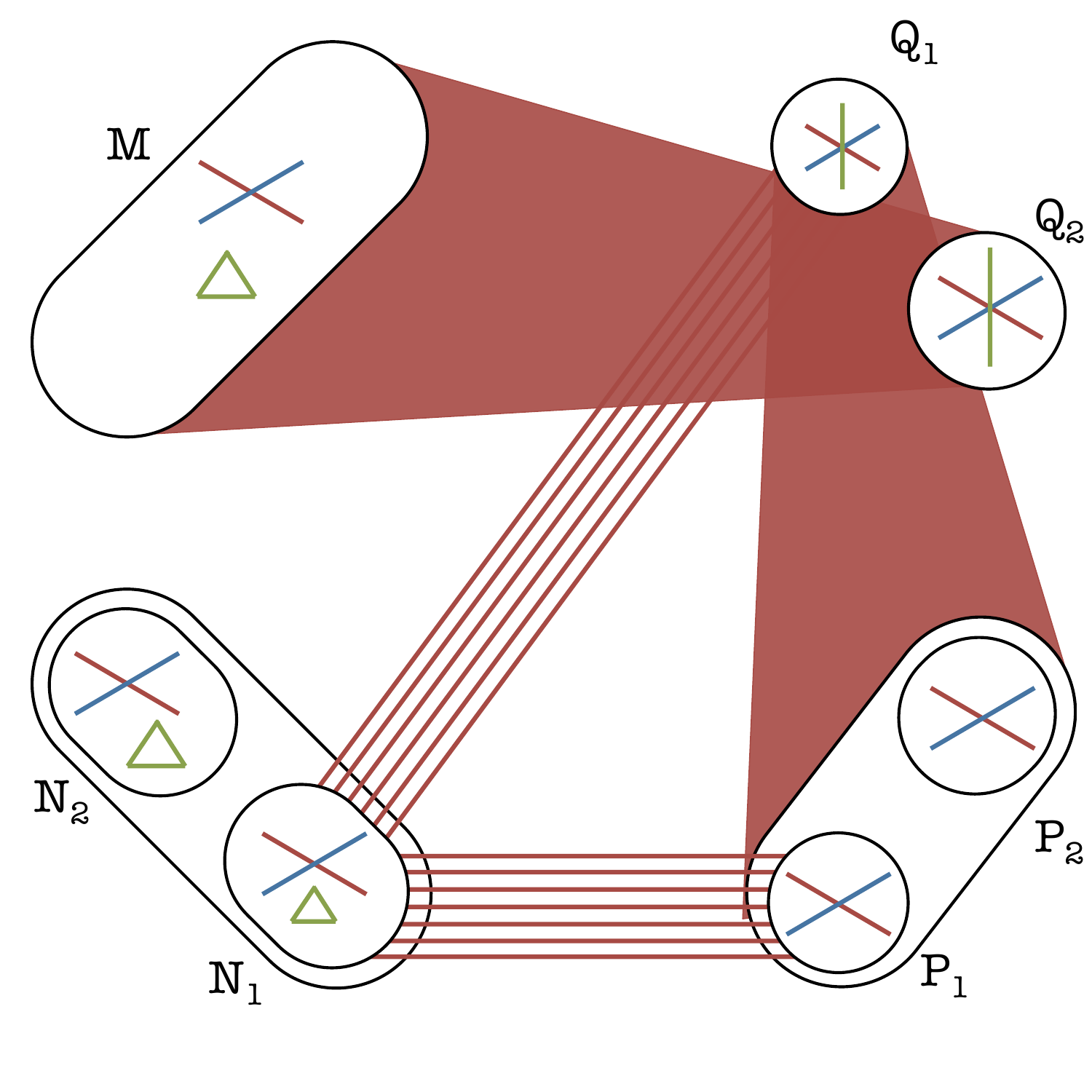}}\quad{\includegraphics[width=64mm, page=116]{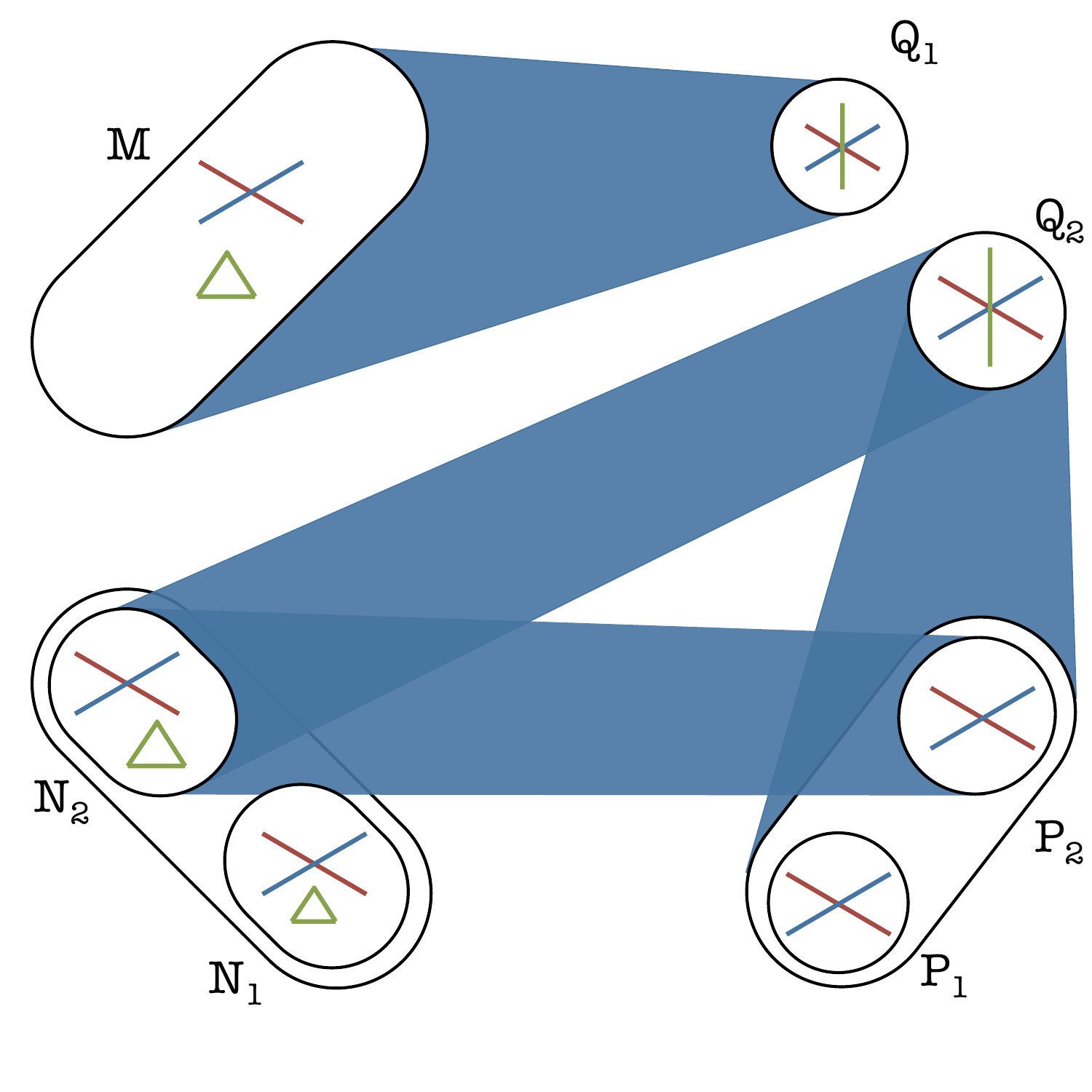}}}}\vspace{-5mm}\caption{Initial red and blue graphs in Case E.iii.b.iv.c. }
\label{f269}\end{figure}

\begin{figure}[!h]
\centering
\includegraphics[width=64mm, page=117]{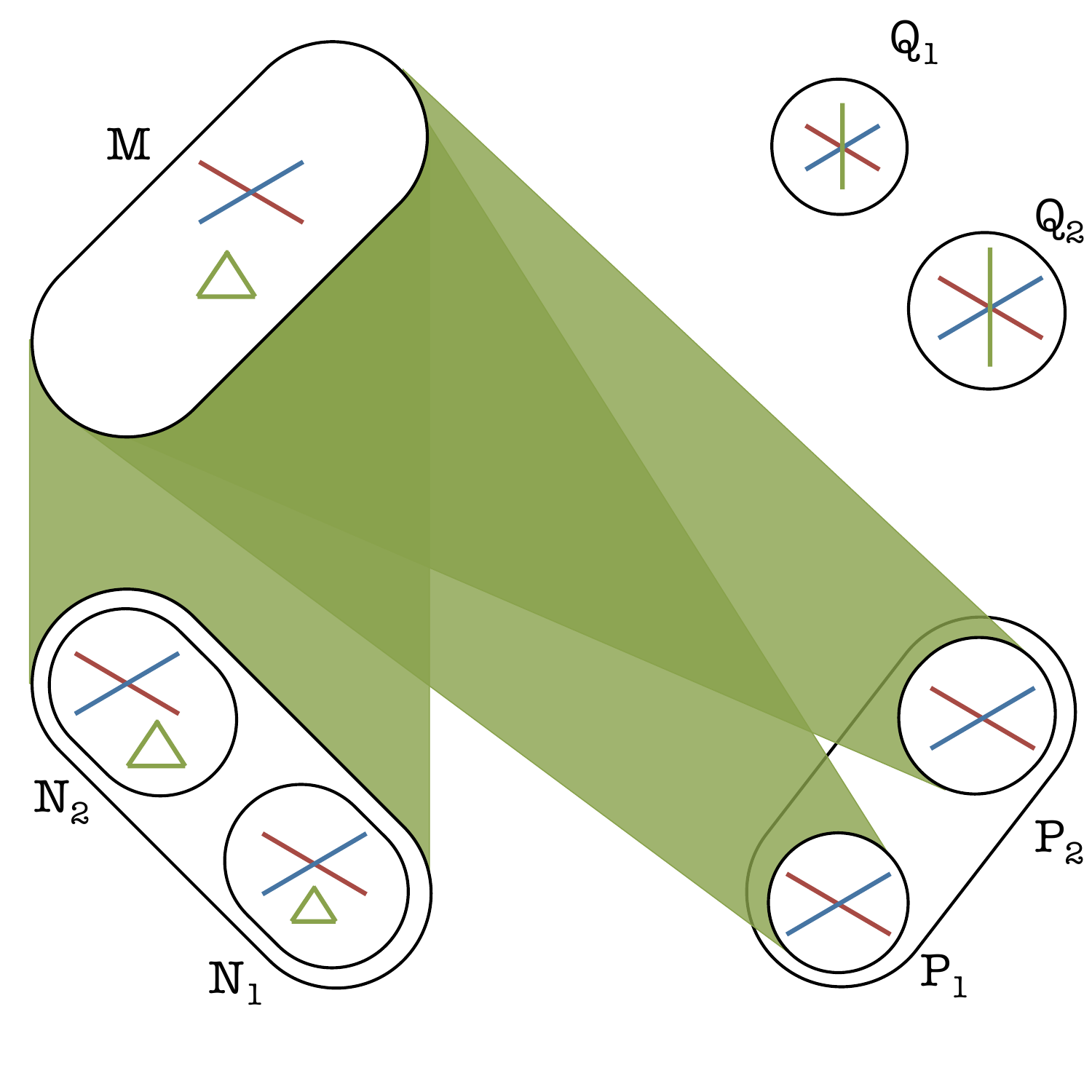}
\vspace{-5mm}\caption{Initial green graph in Case E.iii.b.iv.c. }\label{f270f}
\end{figure}

\FloatBarrier
We begin by proving the following claim which concerns the structure of the red and blue graphs:
\begin{claim}
\label{claimLASTplus2}
\hspace{-3mm} {\rm \bf a.} If $|Q_1|\leq 38\eta^{1/2}k$, we may discard at most $71\eta^{1/2} k$ vertices from $N$, at most $12\eta^{1/2} k$ vertices from $P_1$, at most $59\eta^{1/2} k$ vertices from $P_2$ and at most $38\eta^{1/2}k$ vertices from~$Q_1$ so that, in what remains, all edges present in $G[N,P_1\cup Q_1]$ are coloured exclusively red and all edges present in $G[N,P_2\cup Q_2]$ are coloured exclusively blue.

{\rm \bf Claim~\ref{claimLAST}.b.}  If $|Q_1|\geq 38\eta^{1/2}k$, we may discard at most $57\eta^{1/2} k$ vertices from $N$, at most $19\eta^{1/2} k$ vertices from $P_1$, at most $38\eta^{1/2} k$ vertices from $P_2$ and at most $30\eta^{1/2}k$ vertices from~$Q_1$  so that, in what remains, all edges present in $G[N,P_1\cup Q_1]$ are coloured exclusively red and all edges present in $G[N,P_2\cup Q_2]$ are coloured exclusively blue. 
\end{claim}

\begin{proof}
(a) Observe that, since $|Q_1|\leq 38\eta^{1/2}k<45\eta^{1/2}k$, by (\ref{E8-c}), we have $|Q_2|\geq 45\eta^{1/2}k$. Considering the red graph, we show that all edges present in $G[N\cup P,Q_2]$ are blue as follows: Given (\ref{E9a-c}), since $|M|,|Q_2|\geq 45\eta^{1/2} k$, by Lemma~\ref{l:eleven}, there exists a red connected-matching $R_S$ on at least $88\eta^{1/2}k$ vertices in $G[M,Q_2]$. Then, since $|N|=|P_1\cup Q_1|\geq(\half\aI-8\eta^{1/2})k$, recalling (\ref{E6a-c}), if $G[N\cup P,Q_2]$ contains a red edge, then $R\cup R_S$ forms a red connected-matching on at least $\aI k$ vertices. Thus, recalling~(\ref{E3-c}), all edges present in $G[N\cup P,Q_2]$ are coloured exclusively blue (see Figure~\ref{f271f}).

\clearpage
{~}
\vspace{-7mm}
\begin{figure}[!h]
\centering{
\mbox{\hspace{-2mm}{\includegraphics[width=64mm, page=84]{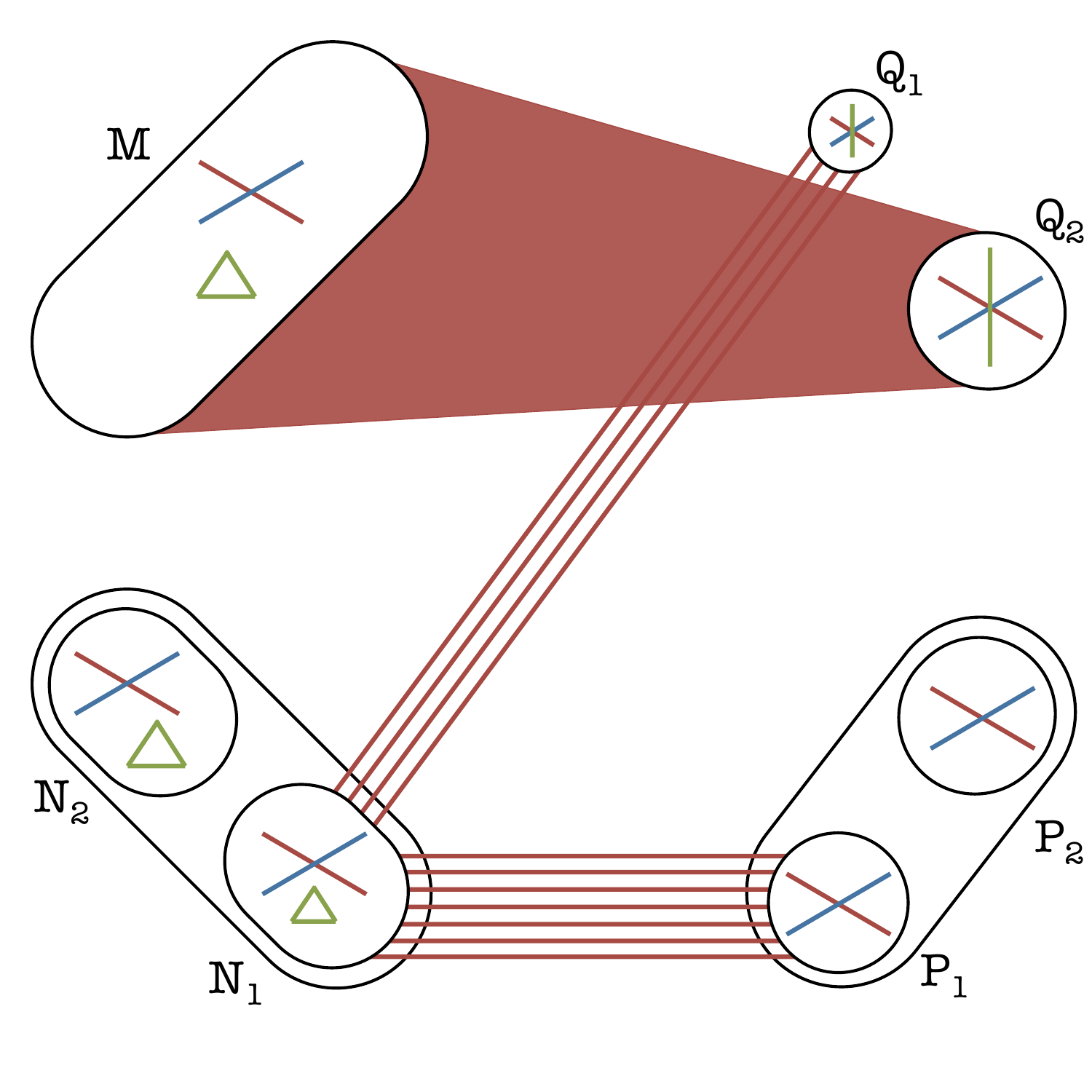}}\quad\quad\quad{\includegraphics[width=64mm, page=85]{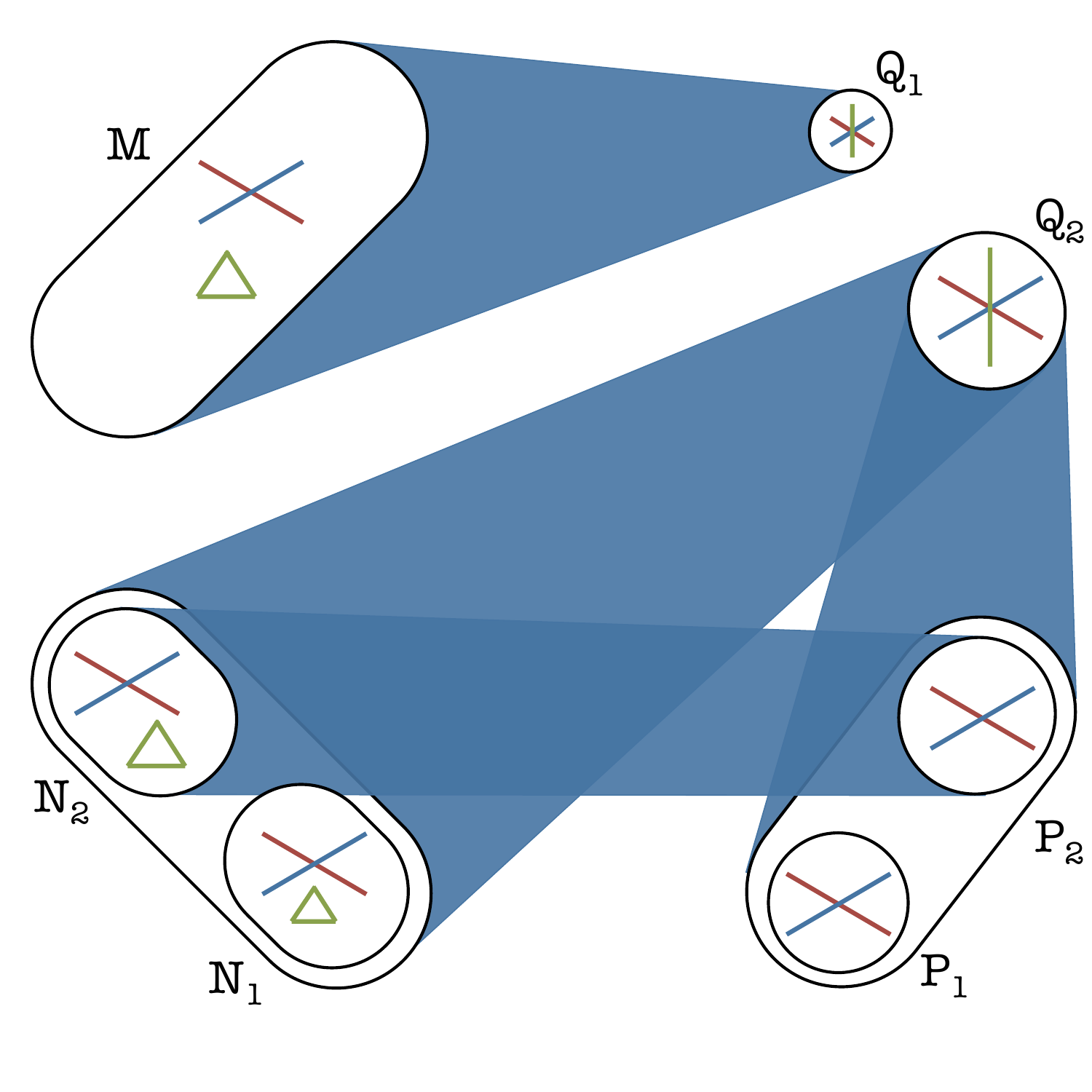}}}}\vspace{-5mm}
\caption{Red and blue subgraphs of $G$, provided $|Q_1|\leq38\eta^{1/2}k$.} \label{f271f}
\end{figure}

\vspace{-3mm}
  Now, suppose there exists a blue matching $B_T$ on $24\eta^{1/2}k$ vertices in $G[N,P_1]$ (see Figure~\ref{f272f}). Then, since $|Q_2|\geq 45\eta^{1/2}k$, there exist subsets $\widetilde{N}_2\subseteq N_2\backslash V(B_T)$ and $\widetilde{Q}_2\subseteq Q_2$ such that $|\widetilde{N}_2|=|P_2\cup \widetilde{Q}_2|\geq (\half\aII-20\eta^{1/2})k$ and 
$|Q_2\backslash \widetilde{Q}_2|\geq 11\eta^{1/2} k$. By Lemma~\ref{l:eleven}, there exist blue connected-matchings $B_U$ on at least $(\aII-42\eta^{1/2})k$ vertices in $G[N_2\backslash V(B_T),P_2\cup \widetilde{Q}_2]$ and $B_V$ on at least $20\eta^{1/2}k$ vertices in $G[N_1\backslash V(B_T),Q_2\backslash \widetilde{Q}_2]$. Since all edges present in $G[N,Q_2]$ are coloured blue, $B_T$, $B_U$ and $B_V$ belong to the same blue component and together, form a blue connected-matching on at least $\aII k$ vertices. Therefore, there can not exist such a matching as $B_T$ and, after discarding at most $12\eta^{1/2}k$ vertices from each of $N$ and $P_1$, we may assume that all edges present in $G[N,P_1]$ are coloured exclusively red.

  \begin{figure}[!h]
\centering{
\mbox{\hspace{-2mm}{\includegraphics[width=64mm, page=86]{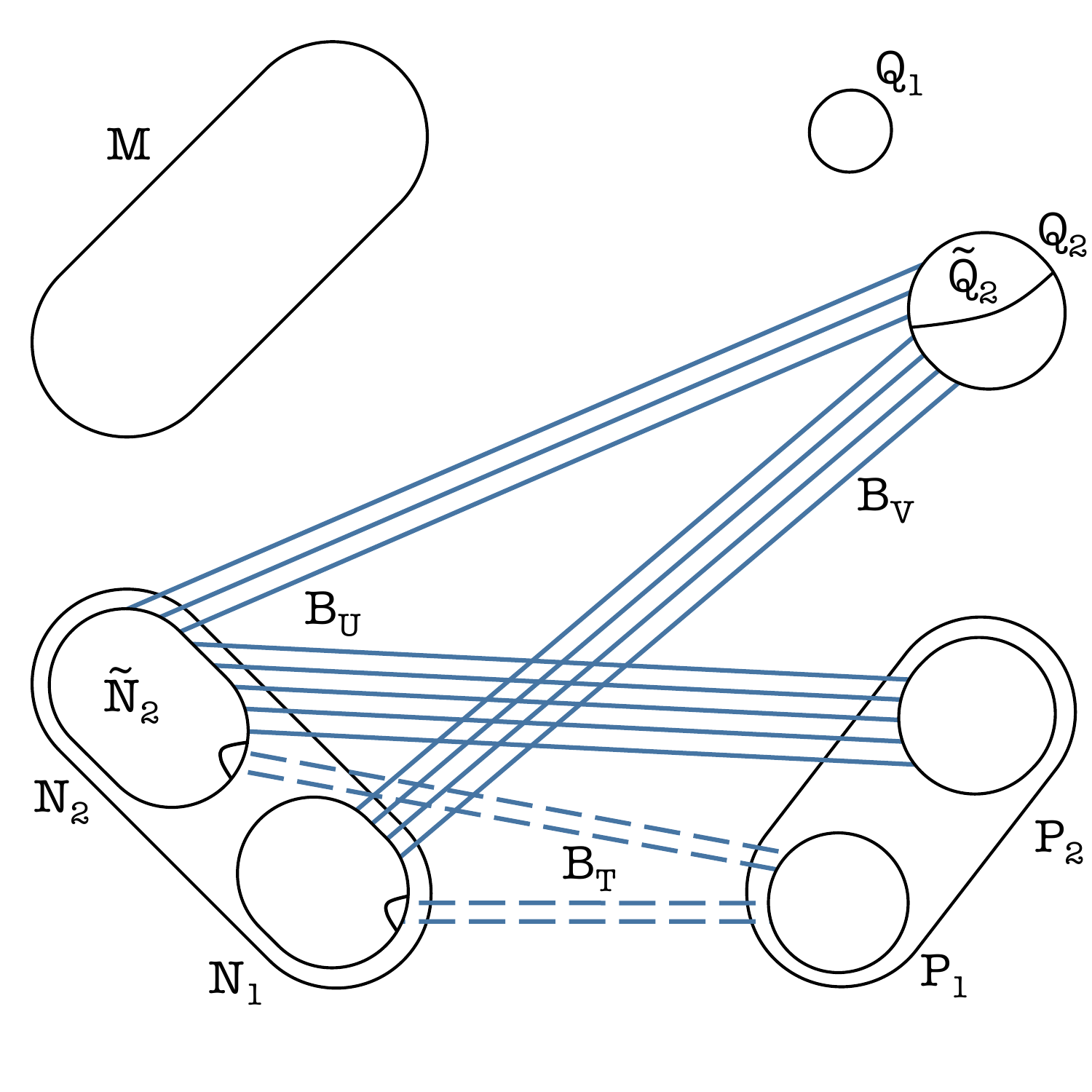}}\quad\quad\quad{\includegraphics[width=64mm, page=87]{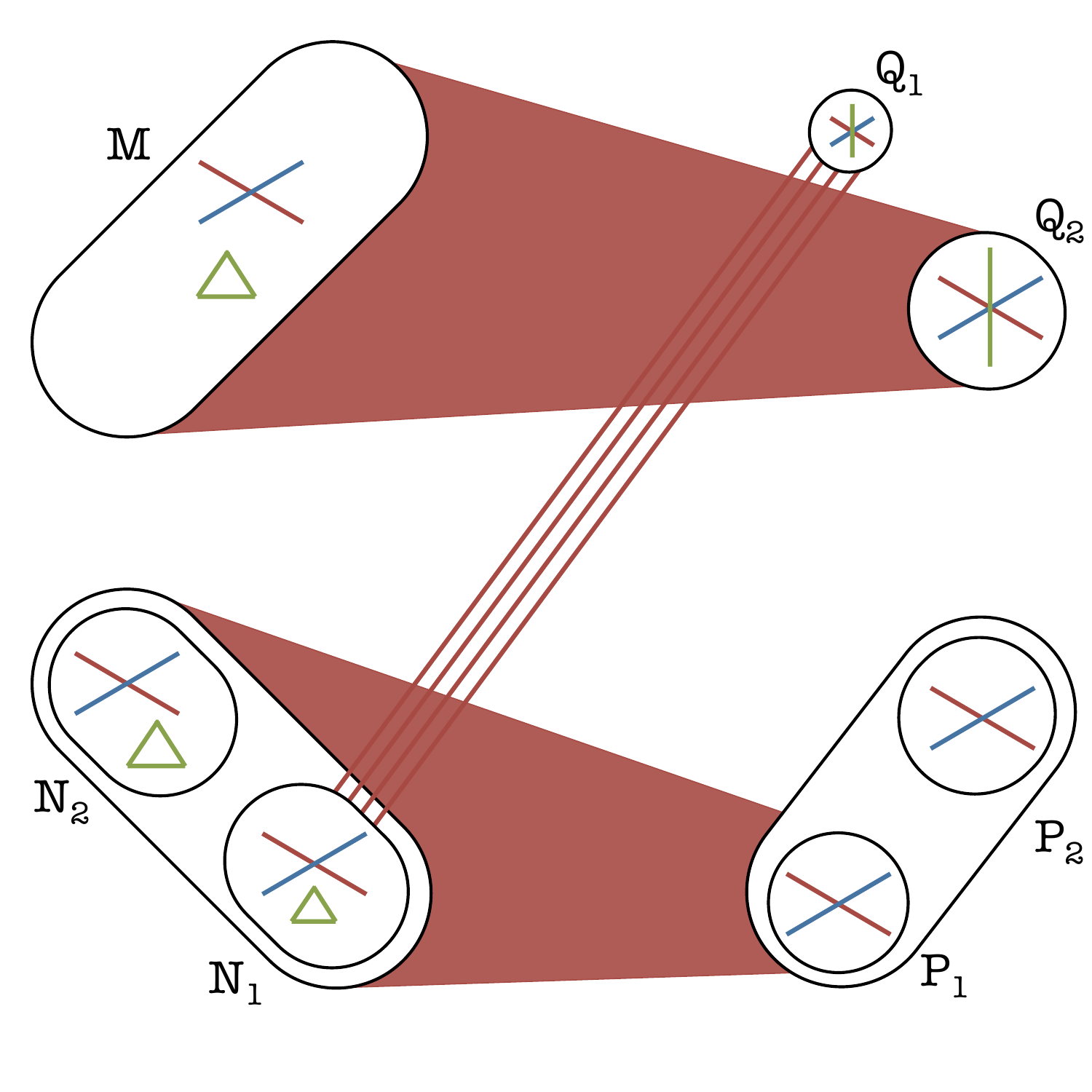}}}}\vspace{-5mm}
\caption{$B_T$ and the resultant colouring of the edges of $G[N,P_2]$.} \label{f272f}
\end{figure}

Having discarded these vertices, we have
\begin{equation}
\left.
\label{Mb1}
\begin{aligned}
\!\!\!\!\! |M|&\geq\big(\!\max\{\tfrac{3}{4}\aI+\tfrac{1}{4}\aII,\half\aIII\}-6\eta^{1/2}\big)k, 
& \,|N_1|,|P_1\cup Q_1|&\geq (\half\aI-20\eta^{1/2})k,\,\,\\
\!\!\!\!\! |N|&\geq\big(\!\max\{\tfrac{3}{4}\aI+\tfrac{1}{4}\aII,\half\aIII\}-19\eta^{1/2}\big)k, 
&  |P_2\cup Q_2| &\geq (\half\aII-8\eta^{1/2})k,\\
&& |N_2|&\geq (\half\aII-20\eta^{1/2})k. \quad
\end{aligned}
\right\}\!\!
\end{equation}
Since $|Q_1|\leq38\eta^{1/2}k$ and $\eta<(\aI/1000)^2$, by~(\ref{E4b-z}) and~(\ref{Mb1}), we have 
\begin{align*}
|N|&\geq|P_1|+200\eta^{1/2}k, & |P_1|&\geq (\half\aI-58\eta^{1/2})k.
\end{align*}

Thus, if there existed a red matching $R_S$ on $118\eta^{1/2}k$ vertices in $G[N,P_2]$, we could obtain a red connected-matching on at least $\aI k$. Indeed, since $|N\backslash V(R_S)|,|P_1|\geq (\half\aI-58\eta^{1/2})k$, by Lemma~\ref{l:eleven}, there exists a red connected-matching $R_L$ on $118\eta^{1/2}k$ vertices in $G[N\backslash V(R_S),P_1]$. Then, since $G[N]$ has a single red effective-component, $R_L$ and $R_S$ belong to the same red component of $G$ and, therefore, together, form a red connected-matching on at least $\aI k$ vertices. Thus, after discarding at most $59\eta^{1/2}k$ vertices from each of $N$ and $P_2$, we may assume that all edges present in $G[P_2,N]$ are coloured exclusively blue. 
  Finally, we discard all vertices from $Q_1$. Having done so, we have $P_1=P_1\cup Q_1$ and know that all edges present in $G[N,P_1\cup Q_1]$ are coloured exclusively red, that all edges present in $G[N,P_2\cup Q_2]$ are coloured exclusively blue.
 
   \begin{figure}[!h]
\centering{
\mbox{\hspace{-2mm}{\includegraphics[width=64mm, page=88]{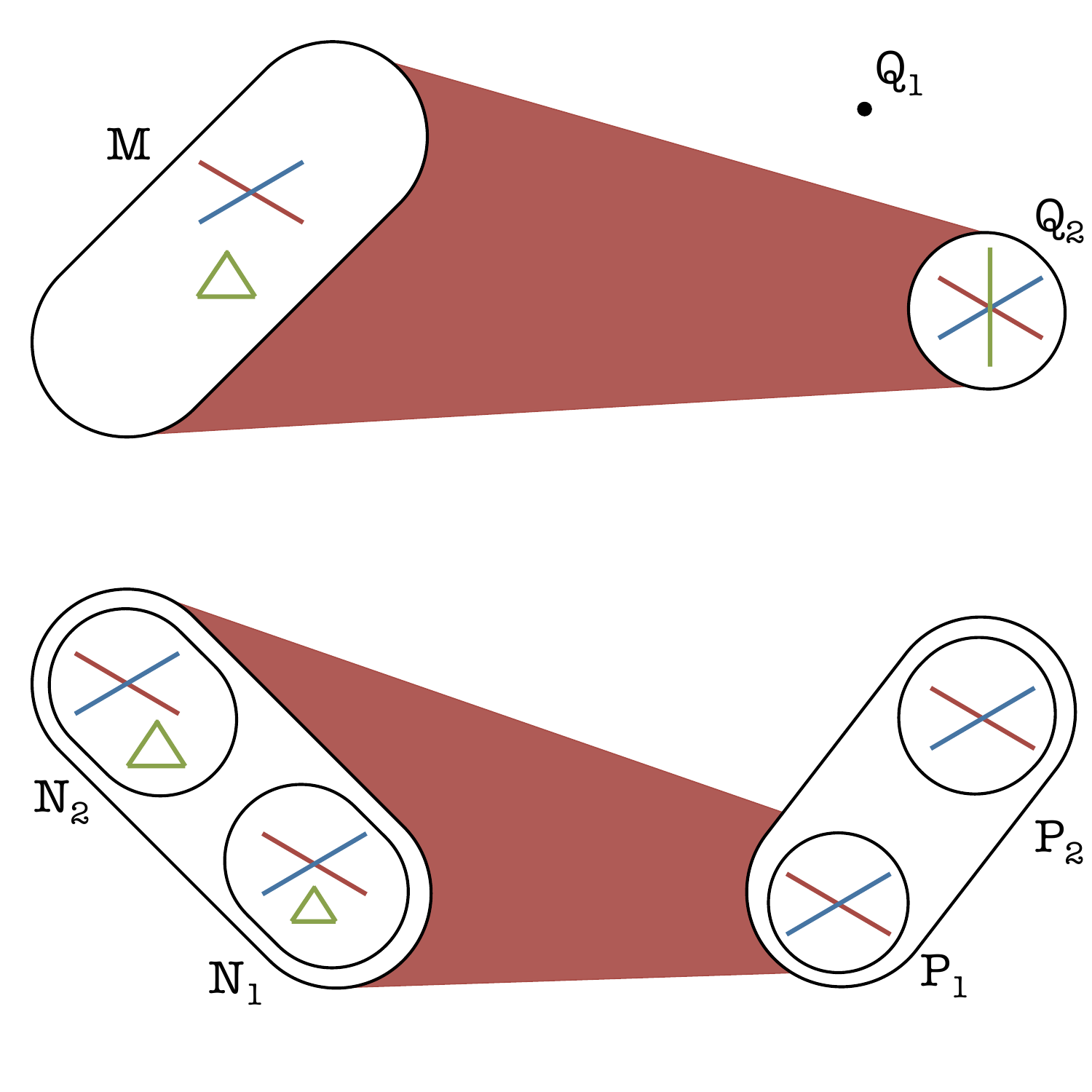}}\quad\quad\quad{\includegraphics[width=64mm, page=89]{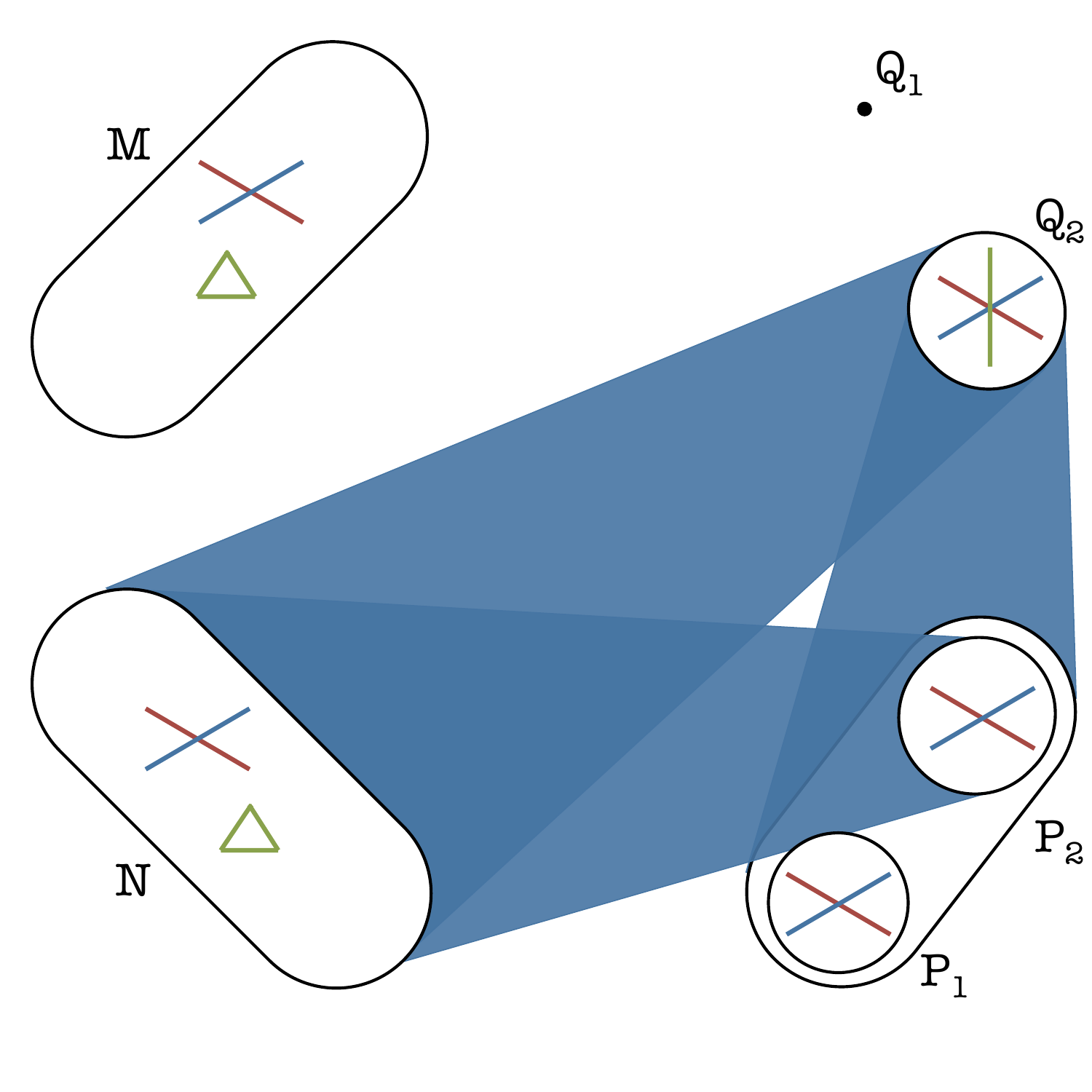}}}}\vspace{-5mm}
\caption{The final red and blue graphs in Claim~\ref{claimLASTplus2}.a.}\end{figure}

We then have
\begin{equation}
\label{M8}
\left.
\begin{aligned}
\,\,\,
 |M|&\geq\big(\!\max\{\tfrac{3}{4}\aI+\tfrac{1}{4}\aII,\half\aIII\}-6\eta^{1/2}\big)k, 
& \,|P_1\cup Q_1|&\geq (\half\aI-58\eta^{1/2})k, \,\,\,\,\,\,\\
 |N|&\geq\big(\!\max\{\tfrac{3}{4}\aI+\tfrac{1}{4}\aII,\half\aIII\}-78\eta^{1/2}\big)k, 
&  |P_2\cup Q_2| &\geq (\half\aII-67\eta^{1/2})k.
\end{aligned}
\, \right\} \!\!
\end{equation}
 thus completing the proof of Claim~\ref{claimLASTplus2}.a.

\FloatBarrier

(b) Suppose $G$ is coloured as shown in Figures~\ref{f269}--\ref{f270f} and that we have $|Q_1|\geq 38\eta^{1/2}k$. Considering the blue graph, we are able to show that all edges present in $G[N\cup P,Q_1]$ are red as follows: Given (\ref{E9b-c}), since $|M|,|Q_1|\geq 38\eta^{1/2}k$, by Lemma~\ref{l:eleven}, there exists a blue connected-matching $B_S$ on at least $74\eta^{1/2}k$ vertices in $G[M,Q_1]$. Then, since $|N_2|,|P_2\cup Q_2|\geq (\half\aII-8\eta^{1/2})k$, if $G[N\cup P,Q_1]$ contains a blue edge, we can obtain a blue connected-matching on at least $\aII k$ vertices utilising edges from $B_S$ and from $G[N_2,P_2\cup Q_2]$. Thus, recalling (\ref{E3-c}), all edges present in $G[N\cup P,Q_1]$ are coloured exclusively red.
  \begin{figure}[!h]
\centering{
\mbox{\hspace{-2mm}{\,\,\,\includegraphics[width=64mm, page=90]{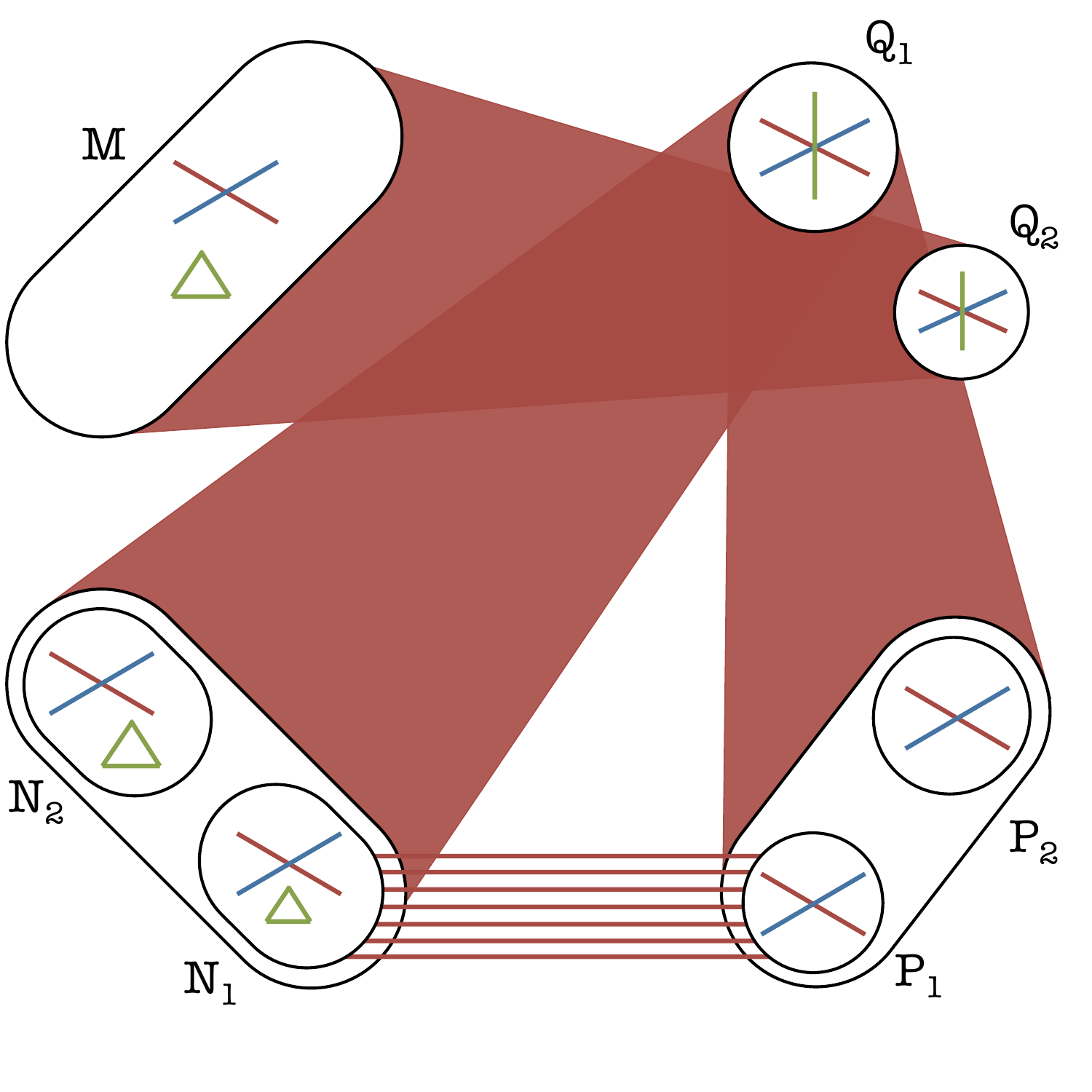}}\quad\quad\quad{\includegraphics[width=64mm, page=91]{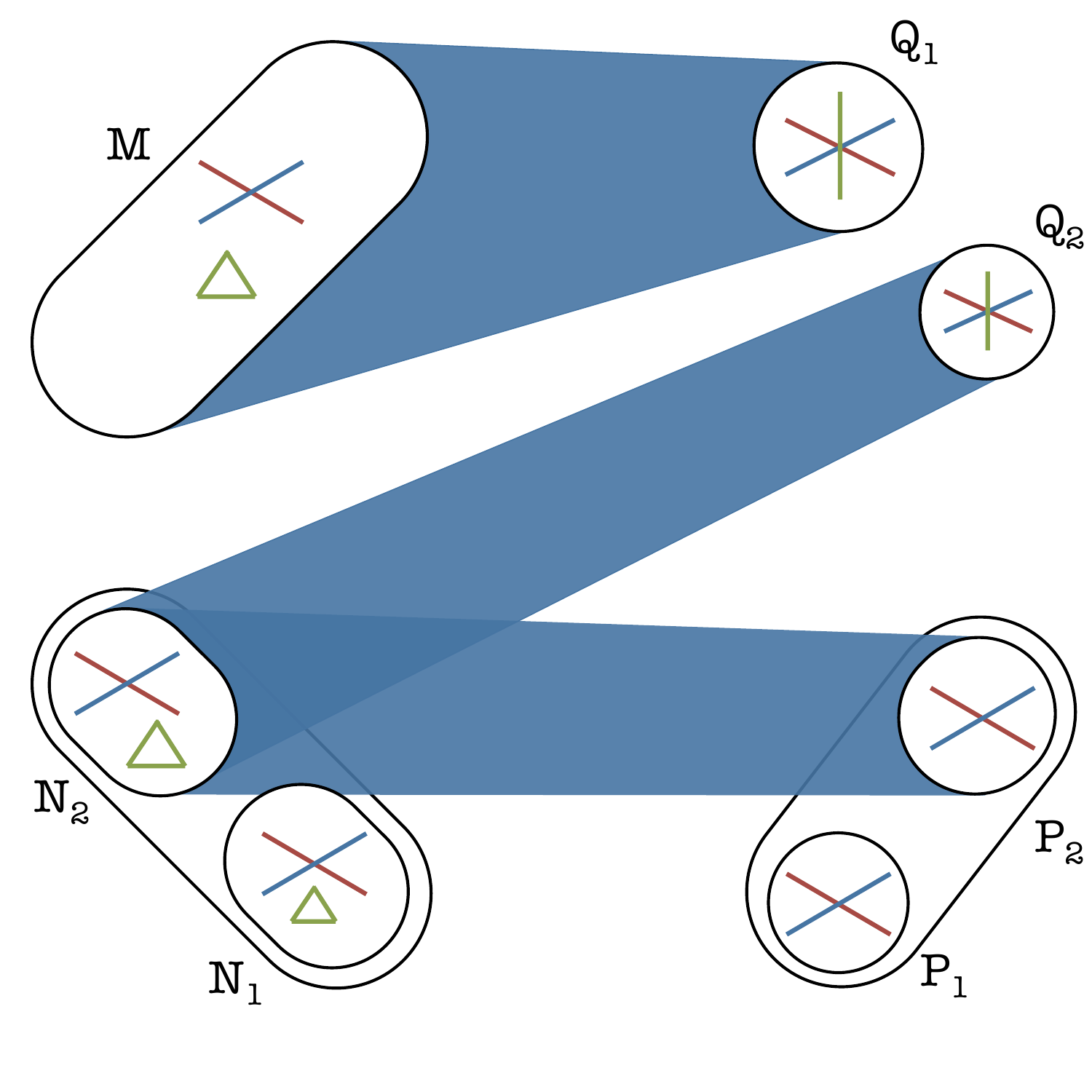}}}}\vspace{-5mm}
\caption{Red and blue subgraphs of $G$ when $|Q_1|\geq38\eta^{1/2}k$.}   
\end{figure}

Now, let $N_{11}=N_1\cap V(R)$ and $N_{12}=N_1\backslash V(R)$. Then, since $|Q_1|\geq 38\eta^{1/2} k$ and $|N_{11}|=|P_1|$, we have $|N_{12}|\geq38\eta^{1/2}k$. If, there exists a red matching $R_S$ on $18\eta^{1/2}k$ vertices in $G[N_{12},P_2]$, then, since $|N_2|,|Q_1|\geq (\half\aI-8\eta^{1/2})k-|P_1|$ and all edges in $G[N,Q_1]$ are coloured red, by Lemma~\ref{l:eleven}, there exists a red connected-matching $R_T$ on at least $(\aI-18\eta^{1/2})k-2|P_1|$ vertices in $G[N_2,Q_1]$. Since $|N_{11}|=|P_1|$, $R_U=R\cap G[N_{11},P_1]$ is a red matching on $2|P_1|$ vertices. Thus, since $G[N]$ has a single red effective-component, $R_S\cup R_T\cup R_U$ forms a red connected-matching on at least $\aI k$ vertices. Thus, after discarding at most $9\eta^{1/2}k$ vertices from each of $N_{12}$ and $P_2$, we may assume that all edges in $G[N_{12},P_2]$ are coloured exclusively blue. We then have $|N_{12}|\geq 28\eta^{1/2}k$ and $|P_2\cup Q_2|\geq (\aII k-17\eta^{1/2})k$.

\begin{figure}[!h]
\centering
\includegraphics[width=64mm, page=93]{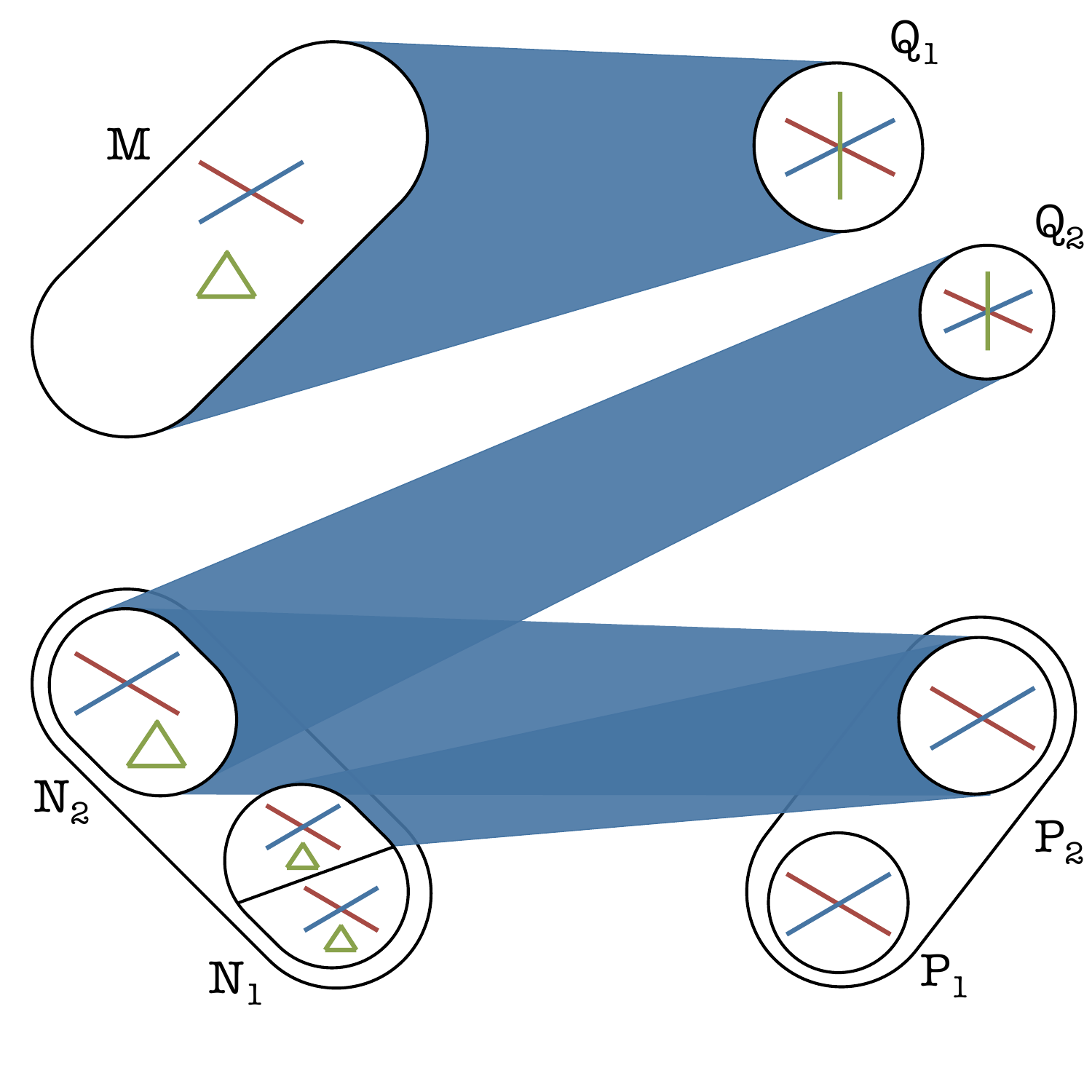}
\vspace{-5mm}\caption{Colouring of $G[N_{12},P_2]$. }
  \end{figure}

Then, suppose there exists a blue matching $B_S$ on $38\eta^{1/2}k$ vertices in $G[N,P_1]$. Then 
\begin{align*}
 |(N_2\cup N_{12})\backslash V(B_S)|\geq(\half\aII &-8\eta^{1/2})k+28\eta^{1/2}k-19\eta^{1/2}k\\
& \geq (\half\aII+\eta^{1/2})k
\geq|P_2\cup Q_2|\geq (\half\aII-17\eta^{1/2})k.
\end{align*} Thus, by Lemma~\ref{l:eleven}, there exists a blue connected-matching on at least $(\aII-38\eta^{1/2})k$ vertices belonging to the same component as $B_S$. Thus, together, these matchings form a blue connected-matching on at least $\aII k$ vertices. Therefore, after discarding at most $19\eta^{1/2}k$ vertices from each of $N$ and $P_1$, we may assume that all edges present in $G[N,P_1]$ are coloured exclusively red and that $|P_1\cup Q_1|\geq(\half\aI-27\eta^{1/2})k$.

Penultimately, suppose there exists a red matching $R_V$ on $58\eta^{1/2}k$ vertices in $G[N_{11},P_2]$. Then, since $|N\backslash V(R_V)|, |P_1\cup Q_1|\geq (\half\aI-27\eta^{1/2})k$, by Lemma~\ref{l:eleven}, there exists a red connected-matching $R_W$ on at least $(\aI-56\eta^{1/2})k$ vertices in $G[N\backslash V(R_V),P_1\cup P_2]$, which together with $R_V$ gives a red connected-matching on at least $\aI k$ vertices. Thus, after discarding at most $29\eta^{1/2}k$ vertices from each of $N_{11}$ and $P_2$, we may assume that all edges in $G[N,P_2]$ are coloured exclusively blue.

  \begin{figure}[!h]
\centering{
\mbox{\hspace{-2mm}{\,\,\,\includegraphics[width=64mm, page=92]{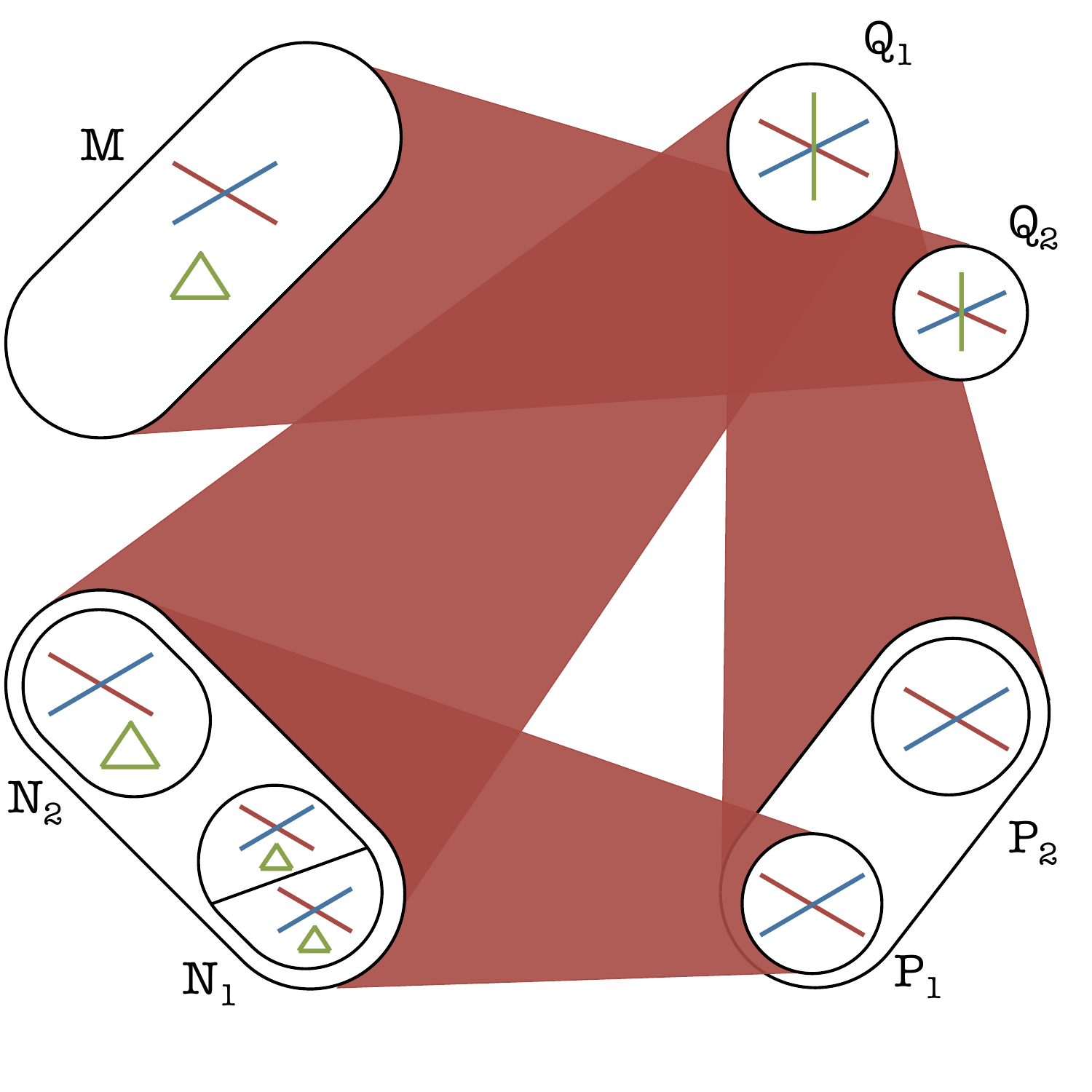}}\quad\quad\quad{\includegraphics[width=64mm, page=94]{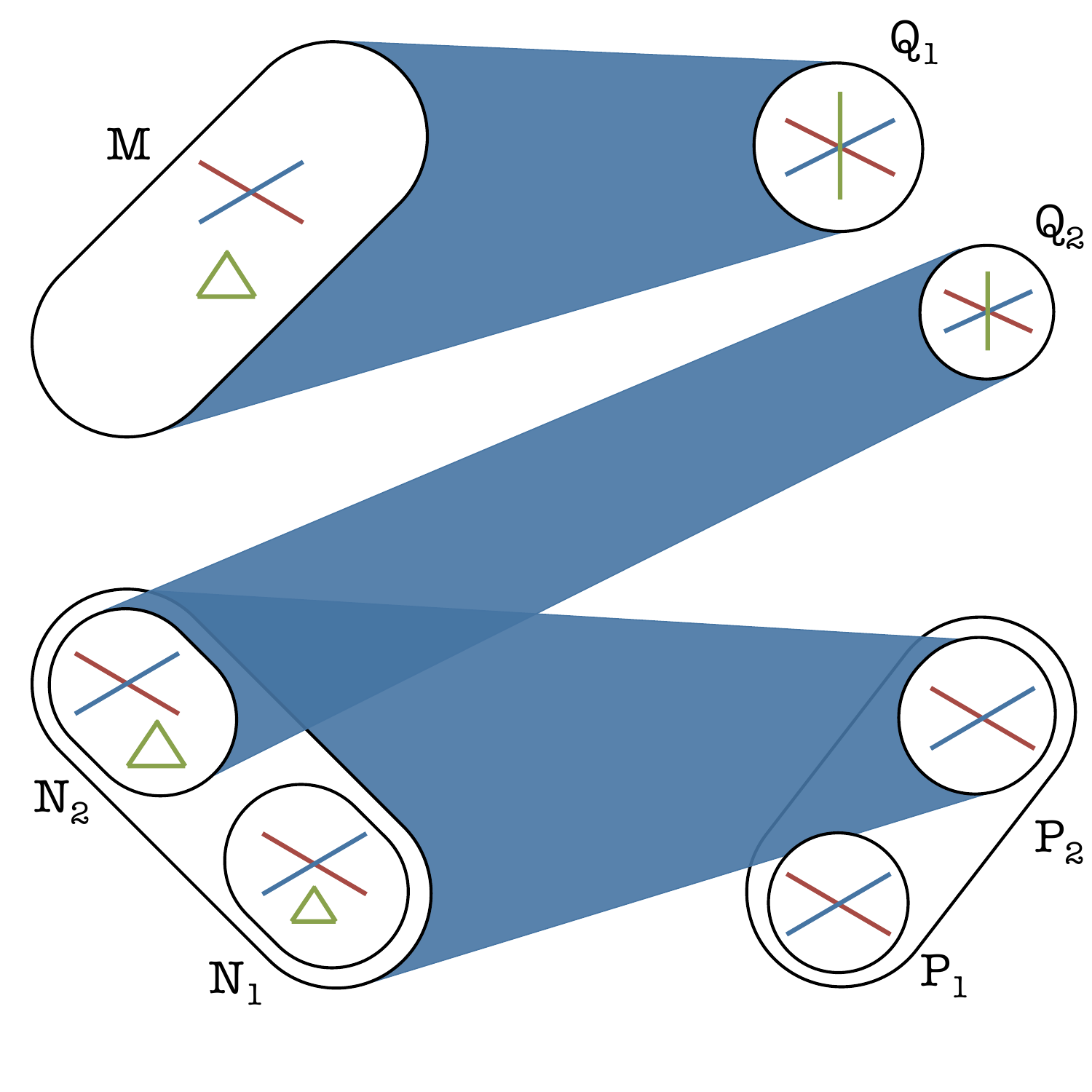}}}}\vspace{-5mm}
\caption{Colouring of $G[N,P]$.}   
\end{figure}

\FloatBarrier
Finally, observe that, if $|Q_2|\geq 30\eta^{1/2}k$, there can be no red edges in $G[N\cup P,Q_2]$. Indeed, in that case, since $|M|,|Q_2|\geq 30\eta^{1/2}k$ and $|N|,|P_1\cup Q_1|\geq (\half\aI -27\eta^{1/2})k$, there exist red connected-matchings $R_W$ in $G[M,Q_2]$ and $R_X$ in $G[P_1\cup Q_1,N]$ which belong to the same component and together span at least $\aI k$ vertices. Alternatively, if $|Q_2|\leq 30\eta^{1/2}k$, we can discard every vertex in $Q_2$, rendering the graph $G[N,Q_2]$ trivial. Thus, in either case, all edges in $G[N,Q_2]$ are coloured exclusively blue.

  \begin{figure}[!h]
\centering{
\mbox{\hspace{-2mm}{\,\,\,\includegraphics[trim=0mm 0mm 0mm 0mm, width=73mm, page=118, clip=true]{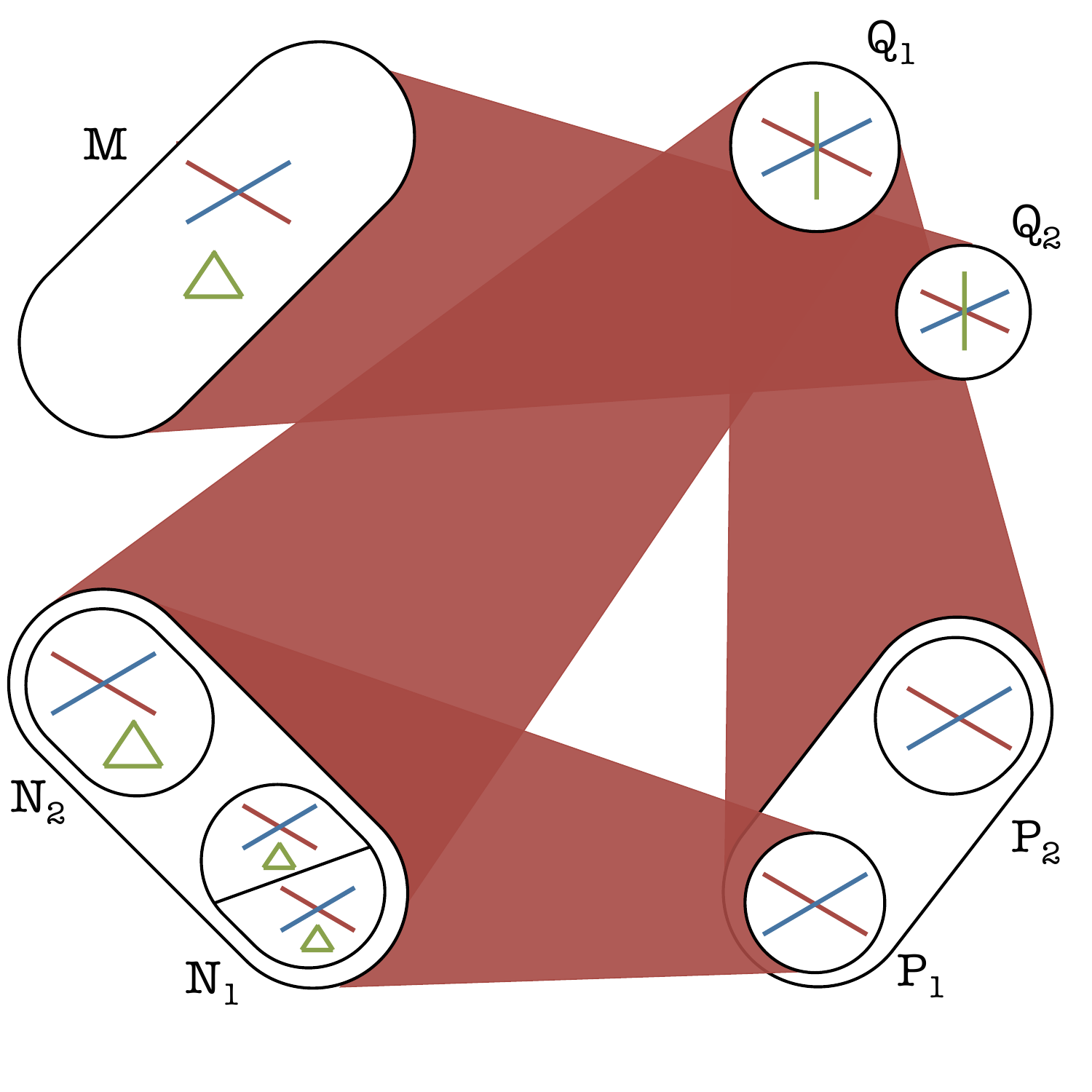}}
\quad\quad\quad{\includegraphics[width=64mm, page=96]{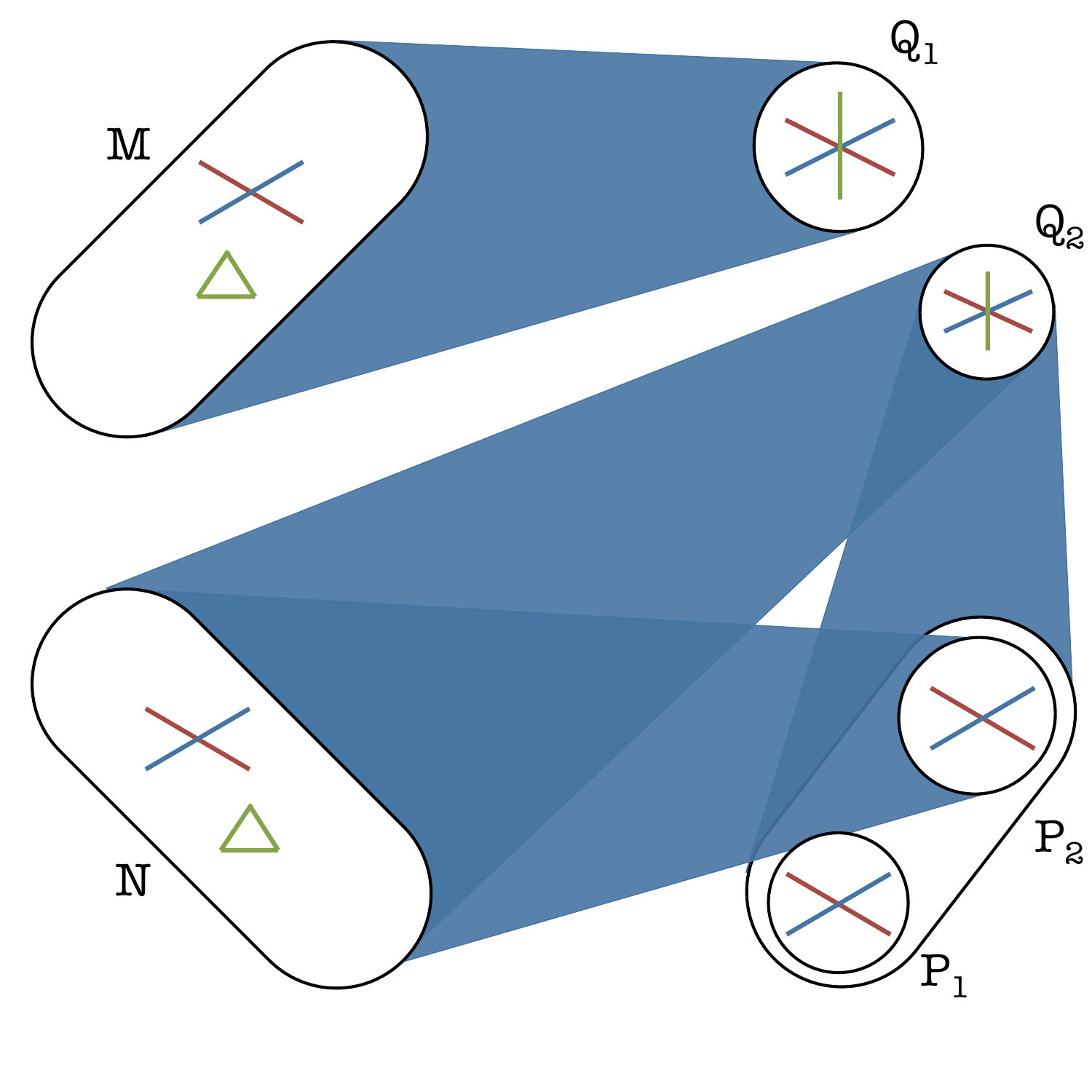}}}}\vspace{-5mm}
\caption{The final red and blue graphs in Claim~\ref{claimLASTplus2}.b. if $|Q_2|\geq 30\eta^{1/2}k$.}   
\end{figure}

In summary, having discarded some vertices, we may assume that all edges present in $G[N,P_1\cup Q_1]$ are coloured exclusively red, that all edges present in $G[N,P_2\cup Q_2]$ are coloured exclusively blue and that 
\begin{equation}
\label{M6}
\left.
\begin{aligned}
\,\,\,
 |M|&\geq\big(\!\max\{\tfrac{3}{4}\aI+\tfrac{1}{4}\aII,\half\aIII\}-6\eta^{1/2}\big)k, 
& \,|P_1\cup Q_1|&\geq (\half\aI-27\eta^{1/2})k, \,\,\,\,\,\,\\
 |N|&\geq\big(\!\max\{\tfrac{3}{4}\aI+\tfrac{1}{4}\aII,\half\aIII\}-64\eta^{1/2}\big)k, 
&  |P_2\cup Q_2| &\geq (\half\aII-76\eta^{1/2})k,
\end{aligned}
\, \right\} \!\!
\end{equation}
thus, completing the proof of Claim~\ref{claimLASTplus2}.b.
\end{proof}

Having proved the claim, we know that all edges in $G[N,P_1\cup Q_1]$ are coloured exclusively red and all edges in $G[N,P_2\cup Q_2]$ are coloured exclusively blue. Combining~(\ref{M6}) and~(\ref{M8}), we have
\begin{equation}
\label{M10}
\left.
\begin{aligned}
\,\,\, |M|&\geq\big(\!\max\{\tfrac{3}{4}\aI+\tfrac{1}{4}\aII,\half\aIII\}-6\eta^{1/2}\big)k, 
& \,|P_1\cup Q_1|&\geq (\half\aI-58\eta^{1/2})k, \,\,\,\,\,\,\\
 |N|&\geq\big(\!\max\{\tfrac{3}{4}\aI+\tfrac{1}{4}\aII,\half\aIII\}-78\eta^{1/2}\big)k, 
&  |P_2\cup Q_2| &\geq (\half\aII-76\eta^{1/2})k.
\end{aligned}
\, \right\} \!\!
\end{equation}

We now consider the green graph (see Figure~\ref{f270f}). Recall that all edges present in 
$G[M,N\cup P]$ are coloured green. Now, suppose there exists a green matching $G_S$ on $17\eta^{1/2}k$ vertices in $G[N]$. Since $|M|\geq|P|\geq 90\eta^{1/2}k$, by Lemma~\ref{l:eleven}, there exists a green connected-matching $G_T$ on $178\eta^{1/2}k$ vertices in $G[M,P]$. Then, by~(\ref{M10}), we have 
\begin{align*}
|M\backslash V(G_T)|,|N\backslash V(G_S)|\geq (\half\aIII - 96\eta^{1/2})k.
\end{align*} 
Thus, by Lemma~\ref{l:eleven}, there exists a green connected-matching $G_L$ on at least $(\aIII-194\eta^{1/2})k$ in $G[M\backslash V(G_T),N\backslash V(G_S)]$. Then, since all edges present in $G[M,N]$ are coloured exclusively green, $G_S, G_T$ and $G_L$ belong to the same green component of $G$ and, thus, together form an green odd connected-matching on at least $\aIII k$ vertices. Therefore, after discarding at most $17\eta^{1/2}k$ vertices from $N$, we may assume that there are no green edges in $G[N]$ and that 
\begin{equation}
\label{M11}
|N|\geq\big(\!\max\{\tfrac{3}{4}\aI+\tfrac{1}{4}\aII,\half\aIII\}-95\eta^{1/2}\big)k.
\end{equation}
We continue to consider $G[N]$. Recall that we know that all edges in $G[N,P_1\cup Q_1]$ are coloured exclusively red, that all edges in $G[N,P_2\cup Q_2]$ are coloured exclusively red and that. From~(\ref{M10}), 
\begin{align*}
|P_1\cup Q_1|&\geq (\half\aI-58\eta^{1/2})k, &  |P_2\cup Q_2| &\geq (\half\aII-76\eta^{1/2})k.
\end{align*}
Recalling (\ref{E4b-iii}) and (\ref{M11}), provided that $\eta\leq(\aI/2000)^2$, we have $$|N|\geq|P_1\cup Q_1|+300\eta^{1/2}k.$$
Then, suppose that there exists a red-matching $R_A$ on at least $118\eta^{1/2}k$ vertices in $G[N]$. Then, since $|N\backslash V(R_A)|\geq|P_1\cup Q_1|\geq(\half\aI-58\eta^{1/2})k$, by Lemma~\ref{l:eleven}, there exists a red connected-matching $R_B$ on at least $(\aI -118\eta^{1/2})k$ vertices in $G[N\backslash V(R_A),P_1\cup Q_1]$. Then, by (\ref{E6a-c}), $R_A$ and $R_B$ belong to the same red component of $G$ and thus, together, form a red connected-matching on at least $\aI k$ vertices.
Likewise, if there exists a blue-matching $B_A$ on at least $154\eta^{1/2}k$ vertices in $G[N]$, then we can obtain a blue connected-matching on at least $\aII k$ vertices, as follows: Since $|N\backslash V(R_B)|\geq|P_2\cup Q_2|\geq(\half\aI-76\eta^{1/2})k$, by Lemma~\ref{l:eleven}, there exists a blue connected-matching $B_B$ on at least $(\aI -154\eta^{1/2})k$ vertices in $G[N\backslash V(B_A),P_2\cup Q_2]$. Then, since all edges present in $G[N,Q_2]$ are coloured exclusively blue, $B_A$ and $B_B$, together, form a blue connected-matching on at least $\aII k$ vertices.

Thus, after discarding at most a further $272\eta^{1/2}k$ vertices from $N$, we can assume that there are no edges of any colour in $G[N]$ and that 
$$|N|\geq\big(\!\max\{\tfrac{3}{4}\aI+\tfrac{1}{4}\aII,\half\aIII\}-367\eta^{1/2}\big)k\geq 10\eta^{1/2}k.$$ This contradicts the fact that $G$ is $4\eta^4 k$-almost-complete, thus completing Case E.iii.b.iv.c. and the proof of Theorem~B.
\qed

\cleardoublepage
\phantomsection
\addcontentsline{toc}{section}{References}
\clearpage
\bibliographystyle{halpha2}
\bibliography{test}

\end{document}